\documentclass[10pt]{amsart}
\usepackage{amsmath}
\usepackage{amsxtra}
\usepackage{amscd}
\usepackage{amsthm}
\usepackage{amsfonts}
\usepackage{amssymb}
\usepackage{eucal}
\usepackage[all]{xy}
\usepackage{graphicx}

\usepackage[linktocpage=true]{hyperref}

\usepackage{xcolor}
\usepackage{soul}

\newtheorem{conj}[subsubsection]{Conjecture}
\newtheorem{prop}[subsubsection]{Proposition}

\newtheorem{cor}[subsubsection]{Corollary}
\newtheorem{lem}[subsubsection]{Lemma}
\newtheorem{defn}[subsubsection]{Definition}

\newtheorem{thm}[subsubsection]{Theorem}

\numberwithin{equation}{section}

\theoremstyle{remark}
\newtheorem{rem}[subsubsection]{Remark}

\newcommand{\lemref}[1]{Lemma~\ref{#1}}
\newcommand{\thmref}[1]{Theorem~\ref{#1}}
\newcommand{\secref}[1]{Sect.~\ref{#1}}
\newcommand{\corref}[1]{Corollary~\ref{#1}}
\newcommand{\propref}[1]{Proposition~\ref{#1}}

\newcommand{\nc}{\newcommand}

\nc{\ssec}{\subsection}
\nc{\sssec}{\subsubsection}

\nc{\renc}{\renewcommand}
\nc{\on}{\operatorname}
\nc\ol{\overline}
\nc\wt{\widetilde}
\nc{\Loc}{\on{Loc}}
\nc{\Bun}{\on{Bun}}
\nc{\BQ}{{\mathbb{Q}}}
\nc{\BA}{{\mathbb{A}}}
\nc{\BC}{{\mathbb{C}}}
\nc{\BE}{{\mathbb{E}}}
\nc{\BH}{{\mathbb{H}}}
\nc{\BG}{{\mathbb{G}}}
\nc{\BK}{{\mathbb{K}}}
\nc{\BN}{{\mathbb{N}}}
\nc{\BD}{{\mathbb{D}}}
\nc{\BV}{{\mathbb{V}}}
\nc{\BL}{{\mathbb{L}}}
\nc{\CA}{{\mathcal{A}}}
\nc{\CC}{{\mathcal{C}}}
\nc{\CG}{{\mathcal{G}}}
\nc{\CI}{{\mathcal{I}}}
\nc{\CJ}{{\mathcal{J}}}
\nc{\CO}{{\mathcal{O}}}
\nc{\CP}{{\mathcal{P}}}
\nc{\CR}{{\mathcal{R}}}
\nc{\CV}{{\mathcal{V}}}
\nc{\CW}{{\mathcal{W}}}
\nc{\CK}{{\mathcal{K}}}
\nc{\CM}{{\mathcal{M}}}
\nc{\CN}{{\mathcal{N}}}
\nc{\CL}{{\mathcal{L}}}
\nc{\CF}{{\mathcal{F}}}
\nc{\CX}{{\mathcal{X}}}
\nc{\CY}{{\mathcal{Y}}}
\nc{\CZ}{{\mathcal{Z}}}

\nc{\D}{{\mathcal{D}}}
\nc{\fd}{{\mathfrak{d}}}
\nc{\fg}{{\mathfrak{g}}}
\nc{\fD}{{\mathfrak{D}}}
\nc{\fh}{{\mathfrak{h}}}
\nc{\fl}{{\mathfrak{l}}}
\nc{\fn}{{\mathfrak{n}}}
\nc{\sM}{{\mathsf M}}
\nc{\ppart}{(\!(t)\!)}
\nc{\qqart}{[\![t]\!]}
\nc{\hg}{{\widehat\fg}}
\nc{\sA}{{\mathsf A}}
\nc{\sB}{{\mathsf B}}
\nc{\sF}{{\mathsf F}}
\nc{\sG}{{\mathsf G}}
\nc{\sk}{{\mathsf k}}
\nc{\sj}{{\mathsf j}}

\nc{\bC}{{\mathbf{C}}}
\nc{\bZ}{{\mathbf{Z}}}
\nc{\bD}{{\mathbf{D}}}
\nc{\bO}{{\mathbf{O}}}
\nc{\bU}{{\mathbf{U}}}
\nc{\bc}{{\mathbf{c}}}
\nc{\be}{{\mathbf{e}}}
\nc{\br}{{\mathbf{r}}}
\nc{\bM}{{\mathbf{M}}}
\nc{\bA}{{\mathbf{A}}}
\nc{\bK}{{\mathbf{K}}}

\nc{\oZ}{\overset{\circ}Z{}}
\nc{\oY}{\overset{\circ}Y{}}
\nc{\oX}{\overset{\circ}X{}}
\nc{\of}{\overset{\circ}f{}}
\nc{\oCX}{\overset{\circ}\CX{}}

\nc{\fW}{{\mathfrak{W}}}
\nc{\reg}{{\text{\rm reg}}}
\nc{\nilp}{{\text{\rm nilp}}}
\nc{\cG}{{\check{G}}}
\nc{\cB}{{\check{B}}}
\nc{\cg}{{\check{\fg}}}
\nc{\cb}{{\check{\fb}}}
\nc{\cn}{{\check{\fn}}}
\nc{\mer}{{\on{mer}}}
\nc{\Const}{\mathsf{Const}}
\nc{\Whit}{\on{Whit}}
\nc{\KL}{\on{KL}}
\nc{\FS}{\on{FS}}
\nc{\LocSys}{\on{LocSys}}
\nc{\QCoh}{\on{QCoh}}
\nc{\Coh}{\on{Coh}}
\nc{\IndCoh}{{\on{IndCoh}}}
\nc{\Cat}{\on{Cat}}
\nc{\Op}{\on{Op}}
\nc{\Gr}{\on{Gr}}
\nc{\Fl}{\on{Fl}}
\nc{\Rep}{\on{Rep}}
\renc{\mod}{{\on{-mod}}}
\nc{\Conn}{\on{Conn}}
\nc{\unit}{{\mathbf{1}}}

\nc{\Hom}{\on{Hom}}
\nc{\End}{\on{End}}
\nc{\Vect}{\on{Vect}}
\nc{\Av}{\on{Av}}
\nc{\Ind}{\on{Ind}}
\nc{\Spec}{\on{Spec}}
\nc{\KG}{K\backslash G}
\nc{\comult}{{co\text{-}mult}}
\nc{\counit}{{co\text{-}unit}}
\nc{\uHom}{{\underline{\Hom}}}
\nc{\dgSch}{\on{DGSch}}
\nc{\dgindSch}{\on{DGindSch}}
\nc{\indSch}{\on{indSch}}
\nc{\Sch}{\on{Sch}}
\nc{\affdgSch}{\on{DGSch}^{\on{aff}}}
\nc{\affSch}{\on{Sch^{\on{aff}}}}

\nc{\Groupoids}{\on{Grpd}}
\nc{\inftygroup}{\infty\on{-Grpd}}
\nc{\inftyPicgroup}{\infty\on{-PicGrpd}}
\nc{\inftyCat}{\infty\on{-Cat}}
\nc{\StinftyCat}{\on{DGCat}}
\nc{\MoninftyCat}{\infty\on{-Cat}^{Mon}}
\nc{\SymMoninftyCat}{\infty\on{-Cat}^{SymMon}}
\nc{\SymMonStinftyCat}{\on{DGCat}^{SymMon}}
\nc{\MonStinftyCat}{\on{DGCat}^{Mon}}
\nc{\inftystack}{\on{Stk}}
\nc{\inftystackalg}{Stk^{1\text{-}alg}}
\nc{\inftyprestack}{\on{PreStk}}
\nc{\inftydgnearstack}{\on{NearStk}}
\nc{\inftydgstack}{\on{Stk}}
\nc{\inftydgstackalg}{DGStk^{1\text{-}alg}}
\nc{\inftydgprestack}{\on{PreStk}}

\nc{\mmod}{{\on{-}{\mathbf{mod}}}}
\nc{\wh}{\widehat}

\nc{\nDG}{^{\leq n}\!\on{DG}}

\nc{\Maps}{\on{Maps}}
\nc{\CMaps}{{\mathcal Maps}}

\begin{document}

\title[Indschemes]{DG indschemes}

\author{Dennis Gaitsgory and Nick Rozenblyum}

\date{\today}

\dedicatory{To Igor Frenkel on the occasion of his 60th birthday}

\begin{abstract}
We develop the notion of indscheme in the context of derived algebraic geometry, 
and study the categories of quasi-coherent sheaves and ind-coherent sheaves on 
indschemes. The main results concern the relation between classical and
derived indschemes and the notion of formal smoothness.
\end{abstract}

\maketitle

\tableofcontents

\section*{Introduction}

\ssec{What is this paper about?}

The goal of this paper is to develop the foundations of the theory of indschemes, especially in the context
of derived algebraic geometry.

\sssec{}

The first question to ask here is ``why bother"? For, it is more
or less clear what DG indschemes are: functors on the category of affine DG schemes, i.e.,
$\infty$-prestacks in the terminology of \cite{Stacks}, that can be written as filtered colimits
of DG schemes with transition maps being closed embeddings. 

\medskip

The definition of the category
of quasi-coherent sheaves on a DG indscheme $\CX$ is also automatic: the category $\QCoh(\CX)$
is defined on any $\infty$-prestack (see \cite[Sect. 1.1]{QCoh} or \cite[Sect. 2.7]{Lu1}), and 
in particular on a DG indscheme.

\medskip

Here is, however, the question, which started life as a remark in another paper, but answering
which in detail was one of the main reasons for writing the present one:  

\sssec{}

Consider the affine Grassmannian $\Gr_G$ corresponding to an algebraic group $G$. This is an
indscheme that figures prominently in the geometric Langlands program. We would like to consider
the category $\QCoh(\Gr_G)$ of quasi-coherent sheaves on $\Gr_G$.\footnote{The other main result 
of this paper, also of direct relevance to geometric Langlands,
is described in \secref{sss:descr of main} below. It expresses the category $\QCoh(\Gr_G)$ in terms
of the corresponding category of ind-coherent sheaves on $\Gr_G$.} However, a moment's
reflection leads one to conclude that the expression $\QCoh(\Gr_G)$ is ambiguous. Namely,
the affine Grassmannian itself can be understood in two, a priori different, ways. 

\medskip

Recall that, as a functor on the category of commutative algebras, $\Gr_G$ assigns to
a commutative algebra $A$ the groupoid of $G$-torsors over $\Spec(A\qqart)$ with a trivialization over
$\Spec(A\ppart)$.

\medskip

Now, we can first take $A$'s to be \emph{classical}, i.e., non-derived, commutative algebras, and 
thus consider $\Gr_G$ as a classical indscheme. Let us denote this version of
$\Gr_G$ by $^{cl}\!\Gr_G$. As for any classical indscheme, we can consider the category
$\QCoh({}^{cl}\!\Gr_G)$.

\medskip

The second possibility is to take $A$'s to be DG algebras, and thus consider $\Gr_G$ right 
away as an object of derived algebraic geometry. Thus, we obtain a different version of $\QCoh(\Gr_G)$. 

\medskip

There is a natural functor 
\begin{equation} \label{e:two versions of Gr}
\QCoh(\Gr_G)\to \QCoh({}^{cl}\!\Gr_G), 
\end{equation}
and our initial question was whether or not it is an equivalence. 

\medskip

If it were not an equivalence, 
it would signify substantial trouble for the geometric Langlands community: on the
one hand, $^{cl}\!\Gr_G$ is a familiar object that people have dealt with for some time
now. However, it is clear that the $\Gr_G$ is ``the right object to consider" if we ever
want to mix derived algebraic geometry into our considerations, which we inevitably
do. \footnote{One might raise an objection to the relevance of the above question 
by remarking that for geometric Langlands we mainly consider D-modules on
$\Gr_G$, and those only depend on the underlying classical indscheme. However,
this is not accurate, since along with D-modules, we consider their global sections
as quasi-coherent sheaves, and the latter do depend on the scheme-theoretic
structure.}

\medskip

To calm the anxious reader, let us say that the functor \eqref{e:two versions of Gr}
is an equivalence, as is guaranteed by \thmref{t:aff gr} of the present paper. 

\medskip

In fact, we show that $\Gr_G$ is ``the same as" $^{cl}\!\Gr_G$, in the sense that the former
is obtained from the latter by the natural procedure of turning classical schemes/indschemes/$\infty$-stacks
into derived ones, \footnote{This procedure is the left Kan extension along the embedding
$\affSch\hookrightarrow \affdgSch$, followed by sheafification in the fppf topology.} 
which preserves the operation of taking $\QCoh$ (see \cite[Lemma 1.2.5]{QCoh} for the latter statement). 

\sssec{}

Another result along these lines, \propref{p:class vs derived formal}, concerns formal completions.

\medskip

Let $X$ be a classical scheme and $Y\subset X$ a Zariski-closed subset. Consider
the formal completion $X^\wedge_Y$. By definition, as a functor on commutative
algebras, $X^\wedge_Y$ assigns to a ring $A$ the groupoid of maps $\Spec(A)\to X$,
such that their image is, set-theoretically, contained in $Y$. 

\medskip

However, again there are two ways to understand $X^\wedge_Y$: as a classical indscheme, 
which we then turn into a DG indscheme by the procedure mentioned above. Or,
we can consider it as a functor of DG algebras, obtaining a DG indscheme right away. 

\medskip

In \propref{p:class vs derived formal} we show that, under the assumption that $X$
is Noetherian, the above two versions of $X^\wedge_Y$ are isomorphic.

\medskip

So, by and large, this paper is devoted to developing the theory in order to prove the above and
similar results.

\ssec{What is done in this paper}

We shall presently proceed to review the main results of this paper
(not necessarily in the order in which they appear in the paper). 

\medskip

We should say that none of these results is really surprising. Rather, they are all
in the spirit of ``things work as they should." \footnote{For the duration of the paper we
make the technical assumption that our DG indschemes are what one could call ``ind-quasi compact
and ``ind-quasi separated."}

\sssec{DG indschemes via deformation theory}

The first theorem of this paper, \thmref{t:char by deform}, addresses the following issue.
Let $\CX$ be an $\infty$-prestack, such that the underlying classical $\infty$-prestack
is a classical indscheme. What are the conditions that would guarantee that $\CX$
is itself a DG indscheme?

\medskip

There is a natural guess: since DG algebras can be thought of as infinitesimal
deformations of classical algebras, if we know the behavior of the functor $\CX$
on the latter, its behavior on the former should be governed by deformation theory.

\medskip

By deformation theory we mean the following: if an algebra $A'$ is the extension of
an algebra $A$ by a square-zero ideal $\CI$, then the groupoid of extensions of a given map
$x:\Spec(A)\to \CX$ to a map $x':\Spec(A')\to \CX$ is determined by the 
\emph{cotangent space} to $\CX$ at $x$, denoted $T^*_x\CX$, which is 
understood just as a functor on the category of $\CI$'s, i.e., on $A\mod$. 

\medskip

If we expect $\CX$ to be a DG indscheme, then the functor 
\begin{equation} \label{e:cotangent as a functor}
T^*_x\CX:A\mod\to \inftygroup
\end{equation}
must have certain properties: for a given algebra $A$, as well as for
algebra homomorphisms $A\to B$.  If an abstract $\infty$-prestack $\CX$ 
has these properties, we shall say that $\CX$ \emph{admits connective deformation theory}.

\medskip

Our \thmref{t:char by deform} asserts that if $\CX$ is such that its underlying classical
$\infty$-prestack is a classical indscheme, and if $\CX$ admits connective deformation theory,
then it is a DG indscheme.

\sssec{Formal smoothness} 

Let us recall the notion of formal smoothness for a classical scheme, or more generally for 
a classical $\infty$-prestack, i.e.,
a functor
\begin{equation} \label{e:for cl formal smoothness}
\CX:(\affSch)^{\on{op}}\to \inftygroup.
\end{equation}
We say that $\CX$ is \emph{formally smooth} if whenever $S\to S'$ is a nilpotent embedding
(i.e., a closed embedding with a nilpotent ideal),
then the restriction map
$$\pi_{0}(\CX(S'))\to \pi_{0}(\CX(S))$$
is surjective.

\medskip

The notion of formal smoothness in the DG setting is less evident. We formulate it as follows.
Let $\CX$ be an $\infty$-prestack, i.e., just a functor
$$(\affdgSch)^{\on{op}}\to \inftygroup.$$
We say that it is formally smooth if:

\begin{itemize}

\item When we restrict $\CX$ to classical affine schemes,
the resulting functor as in \eqref{e:for cl formal smoothness}, is formally smooth
in the classical sense.

\item For an affine DG scheme $S=\Spec(A)$, the $i$-th homotopy
group of the $\infty$-groupoid $\CX(S)$ depends only on the truncation $\tau^{\geq -i}(A)$
(i.e., a map $A_1\to A_2$ that induces an isomorphisms of the $i$-th truncations should
induce an isomorphism of $\pi_i$'s of $\CX(\Spec(A_1))$ and $\CX(\Spec(A_2))$.
\footnote{It is quite possible that a more reasonable definition in
both the classical and derived contexts is when the corresponding properties 
take place not ``on the nose", but after Zariski/Nisnevich/\'etale localization. It is likely
that the notion of formal smoothness defined as above is only sensible for $\infty$-prestacks
that are ``locally of finite type", or more generally of Tate type.}

\end{itemize}

\medskip

It is well-known that if a classical scheme \emph{of finite type} is classically formally smooth,
then it is actually smooth. This implies that it is formally smooth also when viewed as a derived
scheme. \footnote{We do not know whether the latter is true in general without the
finite type hypothesis.}

\medskip

The question we consider is whether the same is true for indschemes. Namely,
if $\CX$ is a classical indscheme, which is classically formally smooth, and \emph{locally of finite type},
is it true that it will be formally smooth also as a DG indscheme? (By ``as a DG indscheme"  we mean the procedure
of turning classical $\infty$-stacks into derived ones by the procedure mentioned above.)

\medskip

The answer turns out to be ``yes", under some additional technical hypotheses, see 
\thmref{t:classical vs derived ft}.  

\medskip

Moreover, the above theorem formally implies that (under the same additional hypotheses),
every formally smooth DG indscheme is classical, i.e., is obtained by the above procedure
from a classical formally smooth indscheme.

\medskip

The theorem about the affine Grassmannian mentioned above is an easy corollary of this
result. 

\sssec{Loop spaces}

We don't know whether \thmref{t:classical vs derived ft} remains valid  if one omits
the locally finite type hypothesis. It is quite possible that this hypothesis is essential. 
However, we do propose the following conjecture:

\medskip

Let $Z$ be a classical affine scheme of finite type, which is smooth. Consider the 
corresponding DG indscheme $Z\ppart$ (see \secref{sss:loops setting} for the definition).
It is easy to see that it is formally smooth. 

\medskip

We conjecture that, although $Z\ppart$ is not locally of finite type, it is classical. The evidence
for this is provided by \cite[Theorem 6.4]{Dr}. This theorem says that $Z\ppart$ violates the 
locally finite type condition by factors isomorphic to the infinite-dimensional affine
space, and the latter does not affect the property of being classical. 

\medskip

We prove this conjecture in the special case when $Z$ is an algebraic group $G$.

\ssec{Quasi-coherent and ind-coherent sheaves on indschemes}

With future applications in mind, the focus of this paper is the categories
$\IndCoh(\CX)$ and $\QCoh(\CX)$ of ind-coherent and quasi-coherent
on a DG indscheme $\CX$.  \footnote{We the refer the reader to \cite{IndCoh} where the category 
$\IndCoh(\CX)$ on a prestack $\CX$ is studied. For it to be defined, $\CX$ needs to
be locally almost of finite type (see \cite[Sect. 1.3.9]{Stacks} for what the latter means).}

\medskip

We shall now proceed to state the main result of this paper.

\sssec{Comparison of $\QCoh$ and $\IndCoh$ on the loop group}  \label{sss:descr of main}

Let us return to the situation of the affine Grassmannian $\Gr_G$, or rather, the loop group $G\ppart$.
As we now know, both of these DG indschemes are classical. 

\medskip

In the study of local geometric Langlands, one considers the notion of category acted on by
the loop group $G\ppart$. This notion may be defined in two, a priori, different ways:

\medskip

\noindent(a) As a \emph{co-action} of the \emph{co-monoidal} category $\QCoh(G\ppart)$, where the
co-monoidal structure is given by pullback with respect to the multiplication map on $G\ppart$.

\medskip

\noindent(b) As an \emph{action} of the \emph{monoidal} category $\IndCoh(G\ppart)$, where the
monoidal structure is given by push-forward with respect to the same multiplication map.
\footnote{We should remark that when talking about $\IndCoh(G\ppart)$, we are leaving the realm of documented
mathematics, as $G\ppart$ is not locally of finite type. However, it is not difficult to give a definition of
$\IndCoh$ ``by hand" in the particular case of $G\ppart$, using the affine Grassmannian.}

\medskip

Obviously, one would like these two notions to coincide. This leads one to believe that the
corresponding categories $\QCoh(G\ppart)$ and $\IndCoh(G\ppart)$ are duals of one another
(duality is understood here in the sense of \cite[Sect. 2.1]{DG}). 

\medskip

Moreover, unless we prove something about $\QCoh(G\ppart)$, it would be a rather unwieldy object, 
as $\QCoh(\CX)$ is for a general DG indscheme $\CX$. For instance, we would not know that it is compactly generated, etc. 

\sssec{}

To formulate a precise statement, we shall return to the case of the affine Grassmannian. We claim
that the functor
$$\QCoh(\Gr_G)\to \IndCoh(\Gr_G)$$
given by tensoring with the dualizing sheaf $\omega_{\Gr_G}\in \IndCoh(\Gr_G)$ is an equivalence.

\medskip

In fact, we prove \thmref{t:QCoh main} that asserts that a similarly defined functor is an equivalence
for any formally smooth DG indscheme locally of finite type (with an additional technical hypothesis). 

\medskip

This theorem was originally stated and proved by J.~Lurie in 2008. 

\medskip

We give a different proof,
but it should be noted that Lurie's original proof was much more elegant. The reason we do not reproduce
it here is that it uses some not yet documented facts about Ext computations on indschemes.

\sssec{$\QCoh$ and $\IndCoh$ on formal completions}

Another set of results we establish concerning $\QCoh$ and $\IndCoh$ is the following.

\medskip

In order to prove \thmref{t:QCoh main} mentioned above, we have to analyze 
in detail the behavior of the categories $\QCoh$ and $\IndCoh$ on a DG indscheme
obtained as a formal completion $X^\wedge_Y$ of a 
DG scheme $X$ along a Zariski-closed subset $Y$. 

\medskip

We show that the category $\QCoh(X^\wedge_Y)$ (resp., $\IndCoh(X^\wedge_Y)$) is equivalent
to the localization of $\QCoh(X)$ (resp., $\IndCoh(X)$) with respect to $\QCoh(U)$ 
(resp., $\IndCoh(U)$), where $U=X\on{-}Y$.

\medskip

This implies some favorable properties of $\QCoh(X^\wedge_Y)$, e.g., that it is compactly 
generated (something, which is not necessarily true for an arbitrary indscheme). We also
endow $\QCoh(X^\wedge_Y)$ with two different t-structures, one compatible with pullbacks
from $X$, and another compatible with push-forwards to $X$.

\medskip

In addition, we show that the functors $\Psi,\Xi,\Psi^\vee,\Xi^\vee$ that act 
between $\QCoh$ and $\IndCoh$ (see \cite[Sects. 1.1, 1.5, 9.3 and 9.6]{IndCoh}) are
compatible for $X^\wedge_Y$ and $X$ under the push-forward and pullback
functors.

\ssec{Conventions and notation}

Our conventions follow closely those of \cite{IndCoh}. Let us recall the most essential ones. 

\sssec{The ground field}

Throughout the paper we will be working over a fixed ground field $k$.  
We assume that $\on{char}(k)=0$.

\sssec{$\infty$-categories}

By an $\infty$-category we always mean an $(\infty,1)$-category. By a slight abuse
of language we will sometimes talk about ``categories" when we actually mean 
$\infty$-categories. Our usage of $\infty$-categories is not tied to any particular model,
but it is their realization as quasi-categories that we actually have in mind, the basic reference
to which is \cite{Lu0}.

\medskip

By $\inftygroup$ we denote the $\infty$-category of $\infty$-groupoids, which is the
same as the category ${\mathcal S}$ of spaces in the notation of \cite{Lu0}.

\medskip

There is a natural functor
$$ \inftyCat \rightarrow \inftygroup $$
which is the right adjoint of the inclusion functor.  It sends an $\infty$-category $\bC$ to its maximal subgroupoid, 
which we will denote by $\bC^{\on{grpd}}$.  I.e., $\bC^{\on{grpd}}$ is obtained from $\bC$ by discarding 
the non-invertible $1$-morphisms.

\medskip

For an $\infty$-category $\bC$, and $x,y\in \bC$, we shall denote by
$\on{Maps}_\bC(x,y)\in \inftygroup$ the corresponding mapping space. By
$\Hom_\bC(x,y)$ we denote the set $\pi_0(\on{Maps}_\bC(x,y))$, i.e.,
what is denoted $\Hom_{h\bC}(x,y)$ in \cite{Lu0}. 

\medskip

When working in a fixed $\infty$-category $\bC$, for two objects $x,y\in \bC$,
we shall call a point of $\on{Maps}_\bC(x,y)$ an \emph{isomorphism} what is in
\cite{Lu0} is called an \emph{equivalence}. I.e., a map that admits a homotopy 
inverse. We reserve the word ``equivalence" to mean a (homotopy) equivalence
between $\infty$-categories.

\sssec{Subcategories}

Let $\phi:\bC'\to \bC$ be a functor between $\infty$-categoris. 

\medskip

We shall say that $\phi$ is \emph{0-fully faithful}, or just \emph{fully faithful}
if for any $\bc'_1,\bc'_2\in \bC'$, the map
\begin{equation} \label{e:map on Homs}
\on{Maps}_{\bC'}(\bc'_1,\bc'_2)\to \on{Maps}_{\bC}(\phi(\bc'_1),\phi(\bc'_2))
\end{equation}
is an isomorphism (=homotopy equivalence)
of $\infty$-groupoids. In this case we shall say that $\phi$ makes $\bC'$ into a \emph{0-full}
(or just \emph{full}) subcategory of $\bC$. 

\medskip

We also consider two weaker notions:

\medskip

We shall say that $\phi$ is \emph{1-fully faithful}, or just \emph{faithful}, if for any $\bc'_1,\bc'_2\in \bC'$, 
the map \eqref{e:map on Homs} is a fully faithful map of $\infty$-groupoids. Equivalently, the map 
\eqref{e:map on Homs} induces an injection on $\pi_0$ and a bijection on the homotopy groups $\pi_i$, $i\geq 1$ 
on each connected component of the space $\on{Maps}_{\bC'}(\bc'_1,\bc'_2)$.

\medskip

I.e., $2$- and higher morphisms between $1$-morphisms in $\bC'$ are the same in $\bC'$
and $\bC$, up to homotopy. 

\medskip

We shall say that $\phi$ is \emph{faithful and groupoid-full} if it is faithful, and 
for any $\bc'_1,\bc'_2\in \bC'$, the map
\eqref{e:map on Homs} is surjective on those connected components of 
$\on{Maps}_{\bC}(\phi(\bc'_1),\phi(\bc'_2))$ that correspond to isomorphisms.  In other words, $\phi$ is faithful and groupoid-full if it is faithful and the restriction
$$ \phi^{\on{grpd}}: \bC'^{\on{grpd}} \rightarrow \bC^{\on{grpd}} $$
is fully faithful.  In this case, we shall say that $\phi$ makes $\bC'$ into a \emph{1-full}
subcategory of $\bC$. 

\sssec{DG categories}

Our conventions regarding DG categories follow \cite[Sects. 0.6.4 and 0.6.5]{IndCoh}.

\medskip

In particular, we denote by $\Vect$ the DG category of chain complexes of $k$-vector
spaces.

\medskip

Unless specified otherwise, we will only consider continuous
functors between DG categories (i.e., exact functors that commute
with direct sums, or equivalently, with all colimits). In other words,
we will be working in the category $\StinftyCat_{\on{cont}}$ in the
notation of \cite{DG}. \footnote{One can replace $\StinftyCat_{\on{cont}}$
by (the equivalent) $(\infty,1)$-category of stable presentable 
$\infty$-categories tensored over $\Vect$, with colimit-preserving functors.}

\medskip

For a DG category $\bC$ and $\bc_1,\bc_2\in \bC$ we let
$$\CMaps_\bC(\bc_1,\bc_2)$$
denote the corresponding object of $\Vect$. We can regard $\CMaps_\bC(\bc_1,\bc_2)$
as a not necessarily connective spectrum and thus identify
$$\Maps_\bC(\bc_1,\bc_2)=\Omega^\infty(\CMaps_\bC(\bc_1,\bc_2)).$$

\medskip

For a DG category $\bC$ equipped with a t-structure, we denote by $\bC^{\leq n}$
(resp., $\bC^{\geq m}$, $\bC^{\leq n,\geq m}$) the corresponding full subcategories.
The inclusion $\bC^{\leq n}\hookrightarrow \bC$ admits
a right adjoint denoted by $\tau^{\leq n}$, and similarly, for the other
categories. We let $\bC^\heartsuit$ denote the heart of the t-structure, and by $H^i:\bC\to \bC^\heartsuit$ 
the functor of $i$th cohomology with respect to our t-structure. Note that if 
$\bc\in \bC^{\leq n}$ (resp., $\bC^{\geq m}$) then $H^i(\bc)=0$
for $i>n$ (resp., $i<m$), but the converse is not true, unless the t-structure is \emph{separated}. 

\sssec{(Pre)stacks and DG schemes}

Our conventions regarding (pre)stacks and DG schemes follow \cite{Stacks}:

\medskip

Let $\affdgSch$ denote the $\infty$-category opposite to that of \emph{connective}
commutative DG algebras over $k$.

\medskip

The category $\inftydgprestack$ of prestacks is by definition that of all accessible\footnote{Recall that an accessible functor is one which commutes with $\kappa$-filtered colimits for some regular cardinal $\kappa$.  This condition ensures that we can avoid set theoretic difficulties when dealing with categories which are not small. See \cite{Lu0} for a discussion of accessible $\infty$-categories and functors.}
 functors
$$(\affdgSch)^{\on{op}}\to \inftygroup.$$
The category $\inftydgstack$ is a full subcategory in $\inftydgprestack$
that consists of those functors that satisfy fppf descent (see \cite[Sect. 2.2]{Stacks}).
This inclusion admits a left adjoint, denoted $L$, referred to as the \emph{sheafification
functor}.

\medskip

We remark that for the purposes of the current paper, the fppf topology can be replaced
by the \'etale, Nisnevich or Zariski  topology: all we need is that a non-affine (DG) scheme 
be isomorphic to the colimit, taken in the category of stacks, of its affine open subschemes. 

\ssec{The notion of $n$-coconnectivity for (pre)stacks}
For the reader's convenience, in this subsection, we briefly review
the material of \cite{Stacks} related to the notion of $n$-connectivity.

\sssec{}

Let $n$ be a non-negative integer. 

\medskip

We denote by $^{\leq n}\!\affdgSch$ the full subcategory of $\affdgSch$ that consists
of affine DG schemes $S=\Spec(A)$, such that $H^{-i}(A)=0$ for $i>n$. 
We shall refer to objects of this category as ``$n$-coconnective affine DG schemes."
When $n=0$ we shall also use the terminology ``classical affine schemes", and
denote this category by $\affSch$. 

\medskip

The inclusion $^{\leq n}\!\affdgSch\hookrightarrow \affdgSch$ admits a right
adjoint given by cohomological truncation below degree $-n$; we denote this
functor by $S\mapsto \tau^{\leq n}(S)$.

\sssec{The case of prestacks}

In this paper, we make extensive use of the operation of restricting a prestack 
$\CY$ to the subcategory $^{\leq n}\!\affdgSch$. We denote this functor by
$$\CY\mapsto {}^{\leq n}\CY:\inftydgprestack\to {}^{\leq n}\!\inftydgprestack,$$
where $^{\leq n}\!\inftydgprestack$ is by definition the category of
all functors
$({}^{\leq n}\!\affdgSch)^{\on{op}}\to \inftygroup$.

\medskip

The above restriction functor admits a (fully faithful) left adjoint, given by left
Kan extension along $^{\leq n}\!\affdgSch\hookrightarrow \affdgSch$; we denote it
by 
$$\on{LKE}_{({}^{\leq n}\!\affdgSch)^{\on{op}}\hookrightarrow (\affdgSch)^{\on{op}}}:{}^{\leq n}\!\inftydgprestack\to
\inftydgprestack.$$
The composition
$$\CY\mapsto \on{LKE}_{({}^{\leq n}\!\affdgSch)^{\on{op}}\hookrightarrow (\affdgSch)^{\on{op}}}({}^{\leq n}\CY)$$
is a colocalization functor on $\inftydgprestack$; we denote it by $\CY\mapsto \tau^{\leq n}(\CY)$.
When $\CY$ is an affine scheme $S$, this coincides with what was denoted above
by $\tau^{\leq n}(S)$.

\medskip

We shall say that a prestack $\CY$ is $n$-\emph{coconnective} if it belongs to the essential
image of $\on{LKE}_{({}^{\leq n}\!\affdgSch)^{\on{op}}\hookrightarrow (\affdgSch)^{\on{op}}}$, or equivalently if
the canonical map $\tau^{\leq n}(\CY)\to \CY$ is an isomorphism. 

\medskip

Thus, the functors of restriction
and left Kan extension identify $^{\leq n}\!\inftydgprestack$ with the full subcategory of 
$\inftydgprestack$ spanned by $n$-\emph{coconnective} prestacks. 

\medskip

We shall say that $\CY$ is \emph{eventually coconnective} if it is $n$-coconnective for some $n$. 

\medskip 

We shall refer to objects of $^{\leq 0}\!\inftydgprestack$ as ``classical prestacks";
we shall denote this category also by $^{cl}\!\inftydgprestack$.
By the above, the category of classical prestacks is canonically equivalent to that
of $0$-coconnective prestacks.

\sssec{The notion of $n$-coconnectivity for stacks}

By considering fppf topology on the category $^{\leq n}\!\affdgSch$, we obtain the
corresponding full subcategory
$$^{\leq n}\!\inftydgstack\subset {}^{\leq n}\!\inftydgprestack.$$

The restriction functor $\inftydgprestack\to {}^{\leq n}\!\inftydgprestack$ sends
\begin{equation} \label{e:restriction of stacks}
\inftydgstack\to {}^{\leq n}\!\inftydgstack,
\end{equation}
but the left adjoint $\on{LKE}_{({}^{\leq n}\!\affdgSch)^{\on{op}}\hookrightarrow (\affdgSch)^{\on{op}}}$ does
not send $^{\leq n}\!\inftydgstack$ to $\inftydgstack$. The left adjoint to the functor
\eqref{e:restriction of stacks} is given by the composition
$$^{\leq n}\!\inftydgstack\hookrightarrow 
{}^{\leq n}\!\inftydgprestack\overset{\on{LKE}_{({}^{\leq n}\!\affdgSch)^{\on{op}}\hookrightarrow (\affdgSch)^{\on{op}}}}\longrightarrow
\inftydgprestack\overset{L}\longrightarrow \inftydgstack,$$
and is denoted $^L\!\on{LKE}_{({}^{\leq n}\!\affdgSch)^{\on{op}}\hookrightarrow (\affdgSch)^{\on{op}}}$. 
The functor $^L\!\on{LKE}_{({}^{\leq n}\!\affdgSch)^{\on{op}}\hookrightarrow (\affdgSch)^{\on{op}}}$ is fully faithful. 
The composition of the functor \eqref{e:restriction of stacks} with 
$^L\!\on{LKE}_{({}^{\leq n}\!\affdgSch)^{\on{op}}\hookrightarrow (\affdgSch)^{\on{op}}}$ is a colocalization functor on
$\inftydgstack$ and is denoted $\CY\mapsto {}^L\tau^{\leq n}(\CY)$.

\medskip

We shall say that a stack $\CY\in \inftydgstack$ is $n$-coconnective \emph{as a stack} if it belongs
to the essential image of the functor $^L\!\on{LKE}_{({}^{\leq n}\!\affdgSch)^{\on{op}}\hookrightarrow (\affdgSch)^{\on{op}}}$,
or equivalently, if the canonical map $^L\tau^{\leq n}(\CY)\to \CY$ is an isomorphism.

\medskip

We emphasize, however, that if $\CY$ is $n$-coconnective as a stack, it is \emph{not} necessarily 
$n$-coconnective as a prestack. The corresponding morphism $\tau^{\leq n}(\CY)\to \CY$ 
becomes an isomorphism only after applying the sheafification functor $L$.

\medskip

Thus, the functor \eqref{e:restriction of stacks} and its left adjoint identify the category 
$^{\leq n}\!\inftydgstack$ with the full subcategory of $\inftydgstack$ spanned by
$n$-coconnective stacks. 

\medskip

We shall say that $\CY$ is \emph{eventually coconnective as a stack} if it is $n$-coconnective as a stack for some $n$. 

\medskip 

We shall refer to objects of $^{\leq 0}\!\inftydgstack$ as ``classical stacks";
we shall also denote this category by $^{cl}\!\inftydgstack$.
By the above, the category of classical stacks is canonically equivalent to that
of $0$-coconnective stacks. 

\sssec{DG schemes}  \label{sss:DG schemes}

The category $\inftydgstack$ (resp., $^{\leq n}\!\inftydgstack$) contains the full subcategory
$\dgSch$ (resp., $^{\leq n}\!\dgSch$), see \cite{Stacks}, Sect. 3.2. 

\medskip

The functors of restriction and $^L\!\on{LKE}_{({}^{\leq n}\!\affdgSch)^{\on{op}}\hookrightarrow (\affdgSch)^{\on{op}}}$
send the categories $\dgSch$ and $^{\leq n}\!\dgSch$ to one another, thereby identifying 
$^{\leq n}\!\dgSch$ with the subcategory of $\dgSch$ that consists of $n$-coconnective
DG schemes, i.e., those DG schemes that are $n$-coconnective as stacks. 

\medskip

For $n=0$ we shall refer to objects of $^{\leq 0}\!\dgSch$ as ``classical schemes", and denote this
category also by $\Sch$.

\medskip

\noindent{\it Notational convention:} In order to avoid unbearably long formulas,
we will sometimes use the following slightly abusive notation: if $Z$ is an object of
$^{\leq n}\!\dgSch$, we will use the same symbol $Z$ for the object of $\dgSch$
that should properly be denoted 
$$^L\!\on{LKE}_{({}^{\leq n}\!\affdgSch)^{\on{op}}\hookrightarrow (\affdgSch)^{\on{op}}}(Z).$$
Similarly, for $n'\geq n$, we shall write $Z$ for the object of $^{\leq n'}\!\dgSch$
that should properly be denoted 
$$^{\leq n'}\left({}^L\!\on{LKE}_{({}^{\leq n}\!\affdgSch)^{\on{op}}\hookrightarrow (\affdgSch)^{\on{op}}}(Z)\right).$$

\sssec{Convergence}  \label{sss:convergence}

An object $\CY$ of $\inftydgprestack$ (resp., $\inftydgstack$) is said to be convergent
if for any $S\in \affdgSch$, the natural map
$$\CY(S)\to \underset{n}{lim}\, \CY(\tau^{\leq n}(S))$$
is an isomorphism.

\medskip

Equivalently, $\CY\in \inftydgprestack$ (resp., $\inftydgstack$) is convergent if the map
$$\CY\to \on{RKE}_{({}^{<\infty}\!\affdgSch)^{\on{op}}\hookrightarrow (\affdgSch)^{\on{op}}}(\CY|_{^{<\infty}\!\affdgSch})$$
is an isomorphism.  Here, $^{<\infty}\!\affdgSch$ denotes the full subcategory of $\affdgSch$
spanned by eventually coconnective affine DG schemes.

\medskip

The full subcategory of $\inftydgprestack$ (resp., $\inftydgstack$) that consists of convergent objects
is denoted $^{\on{conv}}\!\inftydgprestack$ (resp., $^{\on{conv}}\!\inftydgstack$).  The embedding
$$^{\on{conv}}\!\inftydgprestack\hookrightarrow \inftydgprestack$$
admits a left adjoint, called the convergent completion, and denoted $\CY\mapsto {}^{\on{conv}}\CY$.
\footnote{In \cite{Stacks}, this functor was denoted $\CY\mapsto \wh\CY$.}
The restriction of this functor to $\inftydgstack$ sends
$$\inftydgstack\to {}^{\on{conv}}\!\inftydgstack,$$
and is the left adjoint to the embedding $^{\on{conv}}\!\inftydgstack\hookrightarrow \inftydgstack$.

\medskip

Tautologically, we can describe the functor of convergent completion as the composition
$$\CY\mapsto \on{RKE}_{({}^{<\infty}\!\affdgSch)^{\on{op}}\hookrightarrow (\affdgSch)^{\on{op}}}(\CY|_{^{<\infty}\!\affdgSch}).$$
I.e,  the functor of right Kan extension $\on{RKE}_{({}^{<\infty}\!\affdgSch)^{\on{op}}\hookrightarrow (\affdgSch)^{\on{op}}}$
along $$({}^{<\infty}\!\affdgSch)^{\on{op}}\hookrightarrow (\affdgSch)^{\on{op}}$$
identifies the category $^{<\infty}\!\inftydgprestack$ with $^{\on{conv}}\!\inftydgprestack$, and 
$^{<\infty}\!\inftydgstack$ with $^{\on{conv}}\!\inftydgstack$. 

\sssec{Weak $n$-coconnectivity}  \label{sss:weakly coconnective}

For a fixed $n$, the composite functor
$$^{\on{conv}}\!\inftydgprestack\hookrightarrow \inftydgprestack\overset{\on{restriction}}\longrightarrow
{}^{\leq n}\!\inftydgprestack$$
also admits a left adjoint given by
\begin{equation} \label{e:extension and completion}
\CY_n\mapsto {}^{\on{conv}}\!\on{LKE}_{({}^{\leq n}\!\affdgSch)^{\on{op}}\hookrightarrow (\affdgSch)^{\on{op}}}(\CY_n):=
{}^{\on{conv}}\!(\on{LKE}_{({}^{\leq n}\!\affdgSch)^{\on{op}}\hookrightarrow (\affdgSch)^{\on{op}}}(\CY_n)).
\end{equation}
Equivalently, when we identify $^{<\infty}\!\inftydgprestack\simeq {}^{\on{conv}}\!\inftydgprestack$,
the above functor can be described as $\on{LKE}_{({}^{\leq n}\!\affdgSch)^{\on{op}}\hookrightarrow ({}^{<\infty}\!\affdgSch)^{\on{op}}}$.

\medskip

The composite functor
$$\CY\mapsto {}^{\on{conv}}\!\on{LKE}_{({}^{\leq n}\!\affdgSch)^{\on{op}}\hookrightarrow (\affdgSch)^{\on{op}}}(\CY|_{^{\leq n}\!\affdgSch})$$
is a colocalization on $^{\on{conv}}\!\inftydgprestack$, and we will denote it by $^{\on{conv}}\!\tau^{\leq n}$.

\medskip

Similarly, the composite functor
$$^{\on{conv}}\!\inftydgstack\hookrightarrow \inftydgstack \overset{\on{restriction}}\longrightarrow
{}^{\leq n}\!\inftydgstack$$
also admits a left adjoint given by
\begin{equation} \label{e:extension sheaf and completion}
\CY_n\mapsto {}^{\on{conv},L}\!\on{LKE}_{({}^{\leq n}\!\affdgSch)^{\on{op}}\hookrightarrow (\affdgSch)^{\on{op}}}(\CY_n):=
{}^{\on{conv}}\!({}^L\!\on{LKE}_{({}^{\leq n}\!\affdgSch)^{\on{op}}\hookrightarrow (\affdgSch)^{\on{op}}}(\CY_n)).
\end{equation}
Alternatively, when we identify $^{<\infty}\!\inftydgstack\simeq {}^{\on{conv}}\!\inftydgstack$,
the above functor can be described as 
$$^{^{<\infty}\!L}\!\on{LKE}_{({}^{\leq n}\!\affdgSch)^{\on{op}}\hookrightarrow ({}^{<\infty}\!\affdgSch)^{\on{op}}}.$$

\medskip

The composite functor
$$\CY\mapsto {}^{\on{conv},L}\!\on{LKE}_{({}^{\leq n}\!\affdgSch)^{\on{op}}\hookrightarrow (\affdgSch)^{\on{op}}}(\CY|_{^{\leq n}\!\affdgSch})$$
is a colocalization on $^{\on{conv}}\!\inftydgstack$, and we will denote it by $^{\on{conv},L}\!\tau^{\leq n}$.

\medskip

We shall say that an object $\CY$ of $^{\on{conv}}\!\inftydgprestack$ (resp., $^{\on{conv}}\!\inftydgstack$) is
\emph{weakly} $n$-coconnective if it belongs to the essential image of the functor 
\eqref{e:extension and completion}  (resp., \eqref{e:extension sheaf and completion}).
Equivalently, an object as above is \emph{weakly} $n$-coconnective if and only if its
restriction to $^{\leq m}\!\affdgSch$ is $n$-coconnective for any $m\geq n$. 

\medskip

It is clear that if an object is $n$-coconnective, then it is weakly $n$-coconnective. However,
the converse is false.

\ssec{Acknowledgments}

We are much indebted to Jacob Lurie for many helpful discussions (and, really, for teaching
us derived algebraic geometry). We are also grateful to him for sharing with us what 
appears in this paper as \thmref{t:QCoh main}.

\medskip

We are also very grateful to V.~Drinfeld for his help in proving \thmref{t:classical vs derived ft}: he pointed
us to Sect. 7.12 of \cite{BD}, and especially to Proposition 7.12.23, which is crucial for the proof.  Additionally, we thank P.~Pstragowski for pointing out a mistake in a previous version of \secref{ss:monomorph}.

\medskip

The research of D.G. is supported by NSF grant DMS-1063470.
  
\section{DG indschemes}

When dealing with usual indschemes, the definition is straightforward: like any "space" in algebraic geometry,
an indscheme is a presheaf on the category of affine schemes, and the condition we require is that it should
be representable by a filtered family of schemes, where the transition maps are closed embeddings.

\medskip

The same definition is reasonable in the DG setting as long as we restrict ourselves to $n$-coconnective
DG schemes for some $n$. However, when dealing with arbitrary DG indschemes, one
has to additionally require that the presheaf be \emph{convergent}, see
\secref{sss:convergence}.

\medskip

Thus, for reasons of technical convenience we define DG indschemes by requiring the existence of a 
presentation as a filtered colimit  \emph{at the truncated level}. 
We will later show that a DG indscheme defined in this way itself admit a presentation as
a colimit of DG schemes. 

\medskip

In this section we define DG indschemes, first in the $n$-coconnective setting for some $n$,
and then in general, and study the relationship between these two notions. 

\medskip

As was mentioned in the introduction, the class of (DG) indschemes that we consider in this
paper is somewhat smaller than one could in principle consider in general: we will only consider
those (DG) indschemes that are ind-quasi compact and ind-quasi separated. 

\ssec{Definition in the $n$-coconnective case}

\sssec{}

Let us recall the notion of closed embedding in derived algebraic geometry. 

\begin{defn}
A map $X_1\to X_2$ in $\dgSch$ or $^{\leq n}\!\dgSch$ is a closed embedding if the corresponding
map of classical schemes $^{cl}\!X_1\to {}^{cl}\!X_2$ is.
\end{defn}

Recall that the notation $^{cl}\!X$ means $X|_{^{cl}\!\affdgSch}$, i.e., we regard $X$ is a functor
on classical affine schemes, and if $X$ was a DG scheme, then $^{cl}\!X$ is a classical scheme
(see \cite[Sect. 3.2.1]{Stacks}).

\medskip

Let $(\dgSch)_{\on{closed}}$ (resp., $({}^{\leq n}\!\dgSch)_{\on{closed}}$)
denote the 1-full subcategory of $\dgSch$ (resp., $^{\leq n}\!\dgSch$), where we restrict
$1$-morphisms to be closed embeddings. Let 
$$\dgSch_{\on{qsep-qc}}\subset \dgSch,\,\, (\dgSch_{\on{qsep-qc}})_{\on{closed}} \subset (\dgSch)_{\on{closed}},$$
$$^{\leq n}\!\dgSch_{\on{qsep-qc}}\subset {}^{\leq n}\!\dgSch, \,\,
({}^{\leq n}\!\dgSch_{\on{qsep-qc}})_{\on{closed}} \subset ({}^{\leq n}\!\dgSch)_{\on{closed}}$$
be the full subcategories corresponding to quasi-separated and quasi-compact DG schemes
(by definition, this is a condition on the underlying classical scheme).



\sssec{}

We give the following definition:

\begin{defn}
A $\nDG$ indscheme is an object $\CX$ of $^{\leq n}\!\inftydgprestack$ that can be represented
as a colimit of a functor 
$$\sA\to {}^{\leq n}\!\inftydgprestack$$
which \emph{can be factored} as 
$$\sA\to ({}^{\leq n}\!\dgSch_{\on{qsep-qc}})_{\on{closed}}\hookrightarrow {}^{\leq n}\!\inftydgprestack,$$
and where the category $A$ is \emph{filtered}. 
\end{defn}

\medskip

I.e., $\CX\in \on{PreStk}$ is a $\nDG$ indscheme if it can be written as a filtered colimit in $^{\leq n}\!\inftydgprestack$:
\begin{equation} \label{e:indscheme as a colimit}
\underset{\alpha}{colim}\, X_\alpha,
\end{equation}
where $X_\alpha\in {}^{\leq n}\!\dgSch_{\on{qsep-qc}}$ and for $\alpha_1\to \alpha_2$, the corresponding map 
$i_{\alpha_1,\alpha_2}:X_{\alpha_1}\to X_{\alpha_2}$
is a closed embedding. 

\medskip

Let $^{\leq n}\!\dgindSch$ denote the full subcategory of $^{\leq n}\!\inftydgprestack$ spanned
by $\nDG$ indschemes. We shall refer to objects of $^{\leq 0}\!\dgindSch$ as
\emph{classical indschemes}; we shall also use the notation $\indSch:={}^{\leq 0}\!\dgindSch$.

\begin{rem}
Note that the quasi-compactness and quasi-separatedness assumption 
in the definition of $\nDG$ indschemes means that not every $\nDG$ scheme $X$ is a 
$\nDG$ indscheme.  However, a scheme which is an indscheme is not necessarily quasi-separated and quasi-compact:
for example, a disjoint union of quasi-separated and quasi-compact 
$\nDG$ schemes is a $\nDG$ indscheme.
\end{rem}

\ssec{Changing $n$}

\sssec{}

Clearly, for $n'<n$, the functor 
$$^{\leq n}\!\inftydgprestack\to {}^{\leq n'}\!\inftydgprestack,$$
corresponding to restriction along $$^{\leq n'}\!\affdgSch\hookrightarrow {}^{\leq n}\!\affdgSch,$$
sends the subcategory $^{\leq n}\!\dgindSch$ to $^{\leq n'}\!\dgindSch$.

\medskip

Indeed, if $\CX$ is presented as in 
\eqref{e:indscheme as a colimit}, then $^{\leq n'}\CX:=\CX|_{^{\leq n'}\!\affdgSch}$ can be presented as
$$\underset{\alpha}{colim}\,\, ({}^{\leq n'}\!X_\alpha).$$

Thus, restriction defines a functor
$$^{\leq n'}\!\dgindSch \leftarrow {}^{\leq n}\!\dgindSch.$$

\sssec{}  \label{sss:LKE n'->n}

Vice versa, consider the functor
\begin{multline} \label{e:LA to restr stacks}
^{^{\leq n}\!L}\!\on{LKE}_{({}^{\leq n'}\!\affdgSch)^{\on{op}}\hookrightarrow {}({}^{\leq n}\!\affdgSch)^{\on{op}}}:=\\
{}^{\leq n}\!L\circ \on{LKE}_{({}^{\leq n'}\!\affdgSch)^{\on{op}}\hookrightarrow {}({}^{\leq n}\!\affdgSch)^{\on{op}}}:
{}^{\leq n'}\!\inftydgstack\to {}^{\leq n}\!\inftydgstack,
\end{multline}
left adjoint to the restriction functor. In the above formula $^{\leq n}\!L:{}^{\leq n}\!\inftydgprestack\to {}^{\leq n}\inftydgstack$
is the sheafification functor, left adjoint to the embedding $^{\leq n}\!\inftydgstack\hookrightarrow
{}^{\leq n}\!\inftydgprestack$.

\medskip

We claim that it sends $^{\leq n'}\!\dgindSch$ to $^{\leq n}\!\dgindSch$. Indeed,
if $\CX'\in {}^{\leq n'}\!\dgindSch$ is written as
$$\CX'\simeq \underset{\alpha}{colim}\, X_\alpha,\quad X_\alpha\in {}^{\leq n'}\!\dgSch$$
(the colimit taken in $^{\leq n'}\!\on{PreStk}$), then 
$$^{^{\leq n}\!L}\!\on{LKE}_{({}^{\leq n'}\!\affdgSch)^{\on{op}}\hookrightarrow {}({}^{\leq n}\!\affdgSch)^{\on{op}}}(\CX')\simeq 
\underset{\alpha}{colim}\,  X_\alpha,$$
where the colimit is taken in $^{\leq n}\!\on{PreStk}$, and $X_\alpha$ is perceived as an object of
$^{\leq n}\!\dgSch$, see the notational convention in \secref{sss:DG schemes}.

\sssec{}

We obtain a pair of adjoint functors
\begin{equation} \label{e:LA to restr indschemes}
^{\leq n'}\!\dgindSch \rightleftarrows {}^{\leq n}\!\dgindSch,
\end{equation}
with the left adjoint being fully faithful.

\medskip

An object $\CX\in {}^{\leq n}\!\dgindSch$ belongs to the essential image of the left adjoint in
\eqref{e:LA to restr indschemes} if and only if it is $n'$-coconnective as an object of 
$^{\leq n}\!\inftydgstack$, i.e., if it belongs to the essential image of the left adjoint
\eqref{e:LA to restr stacks}. 

\medskip

Moreover, if $\CX \in {}^{\leq n}\!\dgindSch$ has this property,
it admits a presentation as in \eqref{e:indscheme as a colimit},
where the $X_\alpha$ are $n'$-coconnective.

\ssec{Basic properties of $\nDG$ indschemes}

\sssec{}

We observe:

\begin{prop}  \label{p:indschemes fppf n}
Every $\nDG$ indscheme belongs to $^{\leq n}\!\inftydgstack$
i.e., satisfies fppf descent. 
\end{prop}

The proof is immediate from the following general assertion:

\begin{lem} \label{l:no sheafification}
Let $\alpha\mapsto X_\alpha$ be a filtered diagram in $^{\leq n}\!\inftydgprestack$. Set
$$\CX:=\underset{\alpha}{colim}\, X_\alpha.$$
Then if all $X_\alpha$ belong to $^{\leq n}\!\inftydgstack$ and are $k$-truncated
for some $k$ \emph{(}see \cite[Sect. 1.1.7]{Stacks}\emph{)}, then $\CX$ has the same properties.
\end{lem}

\begin{proof}
By assumption,
$$X_\alpha \text{ and }\CX:{}^{\leq n}\!\affdgSch\to \inftygroup$$ take
values in the subcategory $(k+n)$-groupoids. 

\medskip

Recall that for a co-simplicial object $\bc^\bullet$ in 
the category of $m$-groupoids, the totalization $\on{Tot}(\bc^\bullet)$ maps isomorphically to
$\on{Tot}^{m+1}(\bc^\bullet)$, where $\on{Tot}^{m+1}(-)$ denotes the limit taken over the
category of finite ordered sets of cardinality $\leq (m+1)$. 

\medskip

Hence, for an fppf cover $S'\to S$ and its 
\v{C}ech nerve $S'{}^\bullet/S$, for its $(k+n+1)$-truncation $S'{}^{\bullet\leq k+n+1}/S$, the 
restriction maps
$$\on{Tot}(X_\alpha(S'{}^\bullet/S))\to \on{Tot}^{\leq (k+n+1)}(X_\alpha(S'{}^\bullet/S)) \text{ and }
\on{Tot}(\CX(S'{}^\bullet/S))\to \on{Tot}^{\leq (k+n+1)}(\CX(S'{}^{\bullet}/S))$$
are isomorphisms.

\medskip

In particular, it suffices to show that the map $\CX(S)\to \on{Tot}^{\leq (k+n+1)}(\CX(S'{}^{\bullet}/S))$
is an isomorphism. 

\medskip

Consider the commutative diagram
$$
\CD
\underset{\alpha}{colim}\, X_\alpha(S)   @>>>  \on{Tot}^{\leq (k+n+1)} \left(\underset{\alpha}{colim}\, X_\alpha(S'{}^{\bullet}/S)\right) \\
@A{\on{id}}AA    @AAA   \\
\underset{\alpha}{colim}\, X_\alpha(S)  @>>> \underset{\alpha}{colim}\, \left(\on{Tot}^{\leq (k+n+1)}(X_\alpha(S'{}^{\bullet}/S))\right).
\endCD
$$
The bottom horizontal arrow is an isomorphism, since all $X_\alpha$ satisfy descent. The right vertical arrow is an isomorphism, since
filtered colimits commute with \emph{finite} limits. Hence, the top horizontal arrow is also an isomorphism, as desired. 
 
\end{proof}

\sssec{}

We obtain that if $\CX\in {}^{\leq n}\!\inftydgstack$ is written as in \eqref{e:indscheme as a colimit}, 
\emph{but where the colimit is taken in the category $^{\leq n}\!\inftydgstack$}, then $\CX$ is a $\nDG$ indscheme.

\medskip

Indeed, \propref{p:indschemes fppf n} implies that the the natural map from the colimit of \eqref{e:indscheme as a colimit} 
taken in $^{\leq n}\!\inftydgprestack$ to that in $^{\leq n}\!\inftydgstack$ is an isomorphism.







\sssec{}  \label{sss:maps out of arb}

Let $Y$ be an object of $^{\leq n}\!\dgSch$, and let $\CX\in {}^{\leq n}\!\dgindSch$ be presented
as in \eqref{e:indscheme as a colimit}. We have a natural map
\begin{equation}  \label{e:maps from an arb scheme}
\underset{\alpha}{colim}\, \on{Maps}(Y,X_\alpha)\to \on{Maps}(Y,\CX).
\end{equation}

If $Y$ is affine, the above map is an isomorphism by definition, since colimits in 
$$^{\leq n}\!\inftydgprestack=\on{Func}({}^{\leq n}\!\affdgSch,\inftygroup)$$
are computed object-wise.

\medskip

For a general $Y$, the map \eqref{e:maps from an arb scheme}
need not be an isomorphism. However, we have:

\begin{lem} \label{l:maps out of qc}
If $Y$ is quasi-separated and quasi-compact, then the map \eqref{e:maps from an arb scheme}
is an isomorphism.
\end{lem}

\begin{proof}
This follows from the fact that $\CX$ belongs to $^{\leq n}\!\inftydgstack$ , and that 
a quasi-separated and quasi-compact DG scheme can be written 
as a colimit in $^{\leq n}\!\inftydgstack$ of a \emph{finite} diagram whose terms are in
$^{\leq n}\!\affdgSch$, and the fact that filtered colimits in $\inftygroup$ commute with finite limits.
\end{proof}

\begin{rem}
The reason we ever mention sheafification and work with $\on{Stk}$ rather than simply with
$\on{PreStk}$ is \lemref{l:maps out of qc} above. However, the proof of 
\lemref{l:maps out of qc} shows that we could equally well work with \'etale, Nisnevich or
Zariski topologies, instead of fppf.
\end{rem}
 
\ssec{General DG indschemes}  \label{ss:gen DG indschemes}

\sssec{}

We give the following definition:

\begin{defn}
An object $\CX\in \inftydgprestack$ is a DG indscheme if the following
two conditions hold:

\begin{enumerate}

\item  As an object of $\inftydgprestack$, $\CX$ is convergent (see \secref{sss:convergence}).

\item For every $n$, $^{\leq n}\CX:=\CX|_{^{\leq n}\!\affdgSch}$ is a $\nDG$ indscheme.

\end{enumerate}

\end{defn}

\medskip

We shall denote the full subcategory of $\inftydgprestack$ spanned by DG indschemes by
$\dgindSch$.

\sssec{}

We will prove the following (see also \propref{p:indscheme as colimit of its closed} below for a more precise assertion):

\begin{prop}  \label{p:presentation of indschemes}
Any DG indscheme $\CX$ can be presented as a filtered colimit in $\inftydgprestack$
\begin{equation} \label{e:gen indscheme as a colimit}
\underset{\alpha}{colim}\, X_\alpha,
\end{equation}
where $X_\alpha\in \dgSch_{\on{qsep-qc}}$ and for $\alpha_1\to \alpha_2$, the corresponding map 
$i_{\alpha_1,\alpha_2}:X_{\alpha_1}\to X_{\alpha_2}$
is a closed embedding. 
\end{prop}

\medskip

The above proposition allows us to give the following, in a sense, more straightforward,
definition of DG indschemes:

\begin{cor}
An object $\CX\in \inftydgprestack$ is a DG indscheme if and only if:

\begin{itemize}

\item It is convergent;

\item As an object of $\inftydgprestack$ it admits a presentation as in \eqref{e:gen indscheme as a colimit}.

\end{itemize}

\end{cor}

\sssec{}

Note that unlike the case of $\nDG$ indschemes, an object of $\inftydgprestack$ written 
as in \eqref{e:gen indscheme as a colimit} need \emph{not} be a DG indscheme. Indeed, it can
fail to be convergent.

\medskip

However, such a colimit gives rise to a DG indscheme via the following lemma:

\medskip

\begin{lem}  \label{l:convergent completion}
For $\CX\in \inftydgprestack$ given as in \eqref{e:gen indscheme as a colimit}, the
object
$$^{\on{conv}}\CX\in \inftydgprestack$$
belongs to $\dgindSch$.
\end{lem}

\begin{proof}

Indeed, $^{\on{conv}}\CX$ is convergent by definition, and for any $n$, we have 
$^{\leq n}({}^{\on{conv}}\CX)\simeq {}^{\leq n}\CX$. 

\end{proof}

\sssec{}

If $\CX$ is a DG indscheme, then
$$^{\leq n}\CX:=\CX|_{^{\leq n}\!\affdgSch}$$
is a $\nDG$ indscheme. In particular, $^{cl}\CX$ is a classical indscheme.  Thus,
we obtain a functor
\begin{equation} \label{e:restr gen to n}
^{\leq n}\!\dgindSch\leftarrow \dgindSch.
\end{equation}

\medskip

Vice versa, if $\CX_n$ is a $\nDG$ indscheme, set
$$\CX:={}^{\on{conv},L}\!\on{LKE}_{({}^{\leq n}\!\affdgSch)^{\on{op}}\hookrightarrow (\affdgSch)^{\on{op}}}(\CX_n).$$
Explicitly, if $\CX_n$ is given by the colimit as in \eqref{e:indscheme as a colimit}, then 
$\CX$ is the convergent completion of the same colimit taken in $\on{PreStk}$, but where
$X_\alpha$ are understood as objects of $\dgSch$, see notational convention in \secref{sss:DG schemes}.
By \lemref{l:convergent completion}, we obtain that $\CX$ is a DG indscheme.

\medskip

This defines a functor
\begin{equation} \label{e:LKE n to gen}
^{\leq n}\!\dgindSch\to \dgindSch,
\end{equation}
which is left adjoint to the one in \eqref{e:restr gen to n}. It is easy to see that the unit map
defines an isomorphism from the identity functor to 
$$^{\leq n}\!\dgindSch\to  \dgindSch\to {}^{\leq n}\!\dgindSch.$$
I.e., the functor in \eqref{e:LKE n to gen} is fully faithful. 

\sssec{}

In what follows, we shall say that a DG indscheme is \emph{weakly} $n$-coconnective
if it is such as an object of $\inftydgstack$, see \secref{sss:weakly coconnective}, i.e., 
if it belongs to the essential image of the functor \eqref{e:LKE n to gen}. 

\medskip

Thus, the above functor establishes an equivalence between $^{\leq n}\!\dgindSch$ and the full subcategory
of $\dgindSch$ spanned by \emph{weakly} $n$-coconnective DG schemes. In particular, it identifies
classical indschemes with weakly $0$-coconnective DG indschemes. 

\medskip

We shall say that $\CX$ is weakly eventually coconnective if it is weakly $n$-coconnective for some $n$. 

\sssec{}

We shall say that a DG indscheme is $n$-coconnective if it is 
$n$-coconnective as an object of $\inftydgstack$, i.e., if it lies in the essential image of the functor 
\begin{equation} \label{e:LA to restr stacks gen}
^L\!\on{LKE}_{({}^{\leq n}\!\affdgSch)^{\on{op}}\hookrightarrow (\affdgSch)^{\on{op}}}:{}^{\leq n}\!\inftydgstack \to \inftydgstack.
\end{equation}
We shall say that $\CX$ is eventually coconnective if it is $n$-coconnective for some $n$. 

\sssec{The $\aleph_0$ condition}  \label{sss:aleph 0}

We shall say that $\CX\in {}^{\leq n}\!\dgindSch$ is $\aleph_0$ if there exists a presentation as
in \eqref{e:indscheme as a colimit} with the category of indices equivalent to the poset $\BN$. 

\medskip

We shall say that $\CX\in \dgindSch$ is $\aleph_0$ if for it admits a presentation as in
\propref{p:presentation of indschemes}, with the category of indices equivalent to the poset $\BN$. 

\medskip

We shall say that $\CX\in \dgindSch$ is \emph{weakly} $\aleph_0$ if for every $n$,
the object
$$^{\leq n}\CX\in {}^{\leq n}\!\dgindSch$$
is $\aleph_0$.

\ssec{Basic properties of DG indschemes}

\sssec{}

We claim:

\begin{prop}  \label{p:indschemes fppf}
Every $\CX\in \dgindSch$ belongs to $\inftydgstack$, i.e., satisfies fppf descent. 
\end{prop}

\begin{proof}

Let $S'\to S$ be an fppf map in $\affdgSch$, and let $S'{}^\bullet/S$ be its \v{C}ech nerve. We need
to show that the map
$$\Maps(S,\CX)\to \on{Tot}(\Maps(S'{}^\bullet/S,\CX))$$
is an isomorphism. 

\medskip

For an integer $n$, we consider the truncation $^{\leq n}\!S\in {}^{\leq n}\!\affdgSch$
of $S$. Note that since $S'\to S$ is flat, the map $^{\leq n}\!S'\to {}^{\leq n}\!S$ is flat, and 
the simplicial object $^{\leq n}(S'{}^\bullet/S)$ of $^{\leq n}\!\affdgSch$ is the \v{C}ech nerve
of $^{\leq n}\!S'\to {}^{\leq n}\!S$. 

\medskip

We have a commutative diagram
$$
\CD
\Maps(S,\CX)    @>>>   \on{Tot}(\Maps(S'{}^\bullet/S,\CX))  \\
@VVV    @VVV    \\
\underset{n\in \BN^{\on{op}}}{lim}\, \Maps({}^{\leq n}\!S,{}^{\leq n}\CX)   @>>>   
\underset{n\in \BN^{\on{op}}}{lim}\, \on{Tot}(\Maps({}^{\leq n}(S'{}^\bullet/S),{}^{\leq n}\CX))
\endCD.
$$
In this diagram the vertical arrows are isomorphisms, since $\CX$ is convergent.
The bottom horizontal arrow is an isomorphism by \propref{p:indschemes fppf n}.
Hence, the top horizontal arrow is an isomorphism as well, as desired.

\end{proof}

\sssec{}

As in \secref{sss:maps out of arb} we consider maps into a DG indscheme $\CX$ from
an arbitrary DG scheme $Y$, and we have the following analog of \lemref{l:maps out of qc}
(with the same proof, but relying on \propref{p:presentation of indschemes}):

\begin{lem} \label{l:gen maps out of qc}
For $\CX\in \inftydgstack$ written as in \eqref{e:gen indscheme as a colimit}, and $Y\in \dgSch$, 
the natural map $$\underset{\alpha}{colim}\, \on{Maps}(Y,X_\alpha)\to \on{Maps}(Y,\CX)$$
is an isomorphism, provided that $Y$ is quasi-separated and quasi-compact.
\end{lem}

\ssec{The canonical presentation of a DG indscheme}

We shall now formulate a sharper version of \propref{p:presentation of indschemes}, which will be proved
in Sect. \ref{s:pushouts}. 

\sssec{}  \label{sss:general closed emb}

We give the following definition:

\begin{defn}
A map $\CY_1\to \CY_2$ in $\on{PreStk}$ is said to be a \emph{closed embedding} if the corresponding 
map $^{cl}\CY_1\to {}^{cl}\CY_2$ is a closed embedding (i.e., its base change by an affine scheme yields a 
closed embedding). 
\end{defn}

\medskip

Note that in the DG setting, being a closed embedding does not imply that a map is
schematic\footnote{We recall that a map of prestacks is called \emph{schematic} if its
base change by an affine DG scheme yields a DG scheme.}.
Indeed, a closed embedding of a DG scheme into a DG indscheme is typically not schematic.

\medskip

It is easy to see that for maps $\CY_1\to \CY_2\to \CY_3$ with $\CY_2\to \CY_3$ being 
a closed embedding, the map $\CY_1\to \CY_2$ is a closed embedding if and only if $\CY_1\to \CY_3$ is.

\sssec{}

For a DG indscheme $\CX$, let 
$$(\dgSch_{\on{qsep-qc}})_{\on{closed}\,\on{in}\,\CX}\subset (\dgSch_{\on{qsep-qc}})_{/\CX}$$ be
the full subcategory, consisting of those objects for which the map $Z\to \CX$ is a closed
embedding in the above sense.

\medskip

In Sects. \ref{ss:cl into ind} and \ref{ss:proof of indscheme as colimit of its closed} we will prove:

\begin{prop}  \label{p:indscheme as colimit of its closed}
Let $\CX$ be a DG scheme.

\smallskip

\noindent{\em(a)} The category  $(\dgSch_{\on{qsep-qc}})_{\on{closed}\,\on{in}\,\CX}$
is filtered.

\smallskip

\noindent{\em(b)} The natural map
\begin{equation} \label{e:indscheme as a colim of closed}
\underset{Z\in (\dgSch_{\on{qsep-qc}})_{\on{closed}\,\on{in}\,\CX}}{colim}\, Z\to \CX,
\end{equation}
where the colimit is taken in $\inftydgprestack$, is an isomorphism.
\end{prop}

\sssec{}

Combined with \lemref{l:gen maps out of qc}, we obtain the following:

\begin{cor}  \label{c:cofinality of closed} 
Let $\CX$ be a DG indscheme. The functor
\begin{equation} \label{e:closed vs all}
(\dgSch_{\on{qsep-qc}})_{\on{closed}\,\on{in}\,\CX}\to (\dgSch_{\on{qsep-qc}})_{/\CX}
\end{equation}
is cofinal. 
\end{cor}

\begin{proof}
We need to show that for $X\in \dgSch_{\on{qsep-qc}}$ and a map $X\to \CX$, the category of its factorizations
$$X\to Z\to \CX,$$
where $Z\to \CX$ is a closed embedding, is contractible. However, the above category of factorizations is
the fiber of the map of spaces
$$\underset{Z\in (\dgSch_{\on{qsep-qc}})_{\on{closed}\,\on{in}\,\CX}}{colim}\, \Maps(X,Z)\to \Maps(X,\CX)$$
over our given point in $\Maps(X,\CX)$.
\end{proof}

Finally, we can give the following characterization of DG indschemes among $\inftydgprestack$:

\begin{cor}  \label{c:criter for indscheme}
An object $\CX\in {}^{\on{conv}}\!\inftydgprestack$ is a DG indscheme if and only if:
\begin{itemize}

\item The category of closed embeddings $Z\to \CX$, where $Z\in \dgSch_{\on{qsep-qc}}$, is filtered.

\item The functor \eqref{e:closed vs all} is cofinal.

\end{itemize}
\end{cor}

\sssec{}  \label{sss:canonicity of presentation}

Let us also note that \lemref{l:gen maps out of qc} implies that for any presentation of a DG indscheme
as in \propref{p:presentation of indschemes}, the tautological map
$$\sA\to (\dgSch_{\on{qsep-qc}})_{\on{closed}\,\on{in}\,\CX}$$
is cofinal. 

\ssec{The locally almost of finite type condition}

\sssec{}

We shall say that $\CX\in {}^{\leq n}\!\dgindSch$ is locally of finite type if it
is such as an object of $^{\leq n}\!\inftydgprestack$ (see \cite[Sect. 1.3.2]{Stacks}), i.e.,
it belongs to $^{\leq n}\!\inftydgprestack_{\on{lft}}$ in the terminology of {\it loc. cit.} 

\medskip

By
definition, this means that $\CX$, viewed as a functor 
$$({}^{\leq n}\affdgSch)^{\on{op}}\to \inftygroup,$$
equals the left Kan extension under
$$({}^{\leq n}\affdgSch_{\on{ft}})^{\on{op}}\hookrightarrow ({}^{\leq n}\affdgSch)^{\on{op}}$$
of its own restriction to $({}^{\leq n}\affdgSch_{\on{ft}})^{\on{op}}$, where 
$^{\leq n}\affdgSch_{\on{ft}}\subset {}^{\leq n}\affdgSch$ denotes the full subcategory
of $n$-coconnective affine DG schemes \emph{of finite type}. \footnote{We remind that $\Spec(A)\in {}^{\leq n}\affdgSch$
is said to be \emph{of finite type} if $H^0(A)$ is a finitely generated algebra over $k$, and each $H^i(A)$ is finitely generated
as an $H^0(A)$-module.} 

\medskip

We shall denote the full subcategory of $^{\leq n}\!\dgindSch$ spanned
by $\nDG$ indschemes locally of finite type by $^{\leq n}\!\dgindSch_{\on{lft}}$.

\medskip

We shall say that $\CX\in \dgindSch$ is locally almost of finite type if it is such as an object of
$\inftydgprestack$, see \cite[Sect. 1.3.9]{Stacks}, i.e., if in the notation of {\it loc.cit.} it
belongs to the subcategory $\inftydgprestack_{\on{laft}}\subset \inftydgprestack$.
By definition, this means that 
$$^{\leq n}\CX\in {}^{\leq n}\!\dgindSch$$
must be locally of finite type for every $n$. We shall denote the full subcategory of $\dgindSch$ spanned
by DG indschemes locally almost of finite type by $\dgindSch_{\on{laft}}$.

\sssec{}

It is natural to wonder whether one can represent objects of $\dgindSch_{\on{laft}}$ as colimits
of objects of $\dgSch_{\on{aft}}$ under closed embeddings. (We denote by 
$\dgSch_{\on{aft}}$ the category of DG schemes almost of finite type, 
i.e., $\dgSch_{\on{aft}}:=\dgSch_{\on{laft}}\cap \dgSch_{\on{qc}}$, see \cite[Sect. 3.3.1]{Stacks}.)

\medskip

In fact, there are two senses in which one can ask this question: one may want to have
a presentation in a ``weak sense", i.e., as in \lemref{l:convergent completion}, or in the ``strong"
sense, i.e., as in \propref{p:presentation of indschemes}.  

\medskip

The answer to the ``weak" version is affirmative: we will prove
the following:

\begin{prop}  \label{p:presentation of indschemes lft weak} 
For a DG indscheme $\CX$ locally almost of finite type there exists
a filtered family 
$$A\to (\dgSch_{\on{aft}})_{\on{closed}}:\alpha\mapsto X_\alpha,$$ such that
$\CX$ is isomorphic to the convergent completion of 
\begin{equation} \label{e:indscheme as a colimit lft weak}
\underset{\alpha\in \sA}{colim}\, X_\alpha,
\end{equation}
where the colimit is taken in $\inftydgprestack$. 
\end{prop}

\sssec{}

Before we answer the ``strong question", let us note that it is \emph{not} 
true that for any $\CY\in \inftydgprestack_{\on{laft}}$, the functor
$$(\affdgSch_{\on{aft}})_{/\CY}\to (\affdgSch)_{/\CY}$$
is cofinal. However, if $\CX\in \dgindSch_{\on{laft}}$ admitted a
presentation as a colimit of objects of $\dgSch_{\on{aft}}$, it would automatically
have this property. We have:

\begin{prop}  \label{p:presentation of indschemes lft strong}
For a DG indscheme $\CX$ locally almost of finite type there exists
a filtered family 
$$A\to (\dgSch_{\on{aft}})_{\on{closed}}:\alpha\mapsto X_\alpha,$$ such that
$\CX$ is isomorphic to 
\begin{equation} \label{e:indscheme as a colimit lft strong}
\underset{\alpha\in \sA}{colim}\, X_\alpha,
\end{equation}
where the colimit is taken in $\inftydgprestack$. 
\end{prop}


\sssec{}

In fact, we shall prove a more precise version of the above assertions. Namely, in
\secref{ss:proof of indscheme as colimit of its closed} we will prove:

\begin{prop} \label{p:canonical presentation of laft indsch}
Let $\CX$ be an object of $\dgindSch_{\on{laft}}$.

\smallskip

\noindent{\em(a)} The category $(\dgSch_{\on{aft}})_{\on{closed}\,\on{in}\,\CX}$ is filtered.

\smallskip

\noindent{\em(b)} The natural map 
$$\underset{Z\in (\dgSch_{\on{aft}})_{\on{closed}\,\on{in}\,\CX}}{colim}\, Z\to \CX,$$
where the colimit is taken in $\inftydgprestack$, is an isomorphism.
\end{prop}

As a formal consequence, we obtain:

\begin{cor}  \label{c:laft cofinality}
For $\CX\in \dgindSch_{\on{laft}}$ the following functors are cofinal:
\begin{equation} \label{e:closed aft vs all}
(\dgSch_{\on{aft}})_{\on{closed}\,\on{in}\,\CX} \to (\dgSch_{\on{qsep-qc}})_{/\CX}
\end{equation}
\begin{equation} \label{e:closed aft vs closed}
(\dgSch_{\on{aft}})_{\on{closed}\,\on{in}\,\CX} \to (\dgSch_{\on{qsep-qc}})_{\on{closed}\,\on{in}\,\CX} 
\end{equation}
\begin{equation} \label{e:closed aft vs all aft}
(\dgSch_{\on{aft}})_{\on{closed}\,\on{in}\,\CX} \to (\dgSch_{\on{aft}})_{/\CX}
\end{equation} 
and
\begin{equation} \label{e:closed ft vs all ft}
({}^{<\infty}\!\dgSch_{\on{aft}})_{\on{closed}\,\on{in}\,\CX} \to ({}^{<\infty}\!\dgSch_{\on{aft}})_{/\CX}.
\end{equation} 
\end{cor}

\begin{cor}  \label{c:IndSch LKE from aft}
An object $\CX\in \dgindSch_{\on{laft}}\subset \on{PreStk}$
lies in the essential image of the fully faithful functor
\begin{multline*}
\on{LKE}_{(\affdgSch_{\on{aft}})^{\on{op}}\hookrightarrow (\affdgSch)^{\on{op}}}:
\on{Funct}((\affdgSch_{\on{aft}})^{\on{op}},\inftygroup)\to \\
\to \on{Funct}((\affdgSch)^{\on{op}},\inftygroup)=\on{PreStk}.
\end{multline*}
Equivalently, the functor
$$(\affdgSch_{\on{aft}})_{/\CX}\to (\affdgSch)_{/\CX}$$
is cofinal. 
\end{cor}

\begin{cor}
An object $\CX\in {}^{\on{conv}}\!\inftydgprestack$ belongs to $\dgindSch_{\on{laft}}$ if and only if:
\begin{itemize}

\item The category of closed embeddings $Z\to \CX$, where $Z\in \dgSch_{\on{aft}}$, is filtered.

\item The functor \eqref{e:closed aft vs all} is cofinal.

\end{itemize}
\end{cor}

\sssec{}  \label{sss:canonicity of presentation laft}

Note that \lemref{l:gen maps out of qc} implies that for any presentation of $\CX$ as in 
\propref{p:presentation of indschemes lft strong},  
the tautological map
$$\sA\to (\dgSch_{\on{aft}})_{\on{closed}\,\on{in}\,\CX}$$
is cofinal. 

\section{Sheaves on DG indschemes}

\ssec{Quasi-coherent sheaves on a DG indscheme}

\sssec{}

For any $\CY\in \on{PreStk}$, we have the symmetric monoidal category $\QCoh(\CY)$ defined as in \cite[Sect. 1.1.3]{QCoh}.
Explicitly,
$$\QCoh(\CY):=\underset{S\in ((\affdgSch)_{/\CY})^{\on{op}}}{lim}\, \QCoh(S).$$

\sssec{}

In particular, for $\CX\in \dgindSch$ we obtain the symmetric monoidal category $\QCoh(\CX)$.

\medskip

If $\CX\in \dgindSch$ is written as \eqref{e:gen indscheme as a colimit}, we have:
$$\QCoh(\CX)\simeq \underset{\alpha}{lim}\, \QCoh(X_\alpha),$$
where for $\alpha_2\geq \alpha_1$, the map
$\QCoh(X_{\alpha_2})\to \QCoh(X_{\alpha_1})$ is $i^*_{\alpha_1,\alpha_2}$.
This follows from the fact that the functor
$$\QCoh_{\inftydgprestack}:\inftydgprestack^{\on{op}}\to \StinftyCat_{\on{cont}}$$
takes colimits in $\inftydgprestack$ to limits in $\StinftyCat_{\on{cont}}$. 

\medskip

Since the category $\QCoh(\CX)$ is given as a limit, it is not at all guaranteed that it will
be compactly generated. 

\sssec{}

We have the following nice feature of the category $\QCoh$ on DG indschemes that are
locally almost of finite type. Namely, we ``only need to know" $\QCoh$ on affine 
DG schemes that are almost of finite type to recover it. More precisely, from \corref{c:IndSch LKE from aft},
we obtain:

\begin{cor} For $\CX\in \dgindSch_{\on{laft}}$, the functor
$$\QCoh(\CX)=
\underset{S\in ((\affdgSch)_{/\CY})^{\on{op}}}{lim}\, \QCoh(S)\to
\underset{S\in ((\affdgSch_{\on{aft}})_{/\CY})^{\on{op}}}{lim}\, \QCoh(S),$$
given by restriction, is an equivalence.
\end{cor}

\ssec{A digression: perfect objects in $\QCoh$}

\sssec{}  

Recall the notion of a perfect object in $\QCoh(\CY)$ for $\CY\in \on{PreStk}$, see, e.g., \cite[Sect. 4.1.6]{QCoh}.

\medskip

The subcategory $\QCoh(\CY)^{\on{perf}}$ coincides with
that of dualizable objects of $\QCoh(\CY)$ (see, e.g., \cite[Lemma 4.2.2]{QCoh}).

\sssec{}

Recall also that if $\CY=X$ is a quasi-separated and quasi-compact DG scheme, then the category $\QCoh(X)$
is compactly generated and $\QCoh(X)^c=\QCoh(X)^{\on{perf}}$. Moreover, we have the canonical self-duality 
equivalence
$$\bD_X^{\on{naive}}:\QCoh(X)^\vee\simeq \QCoh(X)$$
which can be described in either of the following two equivalent ways:

\begin{itemize}

\item The corresponding \footnote{We recall that for a compactly generated category $\bC$ we have a canonical
equivalence $(\bC^\vee)^c\simeq (\bC^c)^{\on{op}}$.}
equivalence $\BD^{\on{naive}}_{\QCoh(X)}:(\QCoh(X)^c)^{\on{op}}\simeq \QCoh(X)^c$ is the duality functor
with respect to the symmetric monoidal structure on $\QCoh(X)$:
$$\CF\mapsto \CF^\vee:(\QCoh(X)^{\on{perf}})^{\on{op}}\to \QCoh(X)^{\on{perf}}.$$

\smallskip

\item The pairing $\QCoh(X)\otimes \QCoh(X)\to \Vect$ is the composition
$$\QCoh(X)\otimes \QCoh(X)\overset{\otimes}\to \QCoh(X)\overset{\Gamma(X,-)}\longrightarrow \Vect.$$

\end{itemize}

\sssec{}

Note that for an object $\CY\in \on{PreStk}$ (and, in particular, for $\CX\in \dgindSch$),
the functor $\Gamma(\CY,-):\QCoh(\CY)\to \Vect$ is not, in general, continuous. Therefore, the functor
$$\QCoh(\CY)\otimes \QCoh(\CY)\overset{\otimes}\to \QCoh(\CY)\overset{\Gamma(\CY,-)}\longrightarrow \Vect$$
is not continuous either, and as such cannot serve as a candidate the duality paring. 

\sssec{}  \label{sss:rigidity on QCoh}

Let $\CY$ be an arbitrary object of $\on{PreStk}$. We shall say that $\CY$ is quasi-perfect if 

\medskip

\noindent(i) The category $\QCoh(\CY)$ is compactly generated.

\smallskip

\noindent(ii) The compact objects of $\QCoh(\CY)$ are \emph{perfect},
and the duality functor 
\begin{equation} \label{e:duality on perf}
\CF\mapsto \CF^\vee:(\QCoh(\CY)^{\on{perf}})^{\on{op}}\simeq \QCoh(\CY)^{\on{perf}}
\end{equation}
sends $(\QCoh(\CY)^c)^{\on{op}}$ to $\QCoh(\CY)^c$.

\medskip

Note that for $\CY$ quasi-perfect, there exists a canonical equivalence 
$$\bD_{\CY}^{\on{naive}}:\QCoh(\CY)^\vee\simeq \QCoh(\CY),$$ 
given by the equivalence
$$\BD^{\on{naive}}_{\QCoh(\CY)}:(\QCoh(\CY)^c)^{\on{op}}\simeq  \QCoh(\CY)^c$$
induced by the duality functor \eqref{e:duality on perf}.

\medskip

The corresponding pairing $\QCoh(\CY)\otimes \QCoh(\CY)\to \Vect$ can be described as follows:
it is obtained by ind-extending the pairing on compact objects given by 
$$\CF_1,\CF_2\in \QCoh(\CY)^c\mapsto \Gamma(\CX,\CF_1\underset{\CO_\CY}\otimes \CF_2)\in \Vect.$$
Indeed, this follows from the fact that for 
$\CF\in \QCoh(\CY)^{\on{perf}}$ and $\CF'\in \QCoh(\CX)$, we have a functorial isomorphism
$$\CMaps(\CF^\vee,\CF')\simeq \Gamma(\CX,\CF\underset{\CO_\CX}\otimes \CF').$$

\medskip

Furthermore, note that $\QCoh(\CX)^c$ is a monoidal ideal in $\QCoh(X)^{\on{perf}}$.

\sssec{}

We shall see that certain DG indschemes are quasi-perfect in the above sense
(see \secref{ss:self duality formal compl} and \secref{sss:self duality on weird}).

\ssec{Ind-coherent sheaves on a DG indscheme}  \label{ss:IndCoh}

\sssec{}  \label{sss:IndCoh}

Let $\CY$ be an object of $\inftydgprestack_{\on{laft}}$. Following \cite[Sect. 10.1]{IndCoh}, 
we define the category $\IndCoh(\CY)$, which is a module category over $\QCoh(\CY)$
(see \cite[Sect. 10.3]{IndCoh} for the latter piece of structure). 

\medskip

Explicitly, 
$$\IndCoh(\CY)=\underset{S\in (({}^{<\infty}\affdgSch_{\on{aft}})_{/\CY})^{\on{op}}}{lim}\, \IndCoh(S),$$
where for $(f:S_1\to S_2)\in ({}^{\infty}\!\affdgSch_{\on{aft}})_{/\CY}$, the functor
$\IndCoh(S_2)\to \IndCoh(S_1)$ is $f^!$.

\medskip

It follows from \cite[Corollaries 10.2.2 and 10.5.5]{IndCoh} that in the following commutative diagram
all arrows are equivalences:
$$
\CD
\underset{S\in ((\dgSch_{\on{aft}})_{/\CY})^{\on{op}}}{lim}\, \IndCoh(S)   @>>>  
\underset{S\in ((\affdgSch_{\on{aft}})_{/\CY})^{\on{op}}}{lim}\, \IndCoh(S)  \\
@VVV    @VVV  \\
\underset{S\in (({}^{<\infty}\!\dgSch_{\on{aft}})_{/\CY})^{\on{op}}}{lim}\, \IndCoh(S)  @>>>
\underset{S\in (({}^{<\infty}\!\affdgSch_{\on{aft}})_{/\CY})^{\on{op}}}{lim}\, \IndCoh(S)=:\IndCoh(\CY).
\endCD
$$

The following is immediate from the definitions:

\begin{lem} \label{l:IndCoh and colim}
The functor 
$$\IndCoh_{\inftydgprestack_{\on{laft}}}:(\inftydgprestack_{\on{laft}})^{\on{op}}\to 
\StinftyCat_{\on{cont}}$$
takes colimits in 
$\inftydgprestack_{\on{laft}}$ to limits in $\StinftyCat_{\on{cont}}$. 
\end{lem}








\sssec{}

Let us denote by $\IndCoh^!_{\dgindSch_{\on{laft}}}$ the functor
$$(\dgindSch_{\on{laft}})^{\on{op}}\to \StinftyCat_{\on{cont}},$$
obtained from $\IndCoh^!_{\inftydgprestack_{\on{laft}}}$
by restriction along the fully faithful embedding 
$$\dgindSch_{\on{laft}}\hookrightarrow \inftydgprestack_{\on{laft}}.$$

\medskip

Thus, for every $\CX\in \dgindSch_{\on{laft}}$, we have a well-defined DG category $\IndCoh(\CX)$,
which is a module for $\QCoh(\CX)$.

\medskip

We have:

\begin{lem} \label{l:IndCoh expl as lim}
Let $\CX\in \dgindSch$ be written as in \eqref{e:indscheme as a colimit lft weak}. Then the 
natural map 
$$\IndCoh(\CX)\to \underset{\alpha}{lim}\, \IndCoh(X_\alpha),$$
is an equivalence.
\end{lem}

\begin{proof}
Follows from \lemref{l:IndCoh and colim}.
\end{proof}

\begin{rem}
We present results using the presentation of a DG indscheme as in
\eqref{e:indscheme as a colimit lft weak} rather than in \eqref{e:indscheme as a colimit lft strong},
because many DG schemes that occur in practice come in this form. The possibility of
presenting them as in \eqref{e:indscheme as a colimit lft strong} is the result of
\propref{p:presentation of indschemes lft strong} and is seldom explicit. 
\end{rem}









\ssec{Interpretation of $\IndCoh$ as a colimit and compact generation}

\sssec{}

One of the main advantages of the category $\IndCoh(\CX)$ over $\QCoh(\CX)$ for a
DG indscheme $\CX$ is that the former admits an alternative description as a colimit.

\sssec{}  \label{sss:indcoh as colim}

Indeed, recall that for a closed embedding of DG schemes $i:X_1\to X_2$, the functor $$i^!:\IndCoh(X_2)\to \IndCoh(X_1)$$
admits a left adjoint, $i^{\IndCoh}_*$, see \cite[Sect. 3.3]{IndCoh}. 

\medskip

By \lemref{l:IndCoh expl as lim} and \cite[Lemma. 1.3.3]{DG}, we have that for $\CX$ as in \eqref{e:indscheme as a colimit lft weak},
\begin{equation} \label{e:IndCoh of indscheme as colimit}
\IndCoh(\CX)\simeq \underset{\alpha}{colim}\, \IndCoh(X_\alpha),
\end{equation}
where for $\alpha_2\geq \alpha_1$, the map
$\IndCoh(X_{\alpha_2})\to \IndCoh(X_{\alpha_1})$ is $(i_{\alpha_1,\alpha_2})^{\IndCoh}_*$.

\sssec{} \label{sss:coh on ind}

For $\CX\in \dgindSch_{\on{laft}}$, we let $\Coh(\CX)$ denote the full subcategory of $\IndCoh(\CX)$ spanned by objects
$$i^\IndCoh_*(\CF),\quad i:X\to \CX \text{ is a closed embedding and } \CF\in \Coh(X).$$

By \cite[Sect. 2.2.1]{DG}, we obtain:

\begin{cor} \label{c:IndCoh compactly generated}
For $\CX\in \dgindSch$, the category $\IndCoh(\CX)$ is compactly generated by $\Coh(\CX)$.
\end{cor}

\sssec{} 

We are going to prove:

\begin{prop}  \label{p:descr of all compacts} \hfill

\smallskip

\noindent{\em(a)} $\Coh(\CX)$ is a (non-cocomplete) DG subcategory of $\IndCoh(\CX)$.

\smallskip

\noindent{\em(b)} The natural functor $\Ind(\Coh(\CX))\to \IndCoh(\CX))$ is an equivalence.

\smallskip

\noindent{\em(c)} Every compact object of $\IndCoh(\CX)$ can be realized as a direct summand
of an object of $\Coh(\CX)$.

\end{prop}

\sssec{}  \label{sss:push forward and restr}

For the proof of the above proposition, we will need the following observation:

\medskip

Let $$X'\overset{i'}\to \CX \overset{i''}\leftarrow X''$$
be closed embeddings. 

\medskip

We would like to calculate the composition
$$(i')^!\circ (i'')^\IndCoh_*:\IndCoh(X'')\to \IndCoh(X').$$

Let $A$ denote the category $(\dgSch_{\on{aft}})_{\on{closed}\,\on{in}\,\CX}$, so that $X'$ and
$X''$ correspond to indices $\alpha$ and $\alpha'$, respectively. 
Let $B$ be any category cofinal in 
$$A_{\alpha\sqcup \alpha'/}:=A_{\alpha/}\underset{A}\times A_{\alpha'/}.$$ 
For $\beta\in B$, let
$$X'=X_{\alpha}\overset{i_{\alpha,\beta}}\longrightarrow X_\beta
\overset{i_{\alpha',\beta}}\longleftarrow X_{\alpha'}=X''$$
denote the corresponding maps.

\medskip

The next assertion follows from \cite[Sect. 1.3.5]{DG}:

\begin{lem} \label{l:push forward and restr}
Under the above circumstances, we have a canonican isomorphism
$$(i')^!\circ (i'')^\IndCoh_*\simeq 
\underset{b\in B}{colim}\, (i_{\alpha,\beta})^!\circ (i_{\alpha',\beta})^\IndCoh_*.$$
\end{lem}

\sssec{Proof of \propref{p:descr of all compacts}}

To prove point (a), we only need to show that the category $\Coh(\CX)$ is preserved
by taking cones. I.e., we have to show that in the situation of \secref{sss:push forward and restr},
for
$$\CF'\in \Coh(X'),\,\, \CF''\in \Coh(X'')$$
and a map 
$$(i')^\IndCoh_*(\CF')\to (i'')^\IndCoh_*(\CF'')\in \IndCoh(\CX),$$
this map can be realized coming from a map
$$(i_{a',b})^\IndCoh_*(\CF')\to (i_{a'',b})^\IndCoh_*(\CF'')\in \IndCoh(X_b)$$
for some $b\in A_{a'\sqcup a''/}$. However, this readily follows from \lemref{l:push forward and restr}.

\medskip

Point (b) follows from point (a) combined with \corref{c:IndCoh compactly generated}. Point (c)
follows from point (b).

\qed

\ssec{The t-structure on IndCoh}  \label{ss:t-structure on IndCoh} 

\sssec{}

For $\CX\in \dgindSch_{\on{laft}}$, we define a t-structure on $\IndCoh(\CX)$ as follows.
An object $$\CF\in \IndCoh(\CX)$$ belongs to $\IndCoh^{\geq 0}$ if and only if for
every closed embedding $i:X\to \CX$ with $X\in \dgSch_{\on{aft}}$, the object
$i^!(\CF)\in \IndCoh(X)$ belongs to $\IndCoh(X)^{\geq 0}$.

\medskip

By construction, this t-structure is compatible with filtered colimits, i.e., $\IndCoh(\CX)^{\geq 0}$
is preserved by filtered colimits.

\sssec{}

We can describe this t-structure and the category 
$\IndCoh(\CX)^{\leq 0}$ more explicitly. Fix a presentation of $\CX$ as in 
\eqref{e:indscheme as a colimit lft weak}. For each $\alpha$, let $i_\alpha$ denote the corresponding map
$X_\alpha\to \CX$. By \eqref{e:IndCoh of indscheme as colimit}, we have a pair of adjoint functors
$$(i_\alpha)^\IndCoh_*:\IndCoh(X_\alpha)\rightleftarrows \IndCoh(\CX):i_\alpha^!.$$

\begin{lem} \label{l:criter for coconn}
Under the above circumstances we have:

\smallskip

\noindent{\em(a)} An object $\CF\in \IndCoh(\CX)$ belongs to $\IndCoh^{\geq 0}$ if and only if for
every $\alpha$, the object $i_\alpha^!(\CF)\in \IndCoh(X_\alpha)$ belongs to $\IndCoh(X_\alpha)^{\geq 0}$.

\smallskip

\noindent{\em(b)} The category $\IndCoh(\CX)^{\leq 0}$ is generated under colimits by the essential images of 
the functors $(i_\alpha)_*^{\IndCoh}\left(\Coh(X_\alpha)^{\leq 0}\right)$.

\end{lem}

\begin{proof}

It is easy to see that for a quasi-compact DG scheme $X$, the category $\IndCoh(X)^{\leq 0}$
is generated under colimits by $\Coh({}^{cl}\!X)^{\leq 0}$. 
In particular, by adjunction,
an object $\CF\in \IndCoh(X)$ is coconnective if and only if its restriction to 
$^{cl}\!X$ is coconnective.  

\medskip

Hence, in the definition of $\IndCoh(\CX)^{\geq 0}$, instead of all closed embeddings 
$X\to \CX$, it suffices to consider only those with $X$ a classical scheme. 

\medskip

This implies point (a)
of the lemma by \lemref{l:maps out of qc}. Point (b) follows formally from point (a).

\end{proof}

\sssec{}

Suppose $i: X\rightarrow \CX$ is a closed embedding of a DG scheme into a DG indscheme.  We then have:

\begin{lem}\label{l:exactness of closed embedding}
The functor $i^{\IndCoh}_*$ is t-exact.
\end{lem}
\begin{proof}
Since $i^{\IndCoh}_*$ is the left adjoint of $i^!$, it is right t-exact.  Thus we need to show that for $\CF \in \IndCoh(X)^{\geq 0}$, 
we have $i_{\alpha}^! \circ i^{\IndCoh}_*(\CF) \in \IndCoh(X_{\alpha})^{\geq 0}$ for every closed embedding 
$i_{\alpha}: X_{\alpha}\rightarrow \CX$. 
However, this follows from \lemref{l:push forward and restr}.
\end{proof}

\sssec{}

Recall the full (but not cocomplete) subcategory $\Coh(\CX)\subset \IndCoh(\CX)$, see 
\secref{sss:coh on ind} above. From \lemref{l:exactness of closed embedding} we obtain:

\begin{cor}
The full subcategories 
$$\Coh(\CX)\subset \IndCoh(\CX)^c\subset \IndCoh(\CX)$$ are preserved by the truncation
functors.
\end{cor}

Thus, taking into account \propref{p:descr of all compacts},
we obtain that the t-structure on $\IndCoh(\CX)$ can also be described as the ind-extension of
the t-structure on $\Coh(\CX)$:

\begin{cor}  \label{c:generation of t-structure}
The category $\IndCoh(\CX)^{\geq 0}$ 
is generated under filtered colimits by $\Coh(\CX)^{\geq 0}$.
\end{cor}

\ssec{Serre duality on DG indschemes}

We shall now show that the category $\IndCoh(\CX)$ is canonically self-dual, i.e. there exists a canonical equivalence
\begin{equation} \label{e:self-duality IndCoh}
\bD_\CX^{\on{Serre}}:\IndCoh(\CX)^\vee\simeq \IndCoh(\CX).
\end{equation}

\sssec{}

Let us write $\CX$ as in \eqref{e:indscheme as a colimit lft weak}. 
Combining \eqref{e:IndCoh of indscheme as colimit} with \cite[Lemma 2.2.2]{DG} and \cite[Sect. 9.2.3]{IndCoh}, we obtain:

\begin{cor} \label{c:Serre for Ind}
Serre duality defines a canonical equivalence:
$$\IndCoh(\CX)^\vee\simeq \IndCoh(\CX).$$
\end{cor}

\medskip

Note that by \secref{sss:canonicity of presentation laft}, any other way of writing $\CX$ as in 
\eqref{e:indscheme as a colimit lft strong} will give rise to a canonically isomorphic duality functor. 

\sssec{}

Let us describe the equivalence of \corref{c:Serre for Ind} more explicitly. Namely, we would
like to describe the corresponding pairing:
\begin{equation} \label{e:pairing on IndCoh}
\IndCoh(\CX)\otimes \IndCoh(\CX)\to \Vect.
\end{equation}

\sssec{}  \label{sss:Gamma IndCoh}

For a DG scheme $X$ almost of finite type, let
$$\Gamma^{\IndCoh}(X,-):\IndCoh(X)\to \Vect$$
denote the functor $(p_X)^{\IndCoh}_*$ of \cite{IndCoh}, Proposition 3.1.1, where $p_X:X\to \on{pt}$.

\medskip

For a DG indscheme $\CX$, written as in \eqref{e:indscheme as a colimit lft weak}, we define
the functor
$$\Gamma^{\IndCoh}(\CX,-):\IndCoh(\CX)\to \Vect$$
to be given by the compatible family of functors $\Gamma^{\IndCoh}(X_\alpha,-):\IndCoh(X_\alpha)\to \Vect$.

\medskip

Again, by \secref{sss:canonicity of presentation laft}, the above definition of $\Gamma^{\IndCoh}(\CX,-)$
is canonically independent of the choice of the presentation \eqref{e:indscheme as a colimit lft weak}.

\sssec{}

The definition of the functor $\bD^{\on{Serre}}_\CX$ in \eqref{e:self-duality IndCoh}
and \cite[Sect. 9.2.2]{IndCoh} imply: 

\begin{cor} \label{c:pairing on IndCoh}
The functor \eqref{e:pairing on IndCoh} is canonically isomorphic to the composite
$$\IndCoh(\CX)\otimes \IndCoh(\CX)\overset{\boxtimes}\longrightarrow \IndCoh(\CX\times \CX)\overset{\Delta_\CX^!}
\longrightarrow \IndCoh(\CX)\overset{\Gamma^{\IndCoh}(\CX,-)}\longrightarrow \Vect.$$
\end{cor}

\ssec{Functoriality of $\IndCoh$ under pushforwards}

\sssec{}

Recall the functor $\IndCoh_{\dgSch_{\on{aft}}}:\dgSch_{\on{aft}}\to \StinftyCat_{\on{cont}}$ of \cite[Sect. 3.2]{IndCoh}, which assigns
to $X\in \dgSch_{\on{aft}}$ the category $\IndCoh(X)$ and to a map $f:X_1\to X_2$ the functor 
$$f^\IndCoh_*:\IndCoh(X_1)\to \IndCoh(X_2).$$

\medskip

Let
$$(\dgSch_{\on{aft}})_{\on{closed}}\subset (\dgSch_{\on{aft}})_{\on{proper}}\subset \dgSch_{\on{aft}}$$
be the 1-full subcategories, where we restrict $1$-morphisms to be closed embeddings (resp., proper). Let
$$\IndCoh_{(\dgSch_{\on{aft}})_{\on{closed}}} \text{ and } \IndCoh_{(\dgSch_{\on{aft}})_{\on{proper}}}$$
be the restriction of $\IndCoh_{\dgSch_{\on{aft}}}$ to these subcategories.

\sssec{}

We shall say that a map of classical indschemes $f:\CX_1\to \CX_2$ is an \emph{ind-closed embedding} (resp., \emph{ind-proper}) if 
the following condition is satisfied:

\medskip

Whenever $X_i\hookrightarrow \CX_i$ are closed embeddings with $X_i\in \Sch_{\on{ft}}$
such that there exists a commutative diagram
$$
\CD
X_1 @>>>  \CX_1 \\
@V{f'}VV    @VV{f}V  \\
X_2  @>>>  \CX_2,
\endCD
$$
the map $f'$ (which is automatically unique!), is a closed embedding (resp., proper). 

\medskip

Equivalently, one can reformulate this as follows: if $$\CX_1:=\underset{\alpha\in A}{colim}\, X_{1,a} \text{ and }
\CX_2:=\underset{\beta\in A}{colim}\, X_{2,\beta},$$
then for every index $\alpha$, 
and every/some index $\beta$ for which $X_{1,\alpha}\to \CX_1 \to \CX_2$ factors as 
$$X_{1,\alpha}\to X_{2,\beta}\to \CX_2,$$ 
the map $X_{1,\alpha}\to X_{2,\beta}$ is a closed embedding
(resp., proper). 

\medskip

It is easy to see that if $\CX_1=X_1\in \Sch_{\on{qsep-qc}}$, then $f:X_1\to \CX_2$ is an ind-closed embedding
if and only if it is a closed embedding.

\begin{rem}
Note that, in general, ``closed embedding" is stronger than ``ind-closed embedding." For instance,
$$\on{Spf}(k[\![t]\!])\to \Spec(k[t])$$
is an an ``ind-closed emnedding", but not a closed embedding. 
\end{rem}

\sssec{}

We shall say that a map of DG indschemes $f:\CX_1\to \CX_2$ is an \emph{ind-closed embedding} (resp., \emph{ind-proper}) if 
the induced map of classical indschemes $^{cl}\CX_1\to {}^{cl}\CX_2$ is an \emph{ind-closed embedding} (resp., \emph{ind-proper}).

\medskip

Let 
$$(\dgindSch_{\on{laft}})_{\on{ind-closed}}\subset (\dgindSch_{\on{laft}})_{\on{ind-proper}}$$
denote the corresponding 1-full subcategories of $\dgindSch_{\on{laft}}$.

\medskip

Let 
\begin{equation} \label{e:various IndCoh}
\IndCoh_{(\dgindSch_{\on{laft}})_{\on{ind-closed}}},\,\, \IndCoh_{(\dgindSch_{\on{laft}})_{\on{ind-proper}}}
\text{ and } \IndCoh_{\dgindSch_{\on{laft}}}
\end{equation}
denote the left Kan extensions of the functors
$$\IndCoh_{(\dgSch_{\on{aft}})_{\on{closed}}},\,\, \IndCoh_{(\dgSch_{\on{aft}})_{\on{proper}}} \text{ and } \IndCoh_{\dgSch_{\on{aft}}}$$
along the fully faithful embeddings
$$(\dgSch_{\on{aft}})_{\on{closed}}\hookrightarrow 
(\dgindSch_{\on{laft}})_{\on{ind-closed}},\,\, (\dgSch_{\on{aft}})_{\on{proper}}\hookrightarrow 
(\dgindSch_{\on{laft}})_{\on{ind-proper}}$$
and 
$$\dgSch_{\on{aft}}\hookrightarrow \dgindSch_{\on{laft}},$$
respectively.

\medskip

From \eqref{e:IndCoh of indscheme as colimit} and \secref{sss:canonicity of presentation laft} we obtain:

\begin{cor} \label{c:value of IndCoh}
For $\CX\in \dgindSch_{\on{laft}}$, the value of the functor $\IndCoh_{(\dgindSch_{\on{laft}})_{\on{ind-closed}}}$ on $\CX$
is canonically equivalent to $\IndCoh(\CX)$.
\end{cor}

\sssec{}

By construction, we have the natural transformations
\begin{multline} \label{e:pushforwards}
\IndCoh_{(\dgindSch_{\on{laft}})_{\on{ind-proper}}}\to \IndCoh_{\dgindSch_{\on{laft}}}|_{(\dgindSch_{\on{laft}})_{\on{ind-proper}}} 
\text{ and } \\
\IndCoh_{(\dgindSch_{\on{laft}})_{\on{ind-closed}}}\to 
\IndCoh_{(\dgindSch_{\on{laft}})_{\on{proper}}}|_{(\dgindSch_{\on{laft}})_{\on{ind-closed}}}.
\end{multline}

\begin{prop} \label{p:value of IndCoh}
The natural transformations \eqref{e:pushforwards} are equivalences.
\end{prop}

\begin{proof}
For a given $\CX\in \dgindSch$, the value of the functors \eqref{e:various IndCoh} on it are given by
$$\underset{X\in (\dgSch_{\on{aft}})_{\on{closed}\,\on{in}\,\CX}}{colim}\, \IndCoh(X),\,\,
\underset{X\in (\dgSch_{\on{aft}})_{\on{proper}\,\on{over}\,\CX}}{colim}\, \IndCoh(X)$$ and 
$$\underset{X\in (\dgSch_{\on{aft}})_{/\CX}}{colim}\, \IndCoh(X),$$
respectively. 

\medskip

Hence, to prove the proposition, it suffices to show that the functors
$$(\dgSch_{\on{aft}})_{\on{closed}\,\on{in}\,\CX}\to (\dgSch_{\on{aft}})_{\on{proper}\,\on{over}\,\CX}\to
(\dgSch_{\on{aft}})_{/\CX}$$
are cofinal. Since both arrows are fully faithful embeddings, it suffices to show that the functor
$$(\dgSch_{\on{aft}})_{\on{closed}\,\on{in}\,\CX}\to
(\dgSch_{\on{aft}})_{/\CX}$$
is cofinal, but the latter is given by \corref{c:laft cofinality}.

\end{proof}

\sssec{}

Thus, from \propref{p:value of IndCoh} we obtain that for a morphism $f:\CX_1\to \CX_2$
we have a well-defined functor
$$f^\IndCoh_*:\IndCoh(\CX_1)\to \IndCoh(\CX_2).$$

\medskip

Concretely, the functor $f^\IndCoh_*$ can be described as follows. By \eqref{e:IndCoh of indscheme as colimit}, 
objects of $\IndCoh(X)$ are 
colimits of objects of the form $(i_1)^{\IndCoh}_*(\CF_1)$ for
$\CF_1\in \IndCoh(X_1)$, where $X_1$ is a DG scheme almost of finite type equipped with a closed embedding 
$X_1\overset{i_1}\to \CX_1$. By continuity, the functor $f^\IndCoh_*$ is completely determined by its values on such
objects. 

\medskip

By \corref{c:laft cofinality},
we can factor the map 
$$X_1\overset{i_1}\to \CX_1\overset{f}\to \CX_2$$
as
$$X_1\overset{g}\to X_2\overset{i_2}\to \CX_2,$$
where $X_2\in \dgSch_{\on{aft}}$ and $i_2$ being a closed embedding. We set
$$f^\IndCoh_*((i_1)^{\IndCoh}_*(\CF_1)):=(i_2)^\IndCoh_*(g^\IndCoh_*(\CF_1)).$$

\medskip

The content of \propref{p:value of IndCoh} is that this construction extends to a well-defined
functor $f^\IndCoh_*:\IndCoh(\CX_1)\to \IndCoh(\CX_2)$. 

\medskip

Note that the functor $\Gamma^\IndCoh(\CX,-)$ of \secref{sss:Gamma IndCoh} is a particular
instance of this construction for $\CX_1=\CX$ and $\CX_2=\on{pt}$.

\sssec{}

It follows from the definition of the self-duality functors
$$\bD_{\CX_i}^{\on{Serre}}:\IndCoh(\CX_i)^\vee\to \IndCoh(\CX_i),\quad i=1,2$$
that the dual of the functor $f_*^\IndCoh$ identifies canonically with $f^!$. 

\sssec{}

For a morphism of DG indschemes, the pushforward functor on $\IndCoh$ interacts with the t-structure in the usual way:

\begin{lem}
Let $f: \CX_1\rightarrow \CX_2$ be a map of indschemes. Then the functor $f^\IndCoh_*$ is
left t-exact. Furthermore, if $f$ is a closed embedding, then it is t-exact.
\end{lem}
\begin{proof}
Let $\CF\in \IndCoh(\CX_1)^{\geq 0}$. We wish to show that $f^\IndCoh_*(\CF) \in
\IndCoh(\CX_2)^{\geq 0}$. By \corref{c:generation of t-structure}, 
we can assume that $\CF = (i_1)^{\IndCoh}_*(\CF_1)$ for $\CF_1 \in \IndCoh(X_1)^{\geq 0}$ where 
$i_1: X_1 \rightarrow \CX_1$ is a closed embedding.  

\medskip

Let now
$$X_1 \overset{g}{\rightarrow} X_2 \overset{i_2}{\rightarrow} \CX_2 $$
be a factorization of $f\circ i_1$, where $i_2$ is a closed embedding.  We have:
$$f^\IndCoh_*(\CF) \simeq f^\IndCoh_* ((i_1)^\IndCoh_*(\CF_1)) = (i_2)^\IndCoh_*(g^\IndCoh_*(\CF_1)). $$
By \lemref{l:exactness of closed embedding}, $(i_2)^\IndCoh_*(g^\IndCoh_*(\CF_1)) \in \IndCoh(\CX_2)^{\geq 0}$.

\medskip

Suppose now that $f$ is a closed embedding.  In this case, we wish to show that $f^\IndCoh_*$ is also right t-exact.  
Let $\CF\in \IndCoh(\CX_1)^{\leq 0}$.  By \lemref{l:criter for coconn}(b), we can assume that $\CF = (i_1)^\IndCoh_*(\CF_1)$ for $\CF_1 \in \IndCoh(X_1)^{\leq 0}$ where $i_1: X_1 \rightarrow \CX_1$ is a closed embedding.  The result now follows from the fact that the composed map
$$ X_1 \rightarrow \CX_1 \rightarrow \CX_2 $$
is a closed embedding and \lemref{l:exactness of closed embedding}.
\end{proof}

\ssec{Adjunction for proper maps}

\sssec{}

Consider the functor 
$$\IndCoh^!_{\dgindSch_{\on{laft}}}:\dgindSch_{\on{laft}}^{\on{op}}\to \StinftyCat_{\on{cont}},$$
and let 
$$\IndCoh^!_{(\dgindSch_{\on{laft}})_{\on{ind-proper}}} \text{ and } \IndCoh^!_{(\dgindSch_{\on{laft}})_{\on{ind-closed}}}$$
be the restrictions of $\IndCoh^!_{\dgindSch_{\on{laft}}}$ to the corresponding 1-full subcategories.

\medskip

In addition, consider the corresponding functors
$$\IndCoh^!_{\dgSch_{\on{aft}}},\,\, \IndCoh^!_{(\dgSch_{\on{aft}})_{\on{proper}}} \text{ and } \IndCoh^!_{(\dgSch_{\on{aft}})_{\on{closed}}}$$
for $\dgSch_{\on{aft}}$ instead of $\dgindSch_{\on{laft}}$.

\medskip

As in \propref{p:value of IndCoh}, we have:

\begin{lem} \label{l:value of IndCoh !}
The natural maps
$$ \IndCoh^!_{(\dgindSch_{\on{laft}})_{\on{ind-proper}}} \to \on{RKE}_{(\dgSch_{\on{aft}})^{\on{op}}_{\on{proper}}\hookrightarrow (\dgindSch_{\on{laft}})^{\on{op}}_{\on{ind-proper}}}
(\IndCoh^!_{(\dgSch_{\on{aft}})_{\on{proper}}}) $$
and
$$ \IndCoh^!_{(\dgindSch_{\on{laft}})_{\on{ind-closed}}} \to \on{RKE}_{(\dgSch_{\on{aft}})^{\on{op}}_{\on{closed}}\hookrightarrow (\dgindSch_{\on{laft}})^{\on{op}}_{\on{ind-closed}}}
(\IndCoh^!_{(\dgSch_{\on{aft}})_{\on{closed}}}) $$
are isomorphisms.
\end{lem}

We shall now deduce the following:

\begin{cor}  \label{c:! * adj}
The functor 
$$\IndCoh_{(\dgindSch_{\on{laft}})_{\on{ind-proper}}}:(\dgindSch_{\on{laft}})_{\on{ind-proper}}\to \StinftyCat_{\on{cont}}$$ 
is obtained
from the functor 
$$\IndCoh^!_{(\dgindSch_{\on{laft}})_{\on{ind-proper}}}:(\dgindSch_{\on{laft}})^{\on{op}}_{\on{ind-proper}}\to \StinftyCat_{\on{cont}}$$ by passing
to left adjoints.
\end{cor}

This corollary means that for
a proper map $f:\CX_1\to \CX_2$ in $\dgindSch_{\on{laft}}$, the functor
$$f_*^{\IndCoh}:\IndCoh(\CX_1)\to \IndCoh(\CX_2)$$
is the left adjoint of $f^!:\IndCoh(\CX_2)\to \IndCoh(\CX_1)$ in a way compatible with compositions, and that
this data is homotopy-coherent.

\begin{proof}

This follows from the corresponding fact for the functors $\IndCoh_{(\dgSch_{\on{aft}})_{\on{proper}}}$
and $\IndCoh^!_{(\dgSch_{\on{aft}})_{\on{proper}}}$ (see \cite[Theorem 5.2.2(a)]{IndCoh}), and the following
general assertion:

\medskip

Let $F:\bC_1\to \bC_2$ be a functor between $\infty$-categories. Let $\Phi_1:\bC_1\to \StinftyCat_{\on{cont}}$
be a functor such that for every $\bc'_1\to \bc''_1$, the corresponding functor
$$\Phi_1(\bc'_1)\to \Phi_1(\bc''_1)$$
admits a continuous right adjoint. Let $\Psi_1:\bC_1^{\on{op}}\to \StinftyCat$ be the resulting functor given by
taking the right adjoints. 

\medskip

Let $\Phi_2$ and $\Psi_2$ be the left (resp., right) Kan extension of $\Phi_1$ (resp., $\Psi_1)$ along $F$ 
(resp., $F^{\on{op}}$). 

\medskip

The following is a version of \cite[Lemma 1.3.3]{DG}:

\begin{lem}
Under the above circumstances, the functor $\Psi_2$ is obtained from $\Phi_2$ by taking right adjoints.
\end{lem}

\end{proof}

\ssec{Proper base change}

\sssec{}

Let 
$$
\CD
\CY_1  @>{g_1}>>  \CX_1  \\  
@V{f_Y}VV     @VV{f_X}V  \\
\CY_2  @>{g_2}>>  \CX_2 
\endCD
$$
be a Cartesian diagram of DG indschemes, with the maps $f_X$ and $f_Y$ ind-proper. From the isomorphism
of functors
$$g_1^!\circ f_X^!\simeq f_Y^!\circ g_2^!,$$
by adjunction, we obtain a natural
transformation
\begin{equation} \label{e:base change one}
(f_Y)^\IndCoh_*\circ g_1^!\to g_2^!\circ (f_X)^\IndCoh_*.
\end{equation}

\begin{prop} \label{p:base change one}
The natural transformation \eqref{e:base change one} is an isomorphism.
\end{prop}

The proof of this proposition will occupy the next few subsections. 

\sssec{Proof of \propref{p:base change one}, Step 1}

The assertion readily reduces to the case when $\CY_2$ is a DG scheme, denote it $Y_2$. Next,
we are going to show that we can assume $\CX_2$ is also a DG scheme.

\sssec{Interlude}

Consider the following general paradigm. Let $G:\bC_2\to \bC_1$ be a functor between $\infty$-categories.
Let $A$ be a category of indices, and suppose we are given an $A$-family of commutative diagrams
$$
\CD
\bC_{1,\alpha} @<{i_{1,\alpha}}<<  \bC_1  \\
@A{G_\alpha}AA  @AA{G}A   \\
\bC_{2,\alpha} @<{i_{2,\alpha}}<<  \bC_2.
\endCD
$$

Assume that for each $\alpha\in A$, the functor $G_\alpha$ admits a left adjoint $F_\alpha$. Furthermore, assume that for each 
map $\alpha'\to \alpha''$ in $A$, the natural transformation in the diagram 
\begin{equation} \label{e:adjoints index}
\xy
(0,0)*+{\bC_{1,\alpha''}}="A";
(20,0)*+{\bC_{1,\alpha'}}="B";
(0,-20)*+{\bC_{2,\alpha''}}="C";
(20,-20)*+{\bC_{2,\alpha'}}="D";
{\ar@{->}_{i_{1,\alpha',\alpha''}} "B";"A"};
{\ar@{->}^{i_{2,\alpha',\alpha''}} "D";"C"};
{\ar@{->}^{F_{\alpha''}} "A";"C"};
{\ar@{->}^{F_{\alpha'}} "B";"D"};
{\ar@{=>} "A";"D"};
\endxy
\end{equation}
is an isomorphism. 

\medskip

Finally, assume that the functors
$$\bC_1\to \underset{\alpha\in A}{lim}\, \bC_{1,\alpha} \text{ and } \bC_2\to \underset{\alpha\in A}{lim}\, \bC_{2,\alpha}$$
are equivalences.  

\medskip

Under the above circumstances we have:

\begin{lem} \label{l:adjoints via limit}
The functor $G$ admits a left adjoint, denoted $F$, and for every $\alpha\in A$, the natural transformation
in the diagram
$$
\xy
(0,0)*+{\bC_{1,\alpha}}="A";
(20,0)*+{\bC_{1}}="B";
(0,-20)*+{\bC_{2,\alpha}}="C";
(20,-20)*+{\bC_{2}}="D";
{\ar@{->}_{i_{1,\alpha}} "B";"A"};
{\ar@{->}^{i_{2,\alpha}} "D";"C"};
{\ar@{->}^{F_{\alpha}} "A";"C"};
{\ar@{->}^{F} "B";"D"};
{\ar@{=>} "A";"D"};
\endxy
$$
is an isomorphism.
\end{lem}

\sssec{Proof of \propref{p:base change one}, Step 2}

Write 
$$\CX_2\simeq \underset{\alpha\in A}{colim}\, X_{2,\alpha}$$
where $A$ is the category $(\affdgSch_{\on{aft}})_{/\CX_2}$.
Set $$\CX_{1,\alpha}:=X_{2,\alpha}\underset{\CX_2}\times \CX_1.$$
It is clear that
$$\CX_1\simeq \underset{\alpha\in A}{colim}\, \CX_{1,\alpha},$$
where the colimit is taken in $\on{PreStk}_{\on{laft}}$. 

\medskip

Hence, by \lemref{l:IndCoh and colim}, 
$$\IndCoh(\CX_1)\simeq \underset{\alpha\in A^{\on{op}}}{lim}\, \IndCoh(\CX_{1,\alpha}).$$

Set
$$\bC_2=\IndCoh(\CX_2),\,\, \bC_1=\IndCoh(\CX_1),\,\, \bC_{2,\alpha}=\IndCoh(X_{2,\alpha}),\,\,
\bC_{1,\alpha}=\IndCoh(\CX_{1,\alpha}).$$

The condition of \lemref{l:adjoints via limit} is equivalent to the assertion of \propref{p:base change one}
when instead of $\CX_2\in \dgindSch_{\on{laft}}$ we take $X_{2,\alpha}\in \dgSch_{\on{laft}}$. 

\medskip

Thus,
the assertion of \lemref{l:adjoints via limit} reduces the assertion of \propref{p:base change one} to the
case when both $\CY_2=Y_2$ and $\CX_2=X_2$ are DG schemes.

\sssec{Proof of \propref{p:base change one}, Step 3}

Write
$$\CX_1\simeq \underset{\beta\in B}{colim}\, X_{1,\beta},$$ 
where $X_{1,\beta}\in \dgSch_{\on{aft}}$ and $i_{X,\beta}:X_{1,\beta}\to \CX_1$ are closed embeddings.

\medskip

Set
$$Y_{1,\beta}:=Y_2\underset{X_2}\times X_{1,\beta}.$$

We have:
$$\CY_1\simeq \underset{\beta\in B}{colim}\, Y_{1,\beta},$$ 
Let $i_{Y,\beta}$ denote the correspoding closed embedding $Y_{1,\beta}\to \CY_1$,
and let $g_\beta$ denote the map $Y_{1,\beta}\to X_{1,\beta}$. 
Note that the maps $f_X\circ i_{X,\beta}:X_{1,\beta}\to X_2$ and $f_Y\circ i_{Y,\beta}:Y_{1,\beta}\to Y_2$ are proper, by assumption. 

\medskip

By \eqref{e:IndCoh of indscheme as colimit}, we have:
$$\on{Id}_{\IndCoh(\CX_1)}\simeq \underset{\beta\in B}{colim}\, (i_{X,\beta})_*^\IndCoh\circ (i_{X,\beta})^! \text{ and }
\on{Id}_{\IndCoh(\CY_1)}\simeq \underset{\beta\in B}{colim}\, (i_{Y,\beta})_*^\IndCoh\circ (i_{Y,\beta})^!$$

Hence, we can rewrite the functor $(f_Y)^\IndCoh_*\circ g_1^!$
as
$$\underset{\beta\in B}{colim}\, (f_Y)^\IndCoh_*\circ (i_{Y,\beta})_*^\IndCoh\circ (i_{Y,\beta})^! \circ g_1^!,$$
and the functor $g_2^!\circ (f_X)^\IndCoh_*$ as 
$$\underset{\beta\in B}{colim}\,  g_2^!\circ (f_X)^\IndCoh_* \circ (i_{X,\beta})_*^\IndCoh\circ (i_{X,\beta})^! .$$

It follows from the construction that the map in \eqref{e:base change one} is given
by a compatible system of maps for each $\beta\in B$

\begin{multline*} 
(f_Y)^\IndCoh_*\circ (i_{Y,\beta})_*^\IndCoh\circ (i_{Y,\beta})^! \circ g_1^!\simeq 
(f_Y\circ i_{Y,\beta})^\IndCoh_* \circ (g_1\circ i_{Y,\beta})^! \simeq \\
(f_Y\circ i_{Y,\beta})^\IndCoh_*  \circ  (i_{X,\beta}\circ g_\beta)^! \simeq
(f_Y\circ i_{Y,\beta})^\IndCoh_*  \circ g_\beta^! \circ i_{X,\beta}^! \to \\
\to g_2^! \circ (f_X\circ i_{X,\beta})^\IndCoh_*\circ i_{X,\beta}^!\simeq 
g_2^!\circ (f_X)^\IndCoh_* \circ (i_{X,\beta})_*^\IndCoh\circ i_{X,\beta}^!,
\end{multline*}
where the arrow
$$(f_Y\circ i_{Y,\beta})^\IndCoh_*  \circ g_\beta^! \to g_2^! \circ (f_X\circ i_{X,\beta})^\IndCoh_*$$
is base change for the Cartesian square
$$
\CD
Y_{1,\beta}  @>{g_\beta}>>  X_{1,\beta}  \\
@V{f_Y\circ i_{Y,\beta}}VV    @VV{f_X\circ i_{X,\beta}}V  \\
Y_2  @>{g_2}>>  X_2.
\endCD
$$

Hence, the required isomorphism follows from proper base change in the case of DG schemes,
see \cite[Proposition 3.4.2]{IndCoh}.

\qed

\sssec{}

Let 
$$
\CD
\CY_1  @>{g_1}>>  \CX_1  \\  
@V{f_Y}VV     @VV{f_X}V  \\
\CY_2  @>{g_2}>>  \CX_2 
\endCD
$$
now be a Cartesian diagram of DG indschemes, where the maps $g_1$ and $g_2$ 
are ind-proper. From the isomorphism of functors
$$(g_2)^\IndCoh_*\circ (f_Y)^\IndCoh_*\simeq (f_X)^\IndCoh_*\circ (g_1)^\IndCoh_*$$
by adjunction, we obtain a natural
transformation
\begin{equation} \label{e:base change two}
(f_Y)^\IndCoh_*\circ g_1^!\to g_2^!\circ (f_X)^\IndCoh_*.
\end{equation}

\begin{prop} \label{p:base change two}
The natural transformation \eqref{e:base change two} is an isomorphism.
\end{prop}

\begin{rem}
It is easy to see from \corref{c:! * adj} that when both pairs of morphisms (i.e., $(f_X,f_Y)$ and $(g_1,g_2)$) are ind-proper,
then the natural transformations \eqref{e:base change one} and \eqref{e:base change two} are canonically
isomorphic.
\end{rem}

\begin{proof}

By \eqref{e:IndCoh of indscheme as colimit}, we can assime that $\CX_1=X_1\in \dgSch_{\on{aft}}$.
Factor the map $f:X_1\to \CX_2$ as a composition
$$X_1\to X_2\to \CX_2,$$
where $X_2\in \dgSch_{\on{aft}}$ and $X_2\to \CX_2$ is a closed embedding. Such a factorization is possible 
by \corref{c:laft cofinality}.

\medskip

This reduces the assertion of the proposition to the analyses of the following two cases: (1) when the morphism
$f$ is a closed embedding (and, in particular, proper); and (2) when both $\CX_1=X_1$ and $\CX_2=X_2$ are
DG schemes.

\medskip

Now, the assertion in case (1) follows from \propref{p:base change one}. The assertion in case (2) follows
by repeating the argument of Step 3 in the proof of \propref{p:base change one}. 

\end{proof}

\begin{rem} \label{r:base change ind}
The isomorphisms as in \eqref{e:base change one} and \eqref{e:base change two}
can be defined for all Cartesian diagrams of DG indschemes, i.e., we one does not need to require that
either pair of maps be ind-proper. However, the construction is more involved as there is no a priori map in either direction. 

\medskip

For an individual diagram, such an isomorphism is 
easy to deduce from \cite[Sect. 5]{IndCoh}, where the corresponding natural transformations were constructed in the case of DG 
schemes.

\medskip  

A functorial construction of these natural transformations for indschemes compatible with composition requires additional 
work and will be carried out in \cite{GR}. Furthermore, as in \cite[Sect. 10.6]{IndCoh}
one can combine the functors
$$\IndCoh^!_{\on{PreStk}_{\on{laft}}}:(\on{PreStk}_{\on{laft}})^{\on{op}}\to \StinftyCat_{\on{cont}}$$ and 
$$\IndCoh_{\dgindSch_{\on{laft}}}:\dgindSch_{\on{laft}}\to \StinftyCat_{\on{cont}}$$ 
to a functor
$$\IndCoh_{(\on{PreStk}_{\on{laft}})_{\on{corr:ind-sch;all}}}:(\on{PreStk}_{\on{laft}})_{\on{corr:ind-sch;all}}\to 
\StinftyCat_{\on{cont}},$$
where $(\on{PreStk}_{\on{laft}})_{\on{corr:ind-sch;all}}$ is the category of correspondences, whose
objects are prestacks locally almost of finite type $\CY$, and whose morphisms are correspondences
$$
\CD
\CY_{1,2} @>{g}>>  \CY_1 \\
@V{f}VV  \\
\CY_2,
\endCD
$$
where the morphism $g$ is arbitrary, and the morphism $f$ is ind-schematc (i.e., a morphism such that its base change by an
affine DG scheme yields a DG indscheme). 
\end{rem}

\ssec{Groupoids in $\dgindSch$}

\sssec{}  \label{sss:groupoid}

Let $\CX^\bullet$ be a simplicial object in $\dgindSch$, arising from a groupoid
object 
\begin{equation} \label{e:gen groupoid}
p_s,p_t:\CX^1\rightrightarrows \CX^0
\end{equation}
(see \cite{Lu0}, Definition 6.1.2.7).

\medskip

Suppose that the face maps in the above simplicial DG indscheme are ind-proper
(equivalently, the maps $p_s,p_t$ in \eqref{e:gen groupoid} are ind-proper). 

\medskip

In this case, the forgetful functor
$$\on{Tot}(\IndCoh(\CX^\bullet))\to \IndCoh(\CX^0)$$
admits a left adjoint; moreover, the resulting monad on $\IndCoh(\CX^0)$, when
viewed as a plain endo-functor of $\IndCoh(\CX^0)$, is naturally isomorphic to 
$$(p_t)^{\IndCoh}_*\circ (p_s)^!.$$
The proof is the same as that of \cite[Proposition 8.2.3]{IndCoh}.

\sssec{}  \label{sss:Cech}

Assume that in the situation of \secref{sss:groupoid}, the groupoid arises as the
\v{C}ech nerve of a morphism $f:\CX\to \CY$, which is ind-proper and
surjective.\footnote{I.e., the base change of $f$ by an object of
$\affdgSch_{\on{aft}}$ yields a morphism surjective on geometric points.}  Let $\CX^\bullet/\CY$ denote the resulting simplicial object.

\medskip

In this case, the augmentation 
$$\CX^\bullet/\CY\to \CY$$
gives rise to a functor
\begin{equation} \label{e:Cech}
\IndCoh(\CY)\to \on{Tot}(\IndCoh(\CX^\bullet/\CY)).
\end{equation}

As in \cite[Proposition 8.2.3]{IndCoh} we have:
\begin{lem} \label{l:Cech}
Under the above circumstances, the functor \eqref{e:Cech} is an equivalence.
\end{lem}

\medskip

Note that the composition
$$\IndCoh(\CY)\to \on{Tot}(\IndCoh(\CX^\bullet/\CY))\to \IndCoh(\CX)$$
is the functor $f^!$, and hence its left adjoint is $f^{\IndCoh}_*$.

\section{Closed embeddings into a DG indscheme and push-outs}  \label{s:pushouts}

Let $X$ be a scheme, and $Z_1$ and $Z_2$ be two closed subschemes.  In this case, we can consider the subscheme given by the union of $Z_1$ and $Z_2$; in fact,
this is the coproduct
in the category of closed subschemes of $X$ (locally, the ideal of the union is the intersection 
of the ideals of $Z_1$ and $Z_2$). The same operation is well-defined when $X$ is no longer
a scheme, but an indscheme: indeed the union of $Z_1$ and $Z_2$ in $X$ is the same as their
union in $X'$, if $X'$ is another closed subscheme of $X$ which contains $Z_{1}$ and $Z_{2}$.

\medskip

However, one might be suspicious of the operation of union in the DG setting since closed DG subschemes are no longer in bijection with ``ideals.''

\medskip

The goal of this section
is to show that in this case, the operation of union behaves as well as for schemes.

\medskip

In addition, we will consider a particular situation in which push-outs in the category
of DG schemes exist and are well-behaved. This will allow us, in particular, to show
that DG indschemes contain ``many" closed subschemes.



\ssec{Closed embeddings into a DG scheme}

\sssec{}

For a morphism $f:Y\to X$ in $\dgSch_{\on{qsep-qc}}$ consider the category
$$(\dgSch_{\on{qsep-qc}})_{Y/\,/X}$$
of factorizations of $f$; i.e. objects are given by
$$Y\to Z\overset{\phi}\to X$$
and morphisms are commutative diagrams
\begin{gather}  
\xy
(-15,0)*+{Y}="A";
(0,8)*+{Z_1}="B";
(0,-8)*+{Z_2}="C";
(15,0)*+{X.}="D";
{\ar@{->}^{\phi_1} "B";"D"};
{\ar@{->}_{\phi_2} "C";"D"};
{\ar@{->} "B";"C"};
{\ar@{->} "A";"B"};
{\ar@{->} "A";"C"};
\endxy
\end{gather}

\medskip

Let 
$$\dgSch_{Y/,\on{closed}\,\on{in}\,X}\subset (\dgSch_{\on{qsep-qc}})_{Y/\,/X}$$
be the full subcategory, spanned by those objects  $Y\to Z\overset{\phi}\to X$, for
which the map $\phi$ is a closed embedding.

\sssec{}

We shall prove:

\begin{prop} \label{p:colimits as closed} \hfill

\smallskip

\noindent{\em(a)}
The category $\dgSch_{Y/,\on{closed}\,\on{in}\,X}$ contains finite colimits (and, in particular, an initial
object). 

\smallskip

\noindent{\em(b)}
The formation of colimits in $\dgSch_{Y/,\on{closed}\,\on{in}\,X}$ is compatible with Zariski localization 
on $X$.

\end{prop}

\begin{proof} \hfill

\medskip

\noindent{\it Step 1.} Assume first that $X$ is affine, given by $X=\Spec(A)$. Let 
\begin{equation} \label{e:XiY} 
i\rightsquigarrow (Y\to Z_i\overset{\phi_i}\to X),
\end{equation}
be a finite diagram in $\dgSch_{Y/,\,\on{closed}\,\on{in}\,X}$.

\medskip

Set $B:=\Gamma(Y,\CO_Y)$.
This is a (not necessarily connective) commutative $k$-algebra. Set also $Z_i=\Spec(C_i)$. 
Consider the corresponding diagram
\begin{equation} \label{e:AiB}
i\rightsquigarrow (A\to C_i\to B)
\end{equation}
in $\on{ComAlg}_{A/\,/B}$. 

\medskip

Set 
$$(\wt{C}\to B):=\underset{i}{lim}\, (C_i\to B),$$
where the limit taken in $\on{ComAlg}_{/B}$. Note that we have a canonical map $A\to \wt{C}$,
and 
$$(A\to \wt{C}\to B)\in \on{ComAlg}_{A/\,/B}$$
maps isomorphically to the limit of \eqref{e:AiB} taken in, $\on{ComAlg}_{A/\,/B}$. 

\medskip

Set 
$$C:=\tau^{\leq 0}(\wt{C})\underset{H^0(\wt{C})}\times \on{Im}\left(H^0(A)\to H^0(\wt{C})\right),$$
where the fiber product is taken in the category of \emph{connective} commutative algebras
(i.e., it is $\tau^{\leq 0}$ of the fiber product taken in the category of all commutative algebras). 

\medskip

We still have canonical maps
$$A\to C\to B,$$
and it is easy to see that for $Z:=\Spec(C)$, the object
$$(X\to Z\to Y)\in \dgSch_{X/,\,\on{closed}\,\on{in}\,Y}$$
is the colimit of \eqref{e:XiY}.

\medskip

\noindent{\it Step 2.} To treat the general case it suffices to show that the formation of colimits
in the affine case commutes with Zariski localization. I.e., that if $X$ is affine, $\oX\subset X$ is a basic 
open, then for $\oY:=f^{-1}(\oX)$, $\oZ_i:=\phi_i^{-1}(\oX)$, $\oZ:=\phi^{-1}(\oX)$, 
the map
$$\underset{i}{colim}\, \oZ_i \to \oZ,$$
is an isomorphism, where the colimit is taken in $\dgSch_{\oY/,\,\on{closed}\,\,\oX}$. 

\medskip

However, the required isomorphism follows from the description of the colimit in Step 1. 

\end{proof}

\sssec{}

As before, let
$$ i\rightsquigarrow (Y\to Z_i\overset{\phi_i}\to X), $$
be a finite diagram in $\dgSch_{Y/,\,\on{closed}\,\on{in}\,X}$.  In this case, note the following property of colimits.

\medskip

Let $g:X\to X'$ be a closed embedding. Set
$$(Y\to Z\to X)=\underset{i}{colim}\, (Y\to Z_i\to X) \text{ and }
(Y\to Z'\to X')=\underset{i}{colim}\, (Y\to Z_i\to X'),$$
where the colimits are taken in $\dgSch_{Y/,\,\on{closed}\,\on{in}\,X}$ and $\dgSch_{Y/,\,\on{closed}\,X'}$,
respectively.

\medskip

Consider the composition
$$Y\to Z\to X\to X',$$
and the corresponding object
$$(Y\to Z\to X')\in \dgSch_{Y/,\,\on{closed}\,X'}.$$
It is endowed with a compatible family of maps in $\dgSch_{Y/,\,\on{closed}\,X'}$:
$$(Y\to Z_i\to X')\to (Y\to Z\to X').$$

Hence, by the universal property of $(Y\to Z'\to X')\in \dgSch_{Y/,\,\on{closed}\,X'}$, we obtain a canonically defined map 
\begin{equation} \label{e:two closures}
Z'\to Z.
\end{equation}

We claim:

\begin{lem} \label{l:colimit and then closed}
The map \eqref{e:two closures} is an isomorphism.
\end{lem}

\begin{proof}

We construct the inverse map as follows. We note that by the universal property
of $(Y\to Z'\to X')\in \dgSch_{Y/,\,\on{closed}\,X'}$, we have a canonical map
$$(Y\to Z'\to X') \to (Y\to X\to X'),$$
and hence a compatible family of maps
$$(Y\to Z_i\to X') \to (Y\to Z'\to X') \to  (Y\to X\to X').$$

The latter gives rise to a compatible family of maps in $\dgSch_{Y/,\,\on{closed}\,X}$
$$(Y\to Z_i\to X)\to (Y\to Z'\to X),$$
and hence, by the universal property of $(Y\to Z\to X)\in \dgSch_{Y/,\,\on{closed}\,X}$,  the desired map 
$$Z\to Z'.$$

\end{proof}

\sssec{The closure of the image}   \label{sss:closure of image}

For $f: X\rightarrow Y$ a morphism in in $\dgSch_{\on{qsep-qc}}$, let
$$\ol{\on{Im}(f)}\in \dgSch_{Y/,\,\on{closed}\,X}$$ denote the initial object of this category.
We will refer to it as \emph{the closure of the image of $f$}.

\sssec{}

We have the following properties of the formation of colimits in $\dgSch_{Y/,\,\on{closed}\,X}$:

\begin{lem}   \label{c:truncation of push-out} Let $i\mapsto (Y\to Z_i\to X)$
be a finite diagram in  $\dgSch_{Y/,\,\on{closed}\,X}$, and let 
$$Y\to Z\to X$$
be its colimit. 

\smallskip

\noindent{\em(a)} 
Suppose that the DG schemes $Z_i$ are $n$-coconnective. Then so is $Z$.

\smallskip

\noindent{\em(b)} 
Suppose that $f:Y\to X$ is affine \emph{(}resp., of cohomological amplitude $k$ for the functor
$f_*:\QCoh(Y)\to \QCoh(X)$\emph{)}.
For an integer $m$, consider the diagram
$$^{\leq m}Y\to {}^{\leq m}\!Z_i\to X,$$
and let 
$$^{\leq m}Y\to Z'\to X$$
be its colimit in $\dgSch_{^{\leq m}Y/,\,\on{closed}\,X}$. Then the natural map
$$^{\leq n}\!Z'\to {}^{\leq n}\!Z$$
is an isomorphism whenever $m\geq n+1$ \emph{(}resp., $m\geq n+1+k$\emph{)}. 
\end{lem}

\begin{proof}

Both assertions follow from the explicit construction of colimits in Step 1 in the proof 
of \propref{p:colimits as closed}.

\end{proof}




\ssec{The case of DG indschemes}  \label{ss:cl into ind}





\sssec{}

For $\CX\in \dgindSch$, $Y\in \dgSch_{\on{qsep-qc}}$ and a morphism $Y\to \CX$,
we consider the category
$$(\dgSch_{\on{qsep-qc}})_{Y/\,/\CX}$$
and the corresponding full subcategory 
$$\dgSch_{Y/,\on{closed}\,\on{in}\,\CX}.$$

\medskip

\begin{prop}  \label{p:factorization coproducts for indscheme}
The category $\dgSch_{Y/,\on{closed}\,\on{in}\,\CX}$ contains finite colimits.
\end{prop}

As in the case of DG schemes, for a given map $f:Y\to \CX$,  we let
$\ol{\on{Im}(f)}$ denote the initial object of the category $\dgSch_{Y/,\on{closed}\,\on{in}\,\CX}$.

\begin{rem}
As \propref{p:factorization coproducts for indscheme} will be used
in the proof of \propref{p:presentation of indschemes}, we will not be able to use the
existence of a presentation as in \eqref{e:gen indscheme as a colimit}. If we could assume such a presentation, the proof would be immediate. 
\end{rem}

\begin{proof}[Proof of \propref{p:factorization coproducts for indscheme}]

Assume first that $Y$, $Z_1$ and $Z_2$ are eventually coconnective, i.e.,
$n$-coconnective for some $n$. Then we can work in the categories $^{\leq n}\!\dgSch$ and
$^{\leq n}\!\dgindSch$. We replace $\CX$ by 
$^{\leq n}\CX$, and representing it as in \eqref{e:indscheme as a colimit},
we obtain that the statement follows from \lemref{l:colimit and then closed}. 

\medskip

Writing $^{cl}\CX$ as in \eqref{e:indscheme as a colimit}, 
let $\alpha\in A$ be an index such that the map $^{cl}Y\to {}^{cl}\CX$ factors via a map
$$^{cl}\!f_{\alpha}:{}^{cl}Y\to X_\alpha\to {}^{cl}\CX.$$

Let $k$ denote the cohomological amplitude of the functor 
$$({}^{cl}\!f_{\alpha})_*:\QCoh({}^{cl}\!Y)\to \QCoh(X_\alpha).$$

\medskip

Let 
\begin{equation} \label{e:diag in closed in ind}
i\mapsto (Y\to Z_i\to \CX)
\end{equation}
be a finite diagram in $\dgSch_{Y/,\,\on{closed}\,\CX}$. 
For an integer $m$, consider the corresponding diagram
$$^{\leq m}Y\to {}^{\leq m}\!Z_i\to X.$$
Let 
$$^{\leq m}Y\to \wt{Z}^m\to \CX$$
denote its colimit in $\dgSch_{^{\leq m}Y/,\,\on{closed}\,\CX}$.  

\medskip

For an integer $n$ set
$$Z^n={}^{\leq n}\!\wt{Z}^m$$
for any $m\geq n+1+k$. Note that this is independent of the choice of $m$ by 
\corref{c:truncation of push-out}(b). For the same reason, for $n_1\leq n_2$
we have
$$Z^{n_1}\simeq {}^{\leq n_1}\!Z^{n_2}.$$

\medskip

The sought-for colimit of \eqref{e:diag in closed in ind} is $Y\to Z\to \CX$, where
$Z\in \dgSch$ is such that
$$^{\leq n}\!Z=Z^n.$$

\end{proof}




\sssec{}

As a corollary of \propref{p:factorization coproducts for indscheme}, we obtain:

\begin{cor} \label{c:filtered for indscheme}
For $\CX\in \dgindSch$, the category of closed embeddings 
$Z\to \CX$, where $Z\in \dgSch_{\on{qsep-qc}}$, is filtered.
\end{cor}

Note that the assertion of \corref{c:filtered for indscheme} coincides with that of
\propref{p:indscheme as colimit of its closed}(a). 

\ssec{A digression on push-outs}  \label{ss:general push-outs}

Let 
\begin{gather}  
\xy
(-15,0)*+{Y}="A";
(0,10)*+{Y_1}="B";
(0,-10)*+{Y_2}="C";
{\ar@{->}^{f_1} "A";"B"};
{\ar@{->}_{f_2} "A";"C"};
\endxy
\end{gather}
be a diagram in $\dgSch$. 

\medskip

We wish to consider the push-out of this diagram in $\dgSch$. Note that
push-outs of (DG) schemes are not among the standard practices in algebraic 
geometry; this operation is in general quite ill-behaved unless we impose some particular conditions
on morphisms under which we are taking push-outs. In what follows we will consider 
three rather special situations where push-outs are manageable.

\sssec{Push-outs in the category of affine schemes}  \label{sss:diagram of schemes}

Let 
$$i\mapsto Y_i,\quad i\in I$$ 
be an $I$-diagram in $\affdgSch$ for some $I\in \inftyCat$.

\medskip

Let $\wt{Y}$ denote its colimit in the category $\affdgSch$. 
I.e., if $Y_i=\Spec(A_i)$, then $\wt{Y}=\Spec(\wt{A})$, where
$$\wt{A}=\underset{i}{lim}\, A_i,$$
where the limits is taken in the category of connective $k$-algebras. 

\sssec{}  \label{sss:push-out aff}

In particular, consider a diagram $Y_1\leftarrow Y\to Y_2$ in $\affdgSch$ and set
$\wt{Y}:=Y_1\underset{Y}\sqcup\, Y_2$, where the push-out is taken in $\affdgSch$. 
I.e., if $Y_i=\Spec(A_i)$ and $Y=\Spec(A)$, then $\wt{Y}=\Spec(\wt{A})$, where
$$\wt{A}:=A_1\underset{A}\times A_2,$$
where the fiber product is taken in the category of \emph{connective} $k$-algebras. 

\medskip

Note that if $Y\to Y_1$ is a closed embedding, then so is the map $Y_2\to \wt{Y}$. 

\sssec{The case of closed embeddings} 

We observe the following:

\begin{lem}  \label{l:closed pushout}
Suppose that in the setting of \secref{sss:push-out aff}, both maps $Y\to Y_i$ are closed embeddings. Then:

\smallskip

\noindent{\em(a)} The Zariski topology on $\wt{Y}$ is induced by that on $Y_1\sqcup Y_2$.

\smallskip

\noindent{\em(b)} For open affine DG subschemes $\oY_i\subset Y_i$ such that $\oY_1\cap Y=\oY_2\cap Y=:\!\oY$, 
and the corresponding open DG subscheme $\wt{\oY}\subset \wt{Y}$, the
map $$\oY_1\underset{\oY}\sqcup \oY_2\to \wt{\oY}$$ is an isomorphism. 

\smallskip

\noindent{\em(c)} The diagram
$$
\CD
Y  @>>> Y_1  \\
@VVV   @VVV   \\
Y_2  @>>>  \wt{Y}
\endCD
$$
is a push-out diagram in $\dgSch$. 

\end{lem}

\sssec{}

From here we obtain:

\begin{cor}
Let $Y_1\leftarrow Y\to Y_2$ be a diagram in $\dgSch$, where both maps $Y_i\to Y$ are closed embeddings.
Then:

\smallskip

\noindent{\em(a)} The push-out $\wt{Y}:=Y_1\underset{Y}\sqcup Y_2$ in $\dgSch$ exists.

\smallskip

\noindent{\em(b)} The Zariski topology on $\wt{Y}$ is induced by that on $Y_1\sqcup Y_2$.

\smallskip

\noindent{\em(c)} For open DG subschemes $\oY_i\subset Y_i$ such that $\oY_1\cap Y=\oY_2\cap Y=:\!\oY$, 
and the corresponding open DG subscheme $\wt{\oY}\subset \wt{Y}$, the
map $$\oY_1\underset{\oY}\sqcup \oY_2\to \wt{\oY}$$ is an isomorphism. 

\end{cor}

\begin{rem}
Note that if one of the maps $f_i$ fails to be a closed
embedding, it is no longer true that the push-out in the category of affine DG 
schemes is a push-out in the category of schemes. A counter-example is
$$\BA^1\times (\BA^1-0)\hookleftarrow \{0\}\times  (\BA^1-0)\hookrightarrow \{0\}\times \BA^1.$$
\end{rem}

\sssec{}

We give the following definition:

\begin{defn}
A map $f: X_1\to X_2$ in $\dgSch$ is said to be a \emph{nil-immersion} if it induces an isomorphism
$$^{cl,red}\!X_1\to {}^{cl,red}\!X_2,$$
where for a DG scheme $X$, we let $^{cl,red}\!X$ denote the underlying classical reduced scheme. If $f$ is in addition a closed embedding, then it is said to be a \emph{closed nil-immersion}.
\end{defn}



\sssec{Push-outs with respect to nil-immersions}

Consider the following situation. Let $i\mapsto Y_i$ and $\wt{Y}$
be as in \secref{sss:diagram of schemes}.

\medskip

Assume that the maps $Y_i\to \wt{Y}$ are nil-immersions. In particular, the transition maps
$$Y_{i_1}\to Y_{i_2}$$ are nil-immersions as well. In this case we have:

\begin{lem}  \label{l:nil colim} Assume that the maps $Y_i\to \wt{Y}$ are nil-immersions.

\smallskip

\noindent{\em(a)} For an open affine DG subscheme $\wt{\oY}\subset \wt{Y}$, and the corresponding open DG subschemes $\oY_i\subset Y_i$,
the map
$$\underset{i}{colim}\, \oY_i\to \wt{\oY}$$ is an isomorphism, where the colimit is taken in $\affdgSch$.

\smallskip

\noindent{\em(b)} The diagram
$$i\mapsto (Y_i\to \wt{Y})$$
is also a colimit diagram in $\dgSch$.

\end{lem}

\sssec{}

From \lemref{l:nil colim} we obtain:

\begin{cor} Let $Y_1\leftarrow Y\to Y_2$ be a diagram in $\dgSch$ where the maps $Y\to Y_i$ are nil-immersions. Then:

\smallskip

\noindent{\em(a)} The push-out $\wt{Y}:=Y_1\underset{Y}\sqcup\, Y_2$ in $\dgSch$ exists, and the maps
$Y_i\to \wt{Y}$ are nil-immersions. 

\smallskip

\noindent{\em(b)} For an open DG subscheme $\wt{\oY}\subset \wt{Y}$, and the corresponding open DG subschemes $\oY_i\subset Y_i$,
$\oY\subset Y$, the map
$$\oY_1\underset{\oY}\sqcup\, \oY_2 \to \wt{\oY}$$ is an isomorphism.

\end{cor}

\sssec{The push-out of a closed nil-immersion}

Finally, we will consider the following situation. Let
$$Y_1\to Y'_1$$ be a closed nil-immersion of affine DG schemes, and let $f:Y_1\to Y_2$ be a map,
where $Y_2\in \affdgSch$.

\medskip

Let $Y'_2=Y'_1\underset{Y_1}\sqcup\, Y_2$, where the colimit is taken in $\affdgSch$. Note that the map
$$Y_2\to Y'_2$$
is a closed nil-immersion.

\medskip

We claim: 

\begin{lem} \label{l:nil pushout}

\smallskip

\noindent{\em(a)} For an open affine DG subscheme $\oY_2\subset Y_2$, $f^{-1}(\oY_2)=:\oY_1\subset Y_1$,
and the corresponding open affine DG subscheme $\oY'_i\subset Y'_i$, the map
$$\oY'_1\underset{\oY_1}\sqcup\, \oY_2\to \oY'_2$$
is an isomorphism, where the push-out is taken in $\affdgSch$. 

\smallskip

\noindent{\em(b)} The diagram

$$
\CD
Y_1  @>>> Y_2 \\
@VVV   @VVV  \\
Y'_1 @>>>  Y'_2
\endCD
$$
is also a push-out diagram in $\dgSch$.

\end{lem}

\sssec{}

As a corollary we obtain:

\begin{cor}  \label{c:nil pushout}
Let $Y_1\to Y'_1$ be a closed nil-immersion, and $f:Y_1\to Y_2$ be a 
quasi-separated quasi-compact map between DG schemes. Then:

\smallskip

\noindent{\em(a)} The push-out $Y'_2:=Y'_1\underset{Y_1}\sqcup\, Y_2$ exists, and
the map $Y_2\to Y'_2$ is a nil-immersion. 

\smallskip

\noindent{\em(b)} For an open affine DG subscheme $\oY_2\subset Y_2$, $f^{-1}(\oY_2)=:\oY_1\subset Y_1$,
and the corresponding open affine DG subscheme $\oY'_i\subset Y'_i$, the map
$$\oY'_1\underset{\oY_1}\sqcup\, \oY_2\to \oY'_2$$
is an isomorphism, where the push-out is taken in $\dgSch$. 

\smallskip

\noindent{\em(c)} If $f$ is an open embedding, then so is the map $Y'_1\to Y'_2$.

\end{cor}

\begin{proof}

We observe that it suffices to prove the corollary when $Y_2$ is affine. Let us write
$Y_1$ as $\underset{i}{colim}\, U_i$, where $U_i$ are affine and open in $Y_1$.
In this case,
$$Y'_1\simeq \underset{i}{colim}\, U'_i,$$
where $U'_i$ are the corresponding open DG subschemes in $Y'_1$.

\medskip

We construct $Y'_1\underset{Y_1}\sqcup\, Y_2$ as 
$$\underset{i}{colim}\, (U'_i\underset{U_i}\sqcup\, Y_2).$$

This implies points (a) and (b) of the corollary via \lemref{l:nil pushout}. Point (c) follows formally from point (b).

\end{proof}

\sssec{}

We will use the following additional properties of push-outs:

\begin{lem}  \label{l:ppties pushouts}  Let $Y_1,Y'_1,Y_2,Y'_2$ be as in \corref{c:nil pushout}.
%
%
Suppose that the map $f:Y_1\to Y_2$ is such that the cohomological
amplitude of the functor $f_*:\QCoh(Y_1)\to \QCoh(Y_2)$ is bounded by $k$. Then the map
$$^{\leq m}Y'_1 \underset{^{\leq m}Y_1}\sqcup\, {}^{\leq m}Y_2 \to
{}^{\leq m}Y'_2$$ defines an isomorphism of the $n$-coconnective truncations whenever 
$m\geq n+k$.
%


\end{lem}

\ssec{DG indschemes and push-outs}

\sssec{}

Let us observe the following property enjoyed by ind-schemes:

\begin{prop} \label{p:compat with pushouts}
Let 
$$
\CD
Y  @>>>  Y_1  \\
@VVV    @VVV  \\
Y_2  @>>> \wt{Y}
\endCD
$$
be a push-out diagram in $\dgSch_{\on{qsep-qc}}$, where $Y,Y_1,Y_2$ are eventually coconnective. 
Then for $\CX\in \dgindSch$, the natural map
$$\Maps(\wt{Y},\CX)\to \Maps(Y_1,\CX)\underset{\Maps(Y,\CX)}\times \Maps(Y_2,\CX)$$
is an isomorphism.
\end{prop}

\begin{proof}

Suppose that $Y,Y_1,Y_2$ are $n$-coconnective. By adjunction, we obtain that $\wt{Y}$ is $n$-coconnective
as well. 

\medskip

The assertion of the proposition now follows from \lemref{l:maps out of qc} and 
the fact that fiber products commute with filtered colimits.

\end{proof}

\begin{rem}
In the above proposition we had to make the eventual coconnectivity
assumption, because
it will be used for the proof of \propref{p:presentation of indschemes}. However, assuming
this proposition, and hence, \lemref{l:gen maps out of qc}, we will be able to prove
the same assertion for any $Y,Y_1,Y\in \dgSch_{\on{qsep-qc}}$. The next
corollary, which will be also used in the proof of \propref{p:presentation of indschemes},
gives a partial result along these lines.
\end{rem}

\begin{cor}  \label{c:compat with pushouts}
Let 
$$
\CD
Y_1  @>>>  Y_2  \\
@VVV    @VVV  \\
Y'_1  @>>> Y'_2
\endCD
$$
be a push-out diagram as in \lemref{c:nil pushout}, where $Y_1,Y_2\in \dgSch_{\on{qsep-qc}}$.
Then the natural map 
$$\Maps(Y'_2,\CX)\to \Maps(Y'_1,\CX)\underset{\Maps(Y_1,\CX)}\times \Maps(Y_2,\CX)$$
is an isomorphism.
\end{cor}

\begin{proof} 

Consider the following two inverse families of objects of $\dgSch_{\on{qsep-qc}}$:
$$n\mapsto {}^{\leq n}Y'_2  \text{ and } n\mapsto {} ^{\leq n}Y'_1\underset{^{\leq n}Y_1}
\sqcup\, {}^{\leq n}Y_2.$$
There is a natural map $\leftarrow$. By \lemref{l:ppties pushouts}, this
map induces an isomorphism of $m$-coconnective truncations whenever $n\gg m$.

\medskip

Therefore, for any $\CX\in \dgindSch$  (and, indeed, any $\CX\in {}^{\on{conv}}\!\inftydgprestack$),
the induced map
$$\underset{n}{lim}\, \Maps({}^{\leq n}Y'_2,\CX)\to 
\underset{n}{lim}\, \Maps\left({}^{\leq n}Y'_1\underset{^{\leq n}Y_1}\sqcup\, {}^{\leq n}Y_2,\CX\right)$$
is an isomorphism. 

\medskip

Consider the composite map
\begin{multline*}
\Maps(Y'_2,\CX) \to \Maps(Y'_1,\CX)\underset{\Maps(Y_1,\CX)}\times \Maps(Y_2,\CX) \overset{\sim}\to \\
\overset{\sim}\to
\underset{n}{lim}\, \Maps({}^{\leq n}Y'_1,\CX)\underset{\underset{n}{lim}\, 
\Maps({}^{\leq n}Y_1,\CX)}\times
\underset{n}{lim}\, \Maps({}^{\leq n}Y'_2,\CX) \overset{\sim}\to \\
\overset{\sim}\to \underset{n}{lim}\, \left(\Maps({}^{\leq n}Y'_1,\CX)
\underset{\Maps({}^{\leq n}Y_1,\CX)}\times \Maps({}^{\leq n}Y_2,\CX)\right).
\end{multline*}

It equals the map
\begin{multline*}
\Maps(Y'_2,\CX) \to 
\underset{n}{lim}\, \Maps({}^{\leq n}Y'_2,\CX)  \overset{\sim}\to \\
\overset{\sim}\to 
\underset{n}{lim}\, \Maps\left({}^{\leq n}Y'_1\underset{^{\leq n}Y_1}\sqcup\, 
{}^{\leq n}Y_2,\CX\right)   \overset{\sim}\to \\
\overset{\sim}\to
\underset{n}{lim}\, \left(\Maps({}^{\leq n}Y'_1,\CX)
\underset{\Maps({}^{\leq n}Y_1,\CX)}\times \Maps({}^{\leq n}Y_2,\CX)\right),
\end{multline*}
where the last arrow is an isomorphism by \propref{p:compat with pushouts} above.
This shows that 
$$\Maps(Y'_2,\CX)\to \Maps(Y'_1,\CX)\underset{\Maps(Y_1,\CX)}\times \Maps(Y_2,\CX)$$
is an isomorphism as well.

\end{proof}

\ssec{Presentation of indschemes}   \label{ss:proof of indscheme as colimit of its closed}

\sssec{}

We shall now prove point (b) of \propref{p:indscheme as colimit of its closed}. In fact, we will prove a slightly stronger
(but, in fact, equivalent) statement; namely, we will prove \corref{c:cofinality of closed}.

\begin{proof}

We have to show that for $Y\in \dgSch_{\on{qsep-qc}}$ and a map $f:Y\to \CX$, the category of its
factorizations
$$Y\to Z\to \CX,$$
where $Z\in \dgSch_{\on{qsep-qc}}$, and $Z\to \CX$ is a closed embedding, is contractible. 

\medskip

By \propref{p:colimits as closed}, the category in question admits coproducts. Hence, to prove that it is contractible, 
it remains to show that it is non-empty. 

\medskip

Consider the map $^{cl}\!f:{}^{cl}Y\to {}^{cl}\CX$. Since $^{cl}\CX$ is a classical indscheme, 
there exists a factorization
$$^{cl}Y\overset{h_{cl}}\longrightarrow Z_{cl}\overset{g_{cl}}\longrightarrow {}^{cl}\CX,$$
where $Z_{cl}\in \Sch_{\on{qsep-qc}}$ and $g_{cl}$ is a closed embedding. 

\medskip

Let $k$ be the cohomological amplitude of the functor $(h_{cl})_*:\QCoh({}^{cl}Y)\to \QCoh(Z_{cl})$, and
let $n$ be an integer $> k$. 

\medskip

Consider the truncation $^{\leq n}Y$ and its map $^{\leq n}\!f$
to $^{\leq n}\CX$. Since $^{\leq n}\CX$ is a $\nDG$ indscheme, the map $^{\leq n}\!f$ can be factored
as
$$^{\leq n}Y\overset{h_n}\longrightarrow Z_n\overset{g_n}\longrightarrow {}^{\leq n}\CX,$$
where $Z_{n}\in {}^{\leq n}\dgSch_{\on{qsep-qc}}$ and $g_{n}$ is a closed embedding. Moreover, 
without loss of generality, we can assume that we have a commutative square
$$
\CD
^{cl}Y   @>{h_{cl}}>>   Z_{cl} \\
@V{\sim}VV   @VVV    \\
^{cl}({}^{\leq n}Y)  @>{^{cl}\!h_n}>>  ^{cl}\!Z_n,
\endCD
$$
where the right vertical map is automatically a closed embedding. In particular, we obtain
that the cohomological amplitude of the functor $({}^{cl}\!h_n)_*$ also equals $k$. Therefore,
the same is true for the functor
$$(h_n)_*:\QCoh({}^{\leq n}Y)\to \QCoh(Z_n).$$

\medskip

Thus, \lemref{l:ppties pushouts} applies to $h_n$. Let
$$Z:=Y\underset{^{\leq n}Y}\sqcup\, Z_n\in \dgSch.$$

\medskip

By \corref{c:compat with pushouts},
we have a canonical map $g:Z\to \CX$, which is a closed embedding since at the classical
level this map is the same as $g_n$.  Thus
$$Y\to Z\to \CX$$
is the required factorization of $f$. 

\end{proof}

\sssec{}  \label{sss:proof of presentation lft strong}

Let us now prove \propref{p:canonical presentation of laft indsch}. Our proof will rely 
on the notion of square-zero extension, which will be reviewed in  \secref{sss:gen sq zero}. 

\medskip

We begin with the following observation:

\begin{lem}  \label{l:cofinal in filtered}
Let $\bC$ be an $\infty$-category and $i:\bC_1\to \bC$ a fully faithful functor. Assume that $\bC$
is filtered. Then $i$ is cofinal if and only if every object of $\bC$ admits a map to
an object in $\bC_1$. In this case $\bC_1$ is also filtered. 
\end{lem}

We take $\bC:=(\dgSch_{\on{qsep-qc}})_{\on{closed}\,\on{in}\,\CX}$ and
$\bC_1=(\dgSch_{\on{aft}})_{\on{closed}\,\on{in}\,\CX}$. Having
proved \corref{c:cofinality of closed}, it remains to show that every closed embedding
$$f:Y\to X$$
admits a factorization 
$$Y\to Z\overset{g}\longrightarrow \CX,$$
where $Z\in \dgSch_{\on{aft}}$ and $g$ is also a closed embedding. 

\medskip

\noindent{\it Step 1.} Consider a factorization of $^{cl}\!f$
$$^{cl}Y \overset{h_{cl}}\longrightarrow Z_{cl} \overset{g_{cl}}\longrightarrow  {}^{cl}\CX,$$
where $g_{cl}$ is a closed embedding. We claim that the ``locally almost of finite type"
assumption on $\CX$ implies that the classical scheme $Z_{cl}$ is automatically of finite type.

\medskip

This follows from the next lemma:

\begin{lem}  \label{l:classical lft}
If $\CX_{cl}$ is a classical indscheme locally of finite type, and $X_{cl}\to \CX_{cl}$ a closed embedding,
where $X_{cl}\in \Sch$, then $X_{cl}\in \Sch_{\on{ft}}$.
\end{lem}

\begin{proof}[Proof of \lemref{l:classical lft}]

Note that a classical scheme $X_{cl}$ is of finite type if and only if for any classical $k$-algebra
$A$ and a filtered family $i\mapsto A_i$ of \emph{subalegbras} such that $A = \underset{i}{\cup} A_{i}$, the map
$$\underset{i}{colim}\, \on{Maps}(\Spec(A_i),X_{cl})\to \on{Maps}(\Spec(A),X_{cl})$$
is an isomorphism. 

\medskip

Note that since $A_i\to A$ are injective, the diagram
$$
\CD
\underset{i}{colim}\, \on{Maps}(\Spec(A_i),X_{cl})   @>>>  \on{Maps}(\Spec(A),X_{cl})  \\
@VVV   @VVV   \\
\underset{i}{colim}\, \on{Maps}(\Spec(A_i),\CX_{cl})   @>>>  \on{Maps}(\Spec(A),\CX_{cl})
\endCD
$$
is Cartesian. However, the bottom horizontal arrow is an isomorphism since
$\CX\in {}^{cl}\!\inftydgprestack_{\on{lft}}$. 

\end{proof}

\medskip

\noindent{\it Step 2.} We shall construct the required factorization of $f$ by induction on $n\geq 0$. Namely, we shall
construct a sequence of factorizations of $^{\leq n}\!f:{}^{\leq n}Y\to {}^{\leq n}\CX$ as 
$$^{\leq n}Y\overset{h_n}\longrightarrow Z_n\overset{g_n}\to {}^{\leq n}\CX,$$
with $Z_n\in {}^{\leq n}\!\dgSch_{\on{ft}}$, $g_n$ a closed embedding, and 
such that for $n\geq n'$, we have a commutative diagram
$$
\CD
^{\leq n'}\!Z_n   @>{^{\leq n'}\!g_n}>>  {}^{\leq n'}\CX  \\
@A{\sim}AA   @AA{\on{id}}A  \\ 
^{\leq n'}\!Z_{n'}  @>{^{\leq n'}\!g_{n'}}>> {}^{\leq n'}\CX.
\endCD
$$

Setting 
$$Z:=\underset{n}{colim}\,\,  Z_n$$
(where the colimit is taken in $\dgSch$) we will then obtain the desired factorization of $f$.

\medskip

\noindent{\it Step 3.} Suppose $(Z_{n-1},g_{n-1})$ have been constructed. Note that the maps
$$h_{n-1}:{}^{\leq n-1}Y\to Z_{n-1} \text{ and } ^{\leq n-1}Y\to {}^{\leq n}Y$$
satisfy the conditions of \corref{c:nil pushout}. Set
$$Z'_n:=Z_{n-1}
\underset{^{\leq n-1}Y}\sqcup\,{}^{\leq n}Y.$$

We have $^{\leq n-1}\!Z'_n\simeq {}^{\leq n-1}\!Z_{n-1}$, and by \propref{p:compat with pushouts}
we obtain a natural map $g'_n:Z'_{n}\to {}^{\leq n}\CX$. 

\medskip

To find the sought-for pair $(Z_n,g_n)$,
it suffices to find a factorization of $g'_n$ as 
$$Z'_n\to Z_n\overset{g_n}\longrightarrow {}^{\leq n}\CX,$$
so that $Z_n\in {}^{\leq n}\dgSch_{\on{ft}}$, and $^{\leq n-1}\!Z'_n\to {}^{\leq n-1}\!Z_n$ is an isomorphism.

\medskip

\noindent{\it Step 4.}  Note that the closed embedding 
$$^{\leq n-1}Y\to {}^{\leq n}Y$$
has a natural structure of a \emph{square-zero extension},
see \corref{c:can sq zero}, by an ideal 
$$\CI\in \QCoh\left({}^{\leq n-1}Y\right)^\heartsuit[n].$$

\medskip

Hence, the closed embedding $Z_{n-1}\to  Z'_n$ also has a structure of a square-zero extension by
$$\CJ:=(h_{n-1})_*(\CI)\in 
\QCoh(Z_{n-1})^\heartsuit[n].$$

\medskip

\noindent{\it Step 5.}  Write $\CJ$ as a filtered colimit $\underset{\alpha}{colim}\, \CJ_\alpha$, 
where 
$$\CJ_\alpha\in \Coh(Z_{n-1})^\heartsuit[n].$$
The category $\Coh(Z_{n-1})$ is well-defined
since $Z_{n-1}$ is almost of finite type. 

\medskip

By \secref{sss:gen sq zero}, we obtain a family 
$\alpha\mapsto Z_{n,\alpha}$ of objects of $^{\leq n}\!\dgSch$, for all of which $^{\leq n-1}\!Z_{n,\alpha}\simeq 
{}^{\leq n-1}\!Z_{n-1}$;
moreover, we have isomorphisms
$$Z'_n\simeq \underset{\alpha}{lim}\, Z_{n,\alpha}$$
as objects of $^{\leq n}\!\dgSch$. 

\medskip

Now, since $\CX$ is locally almost of finite type as an object of $^{\leq n}\!\inftydgprestack$, the map
$$\underset{\alpha}{colim}\, \on{Maps}(Z_{n,\alpha},{}^{\leq n}\CX)\to \on{Maps}(Z'_n,{}^{\leq n}\CX)$$
is an isomorphism. In particular, the map $g'_n$ factors through some $g_{n,\alpha}:Z_{n,\alpha}\to {}^{\leq n}\CX$.

\medskip

Now, the DG schemes $Z_{n,\alpha}$ all belong to $^{\leq n}\!\dgSch_{\on{ft}}$, by construction. This
gives the required factorization.

\qed

\section{Deformation theory: recollections}     \label{s:deformation theory}

This section is preparation for \secref{s:char via deform}. Our goal is the following:
given $\CX\in \inftydgprestack$ such that $^{cl}\CX$ is a classical indscheme, we would
like to give necessary and sufficient conditions for $\CX$ to be a DG indscheme. In
this section we shall discuss what will be called Conditions (A), (B) and (C) that are
satisfied by every DG indscheme. In \secref{s:char via deform} we will show that these
conditions are also sufficient. 

\medskip

Conditions (A), (B) and (C) say that $\CX$ has a reasonable deformation theory. 
We will encode this by the property of sending certain push-outs 
(in $\affdgSch$) to fiber products (in $\inftygroup$). 

\ssec{Split square-zero extensions and Condition (A)}

\sssec{Split square-zero extensions} \hfill

\medskip

\noindent For $Z\in {}^{\leq n}\!\dgSch_{\on{qsep-qc}}$. 
We define the category $^{\leq n}\!\on{SplitSqZExt}(Z)$ of \emph{split square-zero extensions of $Z$}
to be the opposite of $\QCoh(Z)^{\geq -n,\leq 0}$.

\medskip

There is a natural forgetful functor
$$^{\leq n}\!\on{SplitSqZExt}(Z)\to ({}^{\leq n}\!\dgSch_{\on{qsep-qc}})_{Z/},\quad \CF\mapsto Z_\CF.$$

\medskip

Explicitly, locally in the 
Zariski topology if $Z=S=\Spec(A)$, and $\CM:=\Gamma(S,\CF)$, 
$$S_\CF:=\Spec(A\oplus \CM),$$
where the multiplication on $\CM$ is zero. 

\sssec{}

The category 
$$^{\leq n}\!\on{SplitSqZExt}(Z)=(\QCoh(Z)^{\geq -n,\leq 0})^{\on{op}}$$
has push-outs: for $\CF_1,\CF_2\to \CF\in \QCoh(Z)^{\geq -n,\leq 0}$ the sought-for
push-out is given by $$\CF':=\CF_1\underset{\CF}\times \CF_2,$$
where the fiber product is taken in $\QCoh(Z)^{\geq -n,\leq 0}$, i.e.,
$$\CF'\simeq \tau^{\leq 0}\!\left(\CF_1\underset{\CF}\times \CF_2\right).$$

\medskip

By \corref{c:nil pushout}, the forgetful functor
$$^{\leq n}\!\on{SplitSqZExt}(Z)\to ({}^{\leq n}\!\dgSch_{\on{qsep-qc}})_{Z/}\to {}^{\leq n}\!\dgSch_{\on{qsep-qc}}$$
commutes with push-outs. I.e., for $\CF_1,\CF_2,\CF,\CF'$ as above, the map 
$$Z_{\CF_1}\underset{Z_{\CF}}\sqcup\, Z_{\CF_2}\to Z_{\CF'}$$
is an isomorphism, where the latter push-out is taken in the category 
$^{\leq n}\!\dgSch_{\on{qsep-qc}}$. 
Moreover, if $Z$ is affine, the above push-out agrees 
with the push-out in the category $^{\leq n}\!\affdgSch$.

\sssec{}  \label{sss:condition A}

Let $\CX$ be an object of $^{\leq n}\!\inftydgprestack$. For $S\in {}^{\leq n}\!\affdgSch$ and a map 
$x:S\to \CX$, consider 
the category $^{\leq n}\!\on{SplitSqZExt}(S,x)$ consisting of triples
$$\{\CF\in \QCoh(S)^{\geq -n,\leq 0},\,\,x':S_\CF\to \CX,\,\,x'|_S\simeq x\}.$$
I.e.,
$$^{\leq n}\!\on{SplitSqZExt}(S,x):=
{}^{\leq n}\!\on{SplitSqZExt}(S)\underset{({}^{\leq n}\!\affdgSch)_{S/}}\times{}({}^{\leq n}\!\affdgSch)_{S/\,/\CX}.$$

\medskip

\begin{defn}
We shall say that $\CX$ satisfies indscheme-like Condition (A) if for any $S$ and $x$ as above, 
the category $^{\leq n}\!\on{SplitSqZExt}(S,x)$ is filtered.
\end{defn}

\medskip

We can reformulate the above condition in more familiar terms.  Another familiar reformulation is described in
\secref{sss:cotangent space as pro} below. 
 
\sssec{}

Consider the functor
$$^{\geq -n}(T^*_x\CX):\QCoh(S)^{\geq -n,\leq 0}\to \inftygroup$$
defined by 
\begin{equation} \label{e:cotangent to presheaf}
^{\geq -n}(T^*_x\CX)(\CF):=\{x':S_\CF\to \CX,\,\, x'|_S\simeq x\}.
\end{equation}
I.e.,
$$\CF\mapsto 
\{S_\CF\} \underset{^{\leq n}\!\on{SplitSqZExt}(S)}\times {}^{\leq n}\!\on{SplitSqZExt}(S,x)=
\{S_\CF\}\underset{({}^{\leq n}\!\affdgSch)_{S/}}\times {}({}^{\leq n}\!\affdgSch)_{S/\,/\CX}.$$

\medskip

The following results from \cite[Prop. 5.3.2.9]{Lu0}:

\begin{lem} \label{l:cond A reform}
The prestack $\CX$ satisfies Condition (A) if and only if the functor $^{\geq -n}(T^*_x\CX)$ 
preserves fiber products. 
\end{lem}

\sssec{The pro-cotangent space}  \label{sss:cotangent space as pro}

Recall (\cite[Cor. 5.3.5.4]{Lu0}) that for an arbitrary $\infty$-category $\bC$ that has
fiber products, and a functor $F:\bC\to \inftygroup$, the condition that $F$ preserve
fiber products is equivalent to the condition that $F$ be
pro-representable.

\medskip

Thus, we obtain:

\begin{cor} \label{c:A is pro}
A prestack $\CX$ satisfies Condition (A) if and only if for every 
$$(S,x:S\to \CX)\in {}({}^{\leq n}\!\affdgSch)_{/\CX},$$ the functor
$$^{\geq -n}(T^*_x\CX):\QCoh(S)^{\geq -n,\leq 0}\to \inftygroup$$ is
pro-representable\footnote{Since $\CX$ is an accessible functor, so is $^{\geq -n}(T^*_x\CX)$.}.
\end{cor}

\medskip

Henceforth, whenever $\CX$ satisfies Condition (A), we shall denote by $^{\geq -n}(T^*_x\CX)$ 
the corresponding object of $\on{Pro}(\QCoh(S)^{\geq -n,\leq 0})$. We shall refer to 
$^{\geq -n}(T^*_x\CX)$ as ``the pro-cotangent space to $\CX$ at $x:S\to \CX$."

\medskip

Thus, an alternative terminology for Condition (A) is that the prestack $\CX$ \emph{admits connective pro-cotangent spaces}. 
\footnote{Note that $k$-Artin stacks for $k>0$ viewed as objects of $\on{PreStk}$ typically do not satisfy the above condition,
as their (pro)-cotangent spaces belong to $\QCoh(S)^{\geq -n,\leq k}$ but not
to $\on{Pro}(\QCoh(S)^{\geq -n,\leq 0})$; i.e., they do not satisfy the connectivity condition.}

\sssec{}

Since fiber products in $\QCoh(S)^{\geq -n,\leq 0}$ correspond to push-outs in $^{\leq n}\!\on{SplitSqZExt}(S)$, 
from \lemref{l:cond A reform} we obtain that Condition (A) is equivalent to requiring that the functor 
$$^{\leq n}\!\on{SplitSqZExt}(S)\to \inftygroup$$
given by 
\begin{equation} \label{e:extension functor}
S_\CF\mapsto  \{x':S_\CF\to \CX,\,\, x'|_S\simeq x\}=
\{S_\CF\}\underset{^{\leq n}\!\on{SplitSqZExt}(S)}\times {}^{\leq n}\!\on{SplitSqZExt}(S,x)
\end{equation}
take push-outs to fiber products.

\medskip

Since the forgetful functor
$$^{\leq n}\!\on{SplitSqZExt}(S)\to {}^{\leq n}\!\dgSch_{\on{qsep-qc}}$$ preserves
push-outs, from \propref{p:compat with pushouts}, we obtain:

\begin{cor}  \label{c:A for indsch}
Any $\CX\in {}^{\leq n}\!\dgindSch$ satisfies Condition (A).
\end{cor}

\sssec{}   \label{sss:extn to non-affine}

Going back to a general prestack $\CX$, assume 
that $\CX$ satisfies Zariski descent. This allows us to extend 
$\CX$ to a functor
$$({}^{\leq n}\!\dgSch_{\on{qsep-qc}})^{\on{op}}\to \inftygroup$$
by 
$$Z\mapsto \underset{S\in Zar(Z)}{lim}\, \Maps(S,\CX),$$
where $Zar(Z)$ is the category of affine schemes endowed with an open embedding into $X$.

\medskip

The following is straightforward:

\begin{lem} \label{l:A for non-affine}
If $\CX$ satisfies Condition (A), $Z\in {}^{\leq n}\!\dgSch_{\on{qsep-qc}}$ and $x:Z\to \CX$ is a map, then the functor
$$^{\geq -n}(T^*_{x}\CX):\QCoh(Z)^{\geq -n,\leq 0}\to \inftygroup,\quad
\CF\mapsto \{Z_\CF\}\underset{({}^{\leq n}\!\dgSch_{\on{qsep-qc}})_{Z/}}\times ({}^{\leq n}\!\dgSch_{\on{qsep-qc}})_{S/\,/\CX}$$
preserves fiber products.
\end{lem}

In particular, we obtain that $^{\geq -n}(T^*_{x}\CX)$
is given by an object of $\on{Pro}(\QCoh(Z)^{\geq -n,\leq 0})$.

\sssec{The relative situation}

The functor $^{\geq -n}(T^*_x\CX)$ can be defined
in a relative situation, i.e., when we are dealing with a map of prestacks $\phi:\CX\to \CY$.
Namely, for $x:S\to \CX$ as above, we set $T^*_x\CX/\CY$ to be the functor
$$\QCoh(S)^{\geq -n,\leq 0}\to \inftygroup$$
defined by 
$$\CF\mapsto \{S_\CF\}\underset{^{\leq n}\!\on{SplitSqZExt}(S,\phi\circ x)}\times 
{}^{\leq n}\!\on{SplitSqZExt}(S,x).$$
where $S_\CF$ defines the point of $^{\leq n}\!\on{SplitSqZExt}(S,\phi\circ x)$ equal to the composite
$$S_\CF\overset{\pi}\to S\overset{\phi\circ x}\longrightarrow \CY,$$
and where $\pi:S_\CF\to S$ is the canonical projection.

\medskip

Note that if both $\CX$ and $\CY$ admit connective pro-cotangent spaces, $T^*_x\CX/\CY$, as an object
of $\on{Pro}(\QCoh(S)^{\geq -n,\leq 0})$, is given by
$$\tau^{\geq -n}\left(\on{Cone}(T^*_{\phi\circ x}\CY\to T^*_{x}\CX)\right).$$

\ssec{A digression: pro-objects in $\QCoh$}  \label{ss:pro qc}

\sssec{}

Let $\bC$ be an $\infty$-category. We consider the category $\on{Pro}(\bC)$, which is, by definition,
the full subcategory of $\on{Funct}(\bC,\inftygroup)$ that consists of accessible functors
$$F:\bC\to \inftygroup$$
that can be written as \emph{filtered} colimits of co-representable functors.

\medskip

Let $\Phi:\bC_1\to \bC_2$ be a functor between $\infty$-categories. Then the functor
$$\on{LKE}_\Phi:\on{Funct}(\bC_1,\inftygroup)\to \on{Funct}(\bC_2,\inftygroup)$$
sends $\on{Pro}(\bC_1)$ to $\on{Pro}(\bC_2)$; we shall denote by 
$$\on{Pro}(\Phi):\on{Pro}(\bC_1)\to \on{Pro}(\bC_2)$$
the resulting functor. 

\medskip

Note that if $\Phi$ admits a right adjoint, denoted $\Psi$, then 
$\on{Pro}(\Phi)$ can be computed as
\begin{equation} \label{e:pro right adj}
(\on{Pro}(\Phi)(F))(\bc_2)=F(\Psi(\bc_2)),\quad F\in \on{Pro}(\bC_1),\bc_2\in \bC_2.
\end{equation}

\sssec{}

Let $\bC$ be a stable $\infty$-category. In this case, the category $\on{Pro}(\bC)$ is also stable.
\footnote{Note, however, that even if $\bC$ is presentable, the category $\on{Pro}(\bC)$ is not,
so caution is required when applying such results as the adjoint functor theorem.} 

\medskip

If $\bC_1$ and $\bC_2$ is a pair of stable categories and $\Phi:\bC_1\to \bC_2$ is an exact functor, 
then $\on{Pro}(\Phi)$ is also exact.

\sssec{}  \label{sss:lift to spectra}

Let $\bC$ be a stable $\infty$-category and $F$ an object of $\on{Pro}(\bC)$. Then $F$ gives rise to an exact functor
$$F^{\on{Sp}}:\bC\to \on{Spectra},$$
such that
$$F\simeq \Omega^\infty\circ F^{\on{Sp}}.$$

\medskip

If $\bC$ arises from a DG category (or, equivalently, is tensored over $\Vect$), then the functor $F^{\on{Sp}}$
can be further upgraded to a functor
$$F^{\on{Vect}}:\bC\to \Vect.$$

\sssec{}

Suppose that $\bC$ is endowed with a t-structure. In this case, $\on{Pro}(\bC)$ also inherits
a t-structure: its connective objects are those $F\in \on{Pro}(\bC)$ such that $F(x)=0$
for $x\in \bC^{> 0}$. 

\medskip

Restriction of functors defines a map
$$\on{Pro}(\bC)^{\leq 0}\to \on{Pro}(\bC^{\leq 0}),$$
which is easily seen to be an equivalence. Similarly, for any $n\geq 0$, the natural functor
$$\on{Pro}(\bC)^{\geq -n,\leq 0}\to \on{Pro}(\bC^{\geq -n,\leq 0})$$
is an equivalence. 

\sssec{}

Now consider the following situation specific to $\QCoh$. Let $Z$ be a DG scheme. We have the 
following two categories
$$\on{Pro}(\QCoh(Z)) \text{ and } 
\underset{S\in Zar(Z)}{lim}\, \on{Pro}(\QCoh(S)).$$

\medskip

Left Kan extension along 
$$\CF\mapsto \CF|_S:\QCoh(Z)\to \QCoh(S)$$
defines a functor
\begin{equation} \label{e:pro qc left}
\on{Pro}(\QCoh(Z)) \to \underset{S\in Zar(Z)}{lim}\, \on{Pro}(\QCoh(S)).
\end{equation}

This functor admits a right adjoint, which is tautologically described as follows. To
$$\{S\mapsto (F_S\in \on{Pro}(\QCoh(S)))\}\in 
\underset{S\in Zar(Z)}{lim}\, \on{Pro}(\QCoh(S))$$
it assigns $F\in \on{Pro}(\QCoh(Z))$ given by
$$F(\CF):=\underset{S\in Zar(Z)}{lim}\, F_S(\CF|_{S}).$$

\medskip

We claim:

\begin{lem} \label{l:gluing pro}
Assume that $Z$ is quasi-separated and quasi-compact. Then  the above two functors
\begin{equation} \label{e:pro qc}
\on{Pro}(\QCoh(Z)) \rightleftarrows
\underset{S\in Zar(Z)}{lim}\, \on{Pro}(\QCoh(S))
\end{equation}
are mutually inverse.
\end{lem}

\begin{proof}

A standard argument shows that instead of $Zar(S)$ we can consider a \emph{finite}
limit corresponding to a Zariski hypercovering. 

\medskip

Note that by \eqref{e:pro right adj}, the left Kan extension $\on{Pro}(\QCoh(Z))\to \on{Pro}(\QCoh(S))$
can be also expressed as the functor
$$F\mapsto F\circ (j_S)_*,$$
where $j_S$ denotes the open embedding $S\hookrightarrow Z$. 

\medskip

Then the fact that the two adjunction maps are isomorphisms follows from the fact that 
$$\on{Id}_{\QCoh(S)} \to \underset{S}{lim}\,\, (j_S)_*\circ j_S^*$$
is an isomorphism and the functors $F$ and $F_S$ involved commute with finite limits.

\end{proof}

Note that the lemma (with the same proof) also applies when we replace the category 
$\on{Pro}(\QCoh(Z))$ by $\on{Pro}(\QCoh(Z)^{\geq -n,\leq 0})$ for any $n\geq 0$.

\ssec{Functoriality of split square-zero extensions and Condition (B)}  

\sssec{}      \label{sss:condition B}

Let $\phi:Z_1\to Z_2$ be an map between objects of $^{\leq n}\!\dgSch_{\on{qsep-qc}}$.
Direct image $\phi_*$ composed with the truncation $\tau^{\leq 0}$ defines
a functor
$$^{\leq 0}\phi_*:\QCoh(Z_1)^{\geq -n,\leq 0}\to \QCoh(Z_2)^{\geq -n,\leq 0},$$
i.e., a functor
$$^{\leq n}\!\on{SplitSqZExt}(Z_1)\to {}^{\leq n}\!\on{SplitSqZExt}(Z_2).$$

It follows from \corref{c:nil pushout} that the following diagram is commutative
\begin{equation} \label{e:pushout SqZ}
\CD
^{\leq n}\!\on{SplitSqZExt}(Z_1)  @>>>   {}^{\leq n}\!\on{SplitSqZExt}(Z_2) \\
@VVV   @VVV    \\
(\dgSch_{\on{qsep-qc}})_{Z_1/}  @>>>  (\dgSch_{\on{qsep-qc}})_{Z_2/},
\endCD
\end{equation}
where the bottom horizontal arrow is the push-out functor
$$Z'_1\mapsto Z'_1\underset{Z_1}\sqcup\, Z_2.$$

\sssec{}

Assume now that $Z_1=S_1$ and $Z_2=S_2$ are affine. Let $\CX$ be an object of $^{\leq n}\!\inftydgprestack$,
and $x_2$ an $S_2$-point of $\CX$. Set $x_1:=x_2\circ \phi:S_1\to \CX$. Composition defines a map
\begin{equation} \label{e:cond B}
^{\leq n}\!\on{SplitSqZExt}(S_1)\underset{^{\leq n}\!\on{SplitSqZExt}(S_2)}\times {}^{\leq n}\!\on{SplitSqZExt}(S_2,x_2)\to
{}^{\leq n}\!\on{SplitSqZExt}(S_1,x_1).
\end{equation}

\medskip

\begin{defn}
We shall say that $\CX\in {}^{\leq n}\!\inftydgprestack$
satisfies indscheme-like Condition (B) if the above functor is an equivalence for any $(S_1,S_2,\phi)$.
\end{defn}

\sssec{}

Using \eqref{e:pushout SqZ}, we can reformulate Condition (B) as saying that the 
presheaf $\CX$ should take 
push-outs in $^{\leq n}\!\dgSch_{\on{qsep-qc}}$ of the form
$(S_1)_{\CF_1}\underset{S_1}\sqcup\, S_2$ to fiber products, where $S_1,S_2\in \affdgSch$.

\medskip

By \propref{p:compat with pushouts}, we obtain:

\begin{cor}
Any $\CX\in {}^{\leq n}\!\dgindSch$ satisfies Condition (B).
\end{cor}

\sssec{}

Let us assume that $\CX$ satisfies Condition (A). In this case, by \eqref{e:pro right adj}, the map
\eqref{e:cond B} can be interpreted as a map in $\on{Pro}(\QCoh(S_1)^{\geq -n,\leq 0})$:
\begin{equation} \label{e:cotangent cmpx for presheaf}
^{\geq -n}(T^*_{x_1}\CX)\to \on{Pro}({}^{\geq -n}\!\phi^*)\left({}^{\geq -n}(T^*_{x_2}\CX)\right).
\end{equation}

\medskip

We obtain:

\begin{lem}
As object $\CX\in {}^{\leq n-1}\!\on{PreStk}$, satisfying condition (A), satisfies Condition (B)
if and only if the map
\eqref{e:cotangent cmpx for presheaf} be an isomorphism.
\end{lem}

\sssec{}

We shall use the following terminology:

\begin{defn}
We shall say that $\CX\in {}^{\leq n}\!\inftydgprestack$ 
\emph{admits a connective pro-cotangent complex} if it satisfies both Conditions (A) and (B).
\end{defn}

In other words, $\CX$ admits a connective pro-cotangent complex if it admits connective pro-cotangent spaces, 
whose formation is compatible with pullbacks under morphisms of affine
DG schemes.

\sssec{}

Let us now assume that $\CX$ satisfies Zariski descent, as well as Conditions (A) and (B).

\medskip

Thus, for $Z\in {}^{\leq n}\!\dgSch_{\on{qsep-qc}}$ and $x:Z\to \CX$, we have a well-defined object
$$^{\geq -n}(T^*_{x}\CX)\in \on{Pro}(\QCoh(Z)^{\geq -n,\leq 0}).$$ 

We wish to compare the restriction of $^{\geq -n}(T^*_{x}\CX)$
to a given affine Zariski open $S\subset Z$ with 
$$^{\geq -n}(T^*_{x|_S}\CX)\in \on{Pro}(\QCoh(S)^{\geq -n,\leq 0}).$$ As in \eqref{e:cotangent cmpx for presheaf},
we have a natural map
\begin{equation} \label{e:cotangent cmpx for presheaf nonaffine}
^{\geq -n}(T^*_{x|_S}\CX)\to {}^{\geq -n}(T^*_{x}\CX)|_{S}.
\end{equation}

\medskip

We claim:
\begin{lem}  
The map \eqref{e:cotangent cmpx for presheaf nonaffine} is an isomorphism.
\end{lem}

\begin{proof}

This follows from the description of $\on{Pro}(\QCoh(Z)^{\geq -n,\leq 0})$
given by \lemref{l:gluing pro}.

\end{proof}

\ssec{The cotangent complex of a DG scheme}  

\sssec{}

Assume for a moment that $\CX=X\in {}^{\leq n}\!\dgSch_{\on{qsep-qc}}$.
It is well-known that in this case the object
$^{\geq -n}(T^*_{x}X)\in \on{Pro}\left(\QCoh(S)^{\geq -n,\leq 0}\right)$ actually belongs to 
$\QCoh(S)^{\geq -n,\leq 0}$:

\begin{proof}
It is easy to readuce the assertion to the case when $X$ is affine. 
It is enough to show that the functor $^{\geq -n}(T^*_{x}X)$ commutes with filtered limits.
But filtered limits in $\QCoh(S)^{\geq -n,\leq 0}$ map to filtered colimits in $^{\leq n}\!\affdgSch$,
and the assertion follows.
\end{proof}

\sssec{}

We obtain that for any $X\in {}^{\leq n}\!\dgSch_{\on{qsep-qc}}$ we have a well-defined object ${}^{\geq -n}(T^*X)\in \QCoh(X)^{\geq -n,\leq 0}$,
such that for any affine $S$ with a map $x:S\to X$, we have
\begin{equation} \label{e:pull back of tangent}
^{\geq -n}(T^*_xX)\simeq {}^{\geq -n}\!x^*({}^{\geq -n}(T^*X)).
\end{equation}

\medskip

Moreover, as schemes are sheaves in the Zariski topology, the isomorphism
\eqref{e:pull back of tangent} remains valid when $S\in {}^{\leq n}\!\affdgSch$
is replaced by an arbitrary object $Z\in {}^{\leq n}\!\dgSch_{\on{qsep-qc}}$.

\sssec{}

In particular, taking $Z=X$ and $x$ to be the identity map, we obtain 
that the identity map on ${}^{\geq -n}(T^*Z)$ defines a canonical map
$$\fd_{can}:Z_{{}^{\geq -n}(T^*Z)}\to Z.$$

\sssec{}  \label{sss:again explicit cotangent}

Assume now that $\CX\in {}^{\leq n}\!\dgindSch$, and is written as in \eqref{e:indscheme as a colimit}
for some index set $A$, and let $Z\in {}^{\leq n}\!\dgSch_{\on{qsep-qc}}$. 

\medskip

Let $x:Z\to \CX$ be a map that factors through a map $x_{\alpha_0}:Z\to X_{\alpha_0}$. We obtain
that $^{\geq -n}(T^*_x\CX)$ can be explicitly presented as a pro-object of
$\on{Pro}(\QCoh(Z)^{\geq -n,\leq 0})$. Namely, we have:
\begin{equation} \label{e:expl cotangent indsch}
^{\geq -n}(T^*_x\CX)\simeq \underset{\alpha\in \sA_{\alpha_0/}}{``lim"}\,{}^{\geq -n}(T^*_{x_\alpha}X_\alpha),
\end{equation}
where $x_\alpha$ denotes the composition $Z\overset{x_{\alpha_0}}\longrightarrow X_{\alpha_0}\to X_\alpha$.

\sssec{}  \label{sss:codifferential}

Let $\CX$ again be an arbitrary object of $^{\leq n}\!\inftydgprestack$, satisfying Condition (A),
$S\in {}^{\leq n}\!\affdgSch$
and $x:S\to \CX$ a point. We claim that there exists a canonical map in $\on{Pro}(\QCoh(S)^{\geq -n,\leq 0})$
\begin{equation} \label{e:codifferential map}
(dx)^*:{}^{\geq -n}(T^*_x\CX)\to {}^{\geq -n}(T^*S).
\end{equation}

Indeed, it corresponds to the map $S_{{}^{\geq -n}(T^*S)}\to \CX$ given by the composite
$$S_{{}^{\geq -n}(T^*S)}\overset{\fd_{can}}\to S\overset{x}\to \CX.$$

\medskip

The same remains true with $S\in {}^{\leq n}\!\affdgSch$ replaced by $Z\in {}^{\leq n}\!\dgSch_{\on{qsep-qc}}$, whenever
$\CX$ satisfies Zariski descent. 

\ssec{General square-zero extensions}

\sssec{} \label{sss:gen sq zero}

Let $Z$ be an object of $^{\leq n-1}\!\dgSch_{\on{qsep-qc}}$. 
The category $^{\leq n-1}\!\on{SqZExt}(Z)$ of \emph{square-zero extensions of $Z$}
is defined to be the opposite of
$$\left((\QCoh(Z)^{\geq -n+1,\leq 0})_{^{\geq -n}(T^*Z)[-1]/}\right)^{\on{op}}.$$

\sssec{}

We have a natural forgetful functor
$$^{\leq n-1}\!\on{SqZExt}(Z)\to ({}^{\leq n-1}\!\dgSch_{\on{qsep-qc}})_{Z/},$$
defined as follows.

For $\CI\in \QCoh(Z)^{\geq -n+1,\leq 0}$ and a map $\gamma:{}^{\geq -n}(T^*Z)\to \CI[1]$,
we construct the corresponding scheme $Z'$ as the push-out in $^{\leq n}\!\dgSch_{\on{qsep-qc}}$
\begin{equation} \label{e:all from sq zero}
Z\underset{Z_{\CI[1]}}\sqcup Z,
\end{equation}
where the first map $Z_{\CI[1]}\to Z$ is the projection, and the second map corresponds to $\gamma$
via the universal property of $^{\geq -n}(T^*Z)$. 

\medskip

We note that when $Z$ is affine, by \corref{c:nil pushout}, the push-out in \eqref{e:all from sq zero}
is isomorphic to the corresponding push-out taken in $\affdgSch$.

\sssec{}

Let us denote by $i$ the resulting closed embedding 
$$Z\to Z\underset{Z_{\CI[1]}}\sqcup Z$$
corresponding to the canonical map of the first factor. 

\medskip

We have an exact triangle in $\QCoh(Z')$:
$$i_*(\CI)\to \CO_{Z'}\to i_*(\QCoh(Z)).$$

\begin{rem}
Informally, we can think of the data of $i_*(\CI)\in \QCoh(Z)^{\geq -n+1,\leq 0}$ for 
$$(\CI,\gamma)\in (\QCoh(Z)^{\geq -n+1,\leq 0})_{^{\geq -n}(T^*Z)[-1]/}$$
as the ``ideal" of $Z$ inside $Z'$. The fact that this ``ideal" comes as the direct
image of an object in $\QCoh(Z)$ reflects the fact that its square is zero. This
explains the terminology of ``square-zero extensions." 
\end{rem}

\begin{rem}
Let us emphasize that, unlike the situation of classical schemes,
the forgetful functor
$$^{\leq n-1}\!\on{SqZExt}(Z)\to ({}^{\leq n-1}\!\dgSch_{\on{qsep-qc}})_{Z/}$$
is \emph{not} fully faithful. I.e., being a square-zero extension is not a property, but is extra structure.
\end{rem} 

\sssec{}

However, we have the following:

\begin{lem}  \label{l:can sq zero}
For $Z\in {}^{\leq n-1}\!\dgSch_{\on{qsep-qc}}$, the forgetful functor
$$^{\leq n-1}\!\on{SqZExt}(Z)\to ({}^{\leq n-1}\!\dgSch_{\on{qsep-qc}})_{Z/}$$ 
induces an equivalence between the full subcategories of both sides corresponding to
$Z\hookrightarrow Z'$ for which $^{\leq n-2}\!Z\to {}^{\leq n-2}\!Z'$ is an isomorphism.
\end{lem}

\begin{cor} \label{c:can sq zero}
For $Z'\in ({}^{\leq n}\!\dgSch_{\on{qsep-qc}})$, the canonical map $^{\leq n-1}\!Z\to Z$ has a canonical structure
of an object of $^{\leq n}\!\on{SqZExt}(Z')$.
\end{cor}

In addition, we have:

\begin{lem}  \label{l:cl sq zero}
For $Z\in \Sch_{\on{qsep-qc}}$, the forgetful unctor
$$^{\leq 0}\!\on{SqZExt}(Z)\to ({}^{\leq 0}\!\dgSch_{\on{qsep-qc}})_{Z/}$$ 
is fully faithful and its essential image consists of closed embeddings 
$Z\hookrightarrow Z'$, such that the ideal $\CI$ of $Z$ in $Z'$
satisfies $\CI^2=0$.
\end{lem}

\sssec{}

Let $\phi:Z_1\to Z_2$ be an \emph{affine} map between objects of $^{\leq n-1}\!\dgSch_{\on{qsep-qc}}$. There is a canonically defined functor
\begin{equation} \label{e:pushout for square zero ext}
^{\leq n-1}\!\on{SqZExt}(Z_1)\to {}^{\leq n-1}\!\on{SqZExt}(Z_2),
\end{equation}
which it sends 
$$(\CI_1,\gamma_1)\in (\QCoh(Z_1)^{\geq -n+1,\leq 0})_{^{\geq -n}(T^*Z_1)[-1]/}$$
to
$$(\CI_2,\gamma_2)\in (\QCoh(Z_2)^{\geq -n+1,\leq 0})_{^{\geq -n}(T^*Z_2)[-1]/},$$
where
$$\CI_2:=\phi_*(\CI_1),$$ and $\gamma_2$ is obtained by the
$(\phi^*,\phi_*)$ adjunction from the map 
$$^{\geq -n}\phi^*({}^{\geq -n}(T^*Z_2))\overset{(d\phi)^*}\longrightarrow
{}^{\geq -n}(T^*Z_1)\overset{\gamma_1}\longrightarrow \CI_1.$$

\sssec{}

The following assertion results from the construction:

\begin{lem} \label{l:SqZ and pushout}
The following diagram commutes
$$
\CD
^{\leq n-1}\!\on{SqZExt}(Z_1)  @>>>  ^{\leq n-1}\!\on{SqZExt}(Z_2)  \\
@VVV    @VVV   \\
({}^{\leq n-1}\!\dgSch_{\on{qsep-qc}})_{Z_1/}  @>>>  ({}^{\leq n-1}\!\dgSch_{\on{qsep-qc}})_{Z_2/},
\endCD
$$
where the bottom horizontal arrow is the push-out functor
$$Z'_1\mapsto Z'_1\underset{Z_1}\sqcup\, Z_2.$$
\end{lem}

\ssec{Infinitesimal cohesiveness and Condition (C)}  

\sssec{}  \label{sss:condition C}

Let $S\in {}^{\leq n-1}\!\affdgSch$ and let $(\CI,\gamma)$ be an object of $^{\leq n-1}\!\on{SqZExt}(S)$.
Let $$S':=S\underset{S_{\CI[1]}}\sqcup S$$
be as in \eqref{e:all from sq zero}. 

\medskip

For $\CX\in {}^{\leq n}\!\inftydgprestack$, consider the resulting map
\begin{equation} \label{e:cond C}
\Maps(S',\CX)\to \Maps(S,\CX)\underset{\Maps(S_{\CI[1]},\CX)}\times \Maps(S,\CX).
\end{equation}

\medskip

\begin{defn}
We shall say that $\CX$ satisfies indscheme-like Condition (C) if the map \eqref{e:cond C}
is an isomorphism for any $(S,\CI,\gamma)$ as above.
\end{defn}

An alternative terminology for prestacks satisfying Condition (C) is \emph{infinitesimally cohesive}. 

\sssec{}

Note that from \propref{p:compat with pushouts} we obtain:

\begin{cor}
Any  $\CX\in {}^{\leq n}\!\dgindSch$ satisfies Condition (C).
\end{cor}

\sssec{}  \label{sss:reform C}

For $S\in {}^{\leq n-1}\!\affdgSch$, let $^{\leq n-1}\!\on{SqZExt}(S,x)$ be the category of triples 
$$\{S\hookrightarrow S',\,\,x':S'\to \CX,\,\, x'|_S\simeq x\},$$
where $S\hookrightarrow S'$ is a square-zero extension with
$S'\in {}^{\leq n-1}\!\dgSch$. I.e.,
$$^{\leq n-1}\!\on{SqZExt}(S,x):=
{}^{\leq n-1}\!\on{SqZExt}(S)\underset{^{\leq n-1}\!\affdgSch_{S/}}\times{}^{\leq n-1}\!\affdgSch_{S/\,/\CX}.$$

\medskip

Suppose now that $\CX$ satisfies Condition (A). For $S\in {}^{\leq n-1}\!\affdgSch$, recall
the map in $\on{Pro}(\QCoh(S)^{\geq -n,\leq 0})$
$$(dx)^*:{}^{\geq -n}(T^*_{x}\CX)\to {}^{\geq -n}(T^*S).$$
Consider the object 
$$\on{Cone}((dx)^*)[-1]\in \on{Pro}(\QCoh(S)^{\geq -n+1,\leq 1}).$$

\medskip

Hence, we obtain:

\begin{lem}  \label{l:reform C}
An object $\CX\in {}^{\leq n}\!\on{PreStk}$, satisfying Condition (A), satisfies condition (C) if and only if the naturally defined 
functor 
$$^{\leq n-1}\!\on{SqZExt}(S,x)\to \left(\left(\QCoh(S)^{\geq -n+1,\leq 0}\right)_{\on{Cone}((dx)^*)[-1]/}\right)^{\on{op}}$$
is an equivaence.
\end{lem}

\sssec{}

Assume now that $\CX$ satisfies Zariski descent as well as Conditions (A) and (C). We obtain
that for $Z\in {}^{\leq n-1}\!\dgSch$ and a given $x:Z\to \CX$,  the description of the category 
$^{\leq n-1}\!\on{SqZExt}(Z,x)$ as 
$$\left(\left(\QCoh(Z)^{\geq -n+1,\leq 0}\right)_{\on{Cone}((dx)^*)[-1]/}\right)^{\on{op}}$$
remains valid.

\medskip

Moreover, we have the following: 

\begin{lem}  \label{l:sq zero filtered}
Under the above circumstances the following are equivalent:

\smallskip

\noindent{\em(a)} The category $^{\leq n-1}\!\on{SqZExt}(Z,x)$ is filtered.

\smallskip

\noindent{\em(b)} $T^*_xZ/\CX[-1] \simeq \on{Cone}((dx)^*)[-1]$ belongs to 
$\on{Pro}(\QCoh(Z)^{\geq -n+1,\leq 0})$. 

\smallskip

\noindent{\em(c)} The map 
$$H^0((dx)^*):H^0\left({}^{\geq -n}(T^*_x\CX)\right)\to 
H^0\left({}^{\geq -n}(T^*Z)\right)$$
is surjective. 

\end{lem}

\sssec{} \label{sss:closed embed surj}

We note that condition (c) in \lemref{l:sq zero filtered} is satisfied when $\CX$ is an indscheme, and
the map $x:Z\to \CX$ is a closed embedding. 

\sssec{}

Let $\phi:S_1\to S_2$ be a map in $^{\leq n-1}\!\affdgSch$.
For $x_2:S_2\to \CX$,
composition defines a map
\begin{equation} \label{e:map out of pushout}
^{\leq n-1}\!\on{SqZExt}(S_2,x_2)\underset{^{\leq n-1}\!\on{SqZExt}(S_2)}\times 
{}^{\leq n-1}\!\on{SqZExt}(S_1)\to {}^{\leq n-1}\!\on{SqZExt}(S_1,x_1).
\end{equation}
using the functor \eqref{e:pushout for square zero ext}. 

\medskip

From the definitions, we obtain:

\begin{lem}
If $\CX$ satisfies Conditions (B) and (C), then the map \eqref{e:map out of pushout} is an isomorphism.
\end{lem}

\sssec{} \label{sss:affine map}

If $\CX$ satisfies Zariski descent, then the same continues to be true for
$S_1$ and $S_2$ replaced by arbitrary objects $Z_1,Z_2\in {}^{\leq n-1}\!\dgSch_{\on{qsep-qc}}$,
but keeping the assumption that $f:Z_1\to Z_2$ be affine. 

\medskip

In other words, the map \eqref{e:map out of pushout} is an isomorphism, where 
$$^{\leq n-1}\!\on{SqZExt}(Z_1)\to {}^{\leq n-1}\!\on{SqZExt}(Z_2)$$
is the functor defined in \eqref{e:pushout for square zero ext}.

\sssec{} \label{sss:non-affine map}

Now suppose that $Z_1,Z_2\in {}^{\leq n-1}\!\dgSch_{\on{qsep-qc}}$ are as above, but the
map $f$ is not necessarily affine. Assume that $\CX$ satisfies Zariski descent, and let
$x_2:Z_2\to \CX$ be a map satisfying the equivalent conditions of \lemref{l:sq zero
filtered}. Let $x_1=x_2\circ f$.
In this situation, \secref{sss:affine map} still applies.  Namely, we have:

\begin{lem}\label{l:non-affine pushout}
In the above situation, if $\CX$ satisfies conditions (A), (B) and (C), there is a canonically defined functor
$$^{\leq n-1}\!\on{SqZExt}(Z_1,x_1)\to {}^{\leq n-1}\!\on{SqZExt}(Z_2,x_2),$$
such that the diagram
$$
\CD
^{\leq n-1}\!\on{SqZExt}(Z_1,x_1)  @>>>  ^{\leq n-1}\!\on{SqZExt}(Z_2,x_2) \\
@VVV   @VVV   \\
({}^{\leq n-1}\!\dgSch_{\on{qsep-qc}})_{Z_1/}  @>>>  ({}^{\leq n-1}\!\dgSch_{\on{qsep-qc}})_{Z_2/}
\endCD
$$
commutes, where the bottom horizontal arrow is the push-out functor
$$Z'_1\mapsto Z'_1\underset{Z_1}\sqcup\, Z_2.$$
\end{lem}

\begin{proof}
By definition, an object of $^{\leq n-1}\!\on{SqZExt}(Z_1,x_1)$ is given by a map
$T^*_{x_1}Z_1/\CY \rightarrow \CI[1]$ for
$\CI \in \QCoh(Z_1)^{\geq{-n+1,\leq 0}}$.  This gives a map
$T^*_{x_2}Z_2/\CY \rightarrow f_*(\CI[1])$.  By assumption on $x_2$, this map
canonically factors through $\tau^{\leq -1} f_*(\CI[1]) = \tau^{\leq 0}(f_*\CI)[1]$.  This
gives the desired functor
$$^{\leq n-1}\!\on{SqZExt}(Z_1,x_1)\to {}^{\leq n-1}\!\on{SqZExt}(Z_2,x_2).$$
Let $Z_1'$ be the square zero extension of $Z_1$, and $Z_2'$ the corresponding square zero
extension of $Z_2$; i.e.,
$$ Z_2' = Z_2 \underset{(Z_2)_{\tau^{\leq 0}(f_*\CI)[1]}}{\sqcup} Z_2. $$
It follows from \corref{c:nil pushout} that
$$Z'_1\underset{Z_1}\sqcup\, Z_2 \simeq Z_2 \underset{(Z_2)_{\tau^{\leq 0}f_*(\CI[1])}}{\sqcup} Z_2.$$
Furthermore, by the above discussion, both maps
$(Z_2)_{\tau^{\leq 0}f_*(\CI[1])} \rightarrow Z_2$ canonically factor through
$(Z_2)_{\tau^{\leq 0}(f_*\CI)[1]}$ (compatibly with the map to $\CX$).  This gives the
comparison map
$$ Z'_1\underset{Z_1}\sqcup\, Z_2 \rightarrow Z'_2,$$
and it is easy to show that it is an isomorphism. 
\end{proof}







\ssec{Dropping $n$-coconnectivity}

Finally, note that the above considerations are valid for an object $\CX\in \inftydgprestack$, simply
by omitting the $n$-coconnectivity conditions.

\begin{defn}  \label{def:deform theory}
We shall say that $\CX\in \on{PreStk}$ admits \emph{connective deformation theory} if it is convergent, and
satisfies Conditions (A), (B) and (C).
\end{defn}

\sssec{}
By \corref{c:compat with pushouts}, we obtain:

\begin{cor}\label{c:indsch have def theory}
Any  $\CX\in \dgindSch$ admits connective deformation theory.
\end{cor}

\section{A characterization of DG indschemes via deformation theory}  \label{s:char via deform}

\ssec{The statement}

Let $\CX$ be an object of $^{\leq n}\!\inftydgprestack$, such that
$^{cl}\CX$ is a classical indscheme. We would like to  
give a criterion for when $\CX$ belongs to $^{\leq n}\!\dgindSch$.

\medskip

\begin{thm} \label{t:char by deform}
Under the above circumstances, $\CX\in {}^{\leq n}\!\dgindSch$
if and only if $\CX$ admits an extension to an
object $\CX_{n+1}\in {}^{\leq n+1}\!\inftydgprestack$, which
satisfies indscheme-like Conditions (A), (B) and (C).
\end{thm}

The rest of this subsection is devoted to the proof of this theorem. The ``only if"
direction is clear: if $\CX\in {}^{\leq n}\!\dgindSch$, the extension
$$\CX_{n+1}:={}^{^{\leq n+1}\!L}\!\on{LKE}_{^{\leq n}\!\affdgSch\hookrightarrow {}^{\leq n+1}\!\affdgSch}(\CX)$$
belongs to $^{\leq n+1}\!\dgindSch$, and hence satisfies Conditions (A), (B) and (C). 

\medskip

For the opposite implication, we will argue by induction on $n$, assuming that the statement is true for
$n'<n$. In particular, we can assume that $^{\leq n-1}\CX:=\CX|_{^{\leq n-1}\!\affdgSch}$
belongs to $^{\leq n-1}\!\dgindSch$.

\sssec{Step 0: initial remarks} \hfill

\smallskip

First, we note that by \corref{c:can sq zero}, the induction hypothesis combined with Condition (C) implies
that the prestack $\CX$ satisfies Zariski descent. Hence, deformation theory
of maps into it from objects of $^{\leq n}\!\dgSch_{\on{qsep-qc}}$,
described in the previous section applies. 

\medskip

Thus, for $X\in {}^{\leq n-1}\!\dgSch_{\on{qsep-qc}}$ and a map $f:X\to {}^{\leq n-1}\CX$, we have a well-defined
object
$$^{\geq -n-1}(T^*_f\CX_{n+1})\in \on{Pro}(\QCoh(X)^{\geq -n-1,\leq 0}),$$
whose formation is compatible with pull-backs.

\medskip

Moreover, we have:
$$^{\leq n}\!\on{SqZExt}(X,f)\simeq \left((\QCoh(X)^{\geq -n,\leq 0})_{\on{Cone}({}^{\geq -n-1}(T^*_f\CX_{n+1})\to 
{}^{\geq -n-1}(T^*X))[-1]/}\right)^{\on{op}}.$$

\medskip

Let $^{\leq n-1}\!\dgSch_{\on{closed}\, \on{in}\,\CX}$ denote the full subcategory 
of $({}^{\leq n-1}\!\dgSch_{\on{qsep-qc}})_{/\CX}$ that consists of those $f:X\to \CX$,
for which $f$ is a closed embedding. In particular, a map 
$$(X_1,f_1)\to (X_2,f_2)$$ 
in this category is given by $$(\phi:X_1\to X_2,\,f_1\simeq \phi\circ f_2),$$
where the underlying map
$\phi:X_1\to X_2$ is also a closed embedding, and in particular, affine. 

\medskip

We obtain that push-out makes the
assignment  $$(X,f)\mapsto {}^{\leq n}\!\on{SqZExt}(X,f)$$
into a category co-fibered over $^{\leq n-1}\!\dgSch_{\on{closed}\, \on{in}\,\CX}$.
We denote it by
$$^{\leq n}\!\on{SqZExt}({}^{\leq n-1}\!\dgSch_{\on{closed}\, \on{in}\,\CX}).$$

\medskip

By \secref{sss:closed embed surj}, we have that for $(X,f)\in {}^{\leq n-1}\!\dgSch_{\on{closed}\,\on{ in}\,\CX}$,
\begin{equation} \label{e:cone of codiff}
\on{Cone}({}^{\geq -n-1}(T^*_f\CX_{n+1})\to {}^{\geq -n-1}(T^*X))[-1]\in \on{Pro}(\QCoh(X)^{\geq -n,\leq 0}).
\end{equation}

Hence, by \lemref{l:sq zero filtered}, the category 
$^{\leq n}\!\on{SqZExt}(X,f)$ is filtered.

\sssec{Step 1: creating closed embeddings} \hfill

\smallskip

It is of course not true that for any $(X,f)\in {}^{\leq n-1}\!\dgSch_{\on{closed}\,\on{ in}\,\CX}$
and 
\begin{equation} \label{e:example of extension}
(i:X\hookrightarrow X',\,\, f':X'\to \CX)\in {}^{\leq n}\!\on{SqZExt}(X,f),
\end{equation}
the map $f'$ is also a closed embedding. 

\medskip

Let 
\begin{equation} \label{e:mapping closed embeddings in}
^{\leq n}\!\on{SqZExt}(X,f)_{\on{closed}\,\on{in}\,\CX}\subset {}^{\leq n}\!\on{SqZExt}(X,f)
\end{equation}
denote the full subcategory spanned by objects for which the map $f'$ is a closed embedding. 
We claim that the functor \eqref{e:mapping closed embeddings in} admits a left adjoint.

\medskip

Indeed, for an object \eqref{e:example of extension}, given by a pair
$$\on{Cone}({}^{\geq -n-1}(T^*_f\CX_{n+1})\to {}^{\geq -n-1}(T^*X))[-1]\to \CI,\quad \CI\in  \QCoh(X)^{\geq -n,\leq 0},$$
the image of the map
$$H^0\left(\on{Cone}({}^{\geq -n-1}(T^*_f\CX_{n+1})\to {}^{\geq -n-1}(T^*X))[-1] \right)\to H^0(\CI)$$
is a well-defined object $\CJ\in  \QCoh(X)^\heartsuit$, by \eqref{e:expl cotangent indsch}.

\medskip 

The value of the sough-for left adjoint on the above object of $^{\leq n}\!\on{SqZExt}(X,f)$ is given by 
$$\on{Cone}({}^{\geq -n-1}(T^*_f\CX_{n+1})\to {}^{\geq -n-1}(T^*X))[-1]\to \wt\CI,$$
where $\wt\CI\in \QCoh(X)^{\geq -n,\leq 0}$ fits into the exact triangle
$$\wt\CI\to \CI \to \CJ.$$

\medskip

In particular, we obtain that the embedding \eqref{e:mapping closed embeddings in} is cofinal. We also
obtain that the category $^{\leq n}\!\on{SqZExt}(X,f)_{\on{closed}\,\on{in}\,\CX}$ is also filtered.

\medskip

Let
$$^{\leq n}\!\on{SqZExt}({}^{\leq n-1}\!\dgSch_{/\CX})_{\on{closed}\, \on{in}\,\CX}$$ denote
the corresponding full subcategory of $^{\leq n}\!\on{SqZExt}({}^{\leq n-1}\!\dgSch_{\on{closed}\, \on{in}\,\CX})$.
It follows that the forgetful functor
$$^{\leq n}\!\on{SqZExt}({}^{\leq n-1}\!\dgSch_{/\CX})_{\on{closed}\, \on{in}\,\CX}\to 
{}^{\leq n-1}\!\dgSch_{\on{closed}\,\on{ in}\,\CX}$$
is also a co-Cartesian fibration. 

\sssec{Step 2: construction of the inductive system}

Let  $${}^{\leq n-1}\CX\simeq \underset{\alpha\in \sA}{colim}\, X_{\alpha}$$
be a presentation as in \eqref{e:indscheme as a colimit} with $X_\alpha\in {}^{\leq n-1}\!\dgSch_{\on{qsep-qc}}$.
For every $\alpha\in \sA$, let $f_\alpha$ denote the corresponding map $X_\alpha\to {}^{\leq n-1}\CX$.
For an arrow $\alpha_1\to \alpha_2$,
let $f_{\alpha_1,\alpha_2}$ denote the corresponding map $X_{\alpha_1}\to X_{\alpha_2}$.

\medskip

For each $\alpha$, let $\sB_\alpha$ denote the category 
$$^{\leq n}\!\on{SqZExt}(X_\alpha,f_\alpha)_{\on{closed}\,\on{in}\,\CX}.$$

\medskip

For $\beta$ an object of $\sB_\alpha$, we will denote
by $X_\beta$ the corresponding $\nDG$ scheme $X'_\alpha$, and by $f_\beta$ the closed embedding
$f'_\alpha$. Let $i_\beta$ denote the closed embedding $X_\alpha\to X_\beta$. 
We have an evident functor from $\sB_\alpha$ to the category of $\nDG$ schemes
endowed with a closed embedding into $\CX$. 

\medskip

The above construction makes the
assignment
$$\alpha\mapsto \sB_\alpha$$
into a category co-fibered over $A$. Let $\phi$ denote the tautological map $\sB\to \sA$.
Since $A$ is filtered and all $\sB_\alpha$ are filtered, the category $\sB$ is also filtered.

\medskip

It is also clear that the assignment 
$$(\beta\in \sB)\mapsto (X_\beta\overset{f_\beta}\longrightarrow \CX)$$
is a functor from $B$ to the category of $\nDG$ schemes
equipped with a closed embedding into $\CX$. For an arrow $(\beta_1\to \beta_2)\in \sB$,
let $f_{\beta_1,\beta_2}$ denote the corresponding closed
embedding $X_{\beta_1}\to X_{\beta_2}$.

\medskip

Thus, we obtain a map 
\begin{equation} \label{e:map to prove}
\underset{\beta\in \sB}{colim}\, X_\beta\to \CX,
\end{equation}
and we claim that it is an isomorphism. 

\medskip

In other words, we have to show that for $S'\in {}^{\leq n}\!\affdgSch$,
the maps $f_\beta$ induce an isomorphism:
\begin{equation} \label{e:map one}
\underset{\beta\in \sB}{colim}\, \Maps(S',X_{\beta}) \simeq \Maps(S',\CX).
\end{equation}

\sssec{Step 3: a map in the opposite direction}

Let us construct a map that we shall eventually prove to be the inverse of
\eqref{e:map one}:

\begin{equation} \label{e:map two}
\Maps(S',\CX)\to \underset{\beta\in \sB}{colim}\, \Maps(S',X_{\beta}).
\end{equation}

For $S'\in {}^{\leq n}\!\affdgSch$, set $S:={}^{\leq n-1}S'$. Tautologically, we have:
\begin{equation} \label{e:colim 1}
\Maps(S',\CX)\simeq \underset{\alpha\in \sA}{colim}\,
\{x:S\to X_\alpha,\,x':S'\to \CX,\, x'|_S \simeq f_\alpha\circ x\}.
\end{equation}

We can view $S'$ as a square-zero extension of $S$ by
the object 
$$H^{-n}(\CO_{S'})\in \QCoh(S)^\heartsuit[n]\subset \QCoh(S)^{\geq -n,\leq 0},$$
see \corref{c:can sq zero}.

\medskip
 
By \lemref{l:non-affine pushout}, we obtain an isomorphism in $\inftygroup$: 
\begin{multline*}
\{x:S\to X_\alpha,\,x':S'\to \CX,\, x'|_S \simeq f_\alpha\circ x\}\simeq \\
\simeq \underset{(X_\alpha\hookrightarrow X'_\alpha)\in {}^{\leq n}\!\on{SqZExt}(X_\alpha,f_\alpha)}{colim}\, 
\{x:S\to X_\alpha,\, S'\underset{S}\sqcup\, X_\alpha\simeq X'_\alpha\}.
\end{multline*}

Taking into account that 
$$\sB_\alpha:={}^{\leq n}\!\on{SqZExt}(X_\alpha,f_\alpha)_{\on{closed}\,\on{in}\,\CX}\hookrightarrow
{}^{\leq n}\!\on{SqZExt}(X_\alpha,f_\alpha)$$
is cofinal, we have an isomorphism in $\inftygroup$: 
$$\{x:S\to X_\alpha,\,x':S'\to \CX,\, x'|_S \simeq f_\alpha\circ x\}
\simeq \underset{\beta\in \sB_\alpha}{colim}\, \{x:S\to X_\alpha,\, S'\underset{S}\sqcup\, X_\alpha\simeq X_\beta\}.$$

\medskip

Combining this with \eqref{e:colim 1}, we obtain a canonical isomorphism in $\inftygroup$: 
$$\Maps(S',\CX)\simeq 
\underset{\alpha\in \sA}{colim}\, \left(\underset{\beta\in \sB_\alpha}{colim}\, 
\{x:S\to X_\alpha,\, S'\underset{S}\sqcup\, X_\alpha\simeq X_\beta \}\right).$$

We have a canonical forgetful map
\begin{multline*}
\underset{\alpha\in \sA}{colim}\, \left(\underset{\beta\in \sB_\alpha}{colim}\, 
\{x:S\to X_\alpha,\, S'\underset{S}\sqcup\, X_\alpha\simeq X_\beta\}\right)\to
\underset{\alpha\in \sA}{colim}\, \left(\underset{\beta\in \sB_\alpha}{colim}\, 
\{x':S'\to X_\beta\}\right)\simeq \\
\simeq \underset{\beta\in \sB}{colim}\, \{x':S'\to X_\beta\}.
\end{multline*}

Thus, we obtain the desired map 
$$\Maps(S',\CX)\to \underset{\beta\in \sB}{colim}\, \{x':S'\to X_\beta\}$$
of \eqref{e:map two}.

\medskip

It is immediate from the construction, the composite arrow
$$\Maps(S',\CX)\overset{\text{\eqref{e:map two}}}\longrightarrow \underset{\beta\in \sB}{colim}\, \Hom(S',X_{\beta})
\overset{\text{\eqref{e:map one}}}\longrightarrow 
\Hom(S',\CX)$$
is the identity map.

\sssec{Step 4: computation of the other composition}

It remains to show that the composition
\begin{equation} \label{e:hard composition}
\underset{\beta\in \sB}{colim}\, \Maps(S',X_{\beta}) \overset{\text{\eqref{e:map one}}}\longrightarrow 
\Maps(S',\CX) \overset{\text{\eqref{e:map two}}}\longrightarrow 
\underset{\beta\in \sB}{colim}\, \Maps(S',X_{\beta})
\end{equation}
is isomorphic to the identity map. 

\medskip

To do this, we introduce yet another category, denoted $\Gamma$.  An object of $\Gamma$ is given by the following data.

\begin{itemize}

\item An arrow $(\beta\to \beta_1)\in \sB$, which projects by means of $\phi$ to an arrow $(\alpha\to \alpha_1)\in \sA$,

\item A map $g_{\beta,\alpha_1}:{}^{\leq n-1}\!X_\beta\to X_{\alpha_1}$,

\item A commutative diagram of square-zero extensions compatible with maps to $\CX$
$$
\CD
X_\beta  @>{f_{\beta,\beta_1}}>>  X_{\beta_1}  \\
@A{j_\beta}AA    @AA{i_{\beta_1}}A  \\
{}^{\leq n-1}\!X_\beta  @>{g_{\beta,\alpha_1}}>>  X_{\alpha_1},
\endCD
$$
where $j_\beta$ is the canonical map, corresponding to the truncation (see \corref{c:can sq zero}),

\item An identification of the composition
$$X_\alpha\overset{^{\leq n-1}\!i_\beta}\longrightarrow {}^{\leq n-1}\!X_\beta\overset{g_{\beta,\alpha_1}}\longrightarrow
X_{\alpha_1} \text{ with } f_{\alpha,\alpha_1},$$

\item A homotopy between the resulting two identifications, making the following diagram commutative:
$$
\CD
f_{\alpha_1}\circ g_{\beta,\alpha_1}\circ {}^{\leq n-1}i_\beta  @>{\sim}>> f_{\beta}\circ j_\beta \circ {}^{\leq n-1}i_\beta \simeq 
f_\beta\circ i_\beta \\
@V{\sim}VV   @V{\sim}VV   \\
f_{\alpha_1}\circ f_{\alpha,\alpha_1}  @>>>  f_\alpha.
\endCD
$$

\end{itemize}

We can depict this data in a diagram:

\begin{gather}  
\xy
(-10,0)*+{X_\alpha}="X";
(15,0)*+{X_{\alpha_1}}="Y";
(60,0)*+{\CX.}="Z";
(-10,15)*+{{}^{\leq n-1}\!X_\beta}="W";
(-10,30)*+{X_\beta}="A";
(15,30)*+{X_{\beta_1}}="B";
{\ar@{->}^{f_{\alpha,\alpha_1}} "X";"Y"};
{\ar@{->}_{f_{\alpha_1}} "Y";"Z"};
{\ar@{->}^{^{\leq n-1}i_\beta} "X";"W"};
{\ar@{->}_{j_\beta} "W";"A"};
{\ar@{->}^{f_{\beta,\beta_1}} "A";"B"};
{\ar@{->}^{g_{\beta,\alpha_1}} "W";"Y"};
{\ar@{->}_{f_\beta} "A";"Z"};
{\ar@{->}^{f_{\beta_1}} "B";"Z"};
{\ar@{->}^{i_{\beta_1}} "Y";"B"};
\endxy
\end{gather}

Morphisms in $\Gamma$ are defined naturally (so that the corresponding diagrams of DG schemes
commute). 

\medskip

There are tautological maps $\psi,\psi_1:\Gamma\to \sB$ that remember the data of $\beta$ and
$\beta_1$, respectively. 

\medskip

The colimit 
$$\underset{\gamma\in \Gamma}{colim}\, \Maps(S',X_{\psi(\gamma)}),$$
admits a tautological map
\begin{equation} \label{e:map of colims one}
r:\underset{\gamma\in \Gamma}{colim}\, \Maps(S',X_{\psi(\gamma)})\to 
\underset{\beta\in \sB}{colim}\, \Maps(S',X_{\beta}).
\end{equation}
Note, however, that we have another map
\begin{equation} \label{e:map of colims two}
r_1:\underset{\gamma\in \Gamma}{colim}\, \Maps(S',X_{\psi(\gamma)})\to 
\underset{\beta\in \sB}{colim}\, \Maps(S',X_{\beta}),
\end{equation}
which for $\gamma\in \Gamma$, sends $\Maps(S',X_{\psi(\gamma)})$ to $\Maps(S',X_{\psi_1(\gamma})$
by means of $f_{\psi(\gamma),\psi_1(\gamma)}$. However, the same edge $f_{\psi(\gamma),\psi_1(\gamma)}$
provides a homotopy between these two maps of colimits.

\medskip

It follows from the construction that the composite
$$\underset{\gamma\in \Gamma}{colim}\, \Maps(S,X_{\psi(\gamma)}) \overset{r}\to  
\underset{\beta\in \sB}{colim}\, \Maps(S',X_{\beta}) \overset{\text{\eqref{e:map one}}}\longrightarrow 
\Maps(S',\CX) \overset{\text{\eqref{e:map two}}}\longrightarrow 
\underset{\beta\in \sB}{colim}\, \Maps(S',X_{\beta})$$
coincides with the map $r_1$. 

\medskip

Therefore, to prove that the composition in \eqref{e:hard composition} is isomorphic
to the identity map, it suffices to show that the map $r$ is an isomorphism in $\inftygroup$. To do this,
we will repeatedly use the following observation:

\begin{lem} \label{l:cofinality}
Let $F:\bC'\to \bC$ be a functor between $\infty$-categories.

\smallskip

\noindent{\em(a)} Suppose that $F$ is a Cartesian fibration. Then $F$
is cofinal if and only if it has contractible fibers.

\smallskip

\noindent{\em(b)} Suppose that $F$ is a co-Cartesian fibration, and that
$F$ has contractible fibers. Then it is cofinal.

\end{lem}

It is easy to see that the functor $\psi$ is a Cartesian fibration. Applying \lemref{l:cofinality},
we obtain that it is sufficient to show that the fibers of $\psi$ are contractible. 

\sssec{Step 5: contractibility of the fibers of $\psi$}

For $\beta\in \sB$, let $\Gamma_\beta$ denote the fiber of $\psi$. Explicitly, $\Gamma_\beta$
consists of the data of

\begin{itemize}

\item An object $\alpha_1\in \sA$, and an arrow $\phi(\beta)=:\alpha\to \alpha_1$.

\item A map $g_{\beta,\alpha_1}:{}^{\leq n-1}\!X_\beta\to X_{\alpha_1}$,

\item An identification of the composition
$$X_\alpha\overset{^{\leq n-1}\!i_\beta}\longrightarrow {}^{\leq n-1}\!X_\beta\overset{g_{\beta,\alpha_1}}\longrightarrow
X_{\alpha_1} \text{ with } f_{\alpha,\alpha_1},$$

\item A homotopy between the resulting two identifications, making the following diagram commutative:
$$
\CD
f_{\alpha_1}\circ g_{\beta,\alpha_1}\circ {}^{\leq n-1}\!i_\beta  @>{\sim}>> f_{\beta}\circ j_\beta \circ {}^{\leq n-1}\!i_\beta \simeq 
f_\beta\circ i_\beta \\
@V{\sim}VV   @V{\sim}VV   \\
f_{\alpha_1}\circ f_{\alpha,\alpha_1}  @>>>  f_\alpha.
\endCD
$$

\item A lift of $\alpha\to \alpha_1$ to an arrow $\beta\to \beta_1$. 

\item A commutative diagram of square-zero extensions compatible with maps to $\CX$
$$
\CD
X_\beta  @>{f_{\beta,\beta_1}}>>  X_{\beta_1}  \\
@A{j_\beta}AA    @AA{i_{\beta_1}}A  \\
{}^{\leq n-1}\!X_\beta  @>{g_{\beta,\alpha_1}}>>  X_{\alpha_1}.
\endCD
$$

\end{itemize}

\medskip

We introduce the category $\Delta_\beta$ to consist of the first four out of six of
the pieces of data in the description of $\Gamma_\beta$ given above. I.e.,
an object of $\Delta_\beta$ corresponds to a diagram

\begin{gather}  
\xy
(-10,0)*+{X_\alpha}="X";
(15,0)*+{X_{\alpha_1}}="Y";
(60,0)*+{\CX.}="Z";
(-10,15)*+{{}^{\leq n-1}\!X_\beta}="W";
(-10,30)*+{X_\beta}="A";
{\ar@{->}^{f_{\alpha,\alpha_1}} "X";"Y"};
{\ar@{->}_{f_{\alpha_1}} "Y";"Z"};
{\ar@{->}^{^{\leq n-1}i_\beta} "X";"W"};
{\ar@{->}_{j_\beta} "W";"A"};
{\ar@{->}^{g_{\beta,\alpha_1}} "W";"Y"};
{\ar@{->}_{f_\beta} "A";"Z"};
\endxy
\end{gather}

\medskip

We have a natural forgetful map $\Gamma_\beta\to \Delta_\beta$. It is easy to see
that this functor is a co-Cartesian fibration. Hence, by \lemref{l:cofinality}, it is
enough to show that $\Delta_\beta$ is contractible, and that the fibers of 
$\Gamma_\beta$ over $\Delta_\beta$ are contractible.

\sssec{Step 6: contractibility of $\Delta_\beta$}

By construction, we have a left fibration $\Delta_\beta\to \sA_{\alpha/}$, and hence 
the homotopy type of $\Delta_\beta$ is
$$\underset{\alpha_1\in \sA_{\alpha/}}{colim}\, \left(
\Maps({}^{\leq n-1}\!X_\beta,X_{\alpha_1})
\underset{\Maps({}^{\leq n-1}\!X_\beta,\CX)\underset{\Maps(X_{\alpha},\CX)}\times \Maps(X_\alpha,X_{\alpha_1})}
\times \on{pt}\right),$$
where $\on{pt}\to \Maps({}^{\leq n-1}\!(X_{\beta}),\CX)$ is the map $f_\beta\circ j_\beta$ and 
$\on{pt}\to \Maps(X_\alpha,X_{\alpha_1})$ is $f_{\alpha,\alpha_1}$.

\medskip

Since the category $A_{\alpha/}$ of objects $\alpha_1\in \sA$ \emph{under} $\alpha$ is filtered, we can commute the colimit and
the Caretesian products, and we obtain that the homotopy type of $\Delta_\beta$ is
$$\left(\underset{\alpha_1\in \sA_{\alpha /}}{colim}\, \Maps({}^{\leq n-1}\!X_\beta,X_{\alpha_1}) \right)
\underset{\Maps({}^{\leq n-1}\!X_\beta,\CX)\underset{\Maps(X_{\alpha},\CX)}\times 
\left(\underset{\alpha_1\in \sA_{\alpha /}}{colim}\, \Maps(X_\alpha,X_{\alpha_1})\right)}\times \on{pt}.$$

\medskip

Since the DG schemes and $X_\alpha$ and $X_\beta$ are quasi-separated and quasi-compact, the maps
$$\underset{\alpha_1\in\sA_{\alpha /}}{colim}\, \Maps({}^{\leq n-1}\!X_\beta,X_{\alpha_1}) \simeq
\underset{\alpha_1\in \sA}{colim}\, \Maps({}^{\leq n-1}\!X_\beta,X_{\alpha_1}) 
\to \Maps({}^{\leq n-1}\!X_\beta,{}^{\leq n-1}\CX)$$ and
$$\underset{\alpha_1\in \sA_{\alpha /}}{colim}\, \Maps(X_{\alpha},X_{\alpha_1}) \simeq 
\underset{\alpha_1\in \sA}{colim}\, \Maps(X_{\alpha},X_{\alpha_1}) 
\to \Maps(X_\alpha,{}^{\leq n-1}\CX)\simeq
\Maps(X_\alpha,\CX)$$
are isomorphisms. 

\medskip

We obtain that the homotopy type of $\Delta_\beta$ is

\begin{multline*}
\Maps({}^{\leq n-1}\!X_\beta,\CX) 
\underset{\Maps({}^{\leq n-1}\!X_\beta,\CX)\underset{\Maps(X_{\alpha},\CX)}\times 
\Maps(X_\alpha,\CX)}\times \on{pt}\simeq \\
\simeq \Maps({}^{\leq n-1}\!X_\beta,\CX) \underset{\Maps({}^{\leq n-1}\!X_\beta,\CX)} \times \on{pt}
\simeq \on{pt}.
\end{multline*}

\sssec{Step 7: contractibility of the fibers $\Gamma_\beta\to \Delta_\beta$}

For an object $\delta_\beta\in \Delta_\beta$ as above, the fiber of $\Gamma_\beta$ over it
is the category of

\begin{itemize}

\item $\beta_1\in \sB_{\alpha_1}$,

\item A map of square-zero extensions: 
$$
\CD
X_{\beta}  @>{f_{\beta,\beta_1}}>>  X_{\beta_1}  \\
@A{j_\beta}AA    @AA{i_{\beta_1}}A  \\
{}^{\leq n-1}\!X_\beta  @>{g_{\beta,\alpha_1}}>>  X_{\alpha_1},
\endCD
$$
compatible with the maps to $\CX$.

\end{itemize}

Let $j:Z\hookrightarrow Z'$ be any square-zero extension in $^{\leq n}\!\dgSch$, and
let $x_1:Z\to X_{\alpha_1}$, $x':Z'\to \CX$ be fixed maps equipped with an identification
$$f_{\alpha_1}\circ x_1\simeq x'\circ j.$$
(In our case, we are going to take $Z={}^{\leq n-1}\!X_\beta$ and $Z'=X_\beta$.)
Consider the category of pairs:

\begin{itemize}

\item $\beta_1\in \sB_{\alpha_1}$,

\item A map of square-zero extensions
$$
\CD
Z'  @>{x'_1}>>  X_{\beta_1} \\
@A{j}AA   @AA{i_{\beta_1}}A  \\
Z  @>{x_1}>> X_{\alpha_1},
\endCD
$$
compatible with the maps to $\CX$

\end{itemize}

We claim that this category is contractible. Indeed, if we omit the condition of
compatibility with the given map $x':Z'\to \CX$, we obtain the category
whose homotopy type is 
$$\underset{\beta_1\in \sB_{\alpha_1}}{colim}\, \{\text{maps of square-zero extensions as above}\},$$
which, by the definition of $\sB_{\alpha_1}$, is homotopy equivalent to 
$$\Maps(Z',\CX)\underset{\Maps(Z,\CX)}\times \on{pt},$$
where the map $\on{pt}\to \Maps(Z,\CX)$ is given by $f_{\alpha_1}\circ x_1=x'\circ j$.

\medskip

Reinstating the compatibility condition results in taking the fiber product
$$\underset{\beta_1\in \sB_{\alpha_1}}{colim}\, \left(\{\text{maps of square-zero extensions}\}
\underset{\Maps(Z',\CX)\underset{\Maps(Z,\CX)}\times \on{pt}}\times \on{pt}\right).$$
Since $\sB_{\alpha_1}$ is filtered, the above colimit can be rewritten
as
\begin{multline*}
\underset{\beta_1\in \sB_{\alpha_1}}{colim}\, \{\text{maps of square-zero extensions}\}
\underset{\Maps(Z',\CX)\underset{\Maps(Z,\CX)}\times \on{pt}}\times \on{pt}\simeq \\
\simeq \left(\Maps(Z',\CX)\underset{\Maps(Z,\CX)}\times \on{pt}\right)
\underset{\Maps(Z',\CX)\underset{\Maps(Z,\CX)}\times \on{pt}}\times \on{pt}\simeq \on{pt}.
\end{multline*}

\qed

\ssec{The $\aleph_0$ condition}

In this subsection we will give a characterization of the $\aleph_0$ property in terms of
pro-cotangent spaces.

\sssec{}

Let $\bC$ be a category. We shall say that an object of $\on{Pro}(\bC)$
is $\aleph_0$ if it can be presented as an inverse limit over a category
equivalent to $\BN$ as a poset.

\sssec{}

Let $\CX$ be an object of $^{\leq n}\!\dgindSch$.
We shall denote by $\CX_{n+1}$ its
canonical extension to an object of $^{\leq n+1}\!\dgindSch$, i.e.,
$$\CX_{n+1}:={}^{^{\leq n+1}\!L}\!\on{LKE}(\CX).$$

\begin{prop} \label{p:crit for aleph 0}
An object $\CX\in {}^{\leq n}\!\dgindSch$ is $\aleph_0$ if and only if the following
two conditions hold:

\medskip

\noindent \emph{(a)} The classical indscheme $^{cl}\CX$ is $\aleph_0$.

\medskip

\noindent \emph{(b)} The following equivalent conditions hold:

\smallskip

\emph{(i)} There exists a cofinal family of closed embeddings $x:Z\to {}^{cl}\CX$, 
where $Z\in \Sch_{\on{qsep-qc}}$, such that the object 
$^{\geq -n-1}(T^*_x\CX_{n+1})\in \on{Pro}(\QCoh(Z)^{\geq -n-1,\leq 0})$
is $\aleph_0$. 

\smallskip

\emph{(ii)} Same as \emph{(i)} but for any map 
$x:Z\to {}^{cl}\CX$ (i.e., not necessarily a closed embedding). 

\smallskip

\emph{(iii)} Same as \emph{(ii)}, but with $Z$ required to be affine.

\end{prop}

\sssec{Proof of the equivalence of (i), (ii) and (iii)} \label{sss:equiv aleph 0}

The implication (ii) $\Rightarrow$ (i) is tautological. The implication (i) $\Rightarrow$ (ii) follows
from the fact that the formation of $^{\geq -n-1}(T^*_x\CX_{n+1})$ is compatible with pull-backs,
i.e., Condition (B). The implication (ii) $\Rightarrow$ (iii) is 
again tautological. The implication (iii) $\Rightarrow$ (ii) follows from the next lemma:

\begin{lem} \label{l:t on Pro}
The equivalence of \lemref{l:gluing pro} for $\on{Pro}(\QCoh(-)^{\geq -n,\leq 0})$
preserves the corresponding $\aleph_0$ subcategories. 
\end{lem}

\begin{proof}

It is easy to see that it is enough to prove the lemma for $\on{Pro}(\QCoh(-))$
instead of $\on{Pro}(\QCoh(-)^{\geq -n,\leq 0})$.

\medskip

By induction, the assertion reduces to the following statement: let $Z=Z_1\cup Z_2$
be a covering of $Z$ be two Zariski open subsets.
Let $F\in \on{Pro}(\QCoh(Z))$ be an object such that
$$\on{Pro}(j_1^*)(F)\in \on{Pro}(\QCoh(Z_1)) \text{ and }
\on{Pro}(j_2^*)(F)\in \on{Pro}(\QCoh(Z_2))$$
are $\aleph_0$. Then $F$ is $\aleph_0$.

\medskip

It is easy to see that if $F'\to F''\to F'''$ is an exact triangle in $\on{Pro}(\QCoh(S))$, then the condition of being 
$\aleph_0$ has the ``2 out of 3" property. Considering the exact triangle
$$F\to \on{Pro}(j_1{}_*)\circ \on{Pro}(j_1{}^*)(F)\to \on{Cone}\left(F\to \on{Pro}(j_1{}_*)\circ \on{Pro}(j_1^*)(F)\right),$$
we obtain that it is sufficient to show that $\on{Cone}\left(F\to \on{Pro}(j_1{}_*)\circ \on{Pro}(j_1^*)(F)\right)$ is
$\aleph_0$.

\medskip

However, $\on{Cone}\left(F\to \on{Pro}(j_1{}_*)\circ \on{Pro}(j_1^*)(F)\right)$ is supported on a Zariski-closed
subset contained in $Z_2$ and isomorphic to
$$\on{Cone}\left(\on{Pro}(j_2^*)(F)\to \on{Pro}(j_{12,2}{}_*)\circ \on{Pro}(j_{12,2}^*)(F)\right),$$
(where $j_{12,2}$ denotes the open embedding $Z_1\cap Z_2\hookrightarrow Z_2$), which is $\aleph_0$
by the ``2 out of 3" property. 

\end{proof} 

\medskip

This finishes the proof of the equivalence of properties (i), (ii) and (iii) in \propref{p:crit for aleph 0}.

\qed

\sssec{Proof of the ``only if" direction}  Suppose $\CX$ is $\aleph_0$. Fix its presentation as in \eqref{e:indscheme as a colimit},
where the index set $\sA$ is equivalent to $\BN$. For $\alpha\in \sA$ (resp., for an arrow $\alpha_1\to \alpha_2$)
let $f_\alpha$ (resp., $f_{\alpha_1,\alpha_2}$)
denote the corresponding closed embedding $f_\alpha:X_\alpha\to \CX$ (resp., $X_{\alpha_1}\to X_{\alpha_2}$).

\medskip

For a quasi-separated and quasi-compact $Z\in {}^{\leq n}\!\dgSch$ equipped with a map $x:Z\to\CX$, let
$\alpha_0$ be an index such that $x$ factors through a map $x_{\alpha_0}:Z\to X_{\alpha_0}$. 
By \eqref{e:expl cotangent indsch}, we have:
$$^{\geq -n-1}(T^*_x\CX_{n+1})\simeq \underset{\alpha\in \sA_{\alpha_0/}}{``lim"}\, {}^{\geq -n-1}(T^*_{x_{\alpha_0}\circ f_{\alpha_0,\alpha}}X_\alpha),$$
and the category of indices is explicitly equivalent to $\BN$. 

\qed

\sssec{Proof of the ``if" direction}  

First, we observe that the ``2 out of 3" property of an object of $\QCoh(Z)^{\geq -n-1,\leq 0}$ of being
$\aleph_0$ implies that if conditions (i), (ii) or (iii) hold for $Z\in \Sch_{\on{qsep-qc}}$ equipped with
a map to $^{cl}\CX$, then the same will be true for any $Z\in {}^{\leq n}\!\dgSch_{\on{qsep-qc}}$ 
equipped with a map to $^{\leq n}\CX$. 

\medskip

By induction, we may assume that the truncation $${}^{\leq n-1}\CX:=\CX|_{^{\leq n-1}\!\affdgSch}$$ is 
$\aleph_0$.

\medskip

Fix a presentation of ${}^{\leq n-1}\CX$ as in \eqref{e:indscheme as a colimit}, where the category $A$ is equivalent
to the poset $\BN$. Consider the corresponding category $\sB$ (see Step 3 in the proof of \thmref{t:char by deform}),
mapping to $\sA$ by means of $\phi$. We shall use the following lemma:

\begin{lem}
Let $\phi:\sB\to \sA$ be a co-Cartesian fibration of categories, where $\sA$ is equivalent to $\BN$,
and every fiber admits a cofinal functor from $\BN$. Then $\sB$ also admits a cofinal functor
from $\BN$.
\end{lem}

Hence, by the lemma, it suffices to show that for each $\alpha\in \sA$, the corresponding category 
$\sB_\alpha$ admits a cofinal functor from $\BN$. By construction, the category 
$$\sB_\alpha={}^{\leq n}\!\on{SqZExt}(X_\alpha,f_\alpha)_{\on{closed}\,\on{in}\,\CX}$$
is cofinal in 
$^{\leq n}\!\on{SqZExt}(X_\alpha,f_\alpha)$, and the embedding admits a left adjoint. Therefore, it is enough to
show that the latter admits a cofinal map from $\BN$.

\medskip

We have
$$^{\leq n}\!\on{SqZExt}(X_\alpha,f_\alpha)\simeq \left( (\QCoh(X_\alpha)^{\geq -n,\leq 0})_{\on{Cone}((df_\alpha)^*)[-1]}\right)^{\on{op}},$$
where $(df_\alpha)^*$ is the canonical map in $\on{Pro}(\QCoh(X_\alpha)^{\geq -n-1})$:
$$^{\geq -n-1}(T^*_{f_\alpha}\CX_{n+1})\to {}^{\geq -n-1}(T^*X_{\alpha}).$$
The assertion now follows from the assumption that $^{\geq -n-1}(T^*_{f_\alpha}\CX_{n+1})$ is $\aleph_0$.

\qed

\ssec{The ``locally almost of finite type" condition}

\sssec{}

We shall characterize $\nDG$ indschemes
locally of finite type in terms of their pro-cotangent spaces.

\medskip

As before, let $\CX$ be an object of $^{\leq n}\!\dgindSch$, and set
$$\CX_{n+1}:={}^{^{\leq n+1}\!L}\!\on{LKE}(\CX).$$

\begin{prop} \label{p:char finite type new}
The following conditions are equivalent:

\medskip

\noindent{\em(a)} $\CX$ is locally of finite type.

\medskip

\noindent{\em(b)} $^{cl}\CX$ is locally of finite type and the following equivalent conditions hold:

\smallskip

\emph{(i)} There exists a cofinal family of closed embeddings $x:Z\to {}^{cl}\CX$, 
where $Z\in \Sch_{\on{ft}}$, such that the object $$^{\geq -n-1}(T^*_x\CX_{n+1})\in \on{Pro}(\QCoh(Z)^{\geq -n-1,\leq 0})$$
belongs to $\on{Pro}(\Coh(Z)^{\geq -n-1,\leq 0})$. 

\smallskip

\emph{(ii)} Same as \emph{(i)} but for any map $x:Z\to {}^{cl}\CX$ (i.e., not necessarily a closed embedding).

\smallskip

\emph{(iii)} Same as \emph{(ii)}, but with $Z$ required to be affine.

\end{prop}

\sssec{Proof of the equivalence of (i), (ii), and (iii)}

This is similar to \secref{sss:equiv aleph 0}, using the following interpretation of 
$$\on{Pro}(\Coh(Z)^{\geq -n,\leq 0})\subset \on{Pro}(\QCoh(Z)^{\geq -n,\leq 0})$$
for $Z\in \dgSch_{\on{aft}}$: 

\begin{lem} \label{l:char coh}
For a Noetherian scheme $Z$, an object $F\in \on{Pro}(\QCoh(Z)^{\geq -n,\leq 0})$ belongs to 
$\on{Pro}(\Coh(Z)^{\geq -n,\leq 0})$ if and only if, when viewed as a functor
$$F:\QCoh(Z)^{\geq -n,\leq 0}\to \inftygroup,$$
it commutes with filtered colimits.
\end{lem}

\sssec{Proof of \propref{p:char finite type new}}

The implication (a) $\Rightarrow$ (b) follows using \lemref{l:char coh} from the fact that an
object of $^{\leq n+1}\!\inftydgprestack$, which is locally of finite type, takes limits
in $^{\leq n+1}\!\affdgSch$ to colimits, see \cite[Corollary 1.3.8]{Stacks}.

\medskip

Let us show that (b) implies (a). By induction, we can assume that ${}^{\leq n-1}\CX:=\CX|_{^{\leq n}\!\affdgSch}$
is locally of finite type. 

\medskip

We claim now that conditions (i), (ii) and (iii) of \propref{p:char finite type new}(b) hold
for any $Z\in {}^{\leq n}\dgSch_{\on{ft}}$ mapping to ${}^{\leq n}\CX$ (and not just classical schemes).  This
follows from the next observation:

\begin{lem} \label{l:coh via restr}
Let $Z$ be an object of $\affdgSch_{\on{aft}}$, and $F\in \on{Pro}(\QCoh(Z)^{\geq -n,\leq 0})$. Then $F$ belongs to 
$\on{Pro}(\Coh(Z)^{\geq -n,\leq 0})$ if and only if
its restriction to $^{cl}\!Z$ does.
\end{lem}

\begin{rem}
The assertion of \lemref{l:coh via restr} is valid, with the same proof, when we replace the categories 
$\on{Pro}(\QCoh(Z)^{\geq -n,\leq 0})$ and $\on{Pro}(\Coh(Z)^{\geq -n,\leq 0})$ by the
categories $\QCoh(Z)^{\geq -n,\leq 0}$ and $\Coh(Z)^{\geq -n,\leq 0}$, respectively. 
\end{rem}

\begin{proof}

The property of commutation with filtered colimits is enough to check on
$\QCoh(Z)^\heartsuit$, and direct image defines an equivalence
$\QCoh({}^{cl}\!Z)^\heartsuit\simeq \QCoh(Z)^\heartsuit$.

\end{proof}

\medskip

The rest of the proof of \propref{p:char finite type new}
is the same as that of \propref{p:presentation of indschemes lft strong}
in \secref{sss:proof of presentation lft strong}. 

\section{Formal completions}

\ssec{The setting}

\sssec{}  \label{sss:setting for compl general}

In this section we will study the following situation. Let $\CX$ be an object of
$\inftydgprestack$, and let $\CY$ be an object of 
$$^{cl,red}\!\inftydgprestack:=\on{Funct}(({}^{red}\!\affSch)^{\on{op}},\inftygroup),$$
where $^{red}\!\affSch$ denotes the category of classical reduced affine
schemes. Let $\CY\to {}^{cl,red}\CX$ be a map, where $^{cl,red}\CX:=\CX|_{{}^{red}\!\affSch}$.

\medskip

\begin{defn} 
By the formal completion of $\CX$ along
$\CY$, denoted $\CX^{\wedge}_\CY$, we shall mean the object of $\inftydgprestack$
equal to the fiber product
$$
\CD
\CX^{\wedge}_\CY   @>>>   \CX \\
@VVV    @VVV  \\
\on{RKE}_{({}^{red}\!\affSch)^{\on{op}}\hookrightarrow (\affdgSch)^{\on{op}}}(\CY)   @>>>  
\on{RKE}_{({}^{red}\!\affSch)^{\on{op}}\hookrightarrow (\affdgSch)^{\on{op}}}({}^{cl,red}\CX).
\endCD
$$
\end{defn}

\medskip

In plain terms for $S\in \affdgSch$, we set $\Maps(S,\CX^{\wedge}_\CY)$ to be the groupoid
consisting of pairs $(x,y)$, where $x:S\to \CX$, and $y$ is a lift of the map
$$x|_{{}^{cl,red}\!S}:{}^{cl,red}\!S\to {}^{cl,red}\CX$$
to a map $y:{}^{cl,red}\!S \to \CY$.

\sssec{}  \label{sss:remarks on compl}

Several remarks are in order:

\medskip

\noindent(i) If $\CX$ is convergent, then so is $\CX^{\wedge}_\CY$. 

\medskip

\noindent(ii) $\CX$ and $\CX^{\wedge}_\CY$ have ``the same deformation theory." In particular,
if $\CX$ satisfies Conditions (A), (B) or (C), then so does $\CX^{\wedge}_\CY$, and 
for any $x:S\to \CX^{\wedge}_\CY$, the map
$$T^*_x\CX\to T^*_x\CX^{\wedge}_\CY$$
is an isomorphism of functors out of $\QCoh(S)^{\leq 0}$. 

\medskip

\noindent(iii) The formation of $\CX^{\wedge}_\CY$ is compatible with filtered colimits
in the sense that for a filtered category $\sA$ and functors 
$$\sA\to \inftydgprestack:\alpha\mapsto \CX_\alpha \text{ and }
\sA\to {}^{cl,red}\!\inftydgprestack:\alpha\mapsto \CY_\alpha,$$
and a natural transformation $\CY_\alpha\to {}^{cl,red}\CX_\alpha$, the resulting
map
$$\underset{\alpha\in \sA}{colim}\, (\CX_\alpha)^{\wedge}_{\CY_\alpha}\to \CX^{\wedge}_\CY$$
is an isomorphism, where
$$\CX:=\underset{\alpha\in \sA}{colim}\, \CX_\alpha \text{ and }
\CY:=\underset{\alpha\in \sA}{colim}\, \CY_\alpha.$$

\medskip

\noindent(iv) For a map $\CX'\to \CX$ in $\on{PreStk}$, let $\CY':=\CY\underset{^{cl,red}\CX}\times {}^{cl,red}\CX'$. Then
$$\CX'{}^{\wedge}_{\CY'}\simeq \CX^{\wedge}_\CY\underset{\CX}\times \CX'.$$

\sssec{}  \label{sss:de Rham}

When defining formal completions, we can take $\CY\to \CX|_{{}^{red}\!\affSch}$ to be an arbitrary
map. 

\medskip

For example, taking $\CX=\on{pt}$, we obtain an object of $\inftydgprestack$ isomorphic to
$$\on{RKE}_{({}^{red}\!\affSch)^{\on{op}}\hookrightarrow (\affdgSch)^{\on{op}}}(\CY).$$

\medskip

The latter object is otherwise known as the ``de Rham space of $\CY$" and is denoted $\CY_{\on{dR}}$.

\sssec{}
Given a map of prestacks $\CY\rightarrow \CX$, let $\CX^{\wedge}_\CY$ denote the formal completion of $\CX$ along $^{cl,red}\CY \rightarrow {}^{cl,red}\CX$.  We can express $\CX^{\wedge}_\CY$ in terms of the de Rham spaces of $\CX$ and $\CY$; namely,
$$ \CX^{\wedge}_\CY \simeq \CX \underset{\CX_{\on{dR}}}{\times} \CY_{\on{dR}} .$$

\ssec{Formal completions along monomorphisms}  \label{ss:monomorph}

\sssec{}  \label{sss:monomorph}

Let us now assume that the map $\CY\to {}^{cl,red}\CX$ is a monomorphism. I.e.,
for $S\in {}^{red}\!\affSch$ and a map $S\to {}^{cl,red}\CX$, if there exists
a lifting $S\to \CY$, then the space of such liftings is contractible.

\medskip

Note that in this case the map $\CX^{\wedge}_\CY\to \CX$ is also a monomorphism. In particular,
if $\CZ_i\to \CX$, $i=1,2$ are maps in $\on{PreStk}$ that factor
through $\CX^{\wedge}_\CY$, then the map
\begin{equation} \label{c:Cart over compl}
\CZ_1\underset{\CX^{\wedge}_\CY}\times \CZ_2\to \CZ_1\underset{\CX}\times \CZ_2
\end{equation}
is an isomorphism.

\sssec{}  \label{sss:closed emb into compl}

The above observation implies that if $f:\CZ\to \CX^\wedge_\CY$ is a map such that
the composition $\CZ\to \CX^\wedge_\CY\to \CX$ is a closed embedding, then the original
map $f$ is a closed embedding. 

\begin{rem}
Note that the converse to the above statement is not true: consider $\CX:=\BA^1_{\on{dR}}$, and $\CY=\on{pt}$.
We have $^{cl,red}\CX=\BA^1|_{^{red}\!\affSch}$, and we let $\CY\to {}^{cl,red}\CX$ be the map corresponding 
to $\{0\}\in \BA^1$. Then $\CX^\wedge_\CY=\on{pt}$. The tautological map $\CX^\wedge_\CY\to \CX$ is now
$$\on{pt}\to \BA^1_{\on{dR}},$$
and it is not a closed embedding: indeed, its base change with respect to $\BA^1\to \BA^1_{\on{dR}}$ yields
$(\BA^1)^{\wedge}_{\{0\}}$ which is not a closed subscheme of $\BA^1$. 
\end{rem}

\sssec{}
We now consider the case of open embeddings.

\begin{prop}\label{p:open embedding deRham}
Suppose that $\CY\to {}^{cl,red}\CX$ is an open embedding.  Then the corresponding map $\CY_{\on{dR}} \to \CX_{\on{dR}}$ is an open embedding in $\on{PreStk}$.
\end{prop}
\begin{proof}
For any morphism $S \to \CX_{\on{dR}}$ given by $^{cl, red} S \to \CX$, we have a Cartesian square
$$
\xymatrix{
S^{\wedge}_U \ar[r]\ar[d] & S \ar[d] \\
\CY_{\on{dR}} \ar[r] & \CX_{\on{dR}}
,}
$$
where $U = \ ^{cl, red} S \underset{^{cl, red}\CX}{\times} \CY$.  Now, since $U \to\  ^{cl, red} S$ is an open embedding, we have that $S^{\wedge}_U \to S$ is also an open embedding, as desired.
\end{proof}

\begin{cor}\label{c:completion along open embedding}
Let $\CY' \to \CY$ be an open embedding.  For any map $\CY \to \CZ$, the induced map
$$ \CZ^{\wedge}_{\CY'} \to \CZ^{\wedge}_{\CY} $$
is also an open embedding.
\end{cor}
\begin{proof}
Since $\CZ^{\wedge}_{\CY} \simeq \CZ \underset{\CZ_{\on{dR}}}{\times} \CY_{\on{dR}}$, we have a diagram
$$
\xymatrix{
\CZ^{\wedge}_{\CY'} \ar[r]\ar[d] & \CZ^{\wedge}_{\CY} \ar[r]\ar[d] & \CZ \ar[d]\\
\CY'_{\on{dR}} \ar[r] & \CY_{\on{dR}} \ar[r] & \CZ_{\on{dR}}
}
$$
in which all squares are Cartesian.  The result now follows from \propref{p:open embedding deRham}.
\end{proof}

\sssec{}

We would like to consider descent for $\CX^{\wedge}_\CY$. This is not
completely straightforward since the restriction of the fppf topology to $^{red}\!\affSch$
does not make much sense. For this reason, we now restrict to the Zariski, Nisnevich, or \'etale topologies (which do restrict to $^{red}\!\affSch$ ).  In what follows, let $\tau$ denote the Zariski, Nisnevich, or \'etale topology.\footnote{In a previous version of this paper, we incorrectly asserted that \lemref{l:compl and tau sheafification} below was valid for the flat topology.  We thank P. Pstragowski for pointing out the error.}

\begin{prop}\label{p:completion of tau stacks}
Let $\CZ \in \on{PreStk}$ be a prestack satisfying $\tau$-descent, and let $\CY \to\  ^{cl, red}\CZ$ be a morphism in $^{cl,red}\!\inftydgprestack$, with $\CY$ also satisfying $\tau$-descent.  Then the formal completion $\CZ^{\wedge}_{\CY}$ satisfies $\tau$-descent.
\end{prop}
\begin{proof}
Follows from the fact that if $S \to T$ is an open embedding (resp. Nisnevich, \'etale morphism) in $\affSch$, then
$ ^{cl,red} S \simeq S \underset{T}{\times} (^{cl,red} T)$ and the morphism $^{cl,red} S  \to\  ^{cl,red} T $ is also an open embedding (resp. Nisnevich, \'etale morphism).
\end{proof}

\sssec{}
In what follows, let $L_{\tau}$ denote sheafification with respect to the $\tau$-topology.
Let 
\begin{equation} \label{e:Cart diag for compl and descent}
\CD
\CY  @>>>  ^{cl,red}\CX \\
@VVV  @VVV  \\
\CY'  @>>>  ^{cl,red}(L_{\tau}(\CX))
\endCD
\end{equation}
be a Cartesian diagram in $^{cl,red}\!\inftydgprestack$ in which the horizontal arrows are
monomorphisms and $\CY'$ satisfies $\tau$-descent.  By \propref{p:completion of tau stacks}, the formal completion $L_{\tau}(\CX)^{\wedge}_{\CY'}$ satisfies $\tau$-descent.  Therefore, we have a natural map
\begin{equation}\label{e:tau desc compl comparison}
L_{\tau}(\CX^{\wedge}_\CY)\to (L_{\tau}(\CX))^{\wedge}_{\CY'}
\end{equation}

\begin{lem} \label{l:compl and tau sheafification}
Under the above circumstances, the map \eqref{e:tau desc compl comparison}
is an isomorphism.
\end{lem}
\begin{proof}

Recall that the sheafification functor $L_{\tau}$ maps monomorphisms into
monomorphisms. Therefore both maps
$$L_{\tau}(\CX^{\wedge}_\CY)\to L_{\tau}(\CX) \text{ and } (L_{\tau}(\CX))^{\wedge}_{\CY'}\to L_{\tau}(\CX)$$
are monomorphisms. Hence, the map in the lemma is a monomorphism as well.
It requires to see that it is essentially surjective.

\medskip

Thus, let $x:S\to L_{\tau}(\CX)$ be a map that factors through $(L_{\tau}(\CX))^{\wedge}_{\CY'}$.
We wish to show that it factors through $L_{\tau}(\CX^{\wedge}_\CY)$ as well. Let
$\pi:\wt{S}\to S$ be a $\tau$-cover, such that $x\circ \pi$ lifts to a map
$\wt{x}:\wt{S}\to \CX$. It suffices to show that $\wt{x}|_{^{cl,red}\!S}$ factors through $\CY$. 
However, $\wt{x}|_{^{cl,red}\!S}$ factors through
$$\CY'\underset{^{cl,red}(L(\CX))}\times ^{cl,red}\CX,$$
by construction, and the required factorization follows from the fact that
the diagam \eqref{e:Cart diag for compl and descent} is Cartesian.

\end{proof}

\sssec{}   \label{sss:setting for compl closed}

From now on, we will assume that the map $\CY\to {}^{cl,red}\CX$
is a closed embedding. I.e., for $S\in {}^{red}\!\affSch$ and a map $S\to {}^{cl,red}\CX$,
the fiber product 
$$S\underset{^{cl,red}\CX}\times \CY,$$
taken in $^{cl,red}\!\inftydgprestack$, is representable by a (reduced)
closed subscheme of $S$.

\ssec{Formal completions of DG indshemes}

The next proposition shows that the procedure of formal completion is a way
of generating DG indschemes:

\begin{prop} \label{p:formal compl is indscheme}
Suppose that in the setting of \secref{sss:setting for compl closed}, $\CX$ is a DG indscheme.
Then the formal completion $\CX^{\wedge}_\CY$ is also a DG indscheme.
\end{prop}

We shall give two proofs.

\begin{proof}(an overkill)

We shall prove the proposition by applying \thmref{t:char by deform}. We note that Conditions (A), (B) and (C)
hold for $\CX^{\wedge}_\CY$ because they do for $\CX$, see \secref{sss:remarks on compl}(ii) above. Hence,
it remains to show that $^{cl}(\CX^{\wedge}_\CY)$ is a classical indscheme. However, this is obvious, as the latter
is the colimit 
$$\underset{Z_{cl}\to {}^{cl}\CX}{colim}\, Z_{cl},$$
taken over the (filtered!) category of closed embeddings that at the reduced level factor through $\CY$.

\end{proof}

Note that using \propref{p:char finite type new} and \secref{sss:remarks on compl}(ii), the above argument also gives:

\begin{cor}   \label{c:formal lft}
If $\CX$ is locally almost of finite type, then so is $\CX^{\wedge}_\CY$.
\end{cor}

\sssec{}

The second proof of \propref{p:formal compl is indscheme}
comes along with an explicit description of $\CX^{\wedge}_\CY$ as a colimit
of DG schemes: 

\medskip

For $\CX\in \dgindSch$, consider the full subcategory
$$(\dgSch_{\on{qsep-qc}})_{\on{closed}\,\on{in}\,\CX}\underset{\dgSch_{/\CX}}\times \dgSch_{/\CX^{\wedge}_\CY}.$$
I.e., it consits of those closed embedding $Z\to \CX$, which factor through $\CX^{\wedge}_\CY$. 
Note that by \secref{sss:closed emb into compl}, for any 
$$Z\in (\dgSch_{\on{qsep-qc}})_{\on{closed}\,\on{in}\,\CX}\underset{\dgSch_{/\CX}}\times \dgSch_{/\CX^{\wedge}_\CY},$$
the resulting map $Z\to \CX^{\wedge}_\CY$ is a closed embedding. 

\begin{prop}  \label{p:formal compl is indscheme expl}
The category $(\dgSch_{\on{qsep-qc}})_{\on{closed}\,\on{in}\,\CX}\underset{\dgSch_{/\CX}}\times \dgSch_{/\CX^{\wedge}_\CY}$  is filtered, and the map
$$\underset{Z\in (\dgSch_{\on{qsep-qc}})_{\on{closed}\,\on{in}\,\CX}\underset{\dgSch_{/\CX}}\times \dgSch_{/\CX^{\wedge}_\CY}}{colim}\, Z\to \CX^{\wedge}_\CY,$$
is an isomorphism, where the colimit is taken in $\inftydgprestack$. 
\end{prop}

\begin{proof}

It suffices to show that for $S\in \affdgSch$ and a map $S\to \CX$
that factors through $\CY$ at the reduced classical level, the full subcategory of 
$$\dgSch_{S/,\on{closed}\,\on{in}\,\CX}$$ consisting of 
$$S\to Z\to \CX,\quad Z\in (\dgSch_{\on{qsep-qc}})_{\on{closed}\,\on{in}\,\CX}\underset{\dgSch_{/\CX}}\times \dgSch_{/\CX^{\wedge}_\CY}$$
contains finite colimits.

\medskip

The proof follows from the description of finite colimits in $(\dgSch)_{S/,\on{closed}\,\on{in}\,\CX}$,
given in the proof of \propref{p:factorization coproducts for indscheme}.

\end{proof}

\begin{rem}
It is not difficult to see that the category 
$$(\dgSch_{\on{qsep-qc}})_{\on{closed}\,\on{in}\,\CX}\underset{\dgSch_{/\CX}}\times \dgSch_{/\CX^{\wedge}_\CY}$$ used in the above proof
is the same as $(\dgSch_{\on{qsep-qc}})_{\on{closed}\,\on{in}\,\CX^\wedge_\CY}$, i.e., the 
assertion on \secref{sss:closed emb into compl} is ``if and only if" for $\CX$ a DG indscheme
and $\CZ=Z\in \dgSch_{\on{qsep-qc}}$. 

\medskip

Indeed, let $\ol{Z}$ denote the closure of the image of
$Z$ in $\CX$. It is enough to show that the map
$Z\underset{\CX}\times \ol{Z}\to \ol{Z}$ is a closed
embedding. However, since $\ol{Z}\to \CX$ also factors through $\CX^\wedge_\CY$, the map
$$Z\underset{\CX^\wedge_\CY}\times \ol{Z}\to Z\underset{\CX}\times \ol{Z}$$
is an isomorphism, and the map
$$Z\underset{\CX^\wedge_\CY}\times \ol{Z}\to \ol{Z},$$
being a base change of $Z\to \CX^\wedge_\CY$, is a closed embedding, by assumption.
\end{rem}

\sssec{}

Note also that if $\CX$ is written as in \eqref{e:gen indscheme as a colimit}, then
if we set $Y_\alpha:=\CY\cap {}^{cl,red}\!X_\alpha$, by \secref{sss:remarks on compl} (iii) and (iv),
we have:
$$\CX^{\wedge}_\CY\simeq \underset{\alpha}{colim}\, (X_\alpha)^\wedge_{Y_\alpha},$$
where the colimit is taken in $\inftydgprestack$. 

\ssec{Formal (DG) schemes}

Let us recall the following definition: 

\begin{defn}
A classical indscheme $\CX_{cl}$ is called a formal scheme if $^{red}(\CX_{cl})$ is a scheme.
\footnote{Recall that we denote by $\CY\mapsto {}^{red}\CY$ the 
functor $^{cl}\!\inftydgprestack\to {}^{cl,red}\!\inftydgprestack$ corresponding to restriction along
$^{red}\!\affSch\to \affSch$.}
\end{defn}

In the derived setting, we give the following one:

\begin{defn}
A DG indscheme $\CX$ is called a formal DG scheme if the underlying classical
indscheme $^{cl}\CX$ is formal. 
\end{defn}

We have, tautologically:

\begin{lem}  \label{l:formal compl formal ind}
In the situation of \propref{p:formal compl is indscheme}, if $\CY$ is a scheme,
then $\CX^{\wedge}_\CY$ is a formal DG scheme.
\end{lem}

\ssec{Formal completions of DG schemes}  \label{ss:formal completion of a scheme}

For the rest of this section we will take $\CX$ to be a DG scheme $X$, and $\CY$ to be a Zariski
closed subset $Y$ of $^{cl,red}\!X$. Consider the corresponding formal completion $X^{\wedge}_Y$. 

\medskip

In this situation, we shall always assume $Y$ is quasi-compact 
and quasi-separated, in order for $X^{\wedge}_Y$ to be a DG indscheme according to our definition.

\sssec{}

First, we have:

\begin{prop}\label{p:formal completion of scheme}
$X^{\wedge}_Y$ is a DG indscheme.
\end{prop} 

We note that \propref{p:formal completion of scheme} is not, strictly speaking a consequence
of \propref{p:formal compl is indscheme}, since if $X$ fails to be quasi-separated and quasi-compact,
then it is not a DG indscheme in our definition. However, it is easy to see that either of the first two
proofs of \propref{p:formal compl is indscheme} applies in this case as well.

\medskip

We also note that $^{cl,red}(X^{\wedge}_Y)\simeq Y$. Hence, we obtain:

\begin{cor} \label{c:formal compl formal}
$X^{\wedge}_Y$ is a formal DG scheme.
\end{cor}

\sssec{}

In the present situation, we can slightly improve the presentation of $X^{\wedge}_Y$
given by \propref{p:formal compl is indscheme expl}:

\begin{prop}  \label{p:expl formal compl}
As an object of $\inftydgprestack$, $X^{\wedge}_Y$ is isomorphic to
$$\underset{Y'\to X}{colim}\, Y',$$
where the colimit is taken over the category of closed embeddings whose
set-theoretic image \emph{is} $Y$.
\end{prop}

\begin{proof}
By \corref{c:cofinality of closed}, we know that $X^{\wedge}_Y$ is isomorphic to
$$\underset{Z\to X}{colim}\, Z,$$
where the colimit is taken over the category of closed embeddings $Z\to X$ whose
image is set-theoretically contained in $Y$.

\medskip

By \lemref{l:cofinal in filtered}, is suffices to show that any such $Z\to X$ can be factored
as $Z\to Y'\to Z$, where $Y'\to Z$ is a closed embedding whose set-theoretic image
is exactly $Y$.

\medskip

Let $Y'_{can}$ be the reduced closed subscheme of $X$ corresponding to $Y$. 

\medskip

Consider the map $^{cl,red}\!Z\to {}^{cl}\!X$. The latter canonically factors as $^{cl,red}\!Z\to Y'_{can}\to {}^{cl}\!X$.
The required
$Y'$ is then given by
$$Z\underset{^{cl,red}\!Z}\sqcup\, Y'_{can}.$$

\end{proof}

\sssec{}

Note, however, that in general $X^{\wedge}_Y$ will fail to be weakly 
$\aleph_0$, even at the classical level. E.g., we take $X=\BA^\infty:=\Spec(k[t_1,t_2,...])$
and $Y=\{0\}\subset \BA^\infty$. 

\medskip

However, $X^{\wedge}_Y$ is weakly $\aleph_0$ under the following additional condition:

\begin{prop}  \label{p:formal aleph 0}
Assume that $Y$ can be represented by a subscheme $Y'$ of $^{cl}\!X$, 
whose ideal is locally finitely generated. Then $X_Y^{\wedge}$ is weakly $\aleph_0$ 
as a DG indscheme.
\end{prop}

\begin{rem}
We expect that 
$X_Y^{\wedge}$ is actually $\aleph_0$ (see \secref{sss:aleph 0} for the distinction 
between the two notions), but we cannot prove it at the moment. However,
we will prove this when $X$ is affine, and for general $X$, ``up to sheafification",
see \propref{p:expl aleph 0}.
\end{rem}

\begin{proof} 
We shall deduce the assertion of the proposition from \propref{p:crit for aleph 0}. 

\medskip

We note that condition (b)
of \propref{p:crit for aleph 0} follows from \secref{sss:remarks on compl}(ii), as it
is satisfied for $X$. 

\medskip

It remains to show that the classical indscheme underlying $X^{\wedge}_Y$ is $\aleph_0$. However, the quasi-compactness
hypothesis in $Y$ and the assumption that the ideal $\CI$ of $Y'$ is locally finitely generated imply that the
subschemes $Y'_n$ given by $\CI^n$ are cofinal among all subschemes of $X$ whose underlying set is $Y$.
\end{proof}

\ssec{Formal completion of the affine line at a point} 

\sssec{}  \label{sss:fin codim}

We continue to study formal completions of the form $X^{\wedge}_Y$, where $X$ is a DG
scheme, and $Y$ is a Zariski closed subset of $^{cl}\!X$, which is quasi-separated and quasi-compact.

\medskip

We will impose the assumption made in \propref{p:formal aleph 0}. Namely, 
will assume that $Y$ can be represented by a subscheme $Y'$ of $^{cl}\!X$, 
whose ideal is locally finitely generated.

\medskip

We will show that in this case, the behavior of $X^{\wedge}_Y$ exhibits some additional
favorable features. 

\sssec{}

First, we shall calculate the most basic example: the formal completion of $\BA^1$ at the point $\{0\}$.
Namely, we have:

\begin{prop} \label{p:formal compl of A1}
The natural map
$$\underset{n}{colim}\, \Spec(k[t]/t^n)\to (\BA^1)^\wedge_{\{0\}},$$
where the colimit is taken in $\on{PreStk}$, is an isomorphism.
\end{prop}

The statement of the proposition is obvious at the level of the underlying classical prestacks,
i.e., when we evaluate both sides on $\affSch\subset \affdgSch$. However, some care is needed
in the derived setting. 

\medskip

The rest of this subsection is devoted to the proof of this proposition.



\sssec{Proof of \propref{p:formal compl of A1}, Step 1} \hfill   \label{sss:basic calc step 1}

\medskip

\noindent Both sides of the proposition are a priori functors $(\affdgSch)^{\on{op}}\to \inftygroup$. However, we claim
that they both, along with the map between them, naturally upgrade to functors 
$$(\affdgSch)^{\on{op}}\to \inftyPicgroup,$$
where $\inftyPicgroup$ is the category of $\infty$-Picard groupoids, i.e., connective spectra.

\medskip

Consider first the functor $\Maps_{\on{PreStk}}(-,\BA^1):(\affdgSch)^{\on{op}}\to \inftygroup$ represented by $\BA^1$. We claim that 
it naturally upgrades to a functor
$$\CMaps_{\on{PreStk}}(-,\BA^1):(\affdgSch)^{\on{op}}\to \inftyPicgroup.$$
This comes from the structure on $\BA^1$ of abelian group object in
$$\affSch\subset \affdgSch\subset \on{PreStk}.$$

\medskip

Consider now the object
$$\underset{n}{colim}\, \Spec(k[t]/t^n)\in {}^{cl}\!\on{PreStk},$$
where the colimit in the above formula is taken in $^{cl}\!\on{PreStk}$.
The binomial formula endows the above object with a structure of abelian group object in $^{cl}\!\on{PreStk}$.

\medskip

Consider the object
$$\on{nilp}:=\on{LKE}_{(\affSch)^{\on{op}}\to (\affdgSch)^{\on{op}}}(\underset{n}{colim}\, \Spec(k[t]/t^n))\in \on{PreStk}.$$
It equals 
$$\underset{n}{colim}\, \on{nilp}_n,$$
where the colimit is now taken in $\on{PreStk}$, and where 
$$\on{nilp}_n(S)=\Maps_{\affdgSch}(S,\Spec(k[t]/t^n)),\quad S\in \affdgSch.$$

\medskip

By the functoriality of $\on{LKE}_{(\affSch)^{\on{op}}\to (\affdgSch)^{\on{op}}}$, and since the forgetful functor
$$\inftyPicgroup\to \inftygroup$$
commutes with filtered colimits, we obtain that $\on{nilp}$ canonically lifts to a functor
$${\mathcal Nilp}:(\affdgSch)^{\on{op}}\to \inftyPicgroup.$$ 

\medskip

The same construction shows that the map of functors
$$\on{nilp}\to \Maps_{\on{PreStk}}(-,\BA^1)$$
naturally upgrades to a map of functors with values in $\inftyPicgroup$ 
$${\mathcal Nilp}\to \CMaps_{\on{PreStk}}(-,\BA^1).$$

\medskip

Consider now the functor $\Maps_{\on{PreStk}}\left(-,(\BA^1)^\wedge_{\{0\}}\right)$. 
Since $(\BA^1)^\wedge_{\{0\}}\hookrightarrow \BA^1$ is a monomorphism and gives rise to subgroups
at the level of $\pi_0$, we obtain that this functor also naturally upgrades to a functor
$$\CMaps_{\on{PreStk}}\left(-,(\BA^1)^\wedge_{\{0\}}\right):(\affdgSch)^{\on{op}}\to \inftyPicgroup,$$
and the natural transformation
${\mathcal Nilp}\to \CMaps_{\affdgSch}(-,\BA^1)$ factors canonically as
$${\mathcal Nilp}\to \CMaps_{\on{PreStk}}\left(-,(\BA^1)^\wedge_{\{0\}}\right)\to
\CMaps_{\affdgSch}(-,\BA^1).$$

\sssec{Step 2} \hfill

\medskip

\noindent To prove the proposition, we need to show that for $S=\Spec(A)\in \affdgSch$, the map in $\inftyPicgroup$
$${\mathcal Nilp}(A) \to A$$
is an isomorphism onto those connected components of $A$ that correspond to nilpotent elements in $\pi_0(A)={}^{cl}\!A$. 
In the above formula, we view a connective algebra $A$ as a connective spectrum, i.e.,
object of $\inftyPicgroup$.

\medskip

Hence, it suffices to show that for a connective commutative DG algebra $A$, the map
$$\pi_i({\mathcal Nilp}(A))\to \pi_i(A)$$
is an isomorphism for $i\geq 1$, and that $\pi_0({\mathcal Nilp}(A))$
maps isomorphically to the set of nilpotent elements in $\pi_0(A)={}^{cl}\!A$. 
Here by $\pi_i$ for $i\geq 1$ we mean the $i$th homotopy group of the corresponding space based at the point $0$.

\sssec{Step 3}

We first consider the case $i\geq 1$. 

\medskip

We regard each $\on{nilp}_n(A)$ as a pointed object of $\inftygroup$. Hence, from the isomorphism
$$\Omega^\infty({\mathcal Nilp}(A))=\on{nilp}(A)\simeq \underset{n}{colim}\, \on{nilp}_n(A)$$
in $\inftygroup_{*/}$, for each $i\geq 1$,
we have an isomorphism of (ordinary) groups:
$$\pi_i({\mathcal Nilp}(A))\simeq \underset{n}{colim}\, \pi_i(\on{nilp}_n(A)).$$

\medskip

Hence, it suffices to show that the map
\begin{equation} \label{e:isom for i}
\underset{n}{colim}\, \pi_i(\on{nilp}_n(A))\to \pi_i(\Omega^\infty(A))
\end{equation}
is an isomorphism. 

\medskip

We have a Cartesian
square in $\dgSch$:
$$
\CD
\Spec(k[t]/t^n) @>>> \BA^1 \\
@VVV @VV{\on{power}\,n}V   \\
\{0\}  @>>>  \BA^1,
\endCD
$$
and the corresponding Cartesian square in $\inftygroup$:
$$
\CD
\on{nilp}_n(A)  @>>>  \Omega^\infty(A)  \\
@VVV   @VV{\on{power}\,n}V  \\
*  @>>> \Omega^\infty(A).
\endCD
$$

Hence, we obtain a long exact sequence of homotopy groups
$$...\pi_{i+1}(\Omega^\infty(A)) \overset{\on{power}\,n}\longrightarrow \pi_{i+1}(\Omega^\infty(A)) 
\to  \pi_{i}(\on{nilp}_n(A))
\to \pi_{i}(\Omega^\infty(A)) \overset{\on{power}\,n}
\longrightarrow \pi_{i}(\Omega^\infty(A))...$$

\medskip

However, for $i\geq 1$ and $n>1$, the map 
$\pi_{i}(\Omega^\infty(A)) \overset{\on{power}\,n}\longrightarrow \pi_{i}(\Omega^\infty(A))$
is zero. Indeed, this follows from the fact for any two connective algebras $A_1$ and $A_2$,
the canonical map 
$$\Omega^{\infty}(A_1)\times \Omega^{\infty}(A_2)\to
\Omega^{\infty}(A_1\otimes A_2)$$
induces a zero map
$$\pi_i(\Omega^{\infty}(A_1))\oplus \pi_i(\Omega^{\infty}(A_2))\to
\pi_i(\Omega^{\infty}(A_1\otimes A_2))$$
for $i\geq 1$. 

\medskip

Hence, every $n$ we have a short exact sequence
$$0\to \pi_{i+1}(\Omega^\infty(A)) \to \pi_{i}(\on{nilp}_n(A))\to \pi_{i}(\Omega^\infty(A)) \to 0.$$ 

Moreover, for $n''\geq n'$, in the diagram 
$$
\CD
\pi_{i+1}(\Omega^\infty(A))   @>>>  \pi_{i}(\on{nilp}_{n'}(A))  @>>>  \pi_{i}(\Omega^\infty(A)) \\
@VVV     @VVV   @VVV   \\
\pi_{i+1}(\Omega^\infty(A))   @>>>  \pi_{i}(\on{nilp}_{n''}(A))  @>>>  \pi_{i}(\Omega^\infty(A))
\endCD
$$
the right vertical map is the identity, whereas the left vertical map corresponds 
to the map $\BA^1\to \BA^1$ given by raising to the power $n''-n'$, and so vanishes for $n''>n'$.

\medskip

This shows that \eqref{e:isom for i} is an isomorphism.

\sssec{Step 4}

The fact that 
$$\pi_0({\mathcal Nilp}(A))\to {}^{cl}\!A$$
is an isomorphism onto the set of nilpotent elements is proved similarly. 

\qed(\propref{p:formal compl of A1})

\ssec{Formal completions along subschemes of finite codimension}

We now return to the case of a general $X$ and $Y$ satisfying the assumption of \secref{sss:fin codim}.

\sssec{}

Assume that the DG scheme $X$ is eventually coconnective. It is natural to ask
whether the same will be true for the DG indscheme $X^{\wedge}_Y$. 

\medskip

Note, however, 
that asking for a DG indscheme to be eventually coconnective (i.e., eventually coconnective
as a stack) is a strong requirement, since it is difficult to satisfy it together with convergence,
see \cite[Sect. 1.2.6]{Stacks}. 

\medskip

However, the answer to the above question turns out to
be affirmative: 

\begin{prop}   \label{p:formal compl coconn}
If $X$ is eventually coconnective, then
$X^{\wedge}_Y$ is eventually coconnective as a DG indscheme. 
\end{prop}

\sssec{} \label{sss:Koszul}

In order to prove \propref{p:formal compl coconn}, we will give a more
explicit description of the formal completion $X^{\wedge}_Y$ in the situation
of \secref{sss:fin codim} when $X$ is affine. This description will be 
handy for the proof of several other assertions in this paper. 

\medskip

Let $X=\Spec(A)$, and let $Y'$ be a closed subscheme of $^{cl}\!X$ whose
ideal is generated by elements $\ol{f_1}$,...,$\ol{f_m}$ in 
$$^{cl}\!A=H^0(A)=\pi_0(\Omega^{\infty}({\mathsf {Sp}}(A))).$$ 

\medskip

Let $f_1,...,f_m$ be points of $\Omega^{\infty}(A)$
that project to the $\ol{f_1}$,...,$\ol{f_m}$. 

\medskip

For an integer $n$, set $A_n:=A[t_{n,1},...,t_{n,m}]$, where the generators $t_{n,i}$ are in degree $-1$, and
$d(t_{n,i})=f_i^n$. 

\medskip

For $n'\leq n''$ we have a natural map $A_{n''}\to A_{n'}$ which is identity on $A$,
and which sends $t_{n'',i}\mapsto f_i^{n''-n'}\cdot t_{n',i}$. We will prove:

\begin{prop}  \label{p:expl pres}
The natural map 
\begin{equation} \label{e:compl as colimit}
\underset{n}{colim}\, \Spec(A_n) \to X^{\wedge}_Y,
\end{equation}
where the colimit is taken in $\inftydgprestack$, 
is an isomorphism. 
\end{prop}

\begin{proof}

The functions $f_1,...,f_m$ define a map 
$$\Spec(A)\to \BA^m,$$
and by definition,
$X^{\wedge}_Y$ maps isomorphically to the fiber product
$$(\BA^m)^\wedge_{\{0\}}\underset{\BA^m}\times \Spec(A).$$

Since fiber products commute with filtered colimits, from \propref{p:formal compl of A1}, we obtain that
$$(\BA^m)^\wedge_{\{0\}}\underset{\BA^m}\times \Spec(A)$$
is isomorphic to the colimit over $n$ of
\begin{equation} \label{e:raising n}
(\{0\}\underset{\BA^m}\times \BA^m)\underset{\BA^m}\times \Spec(A),
\end{equation}
where the map $\BA^m\to \BA^m$ is given by raising to the power $n$ along each coordinate. Now,
by definition, the DG scheme in \eqref{e:raising n} is isomorphic to $\Spec(A_n)$, as required.

\end{proof}

\sssec{}   

Let us show how \propref{p:expl pres} implies \propref{p:formal compl coconn}:

\begin{proof}[Proof of \propref{p:formal compl coconn}]

First, note that the assertion is local in the Zariski topology on $X$. Thus, 
we can assume that $X=\Spec(A)$ is affine.

\medskip

Now, the assertion follows from the fact that if $A$ is $l$-coconnective, then 
each of the algebras $A_n$ is $(m+l)$-coconnective, by construction.

\end{proof}

\sssec{}

Here is another corollary of \propref{p:expl pres}:

\begin{prop}  \label{p:expl aleph 0}
The DG indscheme $X^{\wedge}_Y$
can be written as a colimit in $\inftydgstack$
$$\underset{\alpha\in \sA}{colim}\, Y'_\alpha,$$
where $Y'_\alpha\to X$ are closed embeddings with set-theoretic image is equal to $Y$,
and where the category $\sA$ of indices is equivalent to the poset $\BN$.
\end{prop}

\medskip

\begin{rem}
This proposition does not prove that $X^{\wedge}_Y$ is $\aleph_0$, because
the colimit is taken in $\on{Stk}$ and not $\on{PreStk}$.
\end{rem}

\begin{proof}

First, note that \propref{p:expl pres} gives such a presentation if $X$ is affine
(moreover, in this case, the colimit can be taken in $\inftydgprestack$). I.e.,
in this case, $X^{\wedge}_Y$ is $\aleph_0$ as a DG indscheme.

\medskip

Let $S_i$ be a (finite) collection of affine open DG subschemes of $X$ that covers $Y$. For each
$i$, let $\sA_i$ be the corresponding index set (isomorphic to $\BN$) for the formal
completion $(S_i)^{\wedge}_{S_i\cap Y}$. 

\medskip

For $\alpha_i\in \sA_i$ let $Y'_{i,\alpha_i}$ be the
corresponding DG scheme equipped with a closed embedding into $S_i$. Let
$\ol{Y}'_{i,\alpha_i}$ be the closure of its image in $X$, see \secref{sss:closure of image}.

\medskip

For $\alpha:=\{i\mapsto (\alpha_i\in \sA_i)\}$ set $Y'_\alpha$ be the coproduct
of $\ol{Y}'_{i,\alpha_i}$ in $(\dgSch_{\on{qs-qs}})_{\on{closed}\,\on{in}\,X}$.

\medskip

We claim that the family $\alpha\mapsto Y'_\alpha$ has the desired property. Indeed,
it is sufficient to show that for every $i$, the colimit of the family 
$$\alpha\mapsto Y'_\alpha\underset{X}\times S_i$$ is isomorphic to
$(S_i)^{\wedge}_{S_i\cap Y}$. However, this is clear since this colimit is also given by the colimit of the family $i\mapsto Y'_{i,\alpha_i}$.

\end{proof}

\ssec{Classical vs. derived formal completions}  \label{ss:classical vs derived formal}

We shall now show how \propref{p:expl pres} helps answer another natural question
regarding the behavior of $X^{\wedge}_Y$.

\sssec{}

Let $X$ a DG scheme, which is $0$-coconnective (=classical), i.e., the sheafification of a left Kan extension of
a classical scheme. 

\medskip

One can ask whether the DG indscheme $X^{\wedge}_Y$ is also
$0$-coconnective. That is, we consider the classical indscheme $^{cl}(X^{\wedge}_Y)$, and let 
$$\CX:={}^L\!\on{LKE}_{(\affSch)^{\on{op}}\hookrightarrow (\affdgSch)^{\on{op}}}\left({}^{cl}(X^{\wedge}_Y)\right).$$
By adjunction, we obtain a map
\begin{equation} \label{e:classical vs derived}
\CX\to X^{\wedge}_Y,
\end{equation} 
and we wish to know whether it is an isomorphism. 

\medskip

Again, we emphasize that it is a rather strong property for a DG indscheme
(or any convergent stack) to be $0$-coconnective (rather than weakly $0$-coconnective),
see \cite[Remark 1.2.6]{Stacks}.

\medskip

However, the answer to the above question turns out to be affirmative, under
an additional assumption that $X$ be Noetherian (see \cite[Sect. 0.6.9]{IndCoh}
for the notion of Noetherianness in the DG setting):

\begin{prop} \label{p:class vs derived formal}
If $X$ is Noetherian, the map
\eqref{e:classical vs derived} is an isomorphism.
\end{prop}

\begin{rem}\label{r:derived formal as zar sheafification}
In fact, it will be apparent from the proof of \propref{p:class vs derived formal} that the map
$$ ^{\on{Zar}} \CX \to X^{\wedge}_Y $$
is an isomorphism, where $^{\on{Zar}} \CX$ is the \emph{Zariski} sheafification of $\on{LKE}_{(\affSch)^{\on{op}}\hookrightarrow (\affdgSch)^{\on{op}}}\left({}^{cl}(X^{\wedge}_Y)\right)$.
\end{rem}

\sssec{Proof of \propref{p:class vs derived formal}, Step 1}

The assertion readily reduces to the case when $X$ is affine; $X=\Spec(A)$,
where $A$ is a classical $k$-algebra. Let $f_1,...,f_m\in A$ be the generators
of the ideal of some subscheme $Y'\subset X$ whose underlying Zariski-closed 
subset is $Y$. Let $A_n$ be the algebras as in \secref{sss:Koszul}.

\medskip

For each $n$, let $A'_n$ be the classical 
algebra $H^0(A_n)$, so that 
$$\CX\simeq \underset{n}{colim}\, \Spec(A'_n),$$
where the colimit is taken in $\inftydgprestack$. We will show that 
inverse systems $\{A_n\}$ and $\{A'_n\}$ are equivalent, i.e., that the natural map
\begin{equation} \label{e:two systems}
\underset{n\in \BN}{colim}\, \Spec(A'_n) \to \underset{n\in \BN}{colim}\, \Spec(A_n)
\end{equation}
is an isomorphism in $\on{PreStk}$.

\sssec{Proof of \propref{p:class vs derived formal}, Step 2}

We will prove:

\begin{lem}  \label{l:Artin-Rees alg}
For every $n$ there exists $N\geq n$ such that the map $A_{N}\to A_n$ can be factored
as $$A_{N}\to A'_N\to A_n.$$
\end{lem}

\medskip

Let us show how \lemref{l:Artin-Rees alg} implies that \eqref{e:two systems} is an isomorphism. We construct the
sequence $i_1,i_2,...,\subset \BN$ inductively, starting with $i_1=1$. Suppose $i_k$
has been constructed. We take $n:=i_k$, and we let $i_{k+1}$ be the integer $N$ given by \lemref{l:Artin-Rees alg}.

\medskip

We obtain the maps
$$\underset{k\in \BN}{colim}\, \Spec(A'_{i_k})\to \underset{k\in \BN}{colim}\, \Spec(A_{i_k})$$
and 
$$\underset{k\in \BN}{colim}\, \Spec(A_{i_k}) \to \underset{k\in \BN}{colim}\, \Spec(A'_{i_{k+1}})$$
that induce mutually inverse maps in \eqref{e:two systems}. 

\sssec{Proof of \propref{p:class vs derived formal}, Step 3}

We will deduce the assertion of \lemref{l:Artin-Rees alg} from the 
following version of the Artin-Rees lemma:

\begin{lem}  \label{l:actual Artin-Rees}
Let $A$ be a commutative algebra in $\Vect^{\leq 0}$ such that $H^0(A)$
is Noetherian. Let $f_1,....,f_m$ be an $m$-tuple of points in $A$, and 
let $A_n$ have the same meaning as in \secref{sss:Koszul}. Let $M$
be an $A$-module in $\Vect^{\geq -k,\leq 0}$, such that each $H^i(M)$
is finitely generated as a module over $H^0(A)$. Then for every $i>k$, the inverse system
$$n \mapsto H^{-i}(A_n\underset{A}\otimes M)$$
equivalent to zero, i.e., for every $n$ there exists an $N\geq n$, such that the map
$$H^{-i}(A_N\underset{A}\otimes M)\to H^{-i}(A_n\underset{A}\otimes M)$$
is zero.
\end{lem}


We will prove a stronger assertion: for every $n$ there exists an $N\geq n$ such that 
the map
$$A_N\underset{A}\otimes M \to A_n\underset{A}\otimes M$$
factors as 
$$A_N\underset{A}\otimes M \to
\tau^{\geq -k}(A_N\underset{A}\otimes M) \to A_n\underset{A}\otimes M.$$

First, it follows from the Koszul resolution that $A_n\underset{A}\otimes M$ is acyclic in degrees $<-(k+m)$. 
Hence, by induction, it suffices to show that for every $i>k$ and every $n$, there exists an $N\geq n$ such that
the map
$$\tau^{\geq -i}(A_N\underset{A}\otimes M)\to \tau^{\geq -i}(A_n\underset{A}\otimes M)$$
factors as
$$\tau^{\geq -i}(A_N\underset{A}\otimes M)\to \tau^{\geq -i+1}(A_N\underset{A}\otimes M)\to 
\tau^{\geq -i}(A_n\underset{A}\otimes M).$$

Note that the latter assertion is equivalent to the original statement of \lemref{l:actual Artin-Rees}.  

\medskip

We first consider the case $m=1$. Then if $i>k$, the module $H^{-i}(A_n\underset{A}\otimes M)$
is non-zero only for $i=k+1$ and identifies with 
$$H^{-i}(M)^n:=\on{ker}(f^n:H^{-i}(M)\to H^{-i}(M)).$$

By the Noetherian and finite generation hypotheses, 
the submodules $H^{-i}(M)^n \subset H^{-i}(M)$ stabilize. In particular, there exists $N_0$
such that $f^{N_0}$ annihilates all $H^{-i}(M)^n$. Then, for a given $n$, and $N\geq n+N_0$
will do. 

\medskip

For a general $m$, we argue by induction on $m$. Write
$$A_n\underset{A}\otimes M \simeq k\underset{k[t_1,...,t_m]}\otimes M=:M_{m,n}$$
with $t_i$ acting on $M$ by $t_i^n$, so that
$$M_{m,n} \simeq k \underset{k[t_m]}\otimes M_{m-1,n}:=(M_{m-1,n})_{1,n},$$
where $t_m$ acts on $M_{m-1,n}$ by $t_m^n$. Note that we have 
$$\tau^{\geq -i}(M_{m,n})\simeq \tau^{\geq -i}((\tau^{\geq -i}(M_{m-1,n}))_{1,n}).$$

\medskip

By the induction hypothesis, for a given $n$, we can find $N'$ such that the map
$$\tau^{\geq -i}(M_{m-1,N'})\to \tau^{\geq -i}(M_{m-1,n})$$
factors as
$$\tau^{\geq -i}(M_{m-1,N'})\to \tau^{\geq -i+1}(M_{m-1,N'})\to \tau^{\geq -i}(M_{m-1,n}).$$

Hence, the map 
$$\tau^{\geq -i}(M_{m,N'}) \simeq \tau^{\geq -i}((\tau^{\geq -i}(M_{m-1,N'}))_{1,N'}) 
\to \tau^{\geq -i}((\tau^{\geq -i}(M_{m-1,n}))_{1,n}) \simeq \tau^{\geq -i}(M_{m,n})$$
factors as
\begin{multline} \label{e:factor n N'}
\tau^{\geq -i}((\tau^{\geq -i}(M_{m-1,N'}))_{1,N'})\to 
\tau^{\geq -i}((\tau^{\geq -i+1}(M_{m-1,N'}))_{1,N'})\to \\
\to \tau^{\geq -i}((\tau^{\geq -i}(M_{m-1,n}))_{1,N'})
\to \tau^{\geq -i}((\tau^{\geq -i}(M_{m-1,n}))_{1,n}).
\end{multline}
 
Applying the case $m=1$ to $\tau^{\geq -i+1}(M_{m-1,N'})$, we obtain that there exists $N$ such that the map
$$\tau^{\geq -i}((\tau^{\geq -i+1}(M_{m-1,N'}))_{1,N})\to \tau^{\geq -i}((\tau^{\geq -i+1}(M_{m-1,N'}))_{1,N'})$$
factors as
\begin{multline*} 
\tau^{\geq -i}((\tau^{\geq -i+1}(M_{m-1,N'}))_{1,N})\to \tau^{\geq -i+1}((\tau^{\geq -i+1}(M_{m-1,N'}))_{1,N})\to \\
\to \tau^{\geq -i}((\tau^{\geq -i+1}(M_{m-1,N'}))_{1,N'}).
\end{multline*}

Combining with \eqref{e:factor n N'}, we obtain that the composite map
$$\tau^{\geq -i}(M_{m,N}) \simeq \tau^{\geq -i}((\tau^{\geq -i}(M_{m-1,N}))_{1,N})\to 
\tau^{\geq -i}((\tau^{\geq -i}(M_{m-1,n}))_{1,n})\simeq \tau^{\geq -i}(M_{m,n}) $$
factors as
\begin{multline*}  
\tau^{\geq -i}((\tau^{\geq -i}(M_{m-1,N}))_{1,N})\to \tau^{\geq -i+1}((\tau^{\geq -i+1}(M_{m-1,N}))_{1,N}) \to \\
\to \tau^{\geq -i}((\tau^{\geq -i}(M_{m-1,n}))_{1,n}),
\end{multline*}
as required.


\begin{proof}[Proof of \lemref{l:Artin-Rees alg}]

First, it follows from the Koszul resolution that the algebras $A_n$ are a priori $m$-coconnective
(where $m$ is the number of the generators of the ideal). Hence, by descending induction, it
suffices to show the following:

\medskip

For a fixed $n$ and $k>0$, there exists $N>n$ such that the map
$$\tau^{\geq -k}(A_N)\to  \tau^{\geq -k}(A_n)$$
factors as 
$$\tau^{\geq -k}(A_N)\to  \tau^{\geq -(k-1)}(A_N)\to\tau^{\geq -k}(A_n).$$

We claim that this happens whenever the map
$$H^{-k}(A_N)\to H^{-k}(A_n)$$
is zero (the existence of such an $N$ is guaranteed by \lemref{l:actual Artin-Rees}).  

\medskip

Indeed, the algebras $\tau^{\geq -k}(A_n)$ and $\tau^{\geq -k}(A_N)$ are \emph{canonically}
square-zero extensions of $\tau^{\geq -(k-1)}(A_n)$ and $\tau^{\geq -(k-1)}(A_N)$
by the ideals $H^{-k}(A_n)$ and $H^{-k}(A_N)$, respectively, and the map 
$\tau^{\geq -k}(A_N)\to  \tau^{\geq -k}(A_n)$ has a canonical structure of map between
square-zero extensions. 

\medskip

Hence, by \lemref{l:can sq zero},
the datum of the required factorization is equivalent to that of null-homotopy of the map
$H^{-k}(A_N)\to H^{-k}(A_n)$. 

\end{proof}

\sssec{Exponential map}  \label{sss:exp}
Let $\widehat{\BG}_a$ and $\widehat{\BG}_m$ be the formal completions of $\BG_a$ and $\BG_m$ at 0 and 1, respectively.  
These are both formal group schemes.  By \propref{p:class vs derived formal}, we have that $\widehat{\BG}_a$ and $\widehat{\BG}_m$ 
are both $0$-coconnective as prestacks. Hence, the exponential map in $^{cl}\!\on{PreStk}$
$$^{cl}\widehat{\BG}_a\to {}^{cl}\widehat{\BG}_m,$$
defined by the usual fomula, gives rise to a canonical isomorphism in $\on{PreStk}$.
$$ \on{exp}: \widehat{\BG}_a \rightarrow \widehat{\BG}_m.$$
Furthermore, $\on{exp}$ is an isomorphism of $\BE_\infty$-group objects in $\on{PreStk}$, i.e., as functors
$$(\affdgSch)^{\on{op}}\to \inftyPicgroup.$$

\sssec{}

For a connective $k$-akgebra $A$, let ${\mathcal Nilp}(A)$ denote the connective spectrum (i.e., $\infty$-Picard groupoid)
$$\on{ker}\left(A\to {}^{cl,red}\!A\right).$$  

Note that by \propref{p:formal compl of A1}, the above definition of ${\mathcal Nilp}(A)$ agrees with one in
\secref{sss:basic calc step 1}.

\medskip

Let $A^\times$ denote the connective spectrum of invertible elements in the $\BE_\infty$-ring spectrum $A$,
and similarly for $^{cl,red}\!A$. Set
$${\mathcal Unip}(A):=\on{ker}\left(A^\times\to {}^{cl,red}\!A^\times\right).$$ 

We obtain that the exponential map defines an isomorphism
$$\on{exp}: {\mathcal Nilp}(A) \rightarrow {\mathcal Unip}(A)$$
of functors $(\affdgSch)^{\on{op}}\to \inftyPicgroup$.


\section{Quasi-coherent and ind-coherent sheaves on formal completions}

\ssec{Quasi-coherent sheaves on a formal completion}

Let $X$ be a DG scheme, and $Y\to X$ a Zariski closed subset. We shall assume that $Y$ is quasi-separated and quasi-compact. 
Let $U$ be the open DG subscheme of $X$ equal to the complement of $Y$; let $j$ denote the corresponding open embedding. 

\sssec{}

We have a pair of mutually adjoint functors
$$j^*:\QCoh(X)\rightleftarrows \QCoh(U):j_*,$$
which realizes $\QCoh(U)$ as a localization of $\QCoh(X)$.
Note, however, that the functor $j_*$ is not a priori continuous, since $j$ is not necessarily quasi-compact.

\medskip

Let $\QCoh(X)_Y$ denote the full subcategory of $\QCoh(X)$ equal to
$$\on{ker}(j^*:\QCoh(X)\to \QCoh(U)).$$

\medskip

Let $\wh{i}$ denote the canonical map $X^{\wedge}_Y\to X$, and consider the corresponding functor
$$\wh{i}{}^*:\QCoh(X)\to \QCoh(X^{\wedge}_Y).$$

\medskip

We can ask the following questions: 

\smallskip

\noindent(i) Is the composition $\wh{i}{}^*\circ j_*:\QCoh(U)\to \QCoh(X^{\wedge}_Y)$ zero?

\smallskip

\noindent(ii) Does the functor $\wh{i}{}^*$ induce an equivalence $\QCoh(X)_Y\to \QCoh(X^{\wedge}_Y)$?

\medskip

We will answer these questions in the affirmative under an additional hypothesis on the pair $X$ and $Y$.
We learned the corresponding assertion from J.~Lurie. 

\sssec{}

We will impose the assumption of \secref{sss:fin codim}, 
i.e., that $Y$ can be represented
by a closed subscheme $Y'$ of $^{cl}\!X$, whose ideal is locally finitely generated.

\medskip

In this case, the morphism
$j$ is quasi-compact (being an open embedding, it is automatically quasi-separated). In particular,
by \cite[Proposition 2.1.1]{QCoh}, the functor $j_*$ is continuous  and satisfies the base change formula, which immediately
implies that the composition 
$$\wh{i}{}^*\circ j_*:\QCoh(U)\to \QCoh(X^{\wedge}_Y)$$
vanishes.

\medskip

\begin{prop} \label{p:QCoh of compl}
Under the above hypothesis, the composite functor
$$'\wh{i}{}^*:\QCoh(X)_Y\hookrightarrow \QCoh(X)\overset{\wh{i}{}^*}\longrightarrow \QCoh(X^{\wedge}_Y)$$
is an equivalence.
\end{prop}

\sssec{Proof of \propref{p:QCoh of compl}}

The assertion is Zarski-local, so we can assume that $X=\Spec(A)$ is affine. Let
$f_1,...,f_m$ and $A_n$ be as in the proof of \propref{p:expl pres}. 

\medskip

Consider the functor 
$$\QCoh(X^{\wedge}_Y)\to \QCoh(X)$$ given by direct image $\wh{i}_*$ with respect to the morphism 
$\wh{i}$, i.e., the \emph{right} adjoint of the functor $\wh{i}{}^*$. \footnote{Recall that direct image $g_*$, although
in general non-continuous, is defined for any morphism $g:\CY_1\to \CY_2$ in $\on{PreStk}$, by the adjoint
functor theorem.} 

\medskip

\noindent{\it Warning:} The functor $\wh{i}_*$ is not continuous and does not commute with Zariski localization. 

\medskip

We obtain that $\wh{i}_*$ and $\wh{i}^*$ induce a pair of mutually adjoint functors
\begin{equation} \label{e:loc seq}
(\QCoh(X))_{\QCoh(U)}\rightleftarrows \QCoh(X^{\wedge}_Y),
\end{equation}
where $(\QCoh(X))_{\QCoh(U)}$ denotes the localization of $\QCoh(X)$ with respect to
$\QCoh(U)$, and the latter is mapped in by means of $j_*$. To prove the proposition it suffices to show that:

\medskip

\noindent{(a)} The functor $\leftarrow$ in \eqref{e:loc seq} is fully faithful, and 

\smallskip

\noindent{(b)} The functor $\to$ in \eqref{e:loc seq} is conservative.

\medskip

Assertion (a) is equivalent to the functor $\wh{i}_*$ being fully faithful. I.e., we need to show
that the adjunction map $\wh{i}^*\circ \wh{i}_*\to \on{Id}$ is an isomorphism. 

\medskip

Fix an object of $\QCoh(X^{\wedge}_Y)$, thought of as a 
compatible system of $A_n$-modules $\{\CF_{n}\}$; let $\CF$ be its direct image
on $X$. By definition,
$$\CF\simeq\underset{n}{lim}\, \CF_n,$$
where in the right-hand side, the $\CF_n$'s are regarded as $A$-modules. 

\medskip

We need to show that for every 
$n_0$, the map
$$A_{n_0}\underset{A}\otimes \CF\to \CF_{n_0}$$ 
is an isomorphism. 

\medskip

Since $A_{n_0}$ is compact as an $A$-module, we can rewrite the left-hand side
as $\underset{n}{lim}\, \left(A_{n_0}\underset{A}\otimes \CF_n\right)$, and further as
$$\underset{n}{lim}\, \left((A_{n_0}\underset{A}\otimes A_n)\underset{A_n}\otimes \CF_n\right).$$
For $n\geq n_0$ consider the canonical map $A_{n_0}\underset{A}\otimes A_n\to A_0$,
and let $K_n$ denote its kernel. The required assertion follows from the next claim:

\begin{lem}  \label{l:stab ker}
For every $n$ there exists $N\geq n$, such that the map $K_N\to K_n$ is zero
as a map of $A_N$-modules. 
\end{lem}

\begin{proof}

Let $B$ denote the polynomial algebra $k[t_1,...,t_m]$. We have $A_n\simeq A\underset{B}\otimes B_n$,
as commutative $k$-algebras. With respect to this identification, we have
$$K_n\simeq A\underset{B}\otimes K_n^B,$$ 
as $A_n$-modules, where $K^n_B$ denotes the corresponding object for the algebra $B$. Hence, it is enough
to prove the assertion for $A$ replaced by $B$. 

\medskip

To simplify notation, we will only consider the case when $m=1$, i.e., $B=k[t]$. In this case
$$B_n\underset{B}\otimes B_{n_0}\simeq \on{Cone}(t^{n}:k[t]/t^{n_0}\to k[t]/t^{n_0}).$$
When $n\geq n_0$, the map $t^{n}:k[t]/t^{n_0}\to k[t]/t^{n_0}$ is zero, so $K_n\simeq k[t]/t^{n_0}[1]$.
For $n'\geq n$, the corresponding map $K_{n'}\to K_n$ is given by multiplication by $t^{n'-n}$.
Hence, we can take $N=2n$. 

\end{proof}

Let us now prove point (b). Recall the elements $f_1,...,f_m$ of $A$. 
Let $Y_k$ be the closed DG subscheme of $X$ cut out (in the derived sense) by the equations
$f_1,...,f_k$; i.e.,
$$Y_k=\Spec(A[t_1,...,t_k],\,d(t_i)=f_i);$$
let $i_k:Y_k\to X$ denote the corresponding closed embedding.  In particular $^{cl,red}Y_m=Y$. 

\medskip

It suffices to show that if for $\CF\in \QCoh(X)$ we have $i_m^*(\CF)=0$, then $\CF$ belongs to the essential 
image of $j_*$. Taking the cone we can assume that $j^*(\CF)=0$ as well, and
we need to show that $\CF=0$. 

\medskip

By induction on $k$, we may assume that $m=1$. The assumption that $i_1^*(\CF)=0$ means that 
$f_1:\CF\to \CF$ acts invertibly, i.e., $\CF\to (\CF)_{f_1}$ is an isomorphism, where $(\CF)_{f_1}$ denotes the localization
of $\CF$ with respect to $f_1$.  However, $j^*(\CF)=0$ implies $(\CF)_{f_1}=0$.

\qed(\propref{p:QCoh of compl})

\sssec{}

Let us denote by $\be^{\QCoh}$ the tautological embedding
$$\QCoh(X)_Y\to \QCoh(X).$$
We note that it admits a (continuous) right adjoint, denoted $\br^{\QCoh}$, given by
$$\CF\mapsto \on{Cone}(\CF\to j_*\circ j^*(\CF))[-1].$$
The adjoint pair $(\be^{\QCoh},\br^{\QCoh})$ realizes $\QCoh(X)_Y$ as a \emph{co-localization} 
of $\QCoh(X)$.

\medskip

By construction, we have a commutative diagram 
\begin{equation} \label{e:* and r}
\CD
\QCoh(X^{\wedge}_Y)  @<{\wh{i}{}^*}<<  \QCoh(X) \\
@A{'\wh{i}{}^*}AA   @AA{\on{Id}}A  \\
\QCoh(X)_Y   @<{\br^{\QCoh}}<< \QCoh(X),
\endCD
\end{equation}
where the left vertical arrow is the functor from \propref{p:QCoh of compl}.

\medskip

Hence, we obtain that the functor $\wh{i}{}^*:\QCoh(X)\to \QCoh(X^{\wedge}_Y)$,
in addition to having a non-continuous right adjoint $\wh{i}_*$, admits a left
adjoint, which we denote by $\wh{i}_?$. 

\medskip

Thus, we can think of $\QCoh(X^{\wedge}_Y)$ as both a localization and a co-localization
of $\QCoh(X)$ with respect to the essential image of $\QCoh(U)$. 

\medskip

Note that under such circumstances, we have a canonical
natural transformation
\begin{equation} \label{e:? to !}
\wh{i}_?\to \wh{i}_*.
\end{equation}

\sssec{}

Consider now the \emph{non-continuous} functor
$$\QCoh(X)\overset{\wh{i}{}^*}\longrightarrow \QCoh(X^\wedge_Y)\overset{\wh{i}_*}\longrightarrow \QCoh(X),$$
i.e., the localization functor on $\QCoh(X)$ with respect to the essential image of $\QCoh(U)$. 

\medskip

This functor is called \emph{the functor of formal completion} of a quasi-coherent sheaf along $Y$. Its essential image 
(i.e., the essential image of $\wh{i}_*$) is referred to as objects of $\QCoh(X)$ that are \emph{adically-complete} with respect to $Y$. 

\ssec{Compact generation and duality}  \label{ss:self duality formal compl}

Assume now that the scheme $X$ is quasi-separated and quasi-compact. It is well-known
that if $Y$ is locally given by a finitely generated ideal, then the category $\QCoh(X)_Y$
is compactly generated by $\QCoh(X)_Y\cap \QCoh(X)^{\on{perf}}$.

\medskip

Combining this with \propref{p:QCoh of compl} and \eqref{e:* and r} we obtain:

\begin{cor}  \label{c:duality QCoh formal compl}
The category $\QCoh(X^{\wedge}_Y)$ is compactly generated. The compact objects are obtained
as images under $\wh{i}{}^*$ of compact objects of $\QCoh(X)$ that are set-theoretically supported
on $Y$.
\end{cor}

Recall now the notion of quasi-perfectness, see \secref{sss:rigidity on QCoh}. We obtain:

\begin{cor}  \label{c:formal completion quasi-perfect}
For $X$ and $Y$ as above, the DG indscheme $\QCoh(X^{\wedge}_Y)$ is quasi-perfect.
\end{cor}

Let us recall that being quasi-perfect means by definition that the category 
$\QCoh(X^{\wedge}_Y)$ is compactly generated, and that its compact objects belong
to $\QCoh(X^{\wedge}_Y)^{\on{perf}}$. 

\medskip

As was shown in \secref{sss:rigidity on QCoh}, the above property implies that there exists
a canonical equivalence:
\begin{equation} \label{e:self duality QCoh formal compl}
\bD_{X^{\wedge}_Y}^{\on{naive}}:\QCoh(X^{\wedge}_Y)^\vee\simeq \QCoh(X^{\wedge}_Y),
\end{equation}
characterized by either of the following two properties:

\begin{itemize}

\item The canonical anti self-equivalence $\BD^{\on{naive}}_{\QCoh(X^{\wedge}_Y)}:(\QCoh(X^{\wedge}_Y)^c)^{\on{op}}\to (\QCoh(X^{\wedge}_Y)^c$
is given by the restriction of the functor $\CF\mapsto \CF^\vee:((\QCoh(X^{\wedge}_Y)^{\on{perf}})^{\on{op}}\to \QCoh(X^{\wedge}_Y)^{\on{perf}}$.

\medskip

\item The pairing 
\begin{equation} \label{e:pairing on compl}
\QCoh(X^{\wedge}_Y)\otimes \QCoh(X^{\wedge}_Y)\to \Vect
\end{equation}
is given by ind-extension of the pairing
$$\CF_1,\CF_2\in \QCoh(X^{\wedge}_Y)^c\mapsto \Gamma(X^{\wedge}_Y,
\CF_1\underset{\CO_{X^{\wedge}_Y}}\otimes\CF_2)\in \Vect.$$

\end{itemize}

\sssec{}

Note that although the natural transformation \eqref{e:? to !} is not an isomorphism, we have the following:

\begin{lem}  \label{l:? and !}
The natural transformation \eqref{e:? to !} it induces an isomorphism 
when restricted to compact objects of $\QCoh(X^{\wedge}_Y)$.
\end{lem}

\begin{proof}
Follows from the fact that compact objects of $\QCoh(X)$ with set-theoretic support on $Y$
are both left and right orthogonal to the essential image of $\QCoh(U)$.
\end{proof}

\sssec{}

Recall that the category $\QCoh(X)$ is also self-dual. From the description of the functor $\BD^{\on{naive}}_{\QCoh(X^{\wedge}_Y)}$
we obtain that there exists a canonical isomorphism
$$\BD^{\on{naive}}_{\QCoh(X)}\circ \wh{i}_?\simeq \wh{i}_?\circ \BD^{\on{naive}}_{\QCoh(X^{\wedge}_Y)}:(\QCoh(X^{\wedge}_Y)^c)^{\on{op}}\to \QCoh(X)^c.$$

\medskip

By \cite[Lemma 2.3.3]{DG}, this implies:
\begin{cor}  \label{c:dual of ?}
Under the identifications
$$\bD_{X}^{\on{naive}}:\QCoh(X)^\vee\simeq \QCoh(X) \text{ and } 
\bD_{X^{\wedge}_Y}^{\on{naive}}:\QCoh(X^{\wedge}_Y)^\vee\simeq \QCoh(X^{\wedge}_Y)^,$$
the dual of the functor $\wh{i}{}^*$ identifies with $\wh{i}_?$.
\end{cor}

\sssec{}

Note under the identifications
$$\bD_{X}^{\on{naive}}:\QCoh(X)^\vee\simeq \QCoh(X) \text{ and } \bD_{U}^{\on{naive}}:\QCoh(U)^\vee\simeq \QCoh(U)$$
we have $(j_*)^\vee\simeq j^*$. This implies that the category $\QCoh(X)_Y$
is also naturally self-dual, such that the dual of the natural embedding
$$\be^{\QCoh}:\QCoh(X)_Y\to \QCoh(X)$$
is the functor $\br^{\QCoh}$.

\medskip

By \cite[Lemma 2.3.3]{DG}, this implies:
$$\BD^{\on{naive}}_{\QCoh(X)}\circ \be^{\QCoh}
\simeq \be^{\QCoh}\circ \BD^{\on{naive}}_{\QCoh(X)_Y}:(\QCoh(X)_Y^c)^{\on{op}}\to \QCoh(X)^c.$$

It follows that:

\begin{cor}  \label{c:compat self duality}
The above self-duality of $\QCoh(X)_Y$ is compactible with the self-duality $\bD_{X^{\wedge}_Y}^{\on{naive}}$ of 
$\QCoh(X^{\wedge}_Y)$ via the equivalence of \propref{p:QCoh of compl}.
\end{cor}

\ssec{t-structures on $\QCoh(X^{\wedge}_Y)$}

In this subsection we will show that the category $$\QCoh(X^{\wedge}_Y)\simeq \QCoh(X)_Y$$ possesses two natural
t-structures: for one of them the functor $\be^{\QCoh}$ (i.e., the \emph{left} adjoint of $\br^{\QCoh}\simeq \wh{i}{}^*$)
is t-exact, and for the other, the functor $\wh{i}_*$ (i.e., the \emph{right} adjoint of 
$\wh{i}{}^*\simeq \br^{\QCoh}$) is t-exact.

\sssec{}

Let us recall the following general paradigm: let $\bC$ be a DG category equipped with a t-structure.
Let $F:\bC_1\hookrightarrow \bC$ be a fully faithful functor. Assume that $F$ admits a left (resp., right)
adjoint, denoted $F^L$ (resp., $F^R$). 
We have:

\begin{lem} \hfill  \label{l:t-structures and localization}

\smallskip

\noindent{\emph(a)}
If the composition $F\circ F^L$ (resp., $F\circ F^R$) is right (resp., left) t-exact, 
then $\bC_1$ has a unique t-structure such that $F$ is t-exact. With respect to this t-structure, 
the functor $F^L$ (resp., $F^R$) is right (resp., left) t-exact. 

\smallskip

\noindent{\emph(b)} 
If the composition $F\circ F^L$ (resp., $F\circ F^R$) is left (resp., right) t-exact, 
then $\bC_1$ has a unique t-structure such that $F^L$ (resp., $F^R$) is t-exact.
With respect to this t-structure, the functor $F$ is left (resp., right) t-exact.

\end{lem}

\sssec{}  \label{sss:1st t str}

We will apply point (a) of the lemma (with right adjoints)
to $\bC=\QCoh(X)$ and $\bC_1=\QCoh(X)_Y$.

\medskip

Let us first take $F:=\be^{\QCoh}$ and $F^R:=\br^{\QCoh}$. We obtain that 
$\QCoh(X)_Y$ admits a t-structure, compatible with its embedding into $\QCoh(X)$. This
t-structure is compatible with filtered colimits (i.e., truncation functors commute with filtered
colimits). 

\medskip

We shall refer to this t-structure on $\QCoh(X^{\wedge}_Y)$ as the ``inductive t-structure." 

\sssec{}  \label{sss:other t on qcoh}

We shall now introduce another t-structure on $\QCoh(X^{\wedge}_Y)$. 

\medskip

Recall (see \cite[Sect. 1.2.3]{QCoh}) that for any $\CZ\in \inftydgprestack$, the category $\QCoh(\CZ)$ has a canonical t-structure 
defined by the following requirement: an object $\CF$ belongs to $\QCoh(\CZ)^{\leq 0}$ if and only if for
every $S\in \affdgSch$ and $\phi:S\to \CZ$, we have $\phi^*(\CF)\in \QCoh(S)^{\leq 0}$. Let us call
it ``the canonical t-structure on $\QCoh(\CZ)$." 

\begin{prop}  \label{p:2nd t st}
The functor 
$$\wh{i}_*:\QCoh(X^{\wedge}_Y)\to \QCoh(X)$$
is t-exact for the canonical t-structure on $\QCoh(X^{\wedge}_Y)$.
\end{prop}

A few remarks are in order:

\medskip

\noindent(i) Since the functor $\wh{i}{}^*:\QCoh(X)\to \QCoh(X^{\wedge}_Y)$ is right t-exact,
we obtain that the proposition implies that the localization functor
$$\wh{i}_*\circ \wh{i}{}^*:\QCoh(X)\to \QCoh(X)$$ is also right t-exact. 
Thus, the canonical t-structure on $\QCoh(\CZ)$ falls into
the paradigm of \lemref{l:t-structures and localization}(a) with left adjoints.

\medskip

\noindent(ii) The canonical t-structure on $\QCoh(X^{\wedge}_Y)$ 
is different from the one of \secref{sss:1st t str}: for the former the functor $\wh{i}{}^*$ is right t-exact,
and for the latter it is left t-exact.

\medskip

\noindent(iii) Let $\CF$ be an object of $\QCoh(X)^\heartsuit$ which is \emph{scheme-theoretically}
supported on some subscheme $Y'\subset X$ whose underlying set is $Y$. Then it is easy to see
that $\CF$, regarded as an object of $\QCoh(X)_Y$, lies in the heart of both t-structures. 
 
\medskip

\noindent(iv) The canonical t-structure on $\QCoh(X^{\wedge}_Y)$ is typically not compatible with colimits, 
as can be seen in the example of $X=\BA^1$ and $Y=\on{pt}$.

\begin{proof}[Proof of \propref{p:2nd t st}]

The functor $\wh{i}_*$ is left t-exact, being the right adjoint of a right t-exact functor, namely, $\wh{i}^*$.
Hence, we need to show that $\wh{i}_*$ is right t-exact.

\medskip

Let $\CY$ be an object of $\inftydgstack$, and let $f:\CY\to X$ be a morphism, where
$X\in \dgSch$. Assume that $\CY$ is written as a colimit in $\inftydgstack$
$$\underset{g_\alpha:Y'_\alpha\to \CY}{colim}\, Y'_\alpha,$$
where $Y'_\alpha\in \dgSch$.  In this case, the functor
$$\QCoh(\CY)\to \underset{\alpha}{lim}\, \QCoh(Y'_\alpha)$$ 
is an equivalence (this follows from \cite[Corollary 1.3.7 ]{QCoh}, and the fact that the functor
$\QCoh(-)$ takes colimits in $\inftydgprestack$ to limits in $\StinftyCat$). 

\medskip

This implies that the (non-continuous) functor $f_*:\QCoh(\CY)\to \QCoh(X)$ can be calculated as
follows: for $\CF\in \QCoh(\CY)$, given as a compatible family $\CF_\alpha:=g_\alpha^*(\CF)\in \QCoh(Y'_\alpha)$,
$$f_*(\CF)\simeq \underset{\alpha}{lim}\, (f\circ g_\alpha)_*(\CF_\alpha).$$

\medskip

We apply this to $\CY:=X^{\wedge}_Y$ written as a colimit as in \propref{p:expl aleph 0}. Thus,
in order to show that $\wh{i}_*$ is right t-exact, we need to check that if $\CF_\alpha \in 
\QCoh(Y'_\alpha)^{\leq 0}$ for all $\alpha\in \sA$, then 
$$\underset{\alpha\in \sA}{lim}\, (i_\alpha)_*(\CF_\alpha)\in \QCoh(X)^{\leq 0}$$
where $i_\alpha$ denotes the map $Y'_\alpha\to X$.

\medskip

Since the index category is $A$ is $\BN$, we need to show that the functor
$lim^1$ applied to the family $\alpha\mapsto H^0((i_\alpha)_*(\CF_\alpha))$
vanishes. However, this is the case, since the maps in this family are surjective.

\end{proof} 

\ssec{Ind-coherent sheaves on formal completions}

Let $X$ be a DG scheme almost of finite type; 
in particular, it is quasi-compact and quasi-separated. 

\sssec{}

Recall (see \cite[Sect. 4.1]{IndCoh}) that we have a pair
of adjoint functors
$$j^{\IndCoh,*}:\IndCoh(X)\rightleftarrows \IndCoh(U):j^{\IndCoh}_*$$
that realize $\IndCoh(U)$ as a localization of $\IndCoh(X)$. Let
$\IndCoh(X)_Y\subset \IndCoh(X)$ be the full subcategory equal to
$$\on{ker}(j^{\IndCoh,*}):\IndCoh(X)\to \IndCoh(U).$$

\medskip

We let $\be^{\IndCoh}$ denote the tautological embedding
$$\IndCoh(X)_Y\hookrightarrow \IndCoh(X).$$
This functor admits a right adjoint, denoted $\br^{\IndCoh}$ given by
$$\CF\mapsto \on{Cone}\left(\CF\to j^{\IndCoh}_*\circ j^{\IndCoh,*}(\CF)\right)[-1].$$

\sssec{}

As was shown in \corref{c:formal lft}, for a Zariski-closed subset $Y$, the DG indscheme 
$X^{\wedge}_Y$ is locally almost of finite type, so $\IndCoh(X^{\wedge}_Y)$ is well-defined.

\medskip

Consider the functor \footnote{The usage of notation $\wh{i}{}^!$ here is different from 
\cite[Corollary 4.1.5]{IndCoh}. Nevertheless, this notation is consistent as will follow from
\propref{p:localization for IndCoh}.}
$$\wh{i}{}^!:\IndCoh(X)\to \IndCoh(X^{\wedge}_Y),$$
i.e., the !-pullback functor with respect to the morphism $\wh{i}:X^{\wedge}_Y\to X$.
It is easy to see that this functor annihilates the essential image of $\IndCoh(U)$
under $j^{\IndCoh}_*$.

\sssec{}  \label{sss:lower shriek}

We now claim that the functor $\wh{i}{}^!$ admits a left adjoint, to be denoted by
$\wh{i}^{\IndCoh}_*$. 

\medskip

Indeed, by \secref{sss:indcoh as colim}, we have:
\begin{equation} \label{e:present IndCoh on formal}
\IndCoh(X^{\wedge}_Y)\simeq \underset{\alpha}{colim}\, \IndCoh(Y_\alpha),
\end{equation}
where $Y_\alpha$ run over a family of closed DG subschemes of $X$ with the
underlying set contained in $Y$. If we denote by $i_\alpha$ the closed
embedding $Y_\alpha\hookrightarrow X$, the functor $\wh{i}^{\IndCoh}_*$, left adjoint to
$\wh{i}{}^!$, is given by the compatible family of functors
$$(i_\alpha)^{\IndCoh}_*:\IndCoh(Y_\alpha)\to \IndCoh(X).$$

\sssec{}

By construction, the essential image of the functor $\wh{i}^{\IndCoh}_*$ belongs to
$$\IndCoh(X)_Y\subset \IndCoh(X).$$ (Or, equivalently, the right adjoint $\wh{i}{}^!$
of $\wh{i}^{\IndCoh}_*$ factors through the co-localization functor $\br^{\IndCoh}$.)

\medskip

Let
\begin{equation} \label{e:localization for IndCoh}
'\wh{i}^{\IndCoh}_*:\IndCoh(X^{\wedge}_Y)\rightleftarrows \IndCoh(X)_Y:{}'\wh{i}{}^!
\end{equation}
denote the resulting pair of adjoint functors.

\medskip

We will show:

\begin{prop} \label{p:localization for IndCoh}
The adjoint functors of \eqref{e:localization for IndCoh} are equivalences. 
\end{prop}

\begin{proof}

As in the proof of \propref{p:QCoh of compl}, we need to show two things:

\medskip

\noindent{(a)} The functor $\wh{i}^{\IndCoh}_*:\IndCoh(X^{\wedge}_Y)\to \IndCoh(X)$ is fully faithful.

\smallskip

\noindent{(b)} The essential image of the functor $'\wh{i}^{\IndCoh}_*$ generates $\IndCoh(X)_Y$. 

\medskip

We note that (b) follows from \cite[Proposition 4.1.7(a)]{IndCoh}. It remains to prove (a). 
The assertion is Zariski-local, so we can assume $X=\Spec(A)$. Let $A_n$ be as in the 
proof of \propref{p:QCoh of compl}. Set $Y_n:=\Spec(A_n)$. 

\medskip

For $n'\leq n''$, let $i_{n',n''}$ denote the closed embedding $Y_{n'}\to Y_{n''}$,
and $i_n$ the closed embedding $Y_n\to X$. To prove (a), we need to show that for an 
index $n_0$ and $\CF\in \IndCoh(Y_{n_0})$, the map
\begin{equation} \label{e:colimits via n}
\underset{n\geq n_0}{colim}\, i^!_{n_0,n}\circ (i_{n_0,n})^{\IndCoh}_*(\CF)\to i_{n_0}^!\circ (i_{n_0})^{\IndCoh}_*(\CF)
\end{equation}
is an isomorphism. 

\medskip

Both sides in \eqref{e:colimits via n} commute with colimits in the $\CF$ variable. So, we can take $\CF\in \Coh(Y_{n_0})$.
In this case both sides of \eqref{e:colimits via n} belong to $\IndCoh(Y_{n_0})^+$. Hence, by \cite[Proposition 1.2.4]{IndCoh},
it suffices to show that the map in \eqref{e:colimits via n} induces an isomorphisms by applying the functor
$\Psi_{Y_{n_0}}:\IndCoh(Y_{n_0})\to \QCoh(Y_{n_0})$. Since $Y_0$ is affine, we can furthermore test whether 
a map is an isomorphism
by taking global sections.

\medskip

Hence, we obtain that it suffices to show that 
\begin{equation} \label{e:colimits via n bis}
\underset{n\geq n_0}{colim}\, \Maps_{A_n\mod}(A_{n_0},\CF)\to \Maps_{A\mod}(A_{n_0},\CF)
\end{equation}
is an isomorphism.  The map \eqref{e:colimits via n bis} can be rewritten as
$$\underset{n\geq n_0}{colim}\, \Maps_{A_{n_0}\mod}((A_{n_0}\underset{A_n}\otimes A)\underset{A}\otimes A_{n_0},\CF)\to 
\Maps_{A_{n_0}\mod}(A_{n_0}\underset{A}\otimes A_{n_0},\CF).$$

Hence, the required assertion follows from \lemref{l:stab ker}.

\end{proof}

\sssec{}

\propref{p:localization for IndCoh} implies the commutativity of the following diagram,
analogous to \eqref{e:* and r}:

\begin{equation} \label{e:! and r}
\CD
\IndCoh(X^{\wedge}_Y)  @<{\wh{i}{}^!}<<  \IndCoh(X) \\
@A{'\wh{i}{}^!}AA   @AA{\on{Id}}A  \\
\IndCoh(X)_Y   @<{\br^{\IndCoh}}<< \IndCoh(X),
\endCD
\end{equation}

\sssec{Compatibility with the t-structure}

Recall from \secref{ss:t-structure on IndCoh}, that the category $\IndCoh(X^{\wedge}_Y)$ has a natural t-structure.
Note that the category $\IndCoh(X)_Y$ also has a natural t-structure for which the functor $\be^{\IndCoh}$ is
t-exact. Indeed, this follows by \lemref{l:t-structures and localization}(a) from the fact that the functor
$$\be^{\IndCoh}\circ \br^{\IndCoh}:\IndCoh(X)\to \IndCoh(X),\quad \CF\mapsto \on{Cone}\left(\CF\to j^{\IndCoh}_*\circ j^{\IndCoh,*}(\CF)\right)[-1]$$ 
is left t-exact.

\medskip

We claim:

\begin{lem} 
The equivalence in \eqref{e:localization for IndCoh} is t-exact.
\end{lem}

\begin{proof}
By \lemref{l:t-structures and localization}(a) we only have to show that the functor
$$\wh{i}_*^\IndCoh:\IndCoh(X^{\wedge}_Y)\to \IndCoh(X)$$ is t-exact. However, this
follows from the description of this functor given in \secref{sss:lower shriek}.
\end{proof}

\ssec{Comparison of $\QCoh$ and $\IndCoh$ on a formal completion}

\sssec{} Recall (\cite[Sect. 10.3.3]{IndCoh}, Sect. 9.3.2) that for any $\CY\in \inftydgprestack_{\on{laft}}$ we have a canonical
functor
$$\Upsilon_\CY:\QCoh(\CY)\to \IndCoh(\CY),$$
given by tensoring with the dualizing object $\omega_{\CY}\in \IndCoh(\CY)$. 

\medskip

By construction, the following diagram of functors commutes:
\begin{equation} \label{e:qcoh and indcoh compat 1}
\CD
\QCoh(X^{\wedge}_Y)  @<{\wh{i}{}^*}<< \QCoh(X) \\
@V{\Upsilon_{X^{\wedge}_Y}}VV     @VV{\Upsilon_X}V \\
\IndCoh(X^{\wedge}_Y)  @<{\wh{i}{}^!}<< \IndCoh(X)  
\endCD
\end{equation}

\sssec{}  \label{sss:when Psi}

Recall that if $Z$ is a DG scheme, then the category $\IndCoh(Z)$ 
is self-dual, and the functor $\Upsilon_Z$ identifies with the functor $\Psi_Z^\vee$, 
the dual of the naturally defined functor $$\Psi_Z:\IndCoh(Z)\to \QCoh(Z),$$
see \cite[Proposition 9.3.3]{IndCoh}. However, the functor $\Psi_\CZ$ is \emph{not} intrinsically defined for
$\CZ\in \inftydgprestack_{\on{laft}}$. 

\medskip

Nevertheless, for $\CX\in \dgindSch$, we still have a canonical self-duality
$$\bD_\CX^{\on{Serre}}:\IndCoh(\CX)^\vee\simeq \IndCoh(\CX)$$
(see \corref{c:Serre for Ind}), and if $\CX$ is quasi-perfect (see \secref{sss:rigidity on QCoh}), then
we also have a self-duality
$$\bD_\CX^{\on{naive}}:\QCoh(\CX)^\vee\simeq \QCoh(\CX).$$
So, in this case, we can consider the functor $\Upsilon^\vee_\CX:\IndCoh(\CX)\to \QCoh(\CX)$, dual to $\Upsilon_\CX$. 

\medskip

Consider the resulting functor
\begin{equation} \label{e:Psi pairing}
\QCoh(\CX)\otimes \IndCoh(\CX)\overset{\on{Id}\otimes \Upsilon^\vee_\CX}\longrightarrow \QCoh(\CX)\otimes \QCoh(\CX)\to \Vect,
\end{equation}
where the last arrow is the pairing corresponding to the self-duality $\bD_\CX^{\on{naive}}$ of $\QCoh(\CX)$:

\medskip

By construction and \corref{c:pairing on IndCoh}, it is isomorphic to the composite
$$\QCoh(\CX)\otimes \IndCoh(\CX)\to \IndCoh(\CX)\overset{\Gamma^{\IndCoh}(\CX,-)}\longrightarrow \Vect,$$
where the first arrow is the canonical action of $\QCoh(-)$ on $\IndCoh(-)$. 

\sssec{}

The discussion in \secref{sss:when Psi} applies in particular to $\CX=X^{\wedge}_Y$.

\medskip

Passing to dual functors in \eqref{e:qcoh and indcoh compat 1}, and using \corref{c:dual of ?}
we obtain another commutative diagram:
\begin{equation} \label{e:qcoh and indcoh compat 2}
\CD
\QCoh(X^{\wedge}_Y) @>{\wh{i}_?}>> \QCoh(X) \\
@A{\Upsilon_{X^{\wedge}}^\vee}AA   @AA{\Psi_X}A  \\
\IndCoh(X^{\wedge}_Y) @>{\wh{i}{}^{\IndCoh}_*}>> \IndCoh(X).
\endCD
\end{equation}

\begin{lem}
The functor $\Upsilon^\vee_{X^{\wedge}_Y}$ is t-exact, when we consider the t-structure on $\IndCoh(X^{\wedge}_Y)$
of \secref{ss:t-structure on IndCoh} and the inductive t-structure on $\QCoh(X^{\wedge}_Y)$ of \secref{sss:1st t str}.
\end{lem}

\begin{proof}
The assertion follows from the fact that the functors $\wh{i}{}^{\IndCoh}_*$ and $\wh{i}_?$ are t-exact and conservative,
and the fact that $\Psi_X$ is t-exact.
\end{proof}

\sssec{}

Consider now the functors
\begin{equation} \label{e:two monads}
\wh{i}_?\circ \wh{i}{}^{*}:\QCoh(X)\to \QCoh(X) \text{ and }
\wh{i}^{\IndCoh}_{*}\circ \wh{i}^!:\IndCoh(X)\to \IndCoh(X).
\end{equation}

\begin{lem} \label{l:Psi and compl}
The functor $\Psi_X:\IndCoh(X)\to \QCoh(X)$ intertwines the functors
of \eqref{e:two monads}.
\end{lem}

\begin{proof}
The functors of \eqref{e:two monads} are isomorphic to
$$\on{Cone}(\on{Id}\to j_*\circ j^*)[-1] \text{ and }
\on{Cone}(\on{Id}\to j^{\IndCoh}_*\circ j^{\IndCoh,*})[-1],$$
respectively. So, it is enough to show that the functor $\Psi_X$
intertwines the functors $j_*\circ j^*$ and $j^{\IndCoh}_*\circ j^{\IndCoh,*}$.
However, the latter follows from \cite[Propositions 3.1.1 and 3.5.4]{IndCoh}.
\end{proof}

Combining this with the fact that the horizontal arrows in \eqref{e:qcoh and indcoh compat 2}
are conservative (in fact, fully faithful), we obtain:

\begin{cor} \label{c:Psi on formal}
The diagram of functors
\begin{equation} \label{e:qcoh and indcoh compat 3}
\CD
\QCoh(X^{\wedge}_Y)  @<{\wh{i}{}^*}<< \QCoh(X) \\
@A{\Upsilon^\vee_{X^{\wedge}_Y}}AA     @AA{\Psi_X}A \\
\IndCoh(X^{\wedge}_Y)  @<{\wh{i}{}^!}<< \IndCoh(X),
\endCD
\end{equation}
obtained from \eqref{e:qcoh and indcoh compat 2} by passing to right adjoint functors along
the horizontal arrows, and which a priori commutes up to a natural transformation, is 
commutative.
\end{cor}

\medskip

Passing to dual functors in \eqref{e:qcoh and indcoh compat 3}, we obtain yet another
commutative diagram of functors:
\begin{equation} \label{e:qcoh and indcoh compat 4}
\CD
\QCoh(X^{\wedge}_Y) @>{\wh{i}_?}>> \QCoh(X) \\
@V{\Upsilon_{X^{\wedge}_Y}}VV  @VV{\Psi^\vee_X}V  \\
\IndCoh(X^{\wedge}_Y)  @>{\wh{i}^{\IndCoh}_*}>>  \IndCoh(X).
\endCD
\end{equation}

The diagram \eqref{e:qcoh and indcoh compat 4} can be alternatively obtained by passing to
left adjoint functors along the horizontal arrows in \eqref{e:qcoh and indcoh compat 1}. Thus, the resulting diagram,
which a priori commutes up to a natural transformation, is actually commutative.






\ssec{$\QCoh$ and $\IndCoh$ in the eventually coconnective case}

In this subsection we will assume that $X$ is eventually coconnective. By Proposition
\ref{p:formal compl coconn}, the ind-scheme $X^{\wedge}_Y$ is also eventually coconnective.

\sssec{}

Recall (see \cite[Proposition 1.5.3]{IndCoh}) that for $X\in \dgSch_{\on{aft}}$ 
eventually coconnective, the functor $\Psi_X:\IndCoh(X)\to \QCoh(X)$
admits a left adjoint, denoted $\Xi_X$. It is characterized by the property that it sends 
$\QCoh(X)^{\on{perf}}\simeq \QCoh(X)^c$ to $\Coh(X)\simeq \IndCoh(X)^c$ via the tautological map
$$\QCoh(X)^{\on{perf}}\to \Coh(X),$$
which is well-defined because $X$ is eventually coconnective.

\medskip

Also, recall that the functor $\Xi^\vee_X:\IndCoh(X)\to \QCoh(X)$, dual to $\Xi_X$, is the right adjoint of $\Psi^\vee_X$, and
it can be described as
$$\Xi^\vee_X\simeq \uHom_{\QCoh(X)}(\omega_X,-),$$
(see \cite[Lemma 9.6.7]{IndCoh}).

\medskip

We emphasize that the functors $\Xi$ and $\Xi^\vee$ are defined specifically for DG schemes,
and not arbitrary eventually coconnective objects of $\inftydgprestack_{\on{laft}}$. 

\medskip

However, for any
object $\CY\in \on{PreStk}_{\on{laft}}$ we can still ask whether the right adjoint $\Xi^\vee_\CY$ of $\Upsilon_\CY$
is continuous. 

\medskip

If $\CY$ is equipped with self-duality data for $\QCoh(\CY)$ and $\IndCoh(\CY)$, in which case 
the functor $\Upsilon^\vee_\CY$ is well-defined, we can ask whether the left adjoint $\Xi_\CY$ of
$\Upsilon^\vee_\CY$ exists. 

\sssec{}    \label{sss:Xi for formal actual}

By passing to right (resp., left) adjoint functors is Diagrams \eqref{e:qcoh and indcoh compat 4} and 
\eqref{e:qcoh and indcoh compat 3}, respectively, we obtain two more commutative diagrams
\begin{equation} \label{e:qcoh and indcoh compat 5}
\CD
\QCoh(X^{\wedge}_Y)  @<{\wh{i}{}^*}<< \QCoh(X) \\
@A{\Xi^\vee_{X^{\wedge}_Y}}AA     @AA{\Xi^\vee_X}A \\
\IndCoh(X^{\wedge}_Y)  @<{\wh{i}{}^!}<< \IndCoh(X),  
\endCD
\end{equation}
and
\begin{equation} \label{e:qcoh and indcoh compat 6}
\CD
\QCoh(X^{\wedge}_Y) @>{\wh{i}_?}>> \QCoh(X) \\
@V{\Xi_{X^{\wedge}_Y}}VV    @VV{\Xi_X}V  \\
\IndCoh(X^{\wedge}_Y)  @>{\wh{i}^{\IndCoh}_*}>>  \IndCoh(X).
\endCD
\end{equation}

\medskip

In particular, we obtain that the functor $\Xi^\vee_{X^{\wedge}_Y}$ is continuous, and
$\Xi_{X^{\wedge}_Y}$ is defined, for the DG indscheme $X^{\wedge}_Y$.

\sssec{}  \label{sss:Xi for formal}

We now claim the following:

\begin{prop}
The diagrams of functors
\begin{equation} \label{e:qcoh and indcoh compat 7}
\CD
\QCoh(X^{\wedge}_Y)  @<{\wh{i}{}^*}<< \QCoh(X) \\
@V{\Xi_{X^{\wedge}_Y}}VV     @VV{\Xi_X}V \\
\IndCoh(X^{\wedge}_Y)  @<{\wh{i}{}^!}<< \IndCoh(X)  
\endCD
\end{equation}
and
\begin{equation} \label{e:qcoh and indcoh compat 8}
\CD
\QCoh(X^{\wedge}_Y) @>{\wh{i}_?}>> \QCoh(X) \\
@A{\Xi^\vee_{X^{\wedge}_Y}}AA   @AA{\Xi^\vee_X}A  \\
\IndCoh(X^{\wedge}_Y)  @>{\wh{i}^{\IndCoh}_*}>>  \IndCoh(X),
\endCD
\end{equation}
obtained from the diagrams \eqref{e:qcoh and indcoh compat 5} and \eqref{e:qcoh and indcoh compat 6},
respectively, by passing to adjoint functors along the vertical arrows, and which a priori commute up to natural transformations, 
are commutative.
\end{prop}

\begin{proof}
The two diagrams are obtained from one another by passing to dual functors. Therefore,
it is sufficient to show that \eqref{e:qcoh and indcoh compat 8} is commutative. Taking into 
account \eqref{e:qcoh and indcoh compat 5} and the fact that in the latter diagram the
horizontal arrows are co-localizations, it suffices to show that the functor $\Xi^\vee_X$
intertwines the functors
$$\wh{i}_?\circ \wh{i}{}^*:\QCoh(X)\to \QCoh(X) \text{ and }
\wh{i}^{\IndCoh}_*\circ \wh{i}{}^{!}:\IndCoh(X)\to \IndCoh(X).$$
As in the proof of \lemref{l:Psi and compl}, it suffices to show that $\Xi^\vee_X$
intertwines the functors
$$j_*\circ j^*:\QCoh(X)\to \QCoh(X) \text{ and }
j^{\IndCoh}_*\circ j^{\IndCoh,*}:\IndCoh(X)\to \IndCoh(X).$$

It is clear that
$$j^*\circ \Xi^\vee_X\simeq \Xi_U^\vee\circ j^{\IndCoh,*}.$$
So, we have to show that the natural map
$$\Xi^\vee_X\circ j^{\IndCoh}_* \to j_*\circ \Xi^\vee_U$$
is an isomorphism. Let $\CF_X\in \QCoh(X)$ and $\CF_U\in \IndCoh(U)$ be two objects.
We have
\begin{multline*}
\Maps_{\QCoh(X)}(\CF_X,\Xi^\vee_X\circ j^{\IndCoh}_*(\CF_U))\simeq 
\Maps_{\IndCoh(X)}(\CF_X\otimes \omega_X,j^{\IndCoh}_*(\CF_U))\simeq \\
\simeq \Maps_{\IndCoh(U)}(j^{\IndCoh,*}(\CF_X\otimes \omega_X),\CF_U)\simeq
\Maps_{\IndCoh(U)}(j^*(\CF_X)\otimes j^{\IndCoh,*}(\omega_X),\CF_U)\simeq \\
\simeq \Maps_{\IndCoh(U)}(j^*(\CF_X)\otimes \omega_U,\CF_U)\simeq
\Maps_{\QCoh(U)}(j^*(\CF_X),\Xi^\vee_U(\CF_U)).
\end{multline*}

\end{proof}

\section{Formally smooth DG indschemes}  \label{s:class smooth}

\ssec{The notion of formal smoothness}

Let $\CX_{cl}$ be an object of $^{cl}\!\inftydgprestack$. 

\smallskip

\begin{defn}  \label{d:classical formally smooth}
We say that $\CX_{cl}$ is formally smooth if
%
%
%
%
%
for every closed embedding $$S\hookrightarrow S'$$ of classical affine schemes,
\emph{such that the ideal of $S$ inside $S'$ is nilpotent}, 
the map of sets
$$\pi_{0}(\CX_{cl}(S'))\to \pi_{0}(\CX_{cl}(S))$$
is surjective.
\end{defn}

Clearly, in order to test formal smoothness, it is sufficient to consider closed embeddings
of classical affine schemes
$$S\hookrightarrow S',$$
such that the ideal $\CI$ of $S$ inside $S'$ satisfies $\CI^2=0$.



\sssec{}

Let $\CX$ be an object of $\inftydgprestack$. 

\smallskip

\begin{defn}  \label{d:formally smooth}
We say that $\CX$ is formally smooth if:

\smallskip

\begin{enumerate}
\item The classical prestack $^{cl}\CX:=\CX|_{\affSch}$ is formally smooth in the sense of Definition
\ref{d:classical formally smooth}.

\smallskip

\item
For every $n$ and $S\in \affdgSch$, the map $\CX(S)\to \CX({}^{\leq n}\!S)$
induces an isomorphism on $\pi_{n}$.


\end{enumerate}

\end{defn}

\medskip

We can reformulate Definition \ref{d:formally smooth} as follows.

\begin{lem}  \label{l:formal smoothness via homotopy groups}
Let $\CX \in \inftydgprestack$ be such that 
$^{cl}\CX$ is formally smooth.  Then $\CX$ is formally smooth if and only if $\CX$ is convergent and for any integers 
$i\geq j$ and $S\in {}^{\leq i}\!\affdgSch$ the map 
$$\CX(S)\to \CX({}^{\leq j}\!S)$$
induces an isomorphism on $\pi_j$.
\end{lem}

\begin{rem} \label{r:formal smoothness bug} 
As was alluded to in the introduction, the property of formal smoothness, in both the classical
and derived contexts, has a substantial drawback of being non-local in the Zariski topology. For example, we
could have given a different definition by requiring the corresponding properties to hold after Zariski localization
with respect to the test affine scheme $S$. 
We will see a manifestation of this phenomenon in 
\secref{sss:formal smoothness and cotangent bug} for $0$-truncated prestacks that admit connective deformation theory. 

\medskip

However, it will turn out that in the latter case the difference between the two definitions disappears if we restrict ourselves
to prestacks locally almost of finite type (see \secref{ss:formal smoothness for laft}), 
which will be the main case of interest in the rest of this paper. 
\end{rem}

\sssec{} 
All the examples of prestacks that we consider in this paper are 0-truncated in the sense of \cite[Sect. 1.1.7]{Stacks}. 
I.e., we consider prestacks $\CY$ such that for all $n$ and $S\in {}^{\leq n}\!\affdgSch$, $$\CY(S) \in n\on{-}\Groupoids\subset \inftygroup.$$

In this case, we have the following reformulation of the Definition \ref{d:formally smooth}.

\begin{lem}  \label{l:formal smoothness via postnikov tower}
Let $\CX \in \inftydgprestack$ be a $0$-truncated prestack such that 
 $^{cl}\CX$ is formally smooth 
as a classical prestack. Then $\CX$ is formally smooth if and only if
for every $n$ and $S\in \affdgSch$, the map $$\CX(S)\to \CX({}^{\leq n}\!S)$$
identifies the right-hand side with the $n$-truncation of (the Postnikov tower of) $\CX(S)$.
\end{lem}

\sssec{}

Let $\CX$ be a prestack such that ${}^{cl}\CX$ is an indscheme, and let $Y$ be a reduced classical scheme,
equipped with a closed embedding $Y\hookrightarrow {}^{cl,red}\CX$. 
Consider the formal completion $\CX^\wedge_Y$. 

\begin{prop}  \label{p:formal smoothness of formal completions} 
The following conditions are equivalent:

\smallskip

\noindent{\em(a)} $\CX$ is formally smooth.

\smallskip

\noindent{\em(b)} For every $Y\hookrightarrow {}^{cl,red}\CX$ as above,
the formal completion $\CX^\wedge_Y$ is formally smooth.

\end{prop}

\begin{proof}

Since $\CX^\wedge_Y(S)$ is a connected component of $\CX(S)$, condition (a) implies
condition (b). For the opposite implication, write $^{cl}\CX$ as 
$\underset{\alpha}{colim}\, X_{\alpha}$. We claim that it is enough to show that
each $\CX^\wedge_{X_{\alpha}}$ is formally smooth. Indeed, both conditions of
formal smoothness can be checked separately over each point of 
$^{cl,red}\!S\to {}^{cl,red}\CX$, and every such point factors through
some $X_{\alpha}$.

\end{proof}

\ssec{Formal smoothness via deformation theory} 

\sssec{}   \label{sss:formal smoothness and cotangent} 

Let $\CX\in \inftydgprestack$ admit connective deformation theory
(see Definition \ref{def:deform theory}).

\begin{prop} \label{p:formal smoothness through deformations}
Suppose that $^{cl}\CX$ is $0$-truncated. Then $\CX$ is formally smooth if and
only if the following equivalent conditions hold:

\smallskip

\noindent{\em(a)} For every $S\in \affdgSch$ and $x:S\to \CX$, the object
$$T^*_x\CX\in \on{Pro}(\QCoh(S)^{\leq 0})$$ 
has the property that
$$\Hom(T^*_x\CX,\CF[i])=0,\, \forall \CF\in \QCoh(S)^{\heartsuit} \text{ and } i>0.$$

\smallskip

\noindent{\em(b)} Same as {\em(a)}, but for $S$ a classical affine scheme. 

\smallskip

\noindent{\em(c)} Under an additional assumption that $\CX$ is locally almost of finite type,
the same as {\em(b)}, but for $S$ reduced.

\end{prop}

\begin{proof}

It is clear that if $\CX$ is formally smooth, then it satisfies (a): indeed, consider the split square-zero extension
of $S$ corresponding to $\CF[i]$. The converse implication
follows from deformation theory using Lemmas \ref{l:can sq zero}, \ref{l:cl sq zero} and 
\lemref{l:formal smoothness via homotopy groups}.

\medskip

Condition (a) implies condition (b) tautologically. The converse implication follows
from the fact that every object of $\QCoh(S)^{\heartsuit}$ is the direct image 
under the canonical map $^{cl}\!S\to S$. Indeed, for a point $x:S\to \CX$,
the pull-back of $T^*_x\CX$ under $^{cl}\!S\to S$ identifies with $T^*_{^{cl}\!x}\CX$,
where $^{cl}\!x$ is the composition $^{cl}\!S\to S\overset{x}\to \CX$.

\medskip

Condition (b) implies condition (c) tautologically. For the converse implication,
we note that under the assumption that $\CX$ is locally of finite type, by \lemref{l:char coh},
the functor $T^*_x\CX$ commutes with colimits in $\QCoh(S)^{\heartsuit}$. This allows
to replace any $\CF\in \QCoh(S)^{\heartsuit}$ by one obtained as a direct image from
$^{red}\!S$.

\end{proof}

\sssec{}

The following definition will be convenient in the sequel. Let $S$ be an affine DG scheme,
and let $F$ be an object of $\on{Pro}(\QCoh(S)^{\leq 0})$. 

\medskip

We shall say that $F$ is \emph{convergent} if for
every $\CF\in \QCoh(S)^{\leq 0}$, the natural map
\begin{equation} \label{e:conv in Pro}
F(\CF)\to \underset{n\in \BN^{\on{op}}}{lim}\, F(\tau^{\geq -n}(\CF))
\end{equation}
is an isomorphism in $\inftygroup$.

\medskip

We have:

\begin{lem}
Let $\CX\in \on{PreStk}$ admit connective deformation theory, and let $x:S\to \CX$
be a map. Then $T^*_x\CX\in \on{Pro}(\QCoh(S)^{\leq 0})$ is convergent.
\end{lem}

\begin{proof}
Follows from the fact that the condition of admitting connective deformation theory includes convergence.
\end{proof} 

\sssec{}

Let $S$ be an affine classical scheme. Let us characterize those objects
$$F\in \on{Pro}(\QCoh(S)^{\leq 0})$$ that satisfy property (a) of
\propref{p:formal smoothness through deformations}. 


\medskip

We have:

\begin{lem} \label{l:when pro proj one}
For $S\in \affSch$ and $F\in \on{Pro}(\QCoh(S)^{\leq 0})$ 
the following are equivalent:

\smallskip

\noindent{\em(a)} $F$ is convergent and $\pi_0\left(F(\CF[i])\right)=0$ for all $\CF\in \QCoh(S)^\heartsuit$
and $i>0$.

\smallskip

\noindent{\em(a')} $\pi_0\left(F(\CF)\right)=0$ for all $\CF\in \QCoh(S)^{<0}$.

\smallskip

\noindent{\em(b)} $F$ 
belongs to the full subcategory
$$\on{Pro}(\QCoh(S)^{\heartsuit,\on{proj}})\subset \on{Pro}(\QCoh(S)^{\leq 0})$$
where $\QCoh(S)^{\heartsuit,\on{proj}}$ is the full subcategory of projective objects in $\QCoh(S)^{\heartsuit}$.

\smallskip

\noindent{\em(b')} $F$ is convergent, belongs to the full subcategory
$$\on{Pro}(\QCoh(S)^{\heartsuit})\subset \on{Pro}(\QCoh(S)^{\leq 0}),$$
and the functor
$$\CF\mapsto \pi_0\circ F(\CF),\quad \QCoh(S)^\heartsuit\to \on{Sets}$$
is right exact.

\end{lem}

\begin{rem}
This lemma is not specific to $\QCoh(S)$; it is applicable to any 
stable $\infty$-category equipped with a t-structure, whose heart
has enough projectives and injectives. 
\end{rem}

\begin{proof}

The equivalence of (a) and (a') is immediate. It is also clear that (b) implies (a). 

\medskip

Suppose that $F$ satisfies (a'), and let us deduce (b). Consider the category 
$$\{P\in \QCoh(S)^{\heartsuit,\on{proj}},f_P\in H^0(F(P))\}.$$
The assumption implies that this category is cofiltered, and it is easy to see that the resulting map in 
$\on{Pro}(\QCoh(S)^{\geq -n,\leq 0})$
$$F\to \underset{(P,f_P)}{``lim"}\, P$$
is an isomorphism.

\medskip

The implication (b) $\Rightarrow$ (b') is also immediate. Let us show that (b') implies (a).
By assumption, $F$ is given as an object 
$$\underset{\alpha\in \sA}{``lim"}\, \CF_\alpha \in \on{Pro}(\QCoh(S)^\heartsuit),$$
where the category of indices $\sA$ is filtered. By definition,
$$\pi_0(F(\CF))\simeq \underset{\alpha\in \sA}{colim}\, \Hom(\CF_\alpha,\CF).$$

Hence, if $\CF\in \QCoh(S)^\heartsuit$ is injective, then $\pi_0(F(\CF[i]))=0$ for $i>0$.
The exactness of $F$ on the abelian category implies that $\pi_0(F(\CF[1]))=0$ for
any $\CF\in \QCoh(S)^\heartsuit$ by the long exact cohomology sequence.
The assertion that $\pi_0(F(\CF[i]))=0$ for $n>i>1$ and any $\CF\in \QCoh(S)^\heartsuit$ 
follows by induction on $i$, again by the long exact cohomology sequence.

\end{proof}

\sssec{}

In what follows, for $S\in \affSch$, we shall refer to objects of  $F\in \on{Pro}(\QCoh(S)^{\leq 0})$ satisfying the equivalent
conditions of \lemref{l:when pro proj one} as \emph{pro-projective}. 

\sssec{}  \label{sss:formal smoothness and cotangent bug}

We can now better explain the non-locality of the definition of formal smoothness
mentioned in Remark \ref{r:formal smoothness bug}:

\medskip

Let $S$ be an affine classical scheme, and let $F$ be an object of $\on{Pro}(\QCoh(S)^{\heartsuit})$.
It is a natural question to ask whether the property of $F$ to be pro-projective 
is local in the Zariski topology.

\medskip

Namely, if $S_i$ is an open cover of $S$ by affine 
subchemes and $F|_{S_i}\in \on{Pro}(\QCoh(S_i)^{\heartsuit,\on{proj}})$, will it be true that
$F$ itself belongs to $\on{Pro}(\QCoh(S)^{\heartsuit,\on{proj}})$? 

\medskip

Unfortunately, we do not know the answer to this question, but we think that 
it is probably negative.

\begin{rem}
Note, however, if we ask the same question for $F$ being an object of $\QCoh(S)^{\heartsuit}$, rather than
$\on{Pro}(\QCoh(S)^{\heartsuit})$, the answer will be affirmative, due to a non-trivial theorem of Raynaud-Gruson,
\cite{RG}. 
\end{rem}

\ssec{Formal smoothness for prestacks locally of finite type}   \label{ss:formal smoothness for laft}

\sssec{}

Let $S$ be an affine DG scheme, and let $F$ be an object of $\on{Pro}(\QCoh(S)^{\geq -n,\leq 0})$. 

\medskip

We shall say that $F$ is \emph{pro-coherent} if, when viewed as a functor
$$\QCoh(S)^{\geq-n, \leq 0}\to \inftygroup,$$
it commutes with filtered colimits. 

\medskip

Note that this condition is satisfied for $F$ arising as $^{\geq -n}(T^*_x\CX)$ for $x:S\to \CX$, 
where $\CX$ admits connective deformation theory and belongs to $\inftydgprestack_{\on{laft}}$.

\medskip

Also note that when $S$ is Noetherian, by \lemref{l:char coh}, pro-coherence is equivalent to
$F$ belonging to $\on{Pro}(\Coh(S)^{\geq -n,\leq 0})$. 

\medskip

In general, $F$ is pro-coherent if and only if
it can be represented by a complex 
$$P^{-n-1}\to P^{-n}\to...\to P^{-1}\to P^0$$
in $\on{Pro}(\QCoh(S)^{\heartsuit})$, whose terms belong to $\on{Pro}(\QCoh(S)^{\heartsuit,\on{proj,f.g.}})$,
where $$\QCoh(S)^{\heartsuit,\on{proj,f.g.}}\subset \QCoh(S)^{\heartsuit}$$ denotes the category of projective
finitely generated quasi-coherent sheaves.

\sssec{}  \label{sss:locality of formal smoothness}

We have: 

\begin{lem} \label{l:locality of proj} Let $S$ be a classical affine scheme and
let $F\in \on{Pro}(\QCoh(S)^\heartsuit)$ be pro-coherent. Then its 
property of being pro-projective is local in the Zariski topology.
\end{lem}

\begin{proof}

We will check the locality of condition (b') of \lemref{l:when pro proj one}.

\medskip

First, it is easy to see that the property for an object of $\on{Pro}(\QCoh(S)^{\leq 0})$ to be
convergent is Zariski-local.

\medskip

Hence, it remains to check that the property of the functor
$$\CF\mapsto \pi_0(F(\CF)),\quad \QCoh(S)^\heartsuit\to \on{Sets}$$
to be right exact is also Zariski-local, under the assumption that $F$ is pro-coherent. 
We will show that this property is in fact fpqc-local.

\medskip

Thus, let $f:S'\to S$ be an fpqc map, where $S=\Spec(A)$ and $S'=\Spec(B)$. We assume that
the functor 
$$F':=\on{Pro}(f^*)(F): B\mod\to \inftygroup$$
is such that
$$F'{}^\heartsuit:=\pi_0\circ F':(B\mod)^\heartsuit\to \on{Sets}$$
is right exact, and we wish to deduce the same for $$F^\heartsuit:=\pi_0\circ F:(A\mod)^\heartsuit\to \on{Sets}.$$
Note that by adjunction, $F'{}^\heartsuit(\CM)=F^\heartsuit(f_*(\CM))$ for $\CM\in (B\mod)^\heartsuit$.

\medskip

Consider the object $F^\heartsuit(A)\in \on{Sets}$. The action of $A$ on itself as an $A$-module
defines on $F^\heartsuit(A)$ a structure of an $A$-module. There is a natural map of functors
$(A\mod)^\heartsuit\to \on{Sets}$
\begin{equation} \label{e:from tensor product to Hom}
\CN\underset{A}\otimes F^\heartsuit(A)\to F^\heartsuit(\CN),
\end{equation}
where in the above formula we are using the \emph{non-derived} tensor product. 

\medskip

Note the map in \eqref{e:from tensor product to Hom} is an isomorphism whenever $F$ is pro-coherent
and $F^\heartsuit$ is right exact. Indeed, both functors are right exact and commute with filtered colimits,
so the isomorphism for any $\CN$ follows from the case $\CN=A$.

\medskip

And vice versa, if \eqref{e:from tensor product to Hom} is an isomorphism then $F^\heartsuit$ is right exact.
Indeed, the left-hand side is a right exact, and the right-hand side is left exact, so if the map in
question is an isomorphism, and both functors are actually exact. 

\medskip

Also note that \eqref{e:from tensor product to Hom} is an isomorphism for $F^\heartsuit$ pro-coherent whenever
$\CN$ is $A$-flat, by Lazard's lemma. 

\medskip

In order to show that \eqref{e:from tensor product to Hom} is an isomorphism under our
assumptions, let us tensor both sides with $B$, and consider the commutative diagram
$$
\CD
\CN\underset{A}\otimes F^\heartsuit(A)\underset{A}\otimes B  @>>>  F^\heartsuit(\CN)\underset{A}\otimes B   \\
@VVV   @VVV  \\
\CN\underset{A}\otimes F^\heartsuit(B)  @>>>  F^\heartsuit(\CN\underset{A}\otimes B).
\endCD
$$
Since $B$ is faithfully flat over $A$, it is enough to show that the upper horizontal arrow
is an isomorphism.

\medskip

In the above diagram the vertical arrows are isomorphisms since $B$ is $A$-flat. However, the lower
horizontal arrow identifies with
$$\CN\underset{A}\otimes F'{}^\heartsuit(B) \simeq
(\CN\underset{A}\otimes B)\underset{B}\otimes F'{}^\heartsuit(B) \to 
F'{}^\heartsuit(\CN\underset{A}\otimes B),$$
which is an isomorphism by \eqref{e:from tensor product to Hom} applied to
$F'{}^\heartsuit$.

\end{proof}

\sssec{}

In view of \propref{p:formal smoothness through deformations},
the above lemma implies that for 0-truncated prestacks that admit connective deformation theory and 
are locally almost of finite type, the definition of formal smoothness is reasonable, in the 
sense that it is Zariski-local.

\medskip

As a manifestation of this, we have the following assertion that will be useful in the
sequel.

\sssec{}

Let $\CX$ be a formal DG scheme with the underlying reduced classical scheme $X$. Denote
$$T^*\CX|_X:=T^*_x\CX$$ 
where $x:X\to \CX$ is the tautological point.

\begin{cor} \label{c:estimate on cotangent formal}
Suppose that $\CX$ is locally almost of finite type, and that $T^*\CX|_X$ is Zariski-locally 
pro-projective. Then $\CX$ is 
formally smooth. 
\end{cor}

\begin{proof}

We will check that the conditions of \propref{p:formal smoothness through deformations}(c). 
Note that every map $S\to \CX$, where $S$ is a reduced classical affine scheme, factors
through a map $f:S\to X$.  Thus, we need to show that for every such $f$, the object
$$\on{Pro}(f^*)(T^*\CX|_X)\in \on{Pro}(\QCoh(S)^{\leq 0})$$
is pro-projective. 

\medskip

First, the Zariski-locality of the t-structure on $\on{Pro}(\QCoh(X)^{\leq 0})$ implies that 
$T^*\CX|_X$ belongs to the full subcategory 
$$\on{Pro}(\QCoh(X)^\heartsuit)\subset \on{Pro}(\QCoh(X)^{\leq 0}).$$
Now, since $\CX$ is locally almost of finite type, the classical scheme $X$ is of finite type. 
Hence, \propref{p:char finite type new} implies that $T^*\CX|_X$ belongs to
$$\on{Pro}(\Coh(X)^\heartsuit)\subset \on{Pro}(\QCoh(X)^\heartsuit)\subset \on{Pro}(\QCoh(X)^{\leq 0}).$$
Finally, by \lemref{l:locality of proj}, we obtain that $T^*\CX|_X$ belongs to
$$\on{Pro}(\Coh(X)^{\heartsuit,\on{proj}})\subset 
\on{Pro}(\Coh(X)^\heartsuit)\subset \on{Pro}(\QCoh(X)^\heartsuit)\subset \on{Pro}(\QCoh(X)^{\leq 0}),$$
and in particular to
$$\on{Pro}(\QCoh(X)^{\heartsuit,\on{proj}})\subset \on{Pro}(\QCoh(X)^{\leq 0}).$$

\medskip

However, it is clear that for any $f:S\to X$ with $S\in \affSch$, the functor $\on{Pro}(f^*)$ sends
pro-projective objects to pro-projective objects, as required. 

\end{proof}

\ssec{Examples of formally smooth DG indschemes}   \label{ss:examples of formally smooth}

In this subsection we will give three examples of formally smooth DG indschemes.

\sssec{}

The first example is the most basic one: we claim that the affine space $\BA^n$,
considered as an object of $\inftydgprestack$, is formally smooth. Indeed, the
definition of formal smoothness is satisfied on the nose as
$$\Maps(\Spec(A),\BA^n)\simeq \Omega^{\infty}({\mathsf {Sp}}(A))^{\times n}.$$

\sssec{}  \label{sss:classical smooth schemes}

Let $X$ be a classical smooth scheme of finite type over $k$. We claim that $X$,
considered as an object of $\inftydgprestack$ (i.e., $^L\!\on{LKE}_{(\affSch)^{\on{op}}\hookrightarrow (\affdgSch)^{\on{op}}}(X)$), is formally smooth. 

\medskip

Indeed, it suffices to show that the conditions of \propref{p:formal smoothness through deformations}
are satisfied. In fact, we claim that $T^*X$, is an object of $\Coh(X)^{\heartsuit}$, and is
locally projective. 

\medskip

The question is local on $X$, so we can assume that
$X$ fits into a Cartesian square
\begin{equation} \label{e:smooth scheme}
\CD
X @>>> \BA^n \\
@VVV   @VV{f}V \\
0  @>>> \BA^m,
\endCD
\end{equation}
where the map $f$ is smooth, and where the fiber product \emph{is taken in the category of
classical schemes}.

\medskip

Since $f$ is flat, the above square is also Cartesian in the category of DG schemes.
Hence, $T^*X$ can be calculated as
$$\on{Cone}(f^*(T^*\BA^m)|_X\to T^*\BA^n|_X),$$
and the smoothness hypothesis on $f$ implies the required properties of $T^*X$. 

\begin{cor}  \label{c:formal completion formally smooth}
Let $X$ be a smooth classical scheme locally of finite type, and let $Y\subset X$
be a Zariski-closed subset. Then the formal completion $X^{\wedge}_Y$ is
formally smooth as an object of $\inftydgprestack$.
\end{cor}

\begin{proof}
This follows from \propref{p:formal smoothness through deformations} as
$\wh{i}:X^{\wedge}_Y\to X$ induces an isomorphism on pro-cotangent complexes.
\end{proof}

Also, note that by \propref{p:class vs derived formal}, the DG indscheme $X^{\wedge}_Y$ 
is $0$-coconnective, i.e., is a left Kan extension of a classical indscheme. 

\sssec{}  \label{sss:Jacobi ft}

The following example will be needed for the proof of \thmref{t:classical vs derived ft}. 
Consider the formal DG scheme $\BA^{n,m}:=\on{Spf}\left(k[x_1,...,x_n][\![y_1,...,y_m]\!]\right)$, i.e., the
formal completion of $\BA^{n+m}$ along the subscheme $\BA^n\hookrightarrow \BA^{n+m}$
embedded along the first $n$ coordinates.

\medskip

Let $f_1,...,f_k$ be elements of $k[x_1,...,x_n][\![y_1,...,y_m]\!]$, and let
$\bar{f_1},...\bar{f_k}$ be their images under
$$k[x_1,...,x_n][\![y_1,...,y_m]\!]\twoheadrightarrow k[x_1,...,x_n].$$
Set
$$\CX:=0\underset{\BA^k}\times \BA^{n,m}\text{  and  }
X:=0\underset{\BA^k}\times \BA^n.$$

\medskip

Suppose that the Jacobi matrix of $f_1,...,f_k$ is non-degenerate when restricted to $X$.
I.e., the matrix $k\times (m+n)$-matrix $\partial_i(f_j)|_{X}$, viewed as a map 
$$\CO^{\oplus n+m}_{X}\to \CO^{\oplus k}_{X},$$
is a surjective map of vector bundles when restricted to $X$.

\medskip

From \corref{c:estimate on cotangent formal} and \corref{c:formal lft}, we obtain:

\begin{cor}  \label{c:Jacobi matrix ft}
Under the above circumstances, the DG indscheme $\CX$ is formally smooth.
\end{cor}

\medskip

We now claim:
\begin{prop} \label{p:Jacobi classical ft}
The DG indscheme $\CX$ is $0$-coconnective.
\end{prop}

\begin{proof}

Consider the \emph{scheme} $$^{\sim}\!\BA^{n,m}:=\on{Spec}\left(k[x_1,...,x_n][\![y_1,...,y_m]\!]\right)$$
and its map to $\BA^k$ given by $f_1,...,f_k$. The assumption on the Jacobi matrix implies that this
map is flat on a Zariski neighborhood $U$ of $X\subset {}^{\sim}{}\BA^{n,m}$. Therefore, the Cartesian
product \emph{taken in the category of DG schemes}
$$^{\sim}\CX\simeq 0\underset{\BA^k}\times U$$
is $0$-coconnective as a DG scheme. 

\medskip

The formal DG scheme $\CX$ is obtained from $^{\sim}\CX$ as a formal completion along $X$.
Since all the schemes involved are Noetherian, the assertion follows from \propref{p:class vs derived formal}.

\end{proof}

\sssec{}  \label{sss:elementary}

In what follows we shall refer to formal DG schemes $\CX$ of the type described in \secref{sss:Jacobi ft}
as \emph{elementary}.

\medskip

We shall say that a classical formal scheme is elementary if it is of the form $^{cl}\CX$ for $\CX$
an elementary formal DG scheme. 

\section{Classical vs. derived formal smoothness}

The focus of this section is the relation between the notions of formal smoothness
in the classical and derived contexts when $\CX$ is a DG indscheme. Namely, we would like
to know under what circumstances a formally smooth DG indscheme $\CX$ is $0$-coconnective,
i.e., arises as a left Kan extension from a classical indscheme. The reader may have observed
that this was the case in all the examples that we considered in \secref{ss:examples of formally smooth}.

\medskip

And vice versa, we would like to know when, for a classical formally smooth indscheme $\CX_{cl}$, the object 
$$\CX:={}^L\!\on{LKE}_{(\affSch)^{\on{op}}\hookrightarrow (\affdgSch)^{\on{op}}}(\CX_{cl})\in \inftydgstack$$
is a formally smooth DG indscheme. (Note that it is not clear that $\CX$ defined
as above is a DG indscheme, since the convergence condition is not a priori guaranteed.)

\medskip

Unfortunately, we do not have a general answer for this question even in the case of schemes: 
we do not even know that the DG scheme 
$$X:={}^L\!\on{LKE}_{(\affSch)^{\on{op}}\hookrightarrow (\affdgSch)^{\on{op}}}(X_{cl})$$
is smooth when $X_{cl}$ is a smooth classical scheme, except when $X_{cl}$ is locally of 
finite type. 

\ssec{The main result}

The main result of this section and the first of the two main results of this paper is a partial
answer to the above questions, under the assumption that our (DG) indschemes are locally (almost)
of finite type. 

\sssec{}

Let $\CX_{cl}$ be a classical formally smooth $\aleph_0$ indscheme. Assume that 
$\CX_{cl}$ is locally of finite type. Set 
$$\CX:={}^L\!\on{LKE}_{(\affSch)^{\on{op}}\hookrightarrow (\affdgSch)^{\on{op}}}(\CX_{cl})\in \on{PreStk}.$$

\medskip

We will prove:

\begin{thm} \label{t:classical vs derived ft}
Under the above circumstances, $\CX$ is a formally smooth DG indscheme.
\end{thm}

This theorem gives a partial answer to the second of the two questions above. We shall presently show
that it also gives a partial answer to the first question.

\sssec{}

We have the following observation:

\begin{prop} \label{p:classical and derived, general}
If $\CX$ is a formally smooth DG indscheme such that 
$$^L\!\on{LKE}_{(\affSch)^{\on{op}}\hookrightarrow (\affdgSch)^{\on{op}}}({}^{cl}\CX)\in \inftydgstack$$
is also a formally smooth DG indscheme, then the natural map
$$^L\!\on{LKE}_{(\affSch)^{\on{op}}\hookrightarrow (\affdgSch)^{\on{op}}}({}^{cl}\CX)\to \CX$$
is an isomorphism. In particular, $\CX$ is $0$-coconnective.
\end{prop}

\begin{proof}

By assumption, both sides in 
\begin{equation} \label{e:X' and X}
\CX':={}^L\!\on{LKE}_{(\affSch)^{\on{op}}\hookrightarrow (\affdgSch)^{\on{op}}}({}^{cl}\CX)\to \CX
\end{equation}
are formally smooth DG indschemes, and the above map induces an isomorphism of the
underlying classical indschemes.

\medskip

By deformation theory, it suffices to show that for every affine DG scheme
$S$ and a map $x':S\to \CX'$, the map
$$T^*_x\CX\to T^*_{x'}\CX'$$
is an isomorphism, where $x$ is the composition of $x'$ and the map \eqref{e:X' and X}.

\medskip

Using \propref{p:formal smoothness through deformations}(a), we obtain that it suffices to check
that the map
\begin{equation} \label{e:X' and X M}
T^*_{x'}\CX'(\CF)\to T^*_{x}\CX(\CF)
\end{equation}
is an isomorphism for every $\CF\in \QCoh(S)^\heartsuit$. Since any such $\CF$ comes
as a direct image from $^{cl}\!S$, this reduces the assertion to the case 
when $S$ is classical. 

\medskip

We have
$$T^*_{x'}\CX'(\CF)\simeq \on{Maps}(S_{\CF},\CX')\underset{\on{Maps}(S,\CX')}\times x',$$
and similarly for $T^*_{x}\CX(\CF)$. When $S$ is classical and $\CF\in \QCoh(S)^\heartsuit$,
the DG scheme $S_\CF$ is also classical. So, both sides of \eqref{e:X' and X M} only
depend on the restrictions of $\CX|_{\affSch}$ and $\CX'|_{\affSch}$, respectively, and,
hence are isomorphic by construction.

\end{proof}

\sssec{}

Combining \propref{p:classical and derived, general} and \thmref{t:classical vs derived ft}, we obtain:

\begin{thm}  \label{t:derived vs classical ft}
Let $\CX$ be a formally smooth DG indscheme, such that $^{cl}\CX:=\CX|_{\affSch}$ is locally of finite type
and $\aleph_0$. Then $\CX$ is $0$-coconnective, i.e., the natural map
$$^L\!\on{LKE}_{(\affSch)^{\on{op}}\hookrightarrow (\affdgSch)^{\on{op}}}({}^{cl}\CX)\to \CX$$
is an isomorphism.  Moreover, $\CX$ is locally almost of finite type and $\aleph_0$. 
\end{thm}

\begin{proof}
The first assertion is immediate. 

\medskip

Writing $^{cl}\CX$ as a colimit in $^{cl}\!\on{PreStk}$ of $X_\alpha$, 
with $X_\alpha$ being classical schemes closed in $^{cl}\CX$ and hence of finite type, we obtain that
$$\CX\simeq \underset{\alpha}{colim}\, X_\alpha,$$
where the colimit is taken in $\on{PreStk}$, and $X_\alpha$ are now understood as objects of
$\dgSch_{\on{aft}}$. Hence, $\CX\in \on{PreStk}_{\on{laft}}$. 

\medskip

The fact that $\CX$ is $\aleph_0$ also follows.
\end{proof} 

Thus, we obtain:

\begin{cor}  \label{c:main cor}
There exists an equivalence of categories between the category of classical formally smooth
$\aleph_0$ indschemes locally of finite type and that of formally smooth $\aleph_0$ DG 
indschemes locally almost of finite type.  
\end{cor}

\sssec{}

Prior to proving \thmref{t:classical vs derived ft}, let us see some of its corollaries in concrete geometric
situations.

\ssec{Loop spaces}  \label{sss:loops setting}

\sssec{} Let $Z$ be an object of $\inftydgprestack$. We define the objects
$Z[t]/t^k$, $Z\qqart$ and $Z\ppart$ of $\inftydgprestack$ as follows: for $S=\Spec(A)\in \affdgSch$,
$$\Maps(S,Z[t]/t^k):=\Maps(\Spec(A[t]/t^k),Z),\,\, 
\Maps(S,Z\qqart):=\Maps(\Spec(A\qqart),Z)$$
and 
$$\Maps(S,Z\ppart):=\Maps(\Spec(A\ppart),Z).$$

Note that by definition,
$$Z\qqart\simeq \underset{k}{lim}\, Z[t]/t^k,$$
as objects of $\inftydgprestack$.

\begin{lem}  \label{l:loops formally smooth}
Assume that $Z$ is formally smooth as an object of $\inftydgprestack$. Then so are
$Z[t]/t^k$, $Z\qqart$ and $Z\ppart$. 
\end{lem}

\begin{proof}
This is immediate from the fact that for a DG algebra $A$, the maps
$$\tau^{\leq n}(A[t]/t^k)\to (\tau^{\leq n}(A))[t]/t^k,\,\, \tau^{\leq n}(A\qqart)\to (\tau^{\leq n}A)\qqart \text{ and }
\tau^{\leq n}(A\ppart)\to (\tau^{\leq n}A)\ppart$$
are isomorphisms,
and that for a surjection of classical algebras $A_1\twoheadrightarrow A_2$ with a nilpotent kernel, the corresponding maps
$$A_1[t]/t^k\to A_2[t]/t^k,\,\, A_1\qqart\to A_2\qqart \text{ and } A_1\ppart\to A_2\ppart$$
have the same property. 
\end{proof}

\sssec{}

From now on we are going to consider the case when $Z\in \dgSch_{\on{aft}}$. We have:

\begin{prop} \label{p:repr of loops}
Under the above circumstances, we have:

\smallskip

\noindent{\em(a)} $Z[t]/t^k\in \dgSch_{\on{aft}}$, and is affine if $Z$ is affine.

\smallskip

\noindent{\em(b)} $Z\qqart\in \dgSch$, and it is affine if $Z$ is affine.

\smallskip

\noindent{\em(c)} If $Z$ is affine, then $Z\ppart$ is a DG indscheme. 

\end{prop}

\begin{proof}

For all three statements, it is enough to assume that $Z$ is affine. Note that 
$Z[t]/t^k$, $Z\qqart$ and $Z\ppart$, considered as objects of $\inftydgprestack$ are convergent. 
Hence, it is sufficient to show that 
$$^{\leq n}(Z[t]/t^k):=Z[t]/t^k|_{^{\leq n}\!\affdgSch},\,\, ^{\leq n}(Z\qqart):=Z\qqart|_{^{\leq n}\!\affdgSch}  \text{ and }$$
$$^{\leq n}(Z\ppart):=Z\ppart|_{^{\leq n}\!\affdgSch}$$
are representable by objects from $^{\leq n}\!\affdgSch$ (for the first two) 
and $^{\leq n}\!\dgindSch$, respectively. Note that the above objects only depend
on the truncation $^{\leq n}\!Z$. The assertions of the proposition result from combining the following observations: 

\medskip

\noindent(i) The assignments $Z\mapsto Z[t]/t^k$, $Z\mapsto Z\qqart$ and $Z\mapsto Z\ppart$ commute with limits.

\smallskip

\noindent(ii) Every object of $^{\leq n}\!\affdgSch_{\on{ft}}$ can be obtained as the
totalization of a truncated cosimplicial object whose terms are 
isomorphic to affine spaces $\BA^n$. 

\smallskip

\noindent(iii) The subcategories $$^{\leq n}\!\affdgSch_{\on{ft}}\subset {}^{\leq n}\!\affdgSch\subset {}{}^{\leq n}\!\inftydgprestack \text{ and }
^{\leq n}\!\dgindSch\subset {}^{\leq n}\!\inftydgprestack$$ are stable
under finite products.

\smallskip

\noindent(iv) For $Z=\BA^n$, both assertions of the proposition are manifest.

\end{proof}

\sssec{}

Suppose now that $Z$ is a classical scheme which is smooth over $k$ (and in particular,
locally of finite type). We have:

\begin{prop} \label{p:arc classical}
The DG schemes $Z[t]/t^k$ and $Z\qqart$ are $0$-coconnective.
\end{prop}

\begin{proof}

To prove that $Z[t]/t^k$ is $0$-coconnective, by \propref{p:classical and derived, general} and \secref{sss:classical smooth schemes},
it is sufficient to show that the classical scheme 
$$^{cl}(Z[t]/t^k):=Z[t]/t^k|_{\affSch}$$
is smooth. By \lemref{l:loops formally smooth}, $^{cl}(Z[t]/t^k)$ is formally smooth as a classical scheme,
which implies that it is smooth, since $^{cl}(Z[t]/t^k)$ is locally of finite type by \propref{p:repr of loops}(a).

\medskip

To treat the case of $Z\qqart$, we will have to go back to the proof of \propref{p:repr of loops}. We can assume
that $Z$ is affine and that it fits into a Cartesian square \eqref{e:smooth scheme}. Hence,
we have a Cartesian square
$$
\CD
Z\qqart @>>>  \BA^n\qqart  \\
@VVV   @VV{f\qqart}V  \\
0  @>>>  \BA^m\qqart.
\endCD
$$
Since the affine schemes $\BA^n\qqart$ and $\BA^m\qqart$ are $0$-coconnective, to show that $Z\qqart$ is also
$0$-coconnective, it suffices to show that the map $f\qqart$ is flat. The latter is the limit of the maps
$f[t]/t^k:\BA^n[t]/t^k\to \BA^m[t]/t^k$, and smoothness of $f$ implies that each of these maps is flat. Hence,
$f[t]/t^k$ is flat as well.

\end{proof}

\sssec{Question}

What are the conditions on a classical scheme of finite type $Z$
(viewed as a $0$-coconnective DG scheme), that will guarantee that $Z\qqart$ will 
also be $0$-coconnective?

\medskip

It is easy to see that this is not always the case: for instance, consider $Z=\Spec(k[t]/t^2)$. 
However, the smoothness condition on $Z$ is not necessary, as can be seen from the
following example: 

\medskip

Let $\fg$ be a semi-simple Lie algebra, and let $\CN\subset \fg$
be its nilpotent cone. We have:

\begin{cor}
The DG scheme $\CN\qqart$ is $0$-coconnective. 
\end{cor}

\begin{proof}

By definition, $\CN$ fits into a Cartesian square
$$
\CD
\CN @>>>  \fg \\
@VVV   @VV{\varpi}V  \\
0  @>>>  \fg/\!/G,
\endCD
$$
\emph{taken in the category of classical schemes}, 
where $\fg/\!/G$ is the GIT quotient of $\fg$ by the adjoint action of $G$,
i.e., $\Spec(\on{Sym}(\fg^*)^G)$, and $\varpi$ is the Chevalley map.

\medskip

However, by Kostant's theorem, the map $\varpi$ is flat, so the above
square is also Cartesian in the category of DG schemes. Hence,
we have a Cartesian square
$$
\CD
\CN\qqart @>>>  \fg\qqart \\
@VVV   @VV{\varpi\qqart}V  \\
0  @>>>  \fg/\!/G\qqart.
\endCD
$$

Since $\fg$ and $\fg/\!/G$ are smooth schemes of finite type
(in fact, isomorphic to affine spaces), the DG schemes 
$\fg\qqart$ and $\fg/\!/G\qqart$ are $0$-coconnective. Hence, to show
that $\CN\qqart$ is $0$-coconnective, it suffices to know that the map $\varpi\qqart$
is flat. However, the latter is Theorem A.4 in \cite{EF}.

\end{proof}

\sssec{The case of loops}

Let $Z$ be an affine smooth scheme of finite type over the ground field. We propose: 

\begin{conj} \label{conj:formal smoothness}
The DG indscheme $Z\ppart$ is $0$-coconnective.
\end{conj}

In this next subsection we will prove this conjecture in a particular case when $Z$ is an algebraic group $G$.

\ssec{Loop groups and the affine Grassmannian}

\sssec{}

Let $G$ be an algebraic group. We define the affine Grassmannian $\Gr_G$ as an object of $\inftydgprestack$
as follows:

\smallskip

$\Maps(\Spec(A),\Gr_G)$ is the $\infty$-groupoid of principal $G$-bundles on $\Spec(A\qqart)$ equipped
with a trivialization over $\Spec(A\ppart)$. 

\smallskip

It is easy to show that $\Gr_G$ is convergent and that it belongs to
$\inftydgstack$ (i.e., it satisfies fppf descent).
We have a naturally defined map $G\ppart\to \Gr_G$, which identifies $\Gr_G$ with the
quotient of $G\ppart$ by $G\qqart$ in the fppf and the \'etale topology (indeed, it is easy to see that
any $G$-bundle on $\Spec(A\qqart)$ admits a trivialization after an \'etale localization with 
respect to $\Spec(A)$).

\medskip

It is well-known that the underlying object $^{cl}\!\Gr_G\in {}^{cl}\!\inftystack$ is a classical
indscheme, which is $\aleph_0$ and locally of finite type. 
 
\begin{prop} \label{p:aff gr}
$\Gr_G$ is a DG indscheme. Moreover, it is formally smooth.  
\end{prop}

\begin{proof}
To prove that $\Gr_G$ is a DG indscheme, we will apply \thmref{t:char by deform}. For $S\in \affdgSch$
and a point $g:S\to \Gr_G$ we need to study the category of extensions of $g$ to a point $g':S'\to \Gr_G$
for square-zero extensions $S\hookrightarrow S'$. The question is local in the \'etale topology on $S$,
so we can assume that the point $g$ admits a lift to a point $\tilde{g}:S\to G\ppart$. Multiplication
by $\tilde{g}$ defines a map 
$$\on{SplitSqZExt}(S,1_{G\qqart})\to \on{SplitSqZExt}(S,\tilde{g}),$$
where $1_{G\qqart}:S\to G\qqart$ is the constant map to the unit point of $G\qqart$. Consider the
corresponding map
$$\alpha:T^*_{\tilde{g}}G\ppart\to T^*_{1_{G\qqart}}G\qqart.$$
We claim that $\on{Cone}(\alpha)[-1]$ represents $T^*_g\Gr_G$. This follows from the
fact that any extension $g':S'\to \Gr_G$ also admits a lift to $\tilde{g}':S'\to G\ppart$ 
and if $\CF\in \QCoh(S)$ is the ideal of $S$ inside $S'$, the ambiguity for such lift is given 
by the fiber of $\on{SplitSqZExt}(S,1_{G\qqart})$ over $S_\CF\in \on{SplitSqZExt}(S)$. 
This shows that $\Gr_G$ satisfies scheme-like Conditions (A) and (C), while Condition (B)
follows from the construction. 

\medskip

To show that $\Gr_G$ is formally smooth, it suffices to show that $\on{Cone}(\alpha)[-1]$
satisfies property (b) of \propref{p:formal smoothness through deformations}. Multiplication
by the inverse $\tilde{g}$ defines an isomorphism between $\on{Cone}(\alpha)$
and the situation when $\tilde{g}=1_{G\ppart}$. In the latter case, $\on{Cone}(\alpha)$ 
isomorphic to the object of $\on{Pro}(\Coh(S)^{\heartsuit})$ equal to $\underset{\alpha}{``lim"}\, \CO_S\otimes V^*_\alpha$,
where $\alpha\mapsto V_\alpha$ is the filtered family of finite-dimensional $k$-vector spaces, such that
$$\underset{\alpha}{colim}\, V_\alpha\simeq \fg\ppart/\fg\qqart.$$

\end{proof}

\sssec{}

Let us now observe the following corollary of \thmref{t:derived vs classical ft}: 

\begin{thm} \label{t:aff gr}
$\Gr_G$ is $0$-coconnective. Moreover, it is weakly $\aleph_0$, and locally almost of finite type. 
\end{thm}

We shall now use \thmref{t:aff gr} to prove the following:

\begin{thm}  \label{t:loop group classical}
The indscheme $G\ppart$ is $0$-coconnective.
\end{thm}

\begin{proof}

Let $f:\CX_1\to \CX_2$ be a map in $\inftydgstack$ such that $\CX_2$ is $0$-coconnective,
and for any $S\in \affSch$ and a map $S\to \CX_2$, the fiber product
$S\underset{\CX_1}\times \CX_2\in \inftydgstack$ is also $0$-coconnective. 

\begin{lem}
Under the above circumstances, $\CX_2$ is also $0$-coconnective.
\end{lem} 

We apply this lemma to $\CX_1=G\ppart$ and $\CX_2=\Gr_G$. It remains to verify that
for a classical affine scheme $S$ and a map $g:S\to \Gr_G$, the fiber product
$S\underset{\Gr_G}\times G\ppart$ is $0$-coconnective. The question is local in the \'etale
topology on $S$. Hence, we can assume that $g$ admits a lift to an $S$-point of $G\ppart$.
However, this left defines an isomorphism
$$S\underset{\Gr_G}\times G\ppart\simeq S\times G\qqart,$$
and the assertion follows from \thmref{t:aff gr} and \propref{p:arc classical}.

\end{proof}

\ssec{The (pro)-cotangent complex of a classical formally smooth (ind)scheme}

For the proof of \thmref{t:classical vs derived ft} we will need to establish several facts concerning the 
pro-cotangent complex of classical formally smooth indschemes. 

\sssec{}

Let $\CX_{cl}$ be a classical indscheme; set
$$\CX:={}^L\!\on{LKE}_{(\affSch)^{\on{op}}\hookrightarrow (\affdgSch)^{\on{op}}}(\CX_{cl})\in \inftydgprestack.$$

\begin{prop} \label{p:estimate on cl cotangent}
The indscheme $\CX_{cl}$ is classically formally smooth if and only if for every $S\in \affSch$ and $x:S\to \CX_{cl}$,
the object 
$$^{\geq -1}(T^*_x\CX)\in \on{Pro}(\QCoh(S)^{\geq -1,\leq 0})$$
is pro-projective. 
\end{prop}

\begin{proof} 

Let $^{\geq -1}(T^*_x\CX)\in \on{Pro}(\QCoh(S)^{\geq -1,\leq 0})$ be pro-projective. By \lemref{l:cl sq zero},
it suffices to show that
$$\pi_0\left(\Maps({}^{\geq -1}(T^*_x\CX),\CF[1])\right)=0$$
for $\CF\in \QCoh(S)^\heartsuit$. However, the latter is given by condition (a) \lemref{l:when pro proj one}. 

\medskip

For the opposite implication, let us assume that $\CX_{cl}$ is formally smooth. We will
check that $^{\geq -1}(T^*_x\CX)$ satisfies condition (b') of \lemref{l:when pro proj one}.
Let $x$ be a map $S\to \CX$, where $S\in \affSch$. 

\medskip

The fact that the functor
$$\CF\mapsto \pi_0\left(\Maps({}^{\geq -1}(T^*_x\CX),\CF)\right),\quad \QCoh(S)^\heartsuit\to \on{Sets}$$
is right exact follows from the assumption on $\CX_{cl}$ and the 
definition of $T^*_x\CX$
in terms of split square-zero extensions in \secref{sss:condition A}. Hence, it remains to
show that $H^{-1}(T^*_x\CX)=0$.

\medskip

Let $$\CX_{cl}\simeq \underset{\alpha\in \sA}{colim}\, X_{\alpha},$$
where $X_{\alpha}\in \Sch_{\on{qsep-qc}}$. Let $\alpha_0$ be an index
such that $x$ factors through a map $x_{\alpha_0}:S\to X_{\alpha_0}$. 

\medskip

Since the t-structure on $\on{Pro}(\QCoh(S))$ is Zariski-local, we can assume that
the map $x_{\alpha_0}$ factors as
$$S\to U_{\alpha_0}\overset{j}\hookrightarrow X_{\alpha_0},$$
where $U_{\alpha_0}$ is an open affine inside $X_{\alpha_0}$. 

\medskip

Let $\iota_{\alpha_0}:X_{\alpha_0}\to \CX_{cl}$ denote the tautological map.
For $(\alpha_0\to \alpha)\in \sA$, let $\iota_{\alpha_0,\alpha}$ denote 
the corresponding closed embedding $X_{\alpha_0}\to X_{\alpha}$. 

\medskip

It is easy to see that it is sufficient to show that
$$H^{-1}\left(T^*_{\iota_{\alpha_0}\circ j}\CX\right)=0$$
as an object of $\on{Pro}(\QCoh(U_{\alpha_0})^\heartsuit)$. 

\medskip

By \eqref{e:expl cotangent indsch}, we have
$$H^{-1}\left(T^*_{\iota_{\alpha_0}\circ j}\CX\right)\simeq
\underset{\alpha\in \sA_{\alpha_0/}}{``lim"}\, H^{-1}\left(T^*_{\iota_{\alpha_0,\alpha}\circ j}X_\beta\right).$$

\medskip

So, we need to show that for a given
$\CM\in \QCoh(U_{\alpha_0})^\heartsuit$, $\alpha_0\to \alpha$ and 
$$\phi:H^{-1}(T^*_{\iota_{\alpha_0,\alpha}\circ j}X_\alpha)\to \CM,$$
there exists $\alpha\to \beta$, such that the composition
$$H^{-1}(T^*_{\iota_{\alpha_0,\beta}\circ j}X_{\beta})\to 
H^{-1}(T^*_{\iota_{\alpha_0,\alpha}\circ j}X_\alpha)\overset{\phi}\to \CM$$
vanishes. 

\medskip

Since $\iota_{\alpha_0,\alpha}$ is a closed embedding, 
$U_{\alpha_0}$ is the pre-image of an open affine $U_\alpha$ in $X_\alpha$.
Replacing $\CM$ by its direct image under $U_{\alpha_0}\to U_\alpha$, 
we can assume that $\alpha=\alpha_0$. Further, embedding $\CM$ into an injective
sheaf, we can assume that the map $\phi$ extends to a map $\psi:T^*U_\alpha\to \CM[1]$.
We wish to find an index $\beta\in \sA_{\alpha/}$ such that the composition
\begin{equation} \label{e:split beta}
T^*_{\iota_{\alpha,\beta}\circ j}X_\beta\to T^*U_\alpha\overset{\psi}\to \CM[1]
\end{equation}
vanishes.

\medskip

However, the data of $\psi$ as above is equivalent to that of a square-zero extension
$U'_\alpha$ of $U_\alpha$. And the data of a splitting of \eqref{e:split beta} is equivalent
to that of an extension of the map $\iota_{\alpha,\beta}\circ j:U_\alpha\to X_\beta$ to a map
$U'_\alpha\to X_\beta$. Thus, giving such an index $\beta$ is equivalent to
extending the map $U_\alpha\to \CX_{cl}$ to a map $U'_\alpha\to \CX_{cl}$. The
existence of such an extension follows from the classical formal smoothness
of $\CX_{cl}$.

\end{proof}

\begin{cor} \label{c:estimate on cl cotangent}
Let $X_{cl}$ be a classical scheme, and consider it as a DG scheme. 
Then $X_{cl}$ is classically formally smooth if and only if $T^*X_{cl}$ satisfies:

\smallskip

\noindent{\em(a)} $H^{-1}(T^*X_{cl})=0$.

\smallskip

\noindent{\em(b)} $H^0(T^*X_{cl})$ is projective over every affine subscheme 
of $X_{cl}$.

\end{cor}

\begin{proof}

We only need to show that if $S$ is an affine scheme mapping to $X_{cl}$, then the pull-back of
$T^*X_{cl}$ to it is projective. By assumption, we know this locally in the Zariski topology on $S$.
The assertion now follows from the theorem of Raynaud-Gruson mentioned earlier that
projectivity of a module is a Zariski-local property.

\end{proof}

\sssec{}

The following somewhat technical assertion will be needed in the sequel:

\medskip

Let $\CX_{cl}$ be a classical formal scheme with the underlying reduced scheme $X$,
and set 
$$\CX:=^L\!\on{LKE}_{(\affSch)^{\on{op}}\hookrightarrow (\affdgSch)^{\on{op}}}(\CX_{cl}).$$

\begin{cor} \label{c:estimate on cl cotangent formal}
Suppose that $\CX_{cl}$ is locally of finite type, and that $^{\geq -1}(T^*\CX|_X)$ is locally 
pro-projective. Then $\CX_{cl}$ is classically formally smooth.
\end{cor}

This follows from \propref{p:estimate on cl cotangent} in the same way as \corref{c:estimate on cotangent formal} 
follows from \propref{p:formal smoothness through deformations}.

\ssec{Classical formally smooth indschemes locally of finite type case}

In this subsection we will reduce \thmref{t:classical vs derived ft} to a key
proposition (\propref{p:descr of formal completion modified}) that describes
the general shape of formal classical schemes locally of finite type. 

\sssec{}

In fact, we will prove the following stronger assertion. Let $\CX_{cl}$ be as in \thmref{t:classical
vs derived ft}, and let ${}^{\on{Zar}}\CX^{\wedge}_Y$ denote the \emph{Zariski} sheafification of the prestack
$\!\on{LKE}_{(\affSch)^{\on{op}}\hookrightarrow (\affdgSch)^{\on{op}}}(\CX_{cl})$. We will show that
${}^{\on{Zar}}\CX$ is a formally smooth DG indscheme. Indeed, in this case, by \propref{p:indschemes fppf}, the natural
map $$ {}^{\on{Zar}}\CX \to \CX $$ is then an isomorphism.

\sssec{}
In what follows, let $\on{ZarStk} \subset \inftydgprestack$ be the full subcategory of prestacks satisfying Zariski descent,
$L_{\on{Zar}}: \on{PreStk} \to \on{ZarStk}$ denote the left adjoint to the inclusion, i.e. Zariski sheafification, and
$$ ^{L_{\on{Zar}}}\!\on{LKE}_{(\affSch)^{\on{op}}\hookrightarrow (\affdgSch)^{\on{op}}} := L_{\on{Zar}} \circ \on{LKE}_{(\affSch)^{\on{op}}\hookrightarrow (\affdgSch)^{\on{op}}}.$$

\sssec{}
Now, let $\CX_{cl}\simeq \underset{\alpha}{colim}\, X_{\alpha}$, be a presentation of $\CX_{cl}$ where the 
colimit is taken in $^{cl}\!\on{PreStk}$.  Recall that if $Y_{cl}$ is a classical scheme, then
$$^{L_{\on{Zar}}}\!\on{LKE}_{(\affSch)^{\on{op}}\hookrightarrow (\affdgSch)^{\on{op}}}({}^{cl}(Y_{cl})) 
\simeq ^{L}\!\on{LKE}_{(\affSch)^{\on{op}}\hookrightarrow (\affdgSch)^{\on{op}}}({}^{cl}(Y_{cl}))$$
is a DG scheme.  Therefore,
$$ {}^{\on{Zar}}\CX \simeq \underset{\alpha}{colim}\, {}^L\!\on{LKE}_{(\affSch)^{\on{op}}\hookrightarrow 
(\affdgSch)^{\on{op}}}{}^{cl}(X_{\alpha}), $$
where the colimit is taken in the $\on{ZarStk}$.  Now, by \propref{p:indschemes fppf n}, we 
have that for every $n$
$$ ^{\leq n}({}^{\on{Zar}}\CX) \simeq \underset{\alpha}{colim}\,  
^{^{\leq n}\!L}\!\on{LKE}_{(\affSch)^{\on{op}}\hookrightarrow (\affdgSch)^{\on{op}}}{}^{cl}(X_{\alpha}) $$
where the colimit is taken in $^{\leq n}\!\inftydgprestack$.  Thus, $^{\leq n}({}^{\on{Zar}}\CX)$ is a 
$\nDG$ indscheme.  Hence, by definition, ${}^{\on{Zar}}\CX$ is a DG indscheme iff it is convergent.  In 
particular, to prove \thmref{t:classical vs derived ft}, it suffices to prove that ${}^{\on{Zar}}\CX$ is a 
formally smooth prestack.

\medskip

Let $Y$ be a reduced classical scheme and let $Y\subset {}^{red}(\CX_{cl})$ be a closed embedding.
(Note that such a $Y$ is automatically locally of finite type.) By \propref{p:formal smoothness of formal
completions}, in order to prove that ${}^{\on{Zar}}\CX$ is a formally smooth, it suffices to show that that
the formal completion ${}^{\on{Zar}}\CX^{\wedge}_Y$ is formally smooth.

Now, let $\sqcup U_i \to Y$ be a Zariski cover of $Y$ by affine schemes.
By \corref{c:completion along open embedding}, we have that
$$ \sqcup\ {}^{\on{Zar}}\CX^{\wedge}_{U_i} \to {}^{\on{Zar}}\CX^{\wedge}_Y $$
is a Zariski cover. It will suffice to
show that each ${}^{\on{Zar}}\CX^{\wedge}_{U_i}$ is formally smooth (and, in particular, convergent).
Indeed, by the same argument as \cite[Proposition 4.5.2]{Stacks}, convergence is Zariski local. Hence, in
this case, ${}^{\on{Zar}}\CX^{\wedge}_Y$ is convergent and therefore, by the same argument as above, a DG
indscheme. In particular, by \corref{c:indsch have def theory}, it admits connective deformation theory.
Finally, by \propref{p:formal smoothness through deformations} and \lemref{l:locality of proj}, formal
smoothness of the ${}^{\on{Zar}}\CX^{\wedge}_{U_i}$ implies that $ {}^{\on{Zar}}\CX^{\wedge}_Y$ is formally
smooth.  Thus, replacing $Y$ by an $U_i$ (and $\CX_{cl}$ appropriately), we reduce to the case that $Y$ is 
affine.

\sssec{}
We will prove that $\CX^{\wedge}_Y$ is formally smooth by 
quoting/reproving the following result (see \cite[Proposition 7.12.23]{BD}). 
This proposition will also be useful to us in the sequel. 

\begin{prop} \label{p:descr of formal completion modified}
Let $\CZ_{cl}$ be a classical formal scheme. Assume that:

\begin{itemize}

\item As a classical indscheme, $\CZ_{cl}$ is locally of finite type and $\aleph_0$. 

\item The classical scheme $^{red}(\CZ_{cl})$ is affine.

\item $\CZ_{cl}$ is classically formally smooth.

\end{itemize}

Then $\CZ_{cl}$ is isomorphic 
to retract of a filtered colimit, taken in $^{cl}\!\inftydgprestack$, of classical formal schemes
each of which is \emph{elementary} (see \secref{sss:elementary}). 
\end{prop}

\sssec{}

Let us deduce \thmref{t:classical vs derived ft} from \propref{p:descr of formal completion modified}.
This will be done via a series of lemmas. First, we have:

\begin{lem}  \label{l:classical vs derived on ind}
For any $\CX_{cl}\in {}^{\leq 0}\!\dgindSch_{\on{lft}}$,
a reduced classical scheme $Y$ and a closed embedding $Y\subset {}^{red}(\CX_{cl})$, the
canonical map 
$$^{L_{\on{Zar}}}\!\on{LKE}_{(\affSch)^{\on{op}}\hookrightarrow (\affdgSch)^{\on{op}}}({}^{cl}(\CX^\wedge_Y))\to \CX^\wedge_Y,$$
where $\CX:={}^{L_{\on{Zar}}}\!\on{LKE}_{(\affSch)^{\on{op}}\hookrightarrow (\affdgSch)^{\on{op}}}(\CX_{cl})$, is an isomorphism.
\end{lem}

\begin{proof} 

Let $\CX_{cl}\simeq \underset{\alpha}{colim}\, X_{\alpha}$, where the colimit is taken in $^{cl}\!\on{PreStk}$, and each $X_\alpha$ is a finite type scheme.
Then
$$\CX\simeq \underset{\alpha}{colim}\, X_\alpha,$$
where the colimit is taken in $\on{ZarStk}$. 

\medskip

Without loss of generality, we can assume that $Y$ is contained in each $X_{\alpha}$. Then 
$${}^{cl}(\CX^\wedge_Y)\simeq \underset{\alpha}{colim}\, {}^{cl}((X_{\alpha})^\wedge_Y),$$
and the left-hand side in the lemma is
$$\underset{\alpha}{colim}\, 
{}^{L_{\on{Zar}}}\!\on{LKE}_{(\affSch)^{\on{op}}\hookrightarrow (\affdgSch)^{\on{op}}}{}^{cl}((X_{\alpha})^\wedge_Y),$$
where the colimit is taken in $\on{ZarStk}$.

\medskip

We now claim that the map
$$L_{\on{Zar}}\left(\underset{\alpha}{colim}\, ((X_\alpha)^{\wedge}_Y)\right)\to
\left(L_{\on{Zar}}(\underset{\alpha}{colim}\, X_\alpha)\right)^{\wedge}_Y=\CX^\wedge_Y,$$
where both colimits are taken in $\on{ZarStk}$,
is an isomorphism. This follows from \lemref{l:compl and tau sheafification} and 
\secref{sss:remarks on compl}(iii).

\medskip

Thus, the right-hand side in the lemma identifies with 
$\underset{\alpha}{colim}\, (X_\alpha)^{\wedge}_Y$, where the colimit is taken in 
$\on{ZarStk}$.

\medskip

Hence, to prove the lemma, it suffices to show that for every $\alpha$, the canonical map
$$^{L_{\on{Zar}}}\!\on{LKE}_{(\affSch)^{\on{op}}\hookrightarrow (\affdgSch)^{\on{op}}}{}^{cl}((X_{\alpha})^\wedge_Y)\to
(X_\alpha)^{\wedge}_Y$$
is an isomorphism. However, this is the content of \propref{p:class vs derived formal} (see Remark \ref{r:derived formal as zar sheafification}),
which is applicable since $X_{\alpha}$ is of finite type, and in particular, Noetherian. 

\end{proof}

\begin{lem} \label{l:LKE of formal completion}
Let $\CZ_{cl}$ be as in \propref{p:descr of formal completion modified}.
Then
$$^{L_{\on{Zar}}}\!\on{LKE}_{(\affSch)^{\on{op}}\hookrightarrow (\affdgSch)^{\on{op}}}(\CZ_{cl})$$
is a formally smooth DG indscheme.
\end{lem}

Let us assume this lemma and finish the proof of \thmref{t:classical vs derived ft}. 

\begin{proof}[Proof of \thmref{t:classical vs derived ft}]

We need to show that $^{\on{Zar}}\CX^\wedge_Y$ is formally smooth. By \lemref{l:classical vs derived on ind},
this is equivalent to $^{L_{\on{Zar}}}\!\on{LKE}_{(\affSch)^{\on{op}}\hookrightarrow (\affdgSch)^{\on{op}}}({}^{cl}(\CX^\wedge_Y))$ being formally smooth. 
The required assertion follows \lemref{l:LKE of formal completion} applied to 
$\CZ_{cl}:={}^{cl}(\CX^\wedge_Y)$.

\end{proof}

\sssec{Proof of  \lemref{l:LKE of formal completion}}

Since the notion of formal smoothness is stable under taking retracts, 
we can assume that $$\CZ_{cl}\simeq \underset{\alpha}{colim}\, \CZ_{\alpha,cl},$$
(colimit taken in $^{cl}\!\inftydgprestack$),
where each $\CZ_{\alpha,cl}$ is elementary. 

\medskip

Hence, $\CZ:={}^{L_{\on{Zar}}}\!\on{LKE}_{(\affSch)^{\on{op}}\hookrightarrow (\affdgSch)^{\on{op}}}(\CZ_{cl})$ is isomorphic to 
$$\underset{\alpha}{colim}\, \CZ_{\alpha},$$
where the colimit is taken in $\on{ZarStk}$, and
$$\CZ_\alpha:={}^{L_{\on{Zar}}}\!\on{LKE}_{(\affSch)^{\on{op}}\hookrightarrow (\affdgSch)^{\on{op}}}(\CZ_{\alpha,cl}).$$
By Proposition \ref{p:Jacobi classical ft} 
combined with \corref{c:Jacobi matrix ft}, each $\CZ_\alpha$ is a formally smooth DG indscheme. 
Thus, it remains to prove the following:

\begin{lem}
Let $\alpha\mapsto \CZ_\alpha$ be filtered family of objects if $\inftydgstack$, each of
which is formally smooth as an object of $\inftydgprestack$. Assume that for every $n$,
all $\CZ_\alpha|_{^{\leq n}\!\affdgSch}$ are $k$-truncated for some $k$. Then
$\CZ:=\underset{\alpha}{colim}\, \CZ_\alpha$ is a stack and is also formally smooth as an object of 
$\inftydgprestack$, where the colimit is taken in $\inftydgprestack$. In particular, it is also the colimit in $\inftydgstack$ (and hence, also in $\on{ZarStk}$).
\end{lem}

\begin{proof}
 Since
homotopy groups commute with filtered colimits,
we obtain that $\CZ$
is formally smooth as an object of $\inftydgprestack$. In particular,
it is convergent.
It remains to show that $\CZ$ satisfies descent.
By convergence, it is enough to check the
descent condition on $^{\leq n}\!\affdgSch$. But the latter follows from
the truncatedness assumption by \lemref{l:no sheafification}.
\end{proof}

\qed (\lemref{l:LKE of formal completion})

\ssec{Proof of the key proposition}

In this subsection we will prove \propref{p:descr of formal completion modified},
reproducing a slightly modified argument from \cite{BD}, pages 328-331.
\footnote{The reason that we include the proof instead of just quoting the result
from \cite{BD} is that it seems that the considerations that involve derived pro-cotangent 
spaces that we introduce help to make the argument of {\it loc. cit.} more 
conceptual.}

\medskip

\begin{rem}
Note that the statement of \cite[Proposition 7.12.23]{BD} is slightly stronger: it asserts that, under the (innocuous) additional
assumption that $^{red}(\CZ_{cl})$ is connected, we have an isomorphism
$$\CZ_{cl}\simeq \CZ_{0,cl}\times {}^{cl}\!\left(\on{Spf}\left(k[\![z_1,z_2,\ldots]\!]\right)\right),$$
where $\CZ_{0,cl}$ is elementary. The reason we choose the formulation given in
\propref{p:descr of formal completion modified} is that it makes it more amenable for
generalization in the non-finite type situation. 
\end{rem}

\sssec{Step 0: initial remarks} 

Denote $Z:={}^{red}(\CZ_{cl})$ and $\CZ:={}^L\!\on{LKE}_{(\affSch)^{\on{op}}\hookrightarrow (\affdgSch)^{\on{op}}}(\CZ_{cl})$.

\medskip

By \propref{p:char finite type new}, the finite type condition implies that 
the object $T^*\CZ|_Z$ belongs to $\on{Pro}(\Coh(Z)^{\leq 0})$, where
$$T^*\CZ|_Z:=T^*_z\CZ,$$
where $z:Z\to \CZ$ is the tautological map. \propref{p:crit for aleph 0} implies that 
$T^*\CZ|_Z$ is $\aleph_0$ as an object of $\on{Pro}(\Coh(Z)^{\leq 0})$.

\medskip

By \propref{p:estimate on cl cotangent} and \cite[Proposition 7.12.6(iii)]{BD}, $T^*\CZ|_Z$ 
is the dual of a Mittag-Leffler quasi-coherent sheaf $\CM$ on $Z$. By
\cite[Theorem 7.12.8]{BD}, the $\aleph_0$ condition implies that $\CM$ is actually projective. 

\medskip

Multiplying $\CZ$ by a suitable formally smooth
classical indscheme as in \cite[Proposition 7.12.14]{BD}, we can assume that $\CM$ is
a free countably generated $\CO_Z$-module. \footnote{This last procedure is the reason the
word "retract" appears in the formulation of \propref{p:descr of formal completion modified}.}

\medskip

Thus, we obtain that we can assume that
\begin{equation} \label{e:cotangent as proj lim}
H^0(T^*\CZ|_Z)\in \on{Pro}(\Coh(Z)^{\heartsuit})
\end{equation}
can be represented as $P:=\underset{k\in \BN}{``lim"}\, P_k$, where $P_k$ are locally free 
(in fact, free) sheaves on $Z$ of finite rank, and the maps $P_{k+1}\to P_k$ are surjective.

\medskip

Let us write $\CZ_{cl}\simeq \underset{n\in \BN}{colim}\, Z_n$ with $Z=Z_0$. Let 
$\CI_n$ denote the sheaf of ideals of $Z$ in $Z_n$. The finite type hypothesis
implies that $\CI_n\in \Coh(Z_n)^\heartsuit$. Consider
$\CI_n|_Z\simeq \CI_n/\CI_n^2\in \Coh(Z)^\heartsuit$ and denote
$$\CI|_Z:=\underset{n}{``lim"}\, \CI_n|_Z\in \on{Pro}(\Coh(Z)^{\heartsuit}).$$
By \propref{p:estimate on cl cotangent}, the long exact cohomology sequence 
for the map $Z\hookrightarrow \CZ$ gives rise to a 4-term exact
sequence in $\on{Pro}(\Coh(Z)^\heartsuit)$:
\begin{equation} \label{e:4 term seq}
0\to H^{-1}(T^*Z)\to \CI|_Z\to H^0(T^*\CZ|_Z)\to H^0(T^*Z)\to 0.
\end{equation}

\sssec{Step 1: "the finite-dimensional case"}

Let us first assume that $H^0(T^*\CZ|_Z)$ is an object of $\Coh(Z)$. In this case we will prove that $^{cl}\CZ$
is elementary. 

\medskip

By \eqref{e:cotangent as proj lim}, $H^0(T^*\CZ|_Z)$ is locally free of finite rank over $Z$. 

\medskip

As in \cite{BD}, top of page 329, it suffices to show that the system of coherent sheaves 
$n\mapsto \CI_n|_Z$ stabilizes. However, \eqref{e:4 term seq} implies that 
$\CI|_Z$ is in fact an object of $\Coh(Z)^\heartsuit$. Since the maps $\CI_{n+1}|_Z\to \CI_{n}|_Z$ 
are surjective, this implies the stabilization statement.

\sssec{Step 2: choosing generators for the ideal}

For a general $\CZ_{cl}$ we will construct an object $Q\in \on{Pro}(\Coh(Z)^{\heartsuit})$ of the form
$$Q\simeq \underset{m\in \BN}{``lim"}\, Q_m,$$ where $Q_m$ are locally free sheaves
on $Z$ of finite rank, and the maps $Q_{m+1}\to Q_m$ are surjective, and a map
$f:Q\to \CI|_Z$, such that 
the composition
$$Q\to \CI|_Z\to H^0(T^*\CZ|_Z)=:P$$
in injective and has the property that $\on{coker}(Q\to P)$ belongs to $\Coh(Z)$
and is locally free. 

\medskip

Consider again \eqref{e:4 term seq}. Let $k$ be an index such that the map
$P\to H^0(T^*Z)$ factors through a map $P_k\twoheadrightarrow H^0(T^*Z)$, and let
$Q:=\on{ker}(P\to P_k)$. Set 
$$R:=\CI|_Z\underset{P}\times Q.$$
By construction, the map $R\to Q$ is surjective, i.e., 
we have the following short exact sequence in $\on{Pro}(\Coh(Z)^{\heartsuit})$:
$$0\to H^1(T^*Z)\to R \to Q\to 0.$$
Since $Q$ is pro-projective and the category of indices is $\BN$, the map $R\to Q$ admits
a right inverse, which gives rise to the desired map $f:Q\to R\to \CI|_Z$.

\sssec{Step 3} We shall now use the above pair $(Q,f:Q\to \CI|_Z)$ to construct the desired
family of sub-indschemes of $\CZ_{cl}$, each being as in Step 1. 

\medskip

For every $n$ consider the object 
$$\CI|_{Z_n}:=\underset{n'\geq n}{``lim"}\, \CI_{n'}|_{Z_n}\in \on{Pro}(\Coh(Z_n)^\heartsuit).$$
We can extend the locally free sheaves $Q_m$ to a compatible family of locally free finite rank
coherent sheaves $n\mapsto Q_m|_{Z_n}$ and a compatible family of maps
$$f|_{Z_n}:Q|_{Z_n}:=\underset{m}{``lim"}\, Q_m|_{Z_n}\to \CI|_{Z_n}.$$

Let $Q^m|_{Z_n}$ be the kernel of the map $Q|_{Z_n}\to Q_m|_{Z_n}$. For each $m$ we define
the closed sub-scheme $Z^m_n$ of $Z_n$ to be given by the ideal $\CJ^m_n$ equal to the image of
$$Q^m|_{Z_n}\hookrightarrow Q|_{Z_n}\overset{f|_{Z_n}}\longrightarrow \CI|_{Z_n}\to \CI_n.$$
We set
$$\CZ^m_{cl}:=\underset{n}{colim}\, Z^m_n.$$
It is clear from the construction that 
$$\CZ_{cl}\simeq \underset{m}{colim}\, \CZ^m_{cl}.$$

\sssec{Step 4}

It remains to show that each $\CZ^m_{cl}$ is a classical indscheme satisfying
the assumptions of Step 1. Let 
$$\CZ^m:={}^L\!\on{LKE}_{(\affSch)^{\on{op}}\hookrightarrow (\affdgSch)^{\on{op}}}(\CZ^m_{cl}).$$
Using \corref{c:estimate on cl cotangent formal}, it suffices to show that $H^0(T^*\CZ^m|_Z)\in \on{Pro}(\Coh(Z)^\heartsuit)$
is locally free of finite rank and that $H^{-1}(T^*\CZ^m|_Z)=0$.

\medskip

For that it suffices to show that the map
$$\underset{n}{``lim"}\, \CJ^m_n|_Z\to H^0(T^*\CZ|_Z)$$ is injective and that the quotient belongs
to $\Coh(Z)^\heartsuit$ and is locally free of finite rank. However, by construction, we have a surjective
map
$$Q^m:=\on{ker}(Q\to Q_m)\twoheadrightarrow  \underset{n}{``lim"}\, \CJ^m_n|_Z,$$
and the required properties follow from the corresponding properties of the map $f$.

\qed(\propref{p:descr of formal completion modified})

\section{$\QCoh$ and $\IndCoh$ on formally smooth indschemes}

\ssec{The main result}

The goal of this section is to prove the following result, originally established
by J.~Lurie using a different method:

\begin{thm} \label{t:QCoh main}
Let $\CX$ be a formally smooth DG indscheme, which is weakly $\aleph_0$ and locally almost
of finite type. Then the functor
$$\Upsilon_\CX:=-\underset{\CO_\CX}\otimes \omega_\CX:\QCoh(\CX)\to \IndCoh(X)$$
is an equivalence.
\end{thm}

\medskip

Note that by \thmref{t:classical vs derived ft}, the DG indscheme $\CX$ is $0$-coconnected,
so $\QCoh(\CX)$ is equivalent to 
$$\QCoh({}^L\!\tau^{cl}(\CX))\simeq \QCoh(\tau^{cl}(\CX)),$$ 
where the latter equivalence is because of \cite[Corollary 1.3.7]{QCoh}.

\medskip

We also note the following corollary of \thmref{t:QCoh main} and \corref{c:IndCoh compactly generated}:

\begin{cor}
Let $\CX$ be a formally smooth DG indscheme, which is weakly $\aleph_0$ and locally almost
of finite type. Then the category $\QCoh(\CX)$ is compactly generated.
\end{cor}

\medskip

The rest of this section is devoted to the proof of \thmref{t:QCoh main}.

\ssec{Reduction to the ``standard" case}

\sssec{}
Write $^{cl}\CX$ as $\underset{\alpha}{colim}\, X_\alpha$, where $X_\alpha$ are
classical schemes locally of finite type. Let $\CX_\alpha:=\CX^\wedge_{^{red}\!X_\alpha}$ be the formal
completion of $\CX$ along $^{red}\!X_\alpha$. Each $\CX_\alpha$ is DG 
indscheme satisfying the assumptions of the theorem. 

\medskip

Since
$$\underset{\alpha}{colim}\, \CX_\alpha\to \CX,$$
is an isomorphism (the above colimit taken in $\on{PreStk}$), the functors 
$$\QCoh(\CX)\to \underset{\alpha}{lim}\, \QCoh(\CX_\alpha) \text{ and }
\IndCoh(\CX)\to \underset{\alpha}{lim}\, \IndCoh(\CX_\alpha)$$
are both equivalences, where the first limit is taken with respect to the *-pullback
functors, and the second limit is taken with respect to the !-pullback functors.

\medskip

Since for $\alpha\to \beta$ the diagrams
$$
\CD
\QCoh(\CX)  @>>>   \QCoh(\CX_\beta)  @>>> \QCoh(\CX_\alpha)  \\
@V{\Upsilon_\CX}VV   @VV{\Upsilon_{\CX_\beta}}V    @VV{\Upsilon_{\CX_\alpha}}V  \\
\IndCoh(\CX)  @>>>   \IndCoh(\CX_\beta)  @>>> \IndCoh(\CX_{\alpha}) 
\endCD
$$ 
are commutative, it suffices to show that each of the functors
\begin{equation} \label{e:IndCoh vs QCoh}
\Upsilon_{\CX_\alpha}:\QCoh(\CX_\alpha)\to \IndCoh(\CX_\alpha)
\end{equation}
is an equivalence. 

\medskip

So, from now on we will assume that $\CX$ is formal. 

\sssec{}

By \cite[Proposition 4.2.1]{IndCoh}, the functor
$\IndCoh$ satisfies Zariski descent. So, the statement about 
equivalence in \eqref{e:IndCoh vs QCoh} is local in the Zariski topology. Therefore, we can assume that $\CX$ is affine, 
and thus apply \propref{p:descr of formal completion modified}. 

\sssec{}

Since the statement of the theorem survives taking retracts and colimits of DG indschemes, 
we can assume that $\CX$ is \emph{elementary}  (see \secref{sss:elementary}). 
The proof that 
the functor $\Upsilon_\CX$ is an equivalence in this case is a rather straightforward 
but somewhat tedious verification, which we shall presently perform. 

\ssec{The functor $\Upsilon^\vee_\CX$}

\sssec{} \label{sss:self duality on weird}

Recall the notation of \secref{sss:Jacobi ft}. Let us denote by $\CY$ the DG indscheme $\BA^{n,m}$. Let $f:\CX\hookrightarrow \CY$
denote the corresponding closed embedding.

\medskip

Since
$$\CX\simeq 0 \underset{\BA^k}\times \CY,$$
by \cite[Proposition 3.2.1]{QCoh}, we have:
\begin{equation} \label{e:category as ten product}
\QCoh(\CX)\simeq \Vect \underset{\QCoh(\BA^k)}\otimes \QCoh(\CY).
\end{equation}

\medskip

By \secref{ss:self duality formal compl}, the indscheme $\CY$ is quasi-perfect
(i.e., the category $\QCoh(\CY)$ is compactly generated and its compact objects are perfect). 
We claim:

\begin{lem}  \label{l:X tilde qp}
The DG indscheme $\CX$ is quasi-perfect.
\end{lem}

\begin{proof}
From \eqref{e:category as ten product} we obtain that a generating set of compact
objects of $\QCoh(\CX)$ is obtained as the essential image under the functor $f^*$
of compact objects of $\QCoh(\CY)$. The assertion of the lemma follows from the
fact that the pullback functor preserves perfectness.
\end{proof}

\medskip

In particular, from \lemref{l:X tilde qp} we obtain a self-duality equivalence
\begin{equation} \label{e:duality X tilde}
\bD_\CX^{\on{naive}}:(\QCoh(\CX))^\vee\simeq \QCoh(\CX),
\end{equation}
Using also 
$$\bD_\CX^{\on{Serre}}:(\IndCoh(\CX))^\vee\simeq \IndCoh(\CX),$$
we can consider the functor
$$\Upsilon^\vee_\CX:\IndCoh(\CX)\to \QCoh(\CX),$$
dual to $\Upsilon_\CX$.

\medskip

Showing that $\Upsilon_\CX$ is an equivalence is equivalent to showing
that $\Upsilon^\vee_\CX$ is an equivalence.

\sssec{}

Let $^\sim\CY:={}^{\sim}\!\BA^{n,m}$ be the \emph{scheme} introduced in the course 
of the proof of \propref{p:Jacobi classical ft}. Let $^\sim\CX$ be the DG \emph{scheme} 
$$0 \underset{\BA^k}\times ^\sim\CY.$$
As we saw in \secref{sss:Jacobi ft}, the DG scheme $\CX$ is also $0$-coconnective.

\medskip

The formal (DG) scheme $\CX$ is obtained as a formal completion of $^\sim\CX$
along a Zariski-closed subset $X$.  Let
$\wh{i}:\CX\to {}^\sim\CX$ denote the corresponding map, and let 
$$^\sim\CX-X=:U_\CX\overset{j}\hookrightarrow {}^\sim\CX$$
be the complementary open embedding. 

\medskip

Since the DG scheme $^\sim\CX$ is Noetherian, 
the category $\IndCoh(\CX)$ and the functor
$$\Psi_{\CX}:\IndCoh(\CX)\to \QCoh(\CX)$$
are well-defined (see \cite[Sect. 1.1]{IndCoh}).

\sssec{}

We will deduce the fact that $\Upsilon^\vee_\CX$ is an equivalence from the following statement:

\begin{prop}  \label{p:Psi on compl and not}
The diagram of functors
$$
\CD
\IndCoh(\CX)     @>{\wh{i}^{\IndCoh}_*}>> \IndCoh({}^\sim\CX)  \\
@V{\Upsilon^\vee_{\CX}}VV   @V{\Psi_{^\sim\CX}}VV  \\
\QCoh(\CX)   @>{\wh{i}_?}>> \QCoh({}^\sim\CX)
\endCD
$$
commutes.
\end{prop}

\begin{rem}
This proposition does not formally follow from the commutativity of
\eqref{e:qcoh and indcoh compat 2}, because the latter relied on the finite type assumption
of the ambient DG scheme (in our case the ambient scheme is $^\sim\CX$, and it is
not of finite type). 
\end{rem}

\sssec{}

Let us assume this proposition for a moment and finish the proof of the fact that
$\Upsilon^\vee_{\CX}$ is an equivalence (and thereby of \thmref{t:QCoh main}).

\begin{proof}

We have a commutative diagram
$$
\CD
\IndCoh({}^\sim\CX)   @>{j^{\IndCoh,*}}>>  \IndCoh(U_\CX)  \\
@V{\Psi_{^\sim\CX}}VV    @VV{\Psi_{U_\CX}}V   \\
\QCoh({}^\sim\CX)   @>{j^*}>>  \QCoh(U_\CX) 
\endCD
$$
(see \cite[Proposition 3.5.4]{IndCoh}). 

\medskip

By \propref{p:QCoh of compl}, the category $\QCoh(\CX)$ identifies with the kernel of
the functor
$$j^*:\QCoh({}^\sim\CX) \to \QCoh(U_\CX).$$
By \propref{p:localization for IndCoh}, the category $\IndCoh(\CX)$ identifies with the kernel
of the functor
$$j^{\IndCoh,*}:\IndCoh({}^\sim\CX) \to \IndCoh(U_\CX).$$

\medskip

\noindent (We remark that in \propref{p:localization for IndCoh} it was assumed that the ambient scheme
is almost of finite type over the field, but the proof applies in the case when it is only assumed
Noetherian, which is the case for $^\sim\CX$.)

\medskip

The required assertion follows from the fact that the functors
$\Psi_{^\sim\CX}$ and $\Psi_{U_\CX}$ are equivalences, since the
corresponding DG schemes are $0$-coconnective and the underlying
classical schemes are regular (see \cite[Lemma 1.1.6]{IndCoh}).

\end{proof}

\ssec{Proof of \propref{p:Psi on compl and not}}

We shall compare the functors
$$\QCoh({}^\sim\CX)\otimes \IndCoh(\CX) \rightrightarrows \Vect$$
that arise from the two circuits of the diagram and the duality pairing
$$\langle-,-\rangle_{\QCoh({}^\sim\CX)}:\QCoh({}^\sim\CX)\otimes \QCoh({}^\sim\CX)\to \Vect,$$
corresponding to the functor $\bD_\CX^{\on{naive}}$ of \eqref{e:duality X tilde}.



\medskip

For $\CF\in \QCoh({}^\sim\CX)$ and $\CF'\in \IndCoh(\CX)$ we have:
\begin{equation} \label{e:pairing expl}
\langle \CF,\wh{i}_?\circ \Upsilon^\vee_\CX(\CF_1)\rangle_{\QCoh({}^\sim\CX)}\simeq
\langle \wh{i}^*(\CF),\Upsilon^\vee_\CX(\CF_1)\rangle_{\QCoh(\CX)}\simeq 
\Gamma^{\IndCoh}(\CX,\wh{i}^*(\CF)\underset{\CO_{\CX}}\otimes \CF_1),
\end{equation}
where the first isomorphism follows from \corref{c:dual of ?}, and the second one from 
\secref{sss:when Psi}.

\medskip

The description of the functor $\wh{i}^{\IndCoh}_*$ given in \secref{sss:lower shriek}
(which is valid for all Noetherian schemes) implies that we have a canonical isomorphism
$$\Gamma^{\IndCoh}(\CX,-)\simeq \Gamma^{\IndCoh}({}^\sim\CX,-)\circ \wh{i}^{\IndCoh}_*.$$
Hence, the expression in \eqref{e:pairing expl} can be further rewritten as
$$\Gamma^{\IndCoh}\left({}^\sim\CX,\wh{i}^{\IndCoh}_*(\wh{i}^*(\CF)\underset{\CO_{\CX}}\otimes \CF_1)\right),$$
which by the projection formula is canonically isomorphic to
$$\Gamma^{\IndCoh}\left({}^\sim\CX,\CF\underset{\CO_{^\sim\CX}}\otimes \wh{i}^{\IndCoh}_*(\CF_1)\right).$$

\medskip

Now,
$$\langle \CF, \Psi_{^\sim\CX}\circ \wh{i}^{\IndCoh}_*(\CF_1)\rangle_{\QCoh({}^\sim\CX)}\simeq
\Gamma^{\IndCoh}({}^\sim\CX,\CF\underset{\CO_{^\sim\CX}}\otimes \wh{i}^{\IndCoh}_*(\CF_1)),$$
as required.

\qed


\begin{thebibliography}{99}

\bibitem[BD]{BD} A.~Beilinson and V.~Drinfeld, {\it Quantization of Hitchin's integrable system and Hecke
eigensheaves}, available at http://math.uchicago.edu/~mitya/langlands.html.



\bibitem[Dr]{Dr} V.~Drinfeld, {\it Infinite-dimensional vector bundles}, in: Algebraic Geometry and Number Theory, Progr. 
Math. {\bf 253} (2006), 264--304, Birkhauser Boston.

\bibitem[EF]{EF} D.~Eisenbud and E.~Frenkel, appendix to the paper: M.~Mustata,
{\it Jet schemes of locally complete intersection canonical singularities}, Invent. Math. {\bf 145},
397--424 (2001).

\bibitem[GL:DG]{DG} D.~Gaitsgory,  {\it Notes on Geometric Langlands: Generalities on DG categories}, \newline
available at http://www.math.harvard.edu/$~$gaitsgde/GL/.

\bibitem[GL:Stacks]{Stacks} D.~Gaitsgory,  {\it Notes on Geometric Langlands: Stacks}, \newline
available at http://www.math.harvard.edu/$~$gaitsgde/GL/.

\bibitem[GL:QCoh]{QCoh} D.~Gaitsgory,  {\it Notes on Geometric Langlands: Quasi-coherent sheaves on stacks}, \newline
available at http://www.math.harvard.edu/$~$gaitsgde/GL/.

\bibitem[GL:IndCoh]{IndCoh} D.~Gaitsgory,  {\it Ind-coherent sheaves}, arXiv:1105.4857.

\bibitem[GR]{GR} D.~Gaitsgory and N.~Rozenblyum,
{\it A study in derived algebraic geometry}, in preparation, 
preliminary version will gradually become available at http://www.math.harvard.edu/~gaitsgde/GL/.





\bibitem[Lu0]{Lu0} J.~Lurie, {\it Higher Topos Theory}, Princeton Univ. Press (2009). 

\bibitem[Lu1]{Lu1} J.~Lurie, {\it DAG-VIII}, available at  http://www.math.harvard.edu/$~$lurie.



\bibitem[RG]{RG} M.~Raynaud and L.~Gruson, {\it Crit\`eres de platitude et de projectivit\'e}, Invent. Math.,
{\bf 13}, 1--89 (1971).

\end{thebibliography}
\end{document}